\documentclass[12pt]{ttuthes2007}

\usepackage{amsmath,amssymb,latexsym,amsthm}
\usepackage{lscape}
\usepackage{graphicx}
\usepackage{fancyhdr}
\usepackage[usenames]{color}

\newcommand{\arm}[2][i]{\ensuremath{{_nA}_{#2}^{#1}}}
\newcommand{\oneton}[2][1]{\ensuremath{\{{#1},\linebreak[0]\ldots,\linebreak[0]{#2}\}}}

\newcommand{\armzeroi}[1][i]{\ensuremath{\langle v_0,v_{#1} \rangle}}

\newcommand{\armzeronminusone}{\ensuremath{\langle v_0,v_{n-1},v_{n+1} \rangle}}

\newcommand{\armzeron}{\ensuremath{\langle v_0,v_n,v_{n+2} \rangle}}

\newcommand{\armonei}[1][i]{\ensuremath{\langle v_0,\linebreak[0]u_{#1},\linebreak[0]u_{(n-3)+#1},\linebreak[0]u_{2(n-3)+#1},\linebreak[0]\ldots,\linebreak[0]u_{2(n-1)(n-3)+#1},\linebreak[0]v_{#1}
\rangle}}

\newcommand{\armonenminustwo}{\ensuremath{\langle v_0,\linebreak[0]u_{(2n-1)(n-3)+2(n-1)+1},\linebreak[0]v_{n-2} \rangle}}

\newcommand{\armonenminusone}{\ensuremath{\langle v_0,\linebreak[0]u_{(2n-1)(n-3)+2(n-1)+2},\linebreak[0]u_{(2n-1)(n-3)+2(n-1)+3},\linebreak[0]v_{n-1},\linebreak[0]v_{n+1}
\rangle}}

\newcommand{\armonen}{\ensuremath{\langle v_0,\linebreak[0]u_{(2n-1)(n-3)+1},\linebreak[0]\ldots,\linebreak[0]u_{(2n-1)(n-3)+2(n-1)},\linebreak[0]v_n,\linebreak[0]v_{n+2} \rangle}}

\newcommand{\phidef}[3]{\ensuremath{#1_{#2}}&\vline&\ensuremath{v_{#3}}}

\newcommand{\invseq}[3][n]{\ensuremath{\{{_{#1}{#2}}_i,{_{#1}{#3}}_i^j\}_{i=0, j\ge i}^{\infty}}}

\newcommand{\dinvseq}[3][n]{\ensuremath{\{D({_{#1}{#3}}_0^i,{_{#1}{#2}}_i),\linebreak[0]d[{_{#1}{#3}}_0^i,{_{#1}{#3}}_i^j]\}_{i=0, j\ge i}^{\infty}}}

\newcommand{\branch}{\ensuremath{(v_0,\linebreak[0]v_0,\linebreak[0]v_0,\ldots)}}

\newcommand{\X}[2][n]{\ensuremath{{_{#1}X}_{#2}}}

\newcommand{\K}{\ensuremath{{_nK}}}

\newcommand{\nphi}[1][n]{\ensuremath{{_{#1}\phi}}}

\newcommand{\nPhi}[3][n]{\ensuremath{{_{#1}\Phi}_{#3}^{#2}}}

\newcommand{\nS}{\ensuremath{{_nS}}}

\newcommand{\nSdef}[3]{\ensuremath{{_nS}(v_{#1})=\linebreak[0]\{\langle
v_{#2},\linebreak[0]v_{#3} \rangle\}}}

\newcommand{\edge}[4]{\ensuremath{\langle {#1}_{#2},\linebreak[0]{#3}_{#4}
\rangle}}

\newcommand{\brkedge}[4]{\ensuremath{\langle {#1}_{#2},}\\&&\ensuremath{{#3}_{#4}
\rangle}}

\newcommand{\bbrkedge}[4]{\ensuremath{\langle {#1}_{#2},}\\\indent\indent\ensuremath{{#3}_{#4}
\rangle}}

\newcommand{\bbbbrkedge}[4]{\ensuremath{\langle {#1}_{#2},\linebreak{#3}_{#4}
\rangle}}

\newcommand{\subedge}[5]{\ensuremath{(v_{#1},\linebreak[0]\edge{#2}{#3}{#4}{#5},\linebreak[0]\X{1})}}

\newcommand{\imphiedge}[4]{\ensuremath{\nphi(\edge{#1}{#2}{#3}{#4})}}

\newcommand{\imphisubedge}[5]{\ensuremath{\nphi(\subedge{#1}{#2}{#3}{#4}{#5})}}

\newcommand{\imnS}[1]{\ensuremath{\nS(v_{#1})}}

\newcommand{\imphi}[2]{\ensuremath{\nphi(\linebreak[0]{#1}_{#2})}}

\newcommand{\imnSphi}[2]{\ensuremath{\nS(\linebreak[0]\imphi{#1}{#2})}}

\newcommand{\dualver}[2]{\ensuremath{{#1}_{#2}^*}}

\newcommand{\dnphi}{\ensuremath{d[\nphi]}}

\newcommand{\imdnphi}[2]{\ensuremath{b_{#1}}&\vline&\ensuremath{a_{#2}}}

\newcommand{\imdnphiinv}[4]{\ensuremath{\dnphi^{-1}(\edge{a}{#1}{a}{#2})}&=&\ensuremath{\edge{b}{#3}{b}{#4}}}

\newcommand{\Y}[2][n]{\ensuremath{{_{#1}Y}_{#2}}}

\newcommand{\Barm}[2][i]{\ensuremath{{_nB}_{#2}^{#1}}}

\newcommand{\Barmonei}[1][i]{\ensuremath{\langle v_0,\linebreak[0]w_{#1},\linebreak[0]w_{n-3+{#1}},\linebreak[0]w_{2(n-3)+{#1}},
\linebreak[0]\ldots,\linebreak[0]w_{(n-2)(n-3)+{#1}},\linebreak[0]v_{#1}
\rangle}}

\newcommand{\Barmonenminustwo}{\ensuremath{\langle v_0,\linebreak[0]v_{n-2} \rangle=\linebreak[0]\arm[n-2]{0}}}

\newcommand{\Barmonenminusone}{\ensuremath{\langle v_0,\linebreak[0]v_{n-1},\linebreak[0]v_{n+1} \rangle=\linebreak[0]\arm[n-1]{0}}}

\newcommand{\Barmonen}{\ensuremath{\langle v_0,\linebreak[0]w_{(n-1)(n-3)+1},\linebreak[0]\ldots,\linebreak[0]w_{(n-1)(n-3)+n-2},\linebreak[0]v_n,\linebreak[0]v_{n+2} \rangle}}

\newcommand{\nlambda}{\ensuremath{{_n\lambda}}}

\newcommand{\llambdadef}[3]{\ensuremath{#1_{#2}}&\vline&\ensuremath{b_{#3}}}

\newcommand{\imPhibar}[4]{\ensuremath{\overline{\nPhi{#1}{#2}}(\linebreak[0]\arm[#3]{#4})}}

\newcommand{\xseq}[4]{\ensuremath{\langle
x_{#2}^{#1},\linebreak[0]x_{#3}^{#1},\linebreak[0]\ldots,\linebreak[0]x_{#4}^{#1}
\rangle}}

\newcommand{\xxxseq}[4]{\ensuremath{\langle
x_{#2}^{#1},\linebreak
x_{#3}^{#1},\linebreak[0]\ldots,\linebreak[0]x_{#4}^{#1} \rangle}}

\newcommand{\xxxxxseq}[4]{\ensuremath{\langle
x_{#2}^{#1},\linebreak[0]x_{#3}^{#1},\linebreak[0]\hspace{-.03cm}\ldots\hspace{-.03cm},\linebreak[0]x_{#4}^{#1}
\linebreak[0]\rangle}}

\newcommand{\xxxxxxseq}[4]{\ensuremath{\langle
x_{#2}^{#1},\linebreak[0]x_{#3}^{#1},\linebreak[0]\ldots,\linebreak[0]x_{#4}^{#1}
\linebreak[0]\rangle}}

\newcommand{\xxxxxxxseq}[4]{\ensuremath{\langle
\linebreak
x_{#2}^{#1},\linebreak[0]x_{#3}^{#1},\linebreak[0]\ldots,\linebreak[0]x_{#4}^{#1}
\linebreak[0]\rangle}}

\newcommand{\xxxxxxxxseq}[4]{\ensuremath{\langle
x_{#2}^{#1},\linebreak[0]x_{#3}^{#1},\linebreak[0]\hspace{.1cm}\ldots\hspace{.1cm},\linebreak[0]x_{#4}^{#1}
\rangle}}

\newcommand{\subdivedge}[3]{\ensuremath{(\linebreak[0]\edge{v}{#1}{v}{#2},\linebreak[0]\X{#3})}}

\newcommand{\ssssubdivedge}[3]{\ensuremath{(\linebreak[0]\bbbbrkedge{v}{#1}{v}{#2},\linebreak[0]\X{#3})}}

\newcommand{\imedgePhibar}[5]{\ensuremath{\overline{\nPhi{#1}{#2}}(\linebreak[0]\subdivedge{#3}{#4}{#5})}}

\newcommand{\vseqonebase}[4][]{\ensuremath{\langle
v_0,\linebreak[0]v_{n-1},\linebreak[0]v_0,\linebreak[0]v_{n-2},\linebreak[0]v_0,\linebreak[0]v_1,\linebreak[0]v_0,\linebreak[0]\underline{v_2,\linebreak[0]v_0},\linebreak[0]\ldots,\linebreak[0]\underline{v_{n-3},\linebreak[0]v_0},#4
\linebreak[0]v_n,#1\linebreak[0]v_0,\linebreak[0]v_{n-3},\linebreak[0]v_0,#3\linebreak[0]\underline{v_{n-4},\linebreak[0]v_0},\linebreak[0]\ldots,\linebreak[0]\underline{v_1,\linebreak[0]v_0},#2\linebreak[0]v_{n-2},\linebreak[0]v_0,\linebreak[0]v_{n-1},\linebreak[0]v_0
\rangle}}

\newcommand{\vseqone}[1][]{\ensuremath{\vseqonebase[#1]{}{}{}}}

\newcommand{\vvvseqone}[4][]{\ensuremath{\langle
v_0,\linebreak[0]v_{n-1},\linebreak[0]v_0,\linebreak[0]v_{n-2},\linebreak[0]v_0,\linebreak[0]v_1,\linebreak[0]v_0,\hspace{.05cm}\linebreak[0]\underline{v_2,\linebreak[0]v_0}\hspace{.05cm},\hspace{.08cm}\linebreak[0]\ldots\hspace{.08cm},\linebreak[0]\underline{v_{n-3},\linebreak[0]v_0}\hspace{.05cm},#4
\linebreak[0]v_n,#1\linebreak[0]v_0,\linebreak[0]v_{n-3},\linebreak[0]v_0,#3\hspace{.05cm}\linebreak[0]\underline{v_{n-4},\linebreak[0]v_0}\hspace{.05cm},\hspace{.08cm}\linebreak[0]\ldots\hspace{.08cm},\linebreak\underline{v_1,\linebreak[0]v_0},#2\linebreak[0]v_{n-2},\linebreak[0]v_0,\linebreak[0]v_{n-1},\linebreak[0]v_0
\rangle}}

\newcommand{\vseqonea}[1][]{\ensuremath{\vseqonebase[#1]{}{}{}}}

\newcommand{\vvseqonea}[4][]{\ensuremath{\langle
v_0,\linebreak[0]v_{n-1},\linebreak[0]v_0,\linebreak[0]v_{n-2},\linebreak[0]v_0,\linebreak[0]v_1,\linebreak[0]v_0,\linebreak[0]\underline{v_2,\linebreak[0]v_0},\linebreak[0]\ldots,\linebreak[0]\underline{v_{n-3},\linebreak[0]v_0},#4
\linebreak[0]v_n,#1\linebreak[0]v_0,\linebreak[0]v_{n-3},\linebreak[0]v_0,#3\linebreak[0]\underline{v_{n-4},\linebreak[0]v_0},\linebreak\ldots,\linebreak[0]\underline{v_1,\linebreak[0]v_0},#2\linebreak[0]v_{n-2},\linebreak[0]v_0,\linebreak[0]v_{n-1},\linebreak[0]v_0
\rangle}}

\newcommand{\vvvseqonea}[4][]{\ensuremath{\langle
v_0,\linebreak[0]v_{n-1},\linebreak[0]v_0,\linebreak[0]v_{n-2},\linebreak[0]v_0,\linebreak[0]v_1,\linebreak[0]v_0,\linebreak[0]\underline{v_2,\linebreak[0]v_0},\linebreak[0]\ldots,\linebreak[0]\underline{v_{n-3},\linebreak[0]v_0},#4
\linebreak[0]v_n,#1\linebreak[0]v_0,\linebreak[0]v_{n-3},\linebreak[0]v_0,#3\linebreak[0]\underline{v_{n-4},}\linebreak[0]\underline{v_0},\linebreak[0]\ldots,\linebreak[0]\underline{v_1,\linebreak[0]v_0},#2\linebreak[0]v_{n-2},\linebreak[0]v_0,\linebreak[0]v_{n-1},\linebreak[0]v_0
\rangle}}

\newcommand{\vvvvseqonea}[4][]{\ensuremath{\langle
v_0,\linebreak[0]v_{n-1},\linebreak[0]v_0,\linebreak[0]v_{n-2},\linebreak[0]v_0,\linebreak[0]v_1,\linebreak[0]v_0,\linebreak[0]\underline{v_2,\linebreak[0]v_0},\linebreak[0]\hspace{.08cm}\ldots\hspace{.08cm},\linebreak[0]\underline{v_{n-3},\linebreak[0]v_0},#4
\linebreak[0]v_n,#1\linebreak[0]v_0,\linebreak[0]v_{n-3},\linebreak[0]v_0,#3\linebreak[0]\underline{v_{n-4},\linebreak[0]v_0},\linebreak[0]\hspace{.08cm}\ldots\hspace{.08cm},\linebreak[0]\underline{v_1,\linebreak[0]v_0},#2\linebreak[0]v_{n-2},\linebreak[0]v_0,\linebreak[0]v_{n-1},\linebreak[0]v_0
\rangle}}

\newcommand{\vseqoneb}[1][]{\ensuremath{\vseqonebase[#1]{}{}{}}}

\newcommand{\vvvseqoneb}[4][]{\ensuremath{\langle
v_0,\linebreak[0]v_{n-1},\linebreak[0]v_0,\linebreak[0]v_{n-2},\linebreak[0]v_0,\linebreak[0]v_1,\linebreak[0]v_0,\linebreak[0]\underline{v_2,\linebreak[0]v_0},\hspace{.08cm}\linebreak[0]\ldots\hspace{.08cm},\linebreak[0]\underline{v_{n-3},\linebreak[0]v_0},#4
\linebreak[0]v_n,#1\linebreak[0]v_0,\linebreak[0]v_{n-3},\linebreak[0]v_0,#3\linebreak[0]\underline{v_{n-4},\linebreak[0]v_0},\hspace{.08cm}\linebreak[0]\ldots\hspace{.08cm},\linebreak[0]\underline{v_1,\linebreak[0]v_0},#2\linebreak[0]v_{n-2},\linebreak[0]v_0,\linebreak[0]v_{n-1},\linebreak[0]v_0
\rangle}}

\newcommand{\vseqonec}[1][]{\ensuremath{\vseqonebase[#1]{}{}{}}}

\newcommand{\vvseqonec}[4][]{\ensuremath{\langle
v_0,\linebreak[0]v_{n-1},\linebreak[0]v_0,\linebreak[0]v_{n-2},\linebreak[0]v_0,\linebreak[0]v_1,\linebreak[0]v_0,\linebreak[0]\underline{v_2,\linebreak[0]v_0},\linebreak[0]\ldots,\linebreak[0]\underline{v_{n-3},\linebreak[0]v_0},#4
\linebreak[0]v_n,#1\linebreak[0]v_0,\linebreak[0]v_{n-3},\linebreak[0]v_0,#3\linebreak[0]\underline{v_{n-4},\linebreak[0]v_0},\linebreak[0]\ldots,\linebreak[0]\underline{v_1,}\linebreak[0]\underline{v_0},#2\linebreak[0]v_{n-2},\linebreak[0]v_0,\linebreak[0]v_{n-1},\linebreak[0]v_0
\rangle}}

\newcommand{\impj}[2][j]{\ensuremath{p^{#1}({#2})}}

\newcommand{\vseqpbase}[4][j]{\ensuremath{\langle
v_0,\linebreak[0]v_{n-1},\linebreak[0]v_0,\linebreak[0]v_{\impj[#1]{1}},\linebreak[0]v_0,#3\linebreak[0]\underline{v_{\impj[#1]{2}},\linebreak[0]v_0},\linebreak[0]\ldots,\linebreak[0]\underline{v_{\impj[#1]{n-2}},\linebreak[0]v_0},
\linebreak[0]v_n,#4\linebreak[0]v_{n+2},\linebreak[0]v_n,\linebreak[0]v_0,\linebreak[0]v_{\impj[#1]{n-2}},\linebreak[0]v_0,#2\linebreak[0]\underline{v_{\impj[#1]{n-3}},\linebreak[0]v_0},\linebreak[0]\ldots,\linebreak[0]\underline{v_{\impj[#1]{1}},\linebreak[0]v_0},\linebreak[0]v_{n-1},\linebreak[0]v_0
\rangle}}

\newcommand{\vseqpa}[1][j]{\ensuremath{\vseqpbase[#1]{}{}{}}}

\newcommand{\vvvvseqpa}[4][j]{\ensuremath{\langle
v_0,\linebreak[0]v_{n-1},\linebreak[0]v_0,\linebreak[0]v_{\impj[#1]{1}},\linebreak[0]v_0,#3\linebreak[0]\underline{v_{\impj[#1]{2}},\linebreak[0]v_0},\linebreak[0]\ldots,\linebreak[0]\underline{v_{\impj[#1]{n-2}},\linebreak[0]v_0},
\linebreak[0]v_n,#4\linebreak[0]v_{n+2},\linebreak[0]v_n,\linebreak[0]v_0,\linebreak[0]v_{\impj[#1]{n-2}},\linebreak[0]v_0,#2\linebreak[0]\underline{v_{\impj[#1]{n-3}},}\linebreak\underline{v_0},\linebreak[0]\ldots,\linebreak[0]\underline{v_{\impj[#1]{1}},\linebreak[0]v_0},\linebreak[0]v_{n-1},\linebreak[0]v_0
\rangle}}

\newcommand{\vvvseqpb}[4][j]{\ensuremath{\langle
v_0,\linebreak[0]v_{n-1},\linebreak[0]v_0,\linebreak[0]v_{\impj[#1]{1}},\linebreak[0]v_0,#3\linebreak[0]\underline{v_{\impj[#1]{2}},\linebreak[0]v_0},\linebreak\ldots,\linebreak[0]\underline{v_{\impj[#1]{n-2}},\linebreak[0]v_0},
\linebreak[0]v_n,#4\linebreak[0]v_{n+2},\linebreak[0]v_n,\linebreak[0]v_0,\linebreak[0]v_{\impj[#1]{n-2}},\linebreak[0]v_0,#2\linebreak[0]\underline{v_{\impj[#1]{n-3}},\linebreak[0]v_0},\linebreak[0]\ldots,\linebreak[0]\underline{v_{\impj[#1]{1}},\linebreak[0]v_0},\linebreak[0]v_{n-1},\linebreak[0]v_0
\rangle}}

\newcommand{\vseqpc}[1][j]{\ensuremath{\vseqpbase[#1]{}{}{}}}

\newcommand{\vseqtwo}{\ensuremath{\langle
v_0,\linebreak[0]v_{n-1},\linebreak[0]v_{n+1},\linebreak[0]v_{n-1},\linebreak[0]v_0
\rangle}}

\newcommand{\vseqthree}{\ensuremath{\langle
v_0,\linebreak[0]v_{n-1},\linebreak[0]v_0,\linebreak[0]v_n,\linebreak[0]v_0,\linebreak[0]v_{n-1},\linebreak[0]v_0
\rangle}}

\newcommand{\vvvvseqthree}{\ensuremath{\langle
v_0,\linebreak[0]v_{n-1},\linebreak[0]v_0,\linebreak
v_n,\linebreak[0]v_0,\linebreak[0]v_{n-1},\linebreak[0]v_0
\rangle}}

\newcommand{\vseqfour}{\ensuremath{\langle
v_0,\linebreak[0]v_{n-1},\linebreak[0]v_0,\linebreak[0]v_n,\linebreak[0]v_{n+2},\linebreak[0]v_n,\linebreak[0]v_0,\linebreak[0]v_{n-1},\linebreak[0]v_0
\rangle}}

\newcommand{\vseqonehalfbase}[2][\ensuremath{,\linebreak[0]v_{n+2}}]{\ensuremath{\langle
v_0,\linebreak[0]v_{n-1},\linebreak[0]v_0,\linebreak[0]v_{n-2},\linebreak[0]v_0,\linebreak[0]v_1,\linebreak[0]v_0,#2\linebreak[0]\underline{v_2,\linebreak[0]v_0},\linebreak[0]\ldots,\linebreak[0]\underline{v_{n-3},\linebreak[0]v_0},
\linebreak[0]v_n #1 \rangle}}

\newcommand{\vseqonehalf}[1][\ensuremath{,\linebreak[0]v_{n+2}}]{\ensuremath{\vseqonehalfbase[#1]{}}}

\newcommand{\vvvvvseqonehalf}[2][\ensuremath{,\linebreak[0]v_{n+2}}]{\ensuremath{\langle
v_0,\linebreak[0]v_{n-1},\linebreak[0]v_0,\linebreak[0]v_{n-2},\linebreak[0]v_0,\linebreak[0]v_1,\linebreak[0]v_0,#2\linebreak[0]\underline{v_2,\linebreak[0]v_0},\linebreak[0]\hspace{.08cm}\ldots\hspace{.08cm},\linebreak[0]\underline{v_{n-3},\linebreak[0]v_0},
\linebreak[0]v_n #1 \rangle}}

\newcommand{\vvvvseqonehalf}[2][\ensuremath{,\linebreak[0]v_{n+2}}]{\ensuremath{\langle
v_0,\linebreak
v_{n-1},\linebreak[0]v_0,\linebreak[0]v_{n-2},\linebreak[0]v_0,\linebreak[0]v_1,\linebreak[0]v_0,#2\linebreak[0]\underline{v_2,\linebreak[0]v_0},\linebreak[0]\ldots,\linebreak[0]\underline{v_{n-3},\linebreak[0]v_0},
\linebreak[0]v_n #1 \rangle}}

\newcommand{\vvseqonehalf}[2][\ensuremath{,\linebreak[0]v_{n+2}}]{\ensuremath{\langle
v_0,\linebreak[0]v_{n-1},\linebreak[0]v_0,\linebreak
v_{n-2},\linebreak[0]v_0,\linebreak[0]v_1,\linebreak[0]v_0,#2\linebreak[0]\underline{v_2,\linebreak[0]v_0},\linebreak[0]\ldots,\linebreak[0]\underline{v_{n-3},\linebreak[0]v_0},
\linebreak[0]v_n #1 \rangle}}

\newcommand{\vvvseqonehalf}[2][\ensuremath{,\linebreak[0]v_{n+2}}]{\ensuremath{\langle
v_0,\linebreak[0]v_{n-1},\linebreak[0]v_0,\linebreak[0]v_{n-2},\linebreak[0]v_0,\linebreak[0]v_1,\linebreak[0]v_0,#2\linebreak[0]\underline{v_2,\linebreak[0]v_0},\linebreak[0]\ldots,\linebreak[0]\underline{v_{n-3},}\linebreak[0]\underline{v_0},
\linebreak[0]v_n #1 \rangle}}

\newcommand{\vseqonehalfa}[1][\ensuremath{,\linebreak[0]v_{n+2}}]{\ensuremath{\vseqonehalfbase[#1]{}}}

\newcommand{\vvseqonehalfa}[2][\ensuremath{,\linebreak[0]v_{n+2}}]{\ensuremath{\langle
v_0,\linebreak[0]v_{n-1},\linebreak[0]v_0,\linebreak[0]v_{n-2},\linebreak[0]v_0,\linebreak[0]v_1,\linebreak[0]v_0,#2\linebreak[0]\underline{v_2,}\linebreak\underline{v_0},\linebreak[0]\ldots,\linebreak[0]\underline{v_{n-3},\linebreak[0]v_0},
\linebreak[0]v_n #1 \rangle}}

\newcommand{\vseqphalfbase}[2][j]{\ensuremath{\langle
v_0,\linebreak[0]v_{n-1},\linebreak[0]v_0,\linebreak[0]v_{\impj[#1]{1}},\linebreak[0]v_0,\linebreak[0]\underline{v_{\impj[#1]{2}},\linebreak[0]v_0},#2\linebreak[0]\ldots,\linebreak[0]\underline{v_{\impj[#1]{n-2}},\linebreak[0]v_0},
\linebreak[0]v_n,\linebreak[0]v_{n+2} \rangle}}

\newcommand{\vseqphalfa}[1][j]{\ensuremath{\vseqphalfbase[#1]{}}}

\newcommand{\vseqtwohalf}{\ensuremath{\langle
v_0,\linebreak[0]v_{n-1},\linebreak[0]v_{n+1} \rangle}}

\newcommand{\vseqfourhalf}{\ensuremath{\langle
v_0,\linebreak[0]v_{n-1},\linebreak[0]v_0,\linebreak[0]v_n,\linebreak[0]v_{n+2}
\rangle}}

\newcommand{\eseq}[4]{\ensuremath{e_{#2}^{#1}\linebreak[0]\bigvee
\linebreak[0]e_{#3}^{#1}\linebreak[0]\bigvee\linebreak[0]\cdots\linebreak[0]\bigvee
\linebreak[0]e_{#4}^{#1}}}

\newcommand{\eseqq}[4]{\ensuremath{e_{#2}^{#1}\linebreak[0]\hspace{-.1cm}\bigvee
\linebreak[0]\hspace{-.1cm}e_{#3}^{#1}\linebreak[0]\bigvee\linebreak[0]\cdots\linebreak[0]\bigvee
\linebreak[0]e_{#4}^{#1}}}

\newcommand{\eeeseq}[4]{\ensuremath{e_{#2}^{#1}\linebreak[0]\bigvee
\linebreak[0]e_{#3}^{#1}\linebreak[0]\bigvee\linebreak[0]\cdots\linebreak[0]\bigvee
\linebreak e_{#4}^{#1}}}

\newcommand{\eeeeseq}[4]{\ensuremath{e_{#2}^{#1}\linebreak[0]\bigvee
\linebreak
e_{#3}^{#1}\linebreak[0]\bigvee\linebreak[0]\cdots\linebreak[0]\bigvee
\linebreak[0]e_{#4}^{#1}}}

\newcommand{\eseqa}[4]{\ensuremath{e_{#2}^{#1}\linebreak[0]\bigvee
\linebreak[0]e_{#3}^{#1}\linebreak[0]\bigvee\linebreak[0]\cdots\linebreak[0]\bigvee
\linebreak[0]e_{#4}^{#1}}}

\newcommand{\eeseqa}[4]{\ensuremath{e_{#2}^{#1}\linebreak[0]\bigvee
\linebreak[0]e_{#3}^{#1}\linebreak[0]\bigvee\linebreak[0]\cdots\linebreak[0]\bigvee
\linebreak e_{#4}^{#1}}}

\newcommand{\nt}[1][n]{\ensuremath{\tilde{#1}}}

\newcommand{\sseqzero}[2]{\ensuremath{\langle
s_0,\linebreak[0]s_1^{#1},\linebreak[0]\ldots,\linebreak[0]s_{#2}^{#1}
\rangle}}

\newcommand{\sseq}[5][]{\ensuremath{\langle
s_{#3}^{#2},\linebreak[0]s_{#4}^{#2},\linebreak[0]\ldots,\linebreak[0]s_{#5}^{#2}#1
\rangle}}

\newcommand{\vpjseq}[1][\ensuremath{,\linebreak[0]v_n,\linebreak[0]v_{n+2},\linebreak[0]v_n,\linebreak[0]v_0,\linebreak[0]v_{\impj{n-2}},\linebreak[0]v_0,\linebreak[0]\underline{v_{\impj{n-3}},\linebreak[0]v_0},\linebreak[0]\ldots,\linebreak[0]\underline{v_{\impj{1}},\linebreak[0]v_0},\linebreak[0]v_{n-1}}]
{\ensuremath{\langle
v_{\impj{1}},\linebreak[0]v_0,\linebreak[0]\underline{v_{\impj{2}},\linebreak[0]v_0},\linebreak[0]\ldots,\linebreak[0]\underline{v_{\impj{n-2}},\linebreak[0]v_0}#1
\rangle}}

\newcommand{\vvpjseq}[1][\ensuremath{,\linebreak[0]v_n,\linebreak[0]v_{n+2},\linebreak[0]v_n,\linebreak[0]v_0,\linebreak[0]v_{\impj{n-2}},\linebreak[0]v_0,\linebreak[0]\underline{v_{\impj{n-3}},\linebreak[0]v_0},\linebreak[0]\ldots,\linebreak[0]\underline{v_{\impj{1}},\linebreak[0]v_0},\linebreak[0]v_{n-1}}]
{\ensuremath{\langle
v_{\impj{1}},\linebreak[0]v_0,\linebreak[0]\underline{v_{\impj{2}},\linebreak[0]v_0},\linebreak[0]\ldots,\linebreak\underline{v_{\impj{n-2}},\linebreak[0]v_0}#1
\rangle}}

\newcommand{\xedge}[4][x]{\ensuremath{\langle
{#1}_{#3}^{#2},\linebreak[0]{#1}_{#4}^{#2} \rangle}}

\newcommand{\imalpha}[2]{\ensuremath{\alpha(\linebreak[0]x_{#2}^{#1})}}

\newcommand{\imalphabar}[4]{\ensuremath{\overline{\alpha}(\linebreak[0]\xseq{#1}{#2}{#3}{#4})}}

\newcommand{\ntt}[1][n]{\ensuremath{\tilde{\tilde{#1}}}}

\newcommand{\stseqzero}[2]{\ensuremath{\langle
\nt[s]_0,\linebreak[0]\nt[s]_1^{#1},\linebreak[0]\ldots,\linebreak[0]\nt[s]_{#2}^{#1}
\rangle}}

\newcommand{\stseq}[5][]{\ensuremath{\langle
\nt[s]_{#3}^{#2},\linebreak[0]\nt[s]_{#4}^{#2},\linebreak[0]\ldots,\linebreak[0]\nt[s]_{#5}^{#2}#1
\rangle}}

\newcommand{\shseqzero}[2]{\ensuremath{\langle
\hat{s}_0,\linebreak[0]\hat{s}_1^{#1},\linebreak[0]\ldots,\linebreak[0]\hat{s}_{#2}^{#1}
\rangle}}

\newcommand{\bgeq}{\begin{equation}}
\newcommand{\edeq}{\end{equation}}
\newcommand{\bgeqn}{\begin{eqnarray}}
\newcommand{\edeqn}{\end{eqnarray}}
\newcommand{\bgeqnn}{\begin{eqnarray*}}
\newcommand{\edeqnn}{\end{eqnarray*}}

\newtheorem{lemma}{Lemma}[chapter]
\newtheorem{theorem}{Theorem}[chapter]
\newtheorem{theorem1}[lemma]{Theorem}
\newtheorem{corollary}[lemma]{Corollary}
\newtheorem{proposition}[lemma]{Proposition}
\theoremstyle{definition}
\newtheorem{definition}{Definition}[chapter]
\newtheorem{definition1}[lemma]{Definition}
\theoremstyle{plain}

\newtheorem{question}[lemma]{Question}

\renewcommand{\thechapter}{\Roman{chapter}}
\renewcommand{\thesection}{\arabic{chapter}.\arabic{section}}

\renewcommand{\headheight}{14pt}
\renewcommand{\headrulewidth}{0pt}
\renewcommand{\footrulewidth}{0pt}

\begin{document}

%title
\begin{titlepage}
\begin{minipage}[t]{.5in}
\ \
\end{minipage}
\begin{center}
\vspace{.15in} COMPLEXITY OF ATRIODIC CONTINUA
\vspace{.14in}\\
by
\vspace{.14in}\\
CHRISTOPHER TODD KENNAUGH, M.S.
\vspace{.5in}\\
A DISSERTATION
\vspace{.14in}\\
IN
\vspace{.14in}\\
MATHEMATICS
\vspace{.15in}\\
\renewcommand{\baselinestretch}{1}\selectfont
Submitted to the Graduate Faculty\\
of Texas Tech University in\\
Partial Fulfillment of\\
 the Requirements for\\
  the  Degree of\\
\renewcommand{\baselinestretch}{1.5}\selectfont
\vspace{.14in}
DOCTOR OF PHILOSOPHY
\vspace{.4in}\\
APPROVED BY

Wayne Lewis (Chair)\\
Robert Byerly\\
Razvan Gelca\\
Fred Hartmeister,\\
Dean of the Graduate School
\vspace{.5in}\\
May, 2009
\end{center}
\end{titlepage}

\begin{titlepage}
\begin{minipage}[t]{2in}
\ \
\end{minipage}
\vspace{4in}\\
\begin{center}
\copyright 2009, Christopher Kennaugh
\end{center}
\end{titlepage}

\frontmatter

\chapter{ACKNOWLEDGMENTS}

Thanks to Wayne Lewis for being the author's teacher and advisor
and for directing this dissertation.  Thanks to the Department of
Mathematics and Statistics at Texas Tech for having the author as
a graduate student and teaching assistant.  Thanks to Robert
Byerly and Razvan Gelca for being on the committee of this
dissertation. Thanks to Dale Daniel for getting the author in to
topology. Thanks to the author's folks Chris and Pat for all their
support. Dedicated to the Golden Triangle, Texas.

\tableofcontents %Add nothing to this line - This is done Automatically

\chapter{ABSTRACT}

This dissertation investigates the relative complexity between a
continuum and its proper subcontinua (see \cite{young2}), in
particular, providing examples of atriodic $n$-od-like continua.
Let $X$ be a continuum and $n$ be an integer greater than or equal
to three. If $X$ is homeomorphic to an inverse limit of
simple-$n$-od graphs with simplicial bonding maps and is
simple-$(n-1)$-od-like, it is shown that the bonding maps can be
simplicially factored through a simple-$(n-1)$-od.  This implies,
in particular, that $X$ is homeomorphic to an inverse limit of
simple-$(n-1)$-od graphs with simplicial bonding maps.  This
factoring is subsequently used (in a strategy adapted from
\cite{minc2}) to show that a specific inverse limit of
simple-$n$-ods with simplicial bonding maps, having the property
of every proper nondegenerate subcontinuum being an arc, is not
simple-$(n-1)$-od-like.

\listoftables \addcontentsline{toc}{chapter}{LIST OF TABLES} %Add nothing to this line - This is done Automatically

\listoffigures \addcontentsline{toc}{chapter}{LIST OF FIGURES} %Add nothing to this line - This is done Automatically

\mainmatter \setcounter{page}{1} \pagenumbering{arabic}

%%%%%%%%%%%%%%%%%%%%%%%%%%%%%%%%%%%%%%%%%%%%%%%%%%%%%%%%%%%%%%%%%%%%%%%%%%%%%%%%%%%%%%%%%%%%
\chapter{BASIC TERMINOLOGY}
\setcounter{equation}{0}

\begin{definition1} A {\it continuum} is a connected, compact,
metric space.
\end{definition1}

\begin{definition1}
A continuum is {\it decomposable} if it is the union of two of its
proper subcontinua and is {\it indecomposable} otherwise.
\end{definition1}

\begin{definition1}
A continuum is {\it hereditarily decomposable} ({\it hereditarily
indecomposable}) if each of its nondegenerate subcontinua is
decomposable (indecomposable).
\end{definition1}

\begin{definition1}
A continuum $X$ is a {\it triod} (3-od) if it contains a
subcontinuum $M$ so that $X\setminus M$ is the union of three
nonempty mutually separated sets.
\end{definition1}

\begin{definition1}
A continuum is {\it atriodic} if it does not contain a triod.
\end{definition1}

\begin{definition1}
For a positive integer $n$, a {\it simple-$n$-od} is the union of
$n$ arcs joined at an end point.
\end{definition1}

\begin{definition1}
For a sequence of ({\it factor}) spaces $\{X_i\}_{i=0}^{\infty}$
and sequence of ({\it bonding}) maps $\{f_i\}_{i=0}^{\infty}$
where $f_i:X_{i+1}\rightarrow X_i$, the {\it inverse limit} is the
subspace
$M=\{(x_0,x_1,\ldots)\in\prod\limits_{i=0}^{\infty}X_i:f_i(x_{i+1})=x_i\mbox{
for all }i\}$, denoted $M=\varprojlim\invseq[]{X}{f}$, where
$f_i^i=\mbox{identity on }X_i$ \ and for $j>i$ \ $f_i^j=f_i\circ
f_{i+1}\circ\cdots\circ f_{j-1}$.
\end{definition1}

\begin{proposition}
Let $X$ be a continuum and $\mathcal{P}$ be a collection of
connected, compact polyhedra. Then the following are equivalent:

1) $X\approx\varprojlim\invseq[]{X}{f}$ where $X_i\in\mathcal{P}$
and $f_i$ is a continuous surjection for each $i$.\\ \indent 2)
For each $\epsilon>0$ there exist $Y\in\mathcal{P}$ and a
continuous surjection $f:X\rightarrow
Y$ so that $\mathrm{diam}(f^{-1}(y))<\epsilon$ for each $y\in Y$.\\
\indent 3) For each $\epsilon>0$ there exists an open covering
$\mathcal{U}$ of $X$ whose nerve is a member of $\mathcal{P}$ and
so that $\mathrm{diam}(U)<\epsilon$ for each $U\in\mathcal{U}$.
\end{proposition}

\begin{definition1}
A continuum satisfying the conditions of the above proposition is
$\mathcal{P}${\it-like}.
\end{definition1}

\begin{definition1}
A continuum  which is arc-like (simple-1-od-like,
simple-2-od-like) is {\it chainable} (also {\it snake-like}).
\end{definition1}

\begin{definition1}
A continuum is {\it subchainable} if every proper subcontinuum is
chainable.
\end{definition1}

\begin{definition1} \label{brdef}
Let $n$ be a positive integer greater than or equal to $3$, $K$ be
a continuum, and $v\in K$.  Then $v$ is a {\it branch point of $K$
of order $n$} if and only if for each $\epsilon>0$ there exists an
open cover $\mathcal{U}$ of $K$ so that
$\mathrm{mesh}(\mathcal{U})<\epsilon$, the nerve of $\mathcal{U}$
is a simple-$n$-od, and $v$ is in the element of $\mathcal{U}$ of
order $n$.
\end{definition1}

\begin{definition1}
A {\it graph} is a one-dimensional, connected, finite
simplicial complex.  If $G$ is a graph, then $\mathrm{V}(G)$
denotes the set of vertices and $\mathrm{E}(G)$ denotes the set of
edges of $G$.
\end{definition1}

\begin{definition1}
A map $f$ between graphs is {\it simplicial} provided each
edge is either mapped linearly onto an edge or mapped into a
single vertex.
\end{definition1}

\begin{definition1}[Def. 2.1 \cite{kato}] \label{simpler1}
A graph $H$ is {\it simpler} ($\le$) than a graph $G$ if there
exists a simplicial monotone map from $G$ onto $H$.
\end{definition1}

\begin{definition1} \label{simpler2}
For graphs $G$ and $H$, a $G$-like continuum $X$, and an $H$-like
continuum $Y$, $Y$ is {\it simpler} than $X$ if $H$ is simpler
than $G$.
\end{definition1}

%%%%%%%%%%%%%%%%%%%%%%%%%%%%%%%%%%%%%%%%%%%%%%%%%%%%%%%%%%%%%%%%%%%%%%%%%%%%%%%%%%%%%%%%%%%%
\chapter{INTRODUCTION}
\setcounter{equation}{0}

In 1951, R.H. Bing showed in \cite{bing2} that among hereditarily
decomposable tree-like continua, chainability is equivalent to
atriodicity and stated the following question in \cite{bing1} ($M$
denotes the pseudo-arc): ``It would be interesting to know if each
nondegenerate bounded hereditarily indecomposable plane continuum
which does not separate the plane is homeomorphic to $M$.  This
question would be answered in the affirmative if it were shown
that each bounded atriodic plane continuum which does not separate
the plane can be chained.''  In the same year, the following
statement, along with other related claims, appeared in an
abstract \cite{anderson} given by R.D. Anderson: ``The author
demonstrates the existence of a hereditarily indecomposable plane
continuum, not separating the plane, which is not homeomorphic to
a pseudo-arc (that is, a chained hereditarily indecomposable plane
continuum).''  The first published example of such a continuum, in
1979, is contained in \cite{ingram3} by W.T. Ingram, constructed
with a modification to the first published counterexample to
Bing's second question, given in \cite{ingram2}, with crookedness
appropriately built-in so as to retain nonchainability.

Also by Ingram and published in 1972, \cite{ingram2} is an example
of an inverse limit of simple-3-ods with a single bonding map,
where nonchainability and atriodicity are implied by the example's
properties of positive span and of every proper nondegenerate
subcontinuum being an arc, respectively.  A. Lelek, in the 1964
publication \cite{lelek}, defined for a metric space $X$ the {\it
span} $\sigma(X)$ as the least upper bound on the numbers
$\epsilon$ for which there exists a connected subspace $Z\subseteq
X\times X$ with $\pi_1(Z)=\pi_2(Z)$ so that the distance between
$a$ and $b$ is greater than or equal to $\epsilon$ for each
$(a,b)\in Z$ ({\it surjective span} ${\sigma}^{\ast}(X)$ has the
above definition with the additional condition that $\pi_1(Z)=X$)
and showed that if a continuum $X$ is chainable then $\sigma(X)=0$
(the converse is presently open). More generally, for a space $X$,
metric space $Y$, and a map $f$ of $X$ into $Y$, the {\it span of
$f$} $\sigma(f)$ is the least upper bound on the numbers
$\epsilon$ for which there exists a connected subspace $Z\subseteq
X\times X$ with $\pi_1(Z)=\pi_2(Z)$ so that the distance between
$f(a)$ and $f(b)$ is greater than or equal to $\epsilon$ for each
$(a,b)\in Z$. Then, $\sigma(X)$ is the span of the identity map on
$X$. Ingram showed that for the bonding maps $f_i^j$ in the
example from \cite{ingram2}, there exists a positive number
$\epsilon$ so that $\sigma(f_0^j)>\epsilon$ for each $j$ and that
this implies positive span of the inverse limit. In the 1968
publication \cite{ingram1}, Ingram demonstrated that the property
of being atriodic, known to be present in chainable continua, also
holds, more generally, among subchainable continua. The nature of
the single bonding map in \cite{ingram2} ensures that any
nondegenerate subcontinuum of the inverse limit which is not an
arc is not proper, establishing atriodicity of the inverse limit.
Having every proper nondegenerate subcontinuum as an arc also
implies, by the aforementioned result of Bing, that the Ingram
example \cite{ingram2}, being nonchainable, is necessarily
indecomposable. Other inverse limits with the previously discussed
properties in common with \cite{ingram2} have been constructed
utilizing similar techniques as in \cite{ingram2} in demonstrating
positive span, such as in the example by J.F. Davis and Ingram
\cite{davis}. Published in 1988, this continuum also has the
property of admitting a (monotone) map to a chainable continuum
with only one nondegenerate point inverse which is an arc.

Noting that a simple-$n$-od-like continuum is
simple-$(n+1)$-od-like, for example, some properties of a
continuum can be more readily realized when constructing that
continuum as a more ``complex'' space.  Such instances are in
\cite{lewis1} by W. Lewis and in \cite{mayer} by J.C. Mayer,
published in 1983.  In \cite{lewis1}, for each $n$, an inverse
limit of simple-$n$-ods is constructed by Lewis with bonding maps
sufficiently varying so as to produce chainability of the inverse
limit and whose symmetry allows for a homeomorphism on the inverse
limit to be induced, having only one fixed point and every other
point with period $n$. Thereby, construction of a continuum as one
which is simple-$n$-od-like facilitates a description of a period
$n$ homeomorphism on that continuum, controlled so as to be
simpler than would initially appear in the construction, being
simple-$2$-od-like. Thus is given such a homeomorphism for each
$n$ on a chainable continuum and, further, by introducing
crookedness, such a homeomorphism for each $n$ on the pseudo-arc.
In \cite{mayer}, an inverse limit of simple-4-ods, referred to as
the ``X-odic'' continuum, is constructed by Mayer with a single
bonding map, having the previously discussed properties of
\cite{ingram2} and utilizing similar techniques in demonstrating
positive span as \cite{ingram2}, in addition to admitting an
embedding in the plane with a Lake-of-Wada channel.  In the same
publication as \cite{mayer}, in \cite{young1} S.W. Young showed
that the bonding map in \cite{mayer} can be factored through a
simple-triod, giving that the simple-4-od-like X-odic continuum is
simple-3-od-like. Young remarks, ``although the bonding map,''
induced from the factoring, used in the representation of the
(X-odic) continuum $X$ in the simpler form ``does not seem to help
in establishing the main properties of the continuum $X$, it must
inevitably detract from its name.''

Also by Young and in the same publication as \cite{mayer} and
\cite{young1}, \cite{young2} has the following beginning to its
introduction: ``One of the remarkable features of the continuum of
W.T. Ingram (\cite{ingram2}) is the `gap' in complexity between
the continuum and its proper subcontinua. Specifically, the
continuum is T-like (simple-triod-like), not arc-like and every
proper subcontinuum is arc-like.  This combination of structural
properties leads us to ask if there is a continuum with an even
wider `gap'.''  Question 1 follows: ``Does there exist a continuum
which is 4-od-like, not simple-triod-like and every proper
subcontinuum is arc-like?'' Similar questions with slight
variations on this question are Problem 115 from \cite{lewis2}
published in 1983, ``Is there a continuum which is 4-od-like, not
T-like, and every nondegenerate proper subcontinuum of which is an
arc?'' and Problem 5 from \cite{cook} published in 1990, ``Does
there exist an atriodic simple-4-od-like continuum which is not
simple-triod-like?,'' with a positive answer to the second being a
positive answer to all three.

These questions were answered in the affirmative by P. Minc with
the example in \cite{minc2} published in 1993.  In response to
these questions, Minc states in \cite{minc2}, ``Even after a
perfunctory glance at the problems, it becomes apparent that they
should have a positive answer.  It is very easy to get an example
of a simple-4-od-like continuum such that every proper
subcontinuum is an arc.  Most of such continua appear not to be
simple-triod-like and it is very likely that they really are not.
So the only difficulty is a proof,'' where ``a topological
invariant different than the span is needed to distinguish between
those continua which are simple-triod-like and those that are not.
Another way of approaching the problem is to use a continuum with
simplicial bonding maps and prove that they cannot be factored
through a simple triod.'' The following is the abstract from
\cite{minc1} by Minc and published in 1994: ``An operation $d$ on
simplicial maps between graphs is introduced and used to
characterize simplicial maps which can be factored through an arc.
The characterization yields a new technique of showing that some
continua are not chainable,'' demonstrated with the examples in
\cite{ingram2} and \cite{davis}, ``and allows to prove that span
zero is equivalent to chainability for inverse limits of trees
with simplicial bonding maps.''  Theorem 3.3 from \cite{minc1}
establishes that the following are equivalent for an inverse limit
$X$ of trees with simplicial bonding maps $f_i^j$:

1) X is chainable\\
\indent 2) $\sigma^\ast(X)=0$\\
\indent 3) for each $i$ there exists $j>i$ so that $f_i^j$ can be
(simplicially) factored through an arc.

This alternate, combinatorial technique involving factoring and
the operation $d$ is extended by Minc to the example in \cite{minc2}
in demonstrating that it is not simple-3-od-like, where an inverse
limit of simple-4-ods is constructed with a single bonding map,
controlled so that every proper nondegenerate subcontinuum is an
arc and so that the bonding map $f_0^j$ cannot be (simplicially)
factored through a simple-3-od for each $j$.

The purpose of this dissertation is to investigate this gap in the
complexity between a continuum and its proper subcontinua further,
in particular when the proper nondegenerate subcontinua are arcs.
For each integer $n\ge 3$, an inverse limit {\K} of simple-$n$-ods
with a single bonding map {\nphi} is constructed which is not
simple-$(n-1)$-od-like and whose every proper nondegenerate
subcontinuum is an arc, established in Chapter IV. To this end, an
extension and generalization of the strategy employed in
\cite{minc2} is developed with a general theorem on factoring
established in Chapter III.

%%%%%%%%%%%%%%%%%%%%%%%%%%%%%%%%%%%%%%%%%%%%%%%%%%%%%%%%%%%%%%%%%%%%%%%%%%%%%%%%%%%%%%%%%%%%
\chapter{PRELIMINARIES}
\setcounter{equation}{0}

In this chapter, a theorem on factoring is given (Theorem
\ref{claim17}) stating that for certain inverse limits in the
simplicial setting, in the case the inverse limit is simpler
(Definition \ref{simpler2}) than the complexity of the factor
spaces would indicate, the bonding maps are able to be factored
through simpler graphs. In light of Lemma \ref{lemm}, the proof of
Theorem \ref{claim17} follows from the proof of Propostion 2.1 in
\cite{minc2}.

Let $K$ be a continuum and $n$ be
an integer greater than or equal to $3$.

\begin{lemma} \label{lemm} Suppose $\mathcal{V}$ is an open cover of $K$ so that the
nerve of $\mathcal{V}$ is a simple-$n$-od and $\mathcal{W}$ is an
open cover of $K$ refining $\mathcal{V}$ so that the nerve of
$\mathcal{W}$ is a simple-$(n-1)$-od. Then, there exists an amalgamation $\mathcal{U}$
of $\mathcal{W}$ refining $\mathcal{V}$ so that the
nerve of $\mathcal{U}$ is a simple-$\nt$-od for some
$\nt\in\oneton{n-1}$ and $U_0$, the element of $\mathcal{U}$ of
order $\nt$, is contained in $V_0$, the element of $\mathcal{V}$
of order $n$.
\end{lemma}

\begin{proof}
For each $i\in\oneton{n}$ and some $\nu_i\in\{1,2,\ldots\}$, let
$V_0^i,V_1^i,V_2^i,\ldots,V_{\nu_i}^i$ be a linear chain in
$\mathcal{V}$ so that $\mathcal{V}=\{V_j^i:i\in\oneton{n}\mbox{
and }j\in\oneton[0]{\nu_i}\}$, where $V_0^i=V_0$ for each
$i\in\oneton{n}$, and let $V^i=\bigcup\limits_{j=1}^{\nu_i}V_j^i$.
For each $i\in\oneton{n-1}$ and some $\mu_i\in\{1,2,\ldots\}$, let
$W_0,W_1^i,W_2^i,\ldots,W_{\mu_i}^i$ be a linear chain in
$\mathcal{W}$ so that
$\mathcal{W}=\{W_0,W_j^i:i\in\oneton{n-1}\mbox{ and
}j\in\oneton{\mu_i}\}$, where $W_0$ is the element of
$\mathcal{W}$ of order $n-1$.

Suppose $W_0\not\subseteq V_0$, since otherwise the claim holds
with $\mathcal{U}=\mathcal{W}$.  Let $\nt[\imath]\in\oneton{n}$
and $\nt[\jmath]\in\oneton{\nu_{\nt[\imath]}}$ so that
$W_0\subseteq V_{\nt[\jmath]}^{\nt[\imath]}$ and $W_0\not\subseteq
V_{\nt[\jmath]-1}^{\nt[\imath]}$.  For each $i\in\oneton{n-1}$,
let $j_i\in\oneton{\mu_i}$ so that
$\bigcup\limits_{j=1}^{j_i}W_j^i\subseteq V^{\nt[\imath]}\bigcup
V_0$ and $W_{j_i+1}^i\not\subseteq V^{\nt[\imath]}\bigcup V_0$ or $W_{j_i+1}^i$
does not exist. Let
$\nt[\jmath]_m=\mbox{max}\{j\in\oneton[0]{\nu_{\nt[\imath]}}:W_k^i\subseteq
V_j^{\nt[\imath]}\mbox{ and }W_k^i\not\subseteq
V_{j-1}^{\nt[\imath]}\mbox{ if }V_{j-1}^{\nt[\imath]}\mbox{ exists
or }W_k^i\not\subseteq V^{\hat{\imath}}\mbox{ for all
}\hat{\imath}\not={\nt[\imath]}\mbox{ if
}V_{j-1}^{\nt[\imath]}\mbox{ does not exist},\mbox{ for some
}i\in\oneton{n-1}\mbox{ and }j\in\oneton{j_i}\}$, and let
$i^{\prime}\in\oneton{n-1}$ and
$k_{i^{\prime}}\in\oneton{j_{i^{\prime}}}$ so that
$W_{k_{i^{\prime}}}^{i^{\prime}}\subseteq
V_{\nt[\jmath]_m}^{\nt[\imath]}$ and
$W_{k_{i^{\prime}}}^{i^{\prime}}\not\subseteq
V_{\nt[\jmath]_m-1}^{\nt[\imath]}$ if
$V_{\nt[\jmath]_m-1}^{\nt[\imath]}$ exists or
$W_{k_{i^{\prime}}}^{i^{\prime}}\not\subseteq V^i\mbox{ for all
}i\not={\nt[\imath]}$ if $V_{\nt[\jmath]_m-1}^{\nt[\imath]}$ does
not exist.

Without loss of generality, suppose

i)  there exists
$\hat{\imath}\in\oneton{n-1}$ so that
$\{i\in\oneton{n-1}\setminus\{i^{\prime}\}:W_{j_i+1}^i\mbox{ does
not}\linebreak[0]\mbox{ exist}\}=\{\hat{\imath},\hat{\imath}+1,\ldots,n-1\}$ or

ii)
$\{i\in\oneton{n-1}\setminus\{i^{\prime}\}:W_{j_i+1}^i\mbox{ does
not exist}\}=\emptyset$.

Let $\nt=\hat{\imath}-1$ if (i) and
$\nt=n-1$ if (ii).  Let $\nt[\jmath]_M=\nt[\jmath]$ if
$\nt[\jmath]\ge\nt[\jmath]_m$, or $\nt[\jmath]_M=\nt[\jmath]_m-1$
if $\nt[\jmath]<\nt[\jmath]_m$ and $\{W_k^i\subseteq
V_{\nt[\jmath]_m}^{\nt[\imath]}\mbox{ and }W_k^i\not\subseteq
V_{\nt[\jmath]_m-1}^{\nt[\imath]}:i\in\oneton{n-1}\setminus\{i^{\prime}\}\mbox{
and }k\in\oneton{j_i}\}\bigcup\{W_k^{i^{\prime}}\subseteq
V_{\nt[\jmath]_m}^{\nt[\imath]}\mbox{ and
}W_k^{i^{\prime}}\not\subseteq
V_{\nt[\jmath]_m-1}^{\nt[\imath]}:k\in\oneton{k_{i^{\prime}}-1}\}=\emptyset$,
or $\nt[\jmath]_M=\nt[\jmath]_m$ if otherwise.

For each $i\in\oneton{\nt}\setminus\{i^{\prime}\}$, let
$U_1^i,U_2^i,\ldots,U_{\mu_i-j_i}$ be a linear chain of open sets
defined as $U_j^i=W_{j_i+j}^i$ for each $j\in\oneton{\mu_i-j_i}$,
and let $U_0$ be an open set defined as
$U_0=\bigcup(\{W_j^i\subseteq
V_0:i\in\oneton{n-1}\setminus\{i^{\prime}\}\mbox{ and
}j\in\oneton{j_i}\}\bigcup\{W_j^{i^{\prime}}\subseteq
V_0:j\in\oneton{k_{i^{\prime}}-1}\})$, and let
$U_1^{i^{\prime}},U_2^{i^{\prime}},\ldots,U_{\nt[\jmath]_M+\mu_{i^{\prime}}-k_{i^{\prime}}+1}^{i^{\prime}}$
be a linear chain of open sets defined as
$U_j^{i^{\prime}}=\bigcup(\{W\subseteq V_j^{\nt[\imath]}\mbox{ and
}W\not\subseteq
V_{j-1}^{\nt[\imath]}:W\in\{W_0,W_k^i:i\in\oneton{n-1}\setminus\{i^{\prime}\}\mbox{
and }k\in\oneton{j_i}\}\}\bigcup\{W_k^{i^{\prime}}\subseteq
V_j^{\nt[\imath]}\mbox{ and }W_k^{i^{\prime}}\not\subseteq
V_{j-1}^{\nt[\imath]}:k\in\oneton{k_{i^{\prime}}-1}\})$ for each
$j\in\oneton{\nt[\jmath]_M}$ and
$U_{\nt[\jmath]_M+j}^{i^{\prime}}=W_{k_{i^{\prime}}-1+j}^{i^{\prime}}$
for each $j\in\oneton{\mu_{i^{\prime}}-k_{i^{\prime}}+1}$.  Then
the claim holds with
$\mathcal{U}=\{U_0\}\bigcup\{U_j^i:i\in\oneton{\nt}\setminus\{i^{\prime}\}\mbox{
and
}j\in\oneton{\mu_i-j_i}\}\bigcup\{U_j^{i^{\prime}}:j\in\oneton{\nt[\jmath]_M+\mu_{i^{\prime}}-k_{i^{\prime}}+1}\}$.
\end{proof}

\vspace{3cm}
\begin{figure}[here]
\hspace{3.5cm} \setlength{\unitlength}{.83cm}
\begin{picture}(0,1)

\put(7,0){\oval(6,2)} \multiput(7,1.25)(0,1){4}{\oval(4.5,1.5)}
\multiput(7,-1.25)(0,-1){4}{\oval(4.5,1.5)}
\multiput(10.25,0)(1,0){4}{\oval(1.5,4)}
\multiput(3.75,0)(-1,0){2}{\oval(1.5,4)} \put(1.5,0){\oval(2,4)}
\multiput(.25,0)(-1,0){5}{\oval(1.5,4)}

\put(1.25,.15){\circle{1}}

\put(1.4,.7){\circle{.4}}
\multiput(1.65,.85)(.25,0){7}{\circle{.4}}
\put(3.25,.57){\circle{.4}} \put(3.35,.34){\circle{.4}}
\multiput(3.6,.2)(.25,0){20}{\circle{.4}}
\multiput(8.45,.48)(0,.25){2}{\circle{.4}}
\multiput(8.45,.98)(0,.25){14}{\circle{.4}}
\multiput(8.35,4.51)(-.25,0){11}{\circle{.4}}
\multiput(5.75,4.23)(0,-.25){14}{\circle{.4}}
\multiput(5.65,.7)(-.25,0){8}{\circle{.4}}
\put(3.8,.98){\circle{.4}} \put(3.7,1.26){\circle{.4}}
\multiput(3.6,1.54)(-.25,0){31}{\circle{.4}}

\put(.95,-.25){\circle{.35}}
\multiput(.75,-.4)(-.25,0){8}{\circle{.35}}
\put(-1.1,-.68){\circle{.35}}
\multiput(-1,-.89)(.25,0){17}{\circle{.35}}
\put(3.1,-.65){\circle{.35}} \put(3.2,-.42){\circle{.35}}
\multiput(3.3,-.19)(.17,0){20}{\circle{.25}}
\multiput(6.68,-.37)(0,-.25){2}{\circle{.35}}
\multiput(6.68,-.87)(0,-.25){7}{\circle{.35}}
\multiput(6.78,-2.6)(.25,0){7}{\circle{.35}}
\multiput(8.38,-2.37)(0,.25){9}{\circle{.35}}
\multiput(8.62,-.27)(.25,0){21}{\circle{.35}}

\multiput(.75,.15)(-.25,0){14}{\circle{.4}}
\multiput(-2.6,-.13)(0,-.25){5}{\circle{.4}}
\put(-2.5,-1.41){\circle{.4}}
\multiput(-2.25,-1.41)(.25,0){24}{\circle{.4}}
\put(3.6,-1.13){\circle{.4}} \put(3.7,-.85){\circle{.4}}
\multiput(3.8,-.57)(.25,0){8}{\circle{.4}}
\multiput(5.65,-.85)(0,-.25){16}{\circle{.4}}

%\thicklines \put(-2.85,-1.65){\framebox(.57,2.1){}}
%\put(-2.45,-.15){\framebox(1.17,.72){}}
%\put(-1.45,-1.1){\framebox(1.17,1.5){}}
%\put(-.47,-1.05){\framebox(1.16,1.6){}}

%\put(.55,-1.1){\line(1,0){1.67}} \put(.55,-1.1){\line(0,1){2.2}}
%\put(.55,1.1){\line(1,0){1.83}} \put(2.38,1.1){\line(0,-1){.5}}
%\put(2.38,.6){\line(-1,0){.16}} \put(2.22,.6){\line(0,-1){1.7}}

%\put(2.15,1.23){\line(0,-1){1.85}}
%\put(2.15,-.62){\line(-1,0){.1}}
%\put(2.05,-.62){\line(0,-1){.445}}
%\put(2.05,-1.065){\line(1,0){1.25}}
%\put(3.3,-1.065){\line(0,1){1.1}} \put(3.3,.035){\line(1,0){.07}}
%\put(3.37,.035){\line(0,1){1.195}}
%\put(3.37,1.23){\line(-1,0){1.22}}

%\put(3.02,.8){\line(1,0){.55}} \put(3.02,.8){\line(0,-1){1.45}}
%\put(3.02,-.65){\line(1,0){.4}} \put(3.42,-.65){\line(0,1){.31}}
%\put(3.57,.8){\line(0,-1){.33}} \put(3.42,-.34){\line(1,0){.73}}
%\put(3.57,.47){\line(1,0){.78}} \put(4.35,.47){\line(0,-1){.5}}
%\put(4.35,-.03){\line(-1,0){.2}} \put(4.15,-.03){\line(0,-1){.31}}

%\put(4.025,-.32){\line(1,0){2.445}}
%\put(4.025,-.32){\line(0,1){.74}}
%\put(4.025,.42){\line(1,0){4.195}}
%\put(6.47,-.32){\line(0,-1){.53}} \put(8.22,.42){\line(0,1){.53}}
%\put(6.47,-.85){\line(1,0){.45}} \put(8.22,.95){\line(1,0){.45}}
%\put(6.92,-.85){\line(0,1){.85}} \put(8.67,.95){\line(0,-1){.95}}
%\put(8.67,0){\line(-1,0){1.75}}

\end{picture}
\vspace{4cm} \caption{A possibility for $\mathcal{W}$ and
$\mathcal{V}$.}

\end{figure}
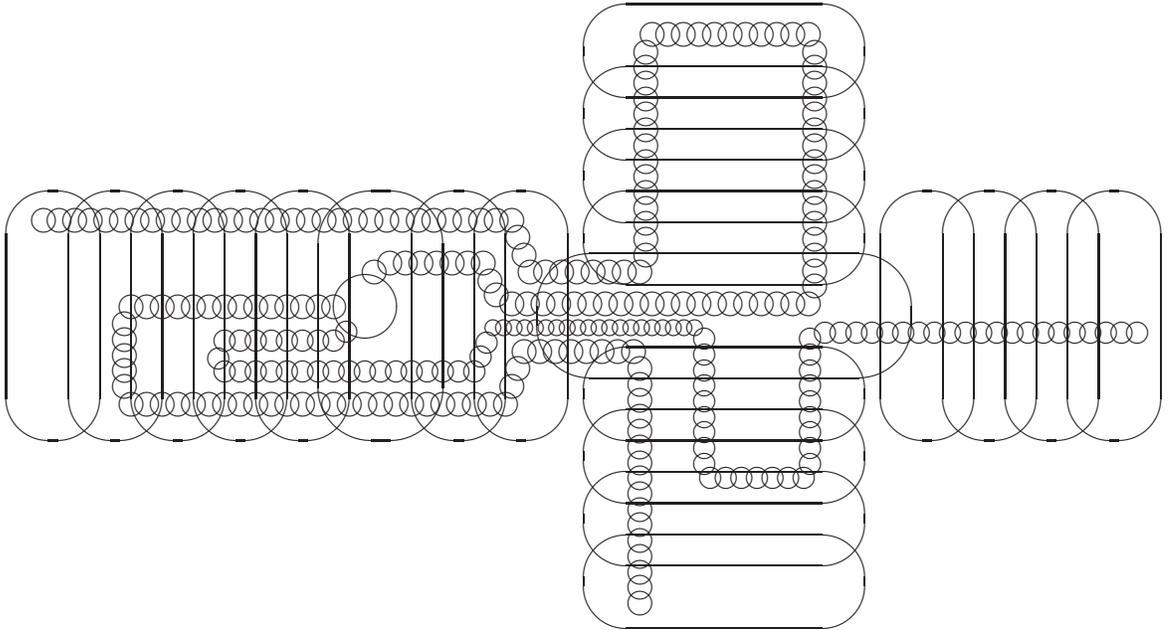

\newpage
\begin{figure}[here]
\vspace{4cm} \hspace{3.5cm} \setlength{\unitlength}{.83cm}
\begin{picture}(0,1)

\put(7,0){\oval(6,2)} \multiput(7,1.25)(0,1){4}{\oval(4.5,1.5)}
\multiput(7,-1.25)(0,-1){4}{\oval(4.5,1.5)}
\multiput(10.25,0)(1,0){4}{\oval(1.5,4)}
\multiput(3.75,0)(-1,0){2}{\oval(1.5,4)} \put(1.5,0){\oval(2,4)}
\multiput(.25,0)(-1,0){5}{\oval(1.5,4)}

\put(1.25,.15){\circle{1}}

\put(1.4,.7){\circle{.4}}
\multiput(1.65,.85)(.25,0){7}{\circle{.4}}
\put(3.25,.57){\circle{.4}} \put(3.35,.34){\circle{.4}}
\multiput(3.6,.2)(.25,0){20}{\circle{.4}}
\multiput(8.45,.48)(0,.25){2}{\circle{.4}}
\multiput(8.45,.98)(0,.25){14}{\circle*{.4}}
\multiput(8.35,4.51)(-.25,0){11}{\circle*{.4}}
\multiput(5.75,4.23)(0,-.25){14}{\circle*{.4}}
\multiput(5.65,.7)(-.25,0){8}{\circle*{.4}}
\put(3.8,.98){\circle*{.4}} \put(3.7,1.26){\circle*{.4}}
\multiput(3.6,1.54)(-.25,0){31}{\circle*{.4}}

\put(.95,-.25){\circle{.35}}
\multiput(.75,-.4)(-.25,0){8}{\circle{.35}}
\put(-1.1,-.68){\circle{.35}}
\multiput(-1,-.89)(.25,0){17}{\circle{.35}}
\put(3.1,-.65){\circle{.35}} \put(3.2,-.42){\circle{.35}}
\multiput(3.3,-.19)(.17,0){20}{\circle{.25}}
\multiput(6.68,-.37)(0,-.25){2}{\circle{.35}}
\multiput(6.68,-.87)(0,-.25){7}{\circle*{.35}}
\multiput(6.78,-2.6)(.25,0){7}{\circle*{.35}}
\multiput(8.38,-2.37)(0,.25){9}{\circle*{.35}}
\multiput(8.62,-.27)(.25,0){21}{\circle*{.35}}

\multiput(.75,.15)(-.25,0){14}{\circle{.4}}
\multiput(-2.6,-.13)(0,-.25){4}{\circle{.4}}
\put(-2.6,-1.13){\circle*{.4}}
\multiput(-2.5,-1.41)(.25,0){25}{\circle*{.4}}
\put(3.6,-1.13){\circle*{.4}} \put(3.7,-.85){\circle*{.4}}
\multiput(3.8,-.57)(.25,0){8}{\circle*{.4}}
\multiput(5.65,-.85)(0,-.25){16}{\circle*{.4}}

\color{Red}

\multiput(8.45,.98)(0,.25){14}{\circle{.4}}
\multiput(8.35,4.51)(-.25,0){11}{\circle{.4}}
\multiput(5.75,4.23)(0,-.25){14}{\circle{.4}}
\multiput(5.65,.7)(-.25,0){8}{\circle{.4}}
\put(3.8,.98){\circle{.4}} \put(3.7,1.26){\circle{.4}}
\multiput(3.6,1.54)(-.25,0){31}{\circle{.4}}

\multiput(6.68,-.87)(0,-.25){7}{\circle{.35}}
\multiput(6.78,-2.6)(.25,0){7}{\circle{.35}}
\multiput(8.38,-2.37)(0,.25){9}{\circle{.35}}
\multiput(8.62,-.27)(.25,0){21}{\circle{.35}}

\put(-2.6,-1.13){\circle{.4}}
\multiput(-2.5,-1.41)(.25,0){25}{\circle{.4}}
\put(3.6,-1.13){\circle{.4}} \put(3.7,-.85){\circle{.4}}
\multiput(3.8,-.57)(.25,0){8}{\circle{.4}}
\multiput(5.65,-.85)(0,-.25){16}{\circle{.4}}

\thicklines \put(-2.85,-1.125){\framebox(.57,1.575){}}
\put(-2.45,-.15){\framebox(1.17,.72){}}
\put(-1.45,-1.1){\framebox(1.17,1.5){}}
\put(-.47,-1.05){\framebox(1.16,1.6){}}

\put(.55,-1.1){\line(1,0){1.67}} \put(.55,-1.1){\line(0,1){2.2}}
\put(.55,1.1){\line(1,0){1.83}} \put(2.38,1.1){\line(0,-1){.5}}
\put(2.38,.6){\line(-1,0){.16}} \put(2.22,.6){\line(0,-1){1.7}}

\put(2.15,1.23){\line(0,-1){1.85}}
\put(2.15,-.62){\line(-1,0){.1}}
\put(2.05,-.62){\line(0,-1){.445}}
\put(2.05,-1.065){\line(1,0){1.25}}
\put(3.3,-1.065){\line(0,1){1.1}} \put(3.3,.035){\line(1,0){.07}}
\put(3.37,.035){\line(0,1){1.195}}
\put(3.37,1.23){\line(-1,0){1.22}}

\put(3.02,.8){\line(1,0){.55}} \put(3.02,.8){\line(0,-1){1.45}}
\put(3.02,-.65){\line(1,0){.4}} \put(3.42,-.65){\line(0,1){.31}}
\put(3.57,.8){\line(0,-1){.33}} \put(3.42,-.34){\line(1,0){.73}}
\put(3.57,.47){\line(1,0){.78}} \put(4.35,.47){\line(0,-1){.5}}
\put(4.35,-.03){\line(-1,0){.2}} \put(4.15,-.03){\line(0,-1){.31}}

\put(4.025,-.32){\line(1,0){2.445}}
\put(4.025,-.32){\line(0,1){.74}}
\put(4.025,.42){\line(1,0){4.195}}
\put(6.47,-.32){\line(0,-1){.53}} \put(8.22,.42){\line(0,1){.53}}
\put(6.47,-.85){\line(1,0){.45}} \put(8.22,.95){\line(1,0){.45}}
\put(6.92,-.85){\line(0,1){.85}} \put(8.67,.95){\line(0,-1){.95}}
\put(8.67,0){\line(-1,0){1.75}} \normalcolor

\end{picture}
\vspace{4cm} \caption{A possibility for $\mathcal{U}$.}

\end{figure}
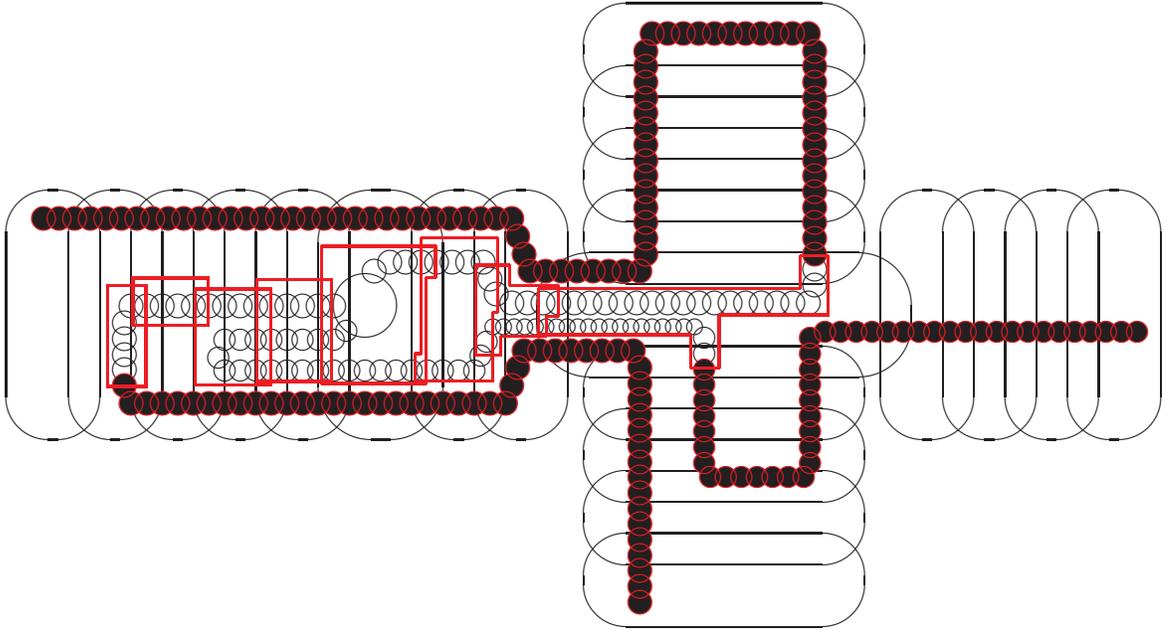

\begin{proposition} \label{claim0} Suppose $v\in K$ is a branch
point of $K$ of order $n$ and $K$ is simple-$(n-1)$-od-like. Then,
for each $\epsilon>0$, there exists an open cover $\mathcal{U}$ of
$K$ so that $\mathrm{mesh}(\mathcal{U})<\epsilon$, the nerve of
$\mathcal{U}$ is a simple-$\nt$-od for some $\nt\in\oneton{n-1}$,
and $v$ is within $\epsilon$ of the element of $\mathcal{U}$ of
order $\nt$.
\end{proposition}

\begin{proof}
The claim follows by applying Lemma \ref{lemm} and Definition \ref{brdef}
\end{proof}

\begin{theorem1} \label{claim17}
Suppose $K$ is homeomorphic to $\varprojlim\invseq[]{X}{\Phi}$
where $X_i$ is a simple-$n$-od graph and $\Phi_i$ is simplicial for each
$i\in\{0,1,\ldots\}$.  If $K$ is simple-$(n-1)$-od-like, then for
each $i\in\{0,1,\ldots\}$ there exist $j\in\{i+1,i+2,\ldots\}$,
a simple-$\nt$-od $T$ for some $\nt\in\oneton{n-1}$, and
simplicial maps $\alpha:X_j\longrightarrow T$ and
$\beta:T\longrightarrow X_i$ so that
$\beta\circ\alpha=\nPhi[]{j}{i}$ and $\beta(t_0)=x_0$ where $t_0$
and $x_0$ are the branch points of $T$ and $X_i$, respectively.
\end{theorem1}

\begin{proof}
Let $\{B_x:x\in\mathrm{V}(X_i)\}$ be a collection of mutually
exclusive connected open sets in $X_i$ so that $x\in B_x$ for each
$x\in\mathrm{V}(X_i)$, and so the nerve of
$\mathcal{V}=\{\pi_i^{-1}(V):V\in\{B_x:x\in\mathrm{V}(X_i)\}\bigcup\{\mbox{open
edges of }X_i\}\}$ is a simple-$n$-od.  Since $K$ is
simple-$(n-1)$-od-like, by Lemma \ref{lemm} there exists an open
cover $\mathcal{U}$ of $K$ refining $\mathcal{V}$ so that the
nerve of $\mathcal{U}$ is a simple-$\nt$-od for some
$\nt\in\oneton{n-1}$ and $U_0$, the element of $\mathcal{U}$ of
order $\nt$, is contained in $\pi_i^{-1}(B_{x_0})$.

Let
$\mathcal{U}_B=\{U\in\mathcal{U}:U\subseteq\pi_i^{-1}(B_x)\mbox{
for some }x\in\mathrm{V}(X_i)\}$ and for each $U\in\mathcal{U}_B$
let $x_U\in\mathrm{V}(X_i)$ so that
$U\subseteq\pi_i^{-1}(B_{x_U})$.  Define an equivalence relation
$\equiv$ on $\mathcal{U}_B$ as $U_p\equiv U_q$ if and only if
$x_U=x_{U_p}$ for all $U\in\langle U_p,U_q
\rangle\bigcap\mathcal{U}_B$, where $\langle U_p,U_q \rangle$
denotes the unique linear chain in $\mathcal{U}$ between $U_p$ and
$U_q$, for each $U_p,U_q\in\mathcal{U}_B$.  Let $T$ be the graph
defined as $\mathrm{V}(T)=\mathcal{U}_B/\equiv$ and
$\widehat{U_p},\widehat{U_q}\in\mathrm{V}(T)$ are adjacent if and
only if $U\in\widehat{U_p}$ or $U\in\widehat{U_q}$ for each
$U\in\langle U_p,U_q \rangle\bigcap\mathcal{U}_B$ for some
$U_p\in\widehat{U_p}$ and $U_q\in\widehat{U_q}$.  Since
$U_0\subseteq\pi_i^{-1}(B_{x_0})$, $U_0\in\mathcal{U}_B$ and
$x_{U_0}=x_0$.  Let $\widehat{U_0}\in\mathrm{V}(T)$ so that
$U_0\in\widehat{U_0}$.  Then, if
$\widehat{U_p},\widehat{U_q}\in\mathrm{V}(T)$ are adjacent or
equal and $U_0\in\langle U_p,U_q \rangle$ for some
$U_p\in\widehat{U_p}$ and $U_q\in\widehat{U_q}$, then
$\widehat{U_p}=\widehat{U_0}$ or $\widehat{U_q}=\widehat{U_0}$.
Thus, $\widehat{U_0}$ is the unique vertex of $T$ with order
greater than 2, provided $T$ is not an arc, and without loss of
generality has order $\nt$.  Then, $T$ is a simple-$\nt$-od.

Define the piecewise-linear map $\beta:T\longrightarrow X_i$ as
$\beta(t)=x_U$ for some $U\in t$ for each $t\in\mathrm{V}(T)$.
Suppose $t_p,t_q\in\mathrm{V}(T)$ are adjacent, and let $U_p\in
t_p$ and $U_q\in t_q$.  If $x_{U_p}$ and $x_{U_q}$ are not
adjacent and not equal then there exists
$x\in\mathrm{V}(X_i)\setminus\{x_{U_p},x_{U_q}\}$ so that
$x\in\langle x_{U_p},x_{U_q} \rangle$, where $\langle
x_{U_p},x_{U_q} \rangle$ denotes the unique arc in $X_i$ between
$x_{U_p}$ and $x_{U_q}$, and so $x_U=x$ for some $U\in\langle
U_p,U_q \rangle\bigcap\mathcal{U}_B$, contradicting $U\in
t_p\bigcup t_q$.  Thus, $\beta(t_p)$ and $\beta(t_q)$ are adjacent
or equal, and so $\beta$ is simplicial.  Also,
$\beta(t_0)=x_{U_0}=x_0$ where $t_0=\widehat{U_0}$.

Let $\epsilon$ be a Lebesgue number for $\mathcal{U}$ and $j>i$ so
that the diameter of $\pi_j^{-1}(x)$ is less than $\epsilon$ for
each $x\in X_j$.  For each $x\in\mathrm{V}(X_j)$ let
$U_x\in\mathcal{U}$ so that $\pi_j^{-1}(x)\subseteq U_x$.  Since
$U_x\bigcap\pi_i^{-1}(\nPhi[]{j}{i}(x))\not=\emptyset$,
$U_x\subseteq\pi_i^{-1}(B_{\nPhi[]{j}{i}(x)})$ for each
$x\in\mathrm{V}(X_j)$, and so $U_{\nt[x]}\in\mathcal{U}_B$ with
$x_{U_{\nt[x]}}=\nPhi[]{j}{i}(\nt[x])$ for each
$\nt[x]\in\mathrm{V}(X_j)$.  Let $t_x\in\mathrm{V}(T)$ so that
$U_x\in t_x$, and define the piecewise-linear map
$\alpha:X_j\longrightarrow T$ as $\alpha(x)=t_x$ for each
$x\in\mathrm{V}(X_j)$.  Then,
$\beta\circ\alpha(\nt[x])=\beta(t_{\nt[x]})=x_{U_{\nt[x]}}=\nPhi[]{j}{i}(\nt[x])$
for all $\nt[x]\in\mathrm{V}(X_j)$.  Suppose
$x_p,x_q\in\mathrm{V}(X_j)$ are adjacent and $t_{x_p}$ and
$t_{x_q}$ are not adjacent and not equal.  Then there exist
$t\in\mathrm{V}(T)\setminus\{t_{x_p},t_{x_q}\}$ so that
$t\in\langle t_{x_p},t_{x_q} \rangle$ and $U_1\in\langle
U_{x_p},U_{x_q} \rangle$ so that $U_1\in t$.  Let
$W_1,W_2,\ldots,W_m\subseteq\langle x_p,x_q \rangle$ for some
$m\in\{1,2,\ldots\}$ be a linear chain of connected sets covering
$\langle x_p,x_q \rangle$ with $x_p\in W_1$ so that the diameter
of $\pi_j^{-1}(W_k)$ is less than $\epsilon$ for each
$k\in\oneton{m}$.  Since $U\bigcap \pi_j^{-1}(W_k)\not=\emptyset$
for some $k\in\oneton{m}$ for all $U\in\langle U_{x_p},U_{x_q}
\rangle$, $U\bigcap\pi_i^{-1}(\langle
\nPhi[]{j}{i}(x_p),\nPhi[]{j}{i}(x_q) \rangle)\not=\emptyset$ for
all $U\in\langle U_{x_p},U_{x_q} \rangle$, and so if $U\in\langle
U_{x_p},U_{x_q} \rangle\bigcap\mathcal{U}_B$ then
$U\subseteq\pi_i^{-1}(B_{\nPhi[]{j}{i}(x_p)})$ or
$U\subseteq\pi_i^{-1}(B_{\nPhi[]{j}{i}(x_q)})$.  If
$\nPhi[]{j}{i}(x_p)=\nPhi[]{j}{i}(x_q)$ then $t_{x_p}=t_{x_q}$,
giving a contradiction.  Thus, without loss of generality,
$\nPhi[]{j}{i}(x_p)\not=\nPhi[]{j}{i}(x_q)$,
$x_{U_1}=\nPhi[]{j}{i}(x_p)$, and there exists $U_2\in\langle
U_{x_p},U_1 \rangle\bigcap\mathcal{U}_B$ so that
$x_{U_2}=\nPhi[]{j}{i}(x_q)$. Let $k\in\oneton{m}$ so that
$U_2\bigcap \pi_j^{-1}(W_k)\not=\emptyset$.  Since
$\pi_i(U_2)\subseteq B_{\nPhi[]{j}{i}(x_q)}$,
$\nPhi[]{j}{i}(W_k)\bigcap B_{\nPhi[]{j}{i}(x_q)}\not=\emptyset$,
and so by linearity $\nPhi{j}{i}(W)\bigcap
B_{\nPhi[]{j}{i}(x_q)}\not=\emptyset$ for all $W\in\langle W_k,W_m
\rangle$.  But, $\pi_j^{-1}(W)\subseteq U_1$ for some $W\in\langle
W_k,W_m \rangle$, and so
$U_1\bigcap\pi_i^{-1}(B_{\nPhi[]{j}{i}(x_q)})\not=\emptyset$
contradicting $U_1\subseteq\pi_i^{-1}(B_{\nPhi[]{j}{i}(x_p)})$.
Thus, $\alpha(x_p)$ and $\alpha(x_q)$ are adjacent or equal, and
so $\alpha$ is simplicial.
\end{proof}

\begin{corollary} Suppose $K$ is homeomorphic to
$\varprojlim\invseq[]{X}{\Phi}$ where $X_i$ is a simple-$n$-od
graph and $\Phi_i$ is simplicial for each $i\in\{0,1,\ldots\}$.  If
$K$ is simple-$(n-1)$-od-like, then $K$ is homeomorphic to
$\varprojlim\invseq[]{Y}{\Psi}$ where $Y_i$ is a simple-$\nt$-od
graph, for some $\nt\in\oneton{n-1}$, and $\Psi_i$ is simplicial
for each $i\in\{0,1,\ldots\}$.
\end{corollary}

\begin{proof}
By Theorem \ref{claim17}, there exist a strictly increasing
sequence $\{n_0,n_1,\ldots\}\subseteq\{0,1,\ldots\}$ with $n_0=0$
and for each $i\in\{0,1,\ldots\}$ simplicial maps
$\alpha_i:X_{n_{i+1}}\longrightarrow Y_i$ and
$\beta_i:Y_i\longrightarrow X_{n_i}$ with $Y_i$ a simple-$\nt$-od
for some $\nt\in\oneton{n-1}$ so that
$\beta_i\circ\alpha_i=\nPhi[]{n_{i+1}}{n_i}$.  For each
$i\in\{0,1,\ldots\}$ define $\Psi_i:Y_{i+1}\longrightarrow Y_i$ as
$\Psi_i=\alpha_i\circ\beta_{i+1}$, and the claim holds.
\end{proof}

%%%%%%%%%%%%%%%%%%%%%%%%%%%%%%%%%%%%%%%%%%%%%%%%%%%%%%%%%%%%%%%%%%%%%%%%%%%%%%%%%%%%%%%%%%%%
\chapter{THE EXAMPLE \K}
\setcounter{equation}{0}

In this chapter, for an integer $n$ greater than or equal to 3, an
example \K{} of an inverse limit of simple-$n$-ods with a single
bonding map \nphi{} is given.  As in \cite{ingram2}, \cite{davis},
and \cite{minc2}, characteristics of the bonding map ensure every
proper nondegenerate subcontinuum being an arc, implying
atriodicity of the continuum. The bonding map for \K{} being a
simplicial map from a graph which is a subdivision of the range
graph, as for the other examples, allows for the inverse limit to
be represented ``simplicially'' where the factor spaces are graphs
with fixed sets of vertices and the bonding maps are simplicial.
In this context, as developed by Minc in \cite{minc1}, certain
combinatorial conditions involving the operation $d$ and factoring
allow for the determination of the nonchainability of the inverse
limit.  If the bonding map defining the inverse limit is
particularly ``nice'', then these conditions are satisfied,
implying that the bonding maps cannot be factored through arcs,
implying nonchainability.

Critical to this program with factoring bonding maps through arcs
(simple-2-ods), when considering the dual of the factoring, the
dual of the arc remains an arc (or is a point or is empty).  When
factoring through a simple-3-od, for example, this is not
necessarily the case, where the dual of the simple-3-od may be
more ``complicated'' than a simple-3-od.  To maintain this control
of complexity when considering an extension of Minc's algorithm
with such factorings, particulars of the specific bonding maps in
question need to be invoked.  Such is done by Minc in
\cite{minc2}, where, again in the simplicial setting, a
combinatorial program involving the factoring of bonding maps
through simple-3-ods is adapted.

For the specific example in \cite{minc2}, the nature of the
defining bonding map allows for a certain factoring to be
arranged, provided the existence of a factoring in general,
thereby avoiding a ``bad'' case where the dual of the simple-3-od
could possibly be more ``complicated.''  Specifically, branch
point of the simple-3-od going to branch point can be avoided in
the factoring, implying, for the dual of simple-3-ods, there being
at most one point of order bigger than 2.  In extending Minc's
program further, a strategy to control complexity is also used in
the factoring of the bonding maps of \K{} through
simple-$(n-1)$-ods. In addition to satisfying Minc's combinatorial
conditions, the nature of the defining bonding map \nphi{} induces
characteristics in the bonding maps that allow for the
rearrangement of the factoring, given by the theorem in the
previous chapter, to one in which the dual of the
simple-$(n-1)$-od remains a simple-$(n-1)$-od. This is
accomplished in conjunction with and facilitated by branch point
of the simple-$(n-1)$-od going to branch point in the factoring as
given by the theorem, which, paradoxically, is the worst-case
scenario with regard to allowing for the greatest possible
complexity of the dual of the simple-$(n-1)$-od.

\section{Definitions}

\begin{definition}[Def. 2.1 \cite{minc1}] \label{def2.1}
For a graph $G$, let $D(G)$ denote the graph such that

i)  $\mathrm{V}(D(G))$ and $\mathrm{E}(G)$ are in one-to-one
correspondence and

ii)  two vertices of $D(G)$ are adjacent if and only if the edges
of $G$ corresponding to the vertices have a non-empty
intersection.

For $v\in\mathrm{V}(D(G))$, $v^*$ denotes the edge of $G$
corresponding to $v$.
\end{definition}

\begin{definition}[Def. 2.4 \cite{minc1}] \label{def2.4}
Let $\phi:G_1\longrightarrow G_0$ be a simplicial map between
graphs, $C=\{c\subseteq G_1:c\mbox{ is a component of
}\phi^{-1}(e)\mbox{ and }\phi(c)=e\mbox{ for some
}e\in\mathrm{E}(G_0)\}$, and $D(\phi,G_1)$ be the graph so that

i)  $\mathrm{V}(D(\phi,G_1))$ and $C$ are in one-to-one
correspondence and

ii)  two vertices of $D(\phi,G_1)$ are adjacent if and only if the
subgraphs of $G_1$ corresponding to the vertices have a non-empty
intersection.

For $v\in\mathrm{V}(D(\phi,G_1))$, $v^*$ denotes the subgraph of
$G_1$ corresponding to $v$.

Let $d[\phi]:D(\phi,G_1)\longrightarrow D(G_0)$ be the simplicial
map determined by $d[\phi](v)=w$ where $w\in\mathrm{V}(D(G_0))$ so
that $\phi(v^*)=w^*$, for every $v\in\mathrm{V}(D(\phi,G_1))$.
\end{definition}

\begin{definition}[Def. 2.10 \cite{minc1}] \label{def2.10}
Let $\phi:G_1\longrightarrow G_0$ and $\psi:G_2\longrightarrow
G_1$ be simplicial maps between graphs.  Let
$d[\phi,\psi]:D(\phi\circ\psi,G_2)\longrightarrow D(\phi,G_1)$ be
the simplicial map determined by $d[\phi,\psi](v)=w$ where $w$ is
the vertex of $D(\phi,G_1)$ so that $\psi(v^*)\subseteq w^*$, for
every vertex $v$ of $D(\phi\circ\psi,G_2)$.
\end{definition}

\begin{definition}[Def. 4.1 \cite{minc1}] \label{def4.1}
Let $\phi:G_1\longrightarrow G_0$ be a simplicial map between
graphs.  Then $\phi$ is {\it ultra light} if it is light and $v^*$
is an edge of $G_1$ for each $v\in\mathrm{V}(D(\phi,G_1))$.
\end{definition}

\begin{definition}[Def. 5.1 \cite{minc1}] \label{def5.1}
A graph $G^{\prime}$ {\it subdivides} a graph $G$ if
$\mathrm{V}(G)\subseteq \mathrm{V}(G^{\prime})$ and for each
$e\in\mathrm{E}(G)$ there is an arc $(e,G^{\prime})\subseteq
G^{\prime}$ so that

i)  $(e,G^{\prime})$ has the same endpoints as $e$,

ii)  $(d,G^{\prime})\bigcap (e,G^{\prime})=d\bigcap e$ for
$d,e\in\mathrm{E}(G)$ and $d\not=e$, and

iii)  for each $v\in\mathrm{V}(G^{\prime})$, $v\in (e,G^{\prime})$
for some $e\in\mathrm{E}(G)$, and for each
$e^{\prime}\in\mathrm{E}(G^{\prime})$, $e^{\prime}$ is an edge of
$(e,G^{\prime})$ for some $e\in\mathrm{E}(G)$.

If $v\in\mathrm{V}(G)$ and $e\in\mathrm{E}(G)$ so that $v\in e$,
then $(v,e,G^{\prime})$ denotes the edge of $(e,G^{\prime})$
containing $v$.
\end{definition}

\begin{proposition}[Prop. 5.2 \cite{minc1}] \label{prop5.2}
If $G^{\prime}$ is a graph subdividing a graph $G$ and
$G^{\prime\prime}$ is a graph subdividing $G^{\prime}$, then
$G^{\prime\prime}$ subdivides $G$.
\end{proposition}

\begin{definition}[Def. 5.3 \cite{minc1}] \label{def5.3}
Let $G_0^{\prime}$ and $G_1^{\prime}$ be graphs subdividing the
graphs $G_0$ and $G_1$, respectively, and $\phi:G_1\longrightarrow
G_0$ and ${\phi}^{\prime}:G_1^{\prime}\longrightarrow
G_0^{\prime}$ be simplicial maps so that
${\phi}^{\prime}(v)=\phi(v)$ for all $v\in\mathrm{V}(G_1)$.  If
for all $e\in\mathrm{E}(G_1)$, $(e,G_1^{\prime})$ is an edge
whenever $\phi(e)$ is degenerate, and ${\phi}^{\prime}$ restricted
to $(e,G_1^{\prime})$ is an isomorphism onto
$(\phi(e),G_0^{\prime})$ whenever $\phi(e)$ is nondegenerate, then
${\phi}^{\prime}$ is a {\it subdivision of $\phi$ matching
$G_0^{\prime}$}.
\end{definition}

\begin{proposition}[Prop. 5.4 \cite{minc1}] \label{prop5.4}
Let $\phi:G_1\longrightarrow G_0$ be a simplicial map between
graphs.  Let $G_0^{\prime}$ be a graph subdividing $G_0$.  Then
there is a subdivision ${\phi}^{\prime}$ of $\phi$ matching
$G_0^{\prime}$, unique up to an isomorphism.
\end{proposition}

\begin{definition}[Def. 5.5 \cite{minc1}] \label{def5.5}
Let $G$ be a graph and
$S:\mathrm{V}(G)\longrightarrow\{X\subseteq\mathrm{E}(G):X\not=\emptyset\}$.
Then $S$ is an {\it edge selection on $G$} if for all
$v\in\mathrm{V}(G)$, $v\in e$ for each $e\in S(v)$.
\end{definition}

\begin{definition}[Def. 5.5 \cite{minc1}] \label{def5.5.1}
Let $G_0$ and $G_1$ be graphs, $S$ be an edge selection on $G_1$,
$G_1^{\prime}$ be a graph subdividing $G_1$, and
$\phi:G_1^{\prime}\longrightarrow G_0$ be simplicial.  If there
exist a graph $H_1$ and a simplicial map
$\lambda:H_1\longrightarrow D(\phi,G_1^{\prime})$ so that $H_1$ is
a subdivision of $G_1$ and $\lambda$ is an isomorphism with

i)  $(v,e,G_1^{\prime})\subseteq{[\lambda(v)]}^*$ for each
$v\in\mathrm{V}(G_1)$ and each $e\in S(v)$ and

ii)  ${[\lambda(v)]}^*\subseteq(e,G_1^{\prime})$ for each
$e\in\mathrm{E}(G_1)$ and
$v\in\mathrm{V}((e,H_1))\setminus\mathrm{V}(G_1)$,

\noindent then {\it $\phi$ is consistent on $S$} and $\lambda$ is
a {\it consistency isomorphism}.
\end{definition}

\begin{definition}[Def. 5.7 \cite{minc1}] \label{def5.7}
Let $G_0$ and $G_1$ be graphs with edge selections $S_0$ and
$S_1$, respectively, $G_1^{\prime}$ be a graph subdividing $G_1$,
and $\phi:G_1^{\prime}\longrightarrow G_0$ be simplicial. Then,
$\phi$ {\it preserves} $(S_0,S_1)$ provided that

i)  $\phi((v,e,G_1^{\prime}))\in S_0(\phi(v))$ for each
$v\in\mathrm{V}(G_1)$ and each $e\in S_1(v)$ and

ii)  for each $e,e^{\prime}\in\mathrm{E}(G_1^{\prime})$ so that
$e\not= e^{\prime}$ and $e\bigcap e^{\prime}=v$ for some
$v\in\mathrm{V}( G_1^{\prime})$, either $\phi(e)\in S_0(\phi(v))$
or $\phi(e^{\prime})\in S_0(\phi(v))$.
\end{definition}

\begin{definition}[Def. 5.10 \cite{minc1}] \label{def5.10}
Let $\{G_i\}_{i=0}^{\infty}$ and $\{G_i^{\prime}\}_{i=1}^{\infty}$
be collections of graphs and $\Phi=\{\phi_i\}_{i=0}^{\infty}$ be a
collection of simplicial maps so that for each
$i\in\{1,2,\ldots\}$, $G_i^{\prime}$ is a subdivision of $G_i$ and
$\phi_{i-1}:G_i^{\prime}\longrightarrow G_{i-1}$.  By Proposition
\ref{prop5.4} there exists a collection of simplicial maps
$\{\psi_i\}_{i=0}^{\infty}$ so that $\psi_0=\phi_0$ and for each
$i\in\{1,2,\ldots\}$, $\psi_i$ is a subdivison of $\phi_i$
matching the domain of $\psi_{i-1}$.  Let $X_0=G_0$, for each
$i\in\{1,2,\ldots\}$, $X_i$ be the domain of $\psi_{i-1}$, and for
$i,j\in\{0,1,\ldots\}$ so that $j>i$ let $\Phi_i^j=\psi_i\circ
\psi_{i+1}\circ\cdots\circ \psi_{j-1}$ and
$\Phi_i^i=\mbox{identity on }X_i$. Then we say {\it
$\varprojlim\invseq[]{X}{\Phi}$ is generated by the sequence
$\Phi$}.

Let $S=\{S_i\}_{i=1}^{\infty}$ where for each
$i\in\{1,2,\ldots\}$, $S_i$ is an edge selection on $G_i$.  We say
{\it $\Phi$ preserves $S$} if $\phi_i$ preserves $(S_i,S_{i+1})$
for each $i\in\{1,2,\ldots\}$.
\end{definition}

\begin{definition}[Def. 5.10 \cite{minc1}] \label{def5.10.1}
We say two inverse limits $\varprojlim\invseq[]{X}{\Phi}$ and
$\varprojlim\invseq[]{Y}{\Psi}$ with simplicial bonding maps are
{\it isomorphic} if there exists a sequence of isomorphisms
$\{\lambda_i\}_{i=0}^{\infty}$ so that
$\lambda_i:X_i\longrightarrow Y_i$ and
$\lambda_i\circ\Phi_i^j=\Psi_i^j\circ\lambda_j$ for all
$i,j\in\{0,1,\ldots\}$ with $j\ge i$.
\end{definition}

\begin{definition} \label{defwedge}
Let $x_0,x_1,\ldots,x_{m_x}$ for some ${m_x}\in\{1,2,\ldots\}$ be
vertices of some graph $G$ and define $\langle
x_0,x_1,\ldots,x_{m_x} \rangle$ to be the path in $G$ given by
$x_0,x_1,\ldots,x_{m_x}$, respecting order, if and only if $x_i$
and $x_{i+1}$ are adjacent for all $i\in\oneton[0]{{m_x}-1}$.  Let
$y_0,y_1,\ldots,y_{m_y}$ for some ${m_y}\in\{1,2,\ldots\}$ be
vertices of $G$ so that $y_i$ and $y_{i+1}$ are adjacent for all
$i\in\oneton[0]{{m_y}-1}$, and define $\langle
x_0,x_1,\ldots,x_{m_x} \rangle\bigvee\langle
y_0,y_1,\ldots,y_{m_y} \rangle=\langle
x_0,x_1,\ldots,x_{m_x},y_1,\ldots,y_{m_y} \rangle$, if and only if
$x_{m_x}=y_0$.
\end{definition}

\newpage
\section{The Bonding Map \nphi}

In this section, a single bonding map \nphi{} between
simple-$n$-ods is defined for an inverse limit \K. Characteristics
of the map ensure that every proper nondegenerate subcontinuum of
\K{} is an arc (Proposition \ref{claim1}) and satisfy certain
conditions essential in the implementation of the previously
mentioned strategy of Minc. In particular, \nphi{} is well-behaved
in the sense of inducing no folding in \dnphi{} (Proposition
\ref{claim4}) and with respect to satisfying conditions sufficient
for demonstrating a ``subdivision'' of \K{} as being isomorphic
with the dual of \K{} (Propositions \ref{claim3} and
\ref{claim6}).  Also, a certain symmetry is exhibited by the
bonding maps (Proposition \ref{claim9}), used in the next section
to obtain a ``well-behaved'' factoring.

For what follows, let \X{0} be a simple-n-od and
$\arm[1]{0},\arm[2]{0},\ldots,\arm[n]{0}$ be $n$ arcs so that
$\bigcup\limits_{i=1}^{n}\arm0=\X{0}$ and
$\arm0\bigcap\arm[j]{0}=\{v_0\}$ for each distinct
$i,j\in\oneton{n}$, and so $v_0$ is an endpoint of \arm0 for each
$i\in\oneton{n}$.

Let \X{0} be the graph defined in the following table.
\vspace{-.4cm}
\begin{table}[here]
\begin{center}
\begin{tabular}[t]{clc}
arc&\vline&as a subgraph\\\hline \arm0&\vline&\armzeroi{} for each $i\in\oneton{n-2}$\\ \arm[n-1]{0}&\vline&\armzeronminusone\\
\arm[n]{0}&\vline&\armzeron
\end{tabular}
\caption{Definition of \X{0}}
\end{center}
\end{table}

\vspace{-.9cm}Let \X{1} be a subdivision of \X{0}, where \arm1 is
a subdivision of \arm0 for each $i\in\oneton{n}$ defined in the
following table. \vspace{-.4cm}
\begin{table}[here]
\begin{center}
\begin{tabular}[t]{clc}
arc&\vline&as a subgraph\\\hline \arm1&\vline&\armonei\\&\vline&
for each $i\in\oneton{n-3}$\\ \arm[n-2]{1}&\vline&\armonenminustwo\\
\arm[n-1]{1}&\vline&\armonenminusone\\
\arm[n]{1}&\vline&\armonen
\end{tabular}
\caption{Definition of \X{1}}
\end{center}
\end{table}

\newpage
Let $\nphi:\X{1}\longrightarrow\X{0}$ be the light simplicial map
determined by the following table where $\nphi(v)$ is defined for
$v\in\mathrm{V}(\X{1})$.
\begin{table}[here]
\begin{center}
\begin{tabular}[t]{clc}
$v$&\vline&$\nphi(v)$\\\hline \phidef{v}{0}{0}\\ \phidef{v}{i}{n+2} for all $i\in\oneton{n-3}$\\
\phidef{v}{n-2}{n+1}\\ \phidef{v}{n-1}{n}\\ \phidef{v}{n}{n}\\
\phidef{v}{n+1}{n+2}\\ \phidef{v}{n+2}{n+2}\\ \phidef{u}{i}{n-1}
for all $i\in\oneton{n-3}$\\ \phidef{u}{j(n-3)+i}{0} for all
$i\in\oneton{n-3}$\\&\vline& and for all odd $j\in\oneton{2(n-1)-1}$\\
\phidef{u}{j(n-3)+i}{{i-1+\frac{j}{2}}\bmod{n-2}} for all
$i\in\oneton{n-3}$\\&\vline& and for all even $j\in\oneton[2]{2(n-1)-2}$\\
\phidef{u}{2(n-1)(n-3)+i}{n} for all $i\in\oneton{n-3}$\\
\phidef{u}{(2n-1)(n-3)+1}{n-1}\\ \phidef{u}{(2n-1)(n-3)+i}{0} for
all even $i\in\oneton[2]{2(n-1)}$\\
\phidef{u}{(2n-1)(n-3)+3}{n-2}\\
\phidef{u}{(2n-1)(n-3)+i}{\frac{i-3}{2}} for all odd
$i\in\oneton[5]{2(n-1)-1}$\\ \phidef{u}{(2n-1)(n-3)+2(n-1)+1}{n-1}\\
\phidef{u}{(2n-1)(n-3)+2(n-1)+2}{n-1}\\
\phidef{u}{(2n-1)(n-3)+2(n-1)+3}{0}
\end{tabular}
\caption{Definition of \nphi}
\end{center}
\end{table}

\newpage

\begin{figure}[here] \vspace{4.3cm}
\hspace{2.8cm} \setlength{\unitlength}{.85cm}
\begin{picture}(0,0)

\put(9.3,-3.8){\LARGE\ensuremath{{_{3}X}_{1}}}

\color{LimeGreen} \put(6,-1.1){\large\ensuremath{{_3A}_{1}^{2}}}
\color{Red} \put(13.3,-1.1){\large\ensuremath{{_3A}_{1}^{3}}}
\color{RawSienna} \put(9.6,2.9){\large\ensuremath{{_3A}_{1}^{1}}}
\normalcolor \put(14,-1.8){\line(-1,0){8}}
\put(10,-1.8){\line(0,1){4}}
\put(9.5,-1.4){\small\ensuremath{v_0}} \put(10,-1.8){\circle*{.4}}
\put(10.4,2.1){\small\ensuremath{v_1}} \put(10,2.2){\circle*{.4}}
\put(5.8,-2.4){\small\ensuremath{v_4}} \put(6,-1.8){\circle*{.4}}
\put(13.8,-2.4){\small\ensuremath{v_5}}
\put(14,-1.8){\circle*{.4}} \put(6.5,-2.4){\small\ensuremath{v_2}}
\put(6.7,-1.8){\circle*{.4}}
\put(13.1,-2.4){\small\ensuremath{v_3}}
\put(13.3,-1.8){\circle*{.4}}

\put(9.5,1.45){\tiny\ensuremath{u_5}} \put(10,1.5){\circle*{.2}}
\put(8.75,-1.55){\tiny\ensuremath{u_6}}
\put(8.9,-1.8){\circle*{.2}}
\put(7.65,-1.55){\tiny\ensuremath{u_7}}
\put(7.8,-1.8){\circle*{.2}}
\put(10.51,-1.55){\tiny\ensuremath{u_1}}
\put(10.66,-1.8){\circle*{.2}}
\put(11.17,-1.55){\tiny\ensuremath{u_2}}
\put(11.32,-1.8){\circle*{.2}}
\put(11.83,-1.55){\tiny\ensuremath{u_3}}
\put(11.98,-1.8){\circle*{.2}}
\put(12.49,-1.55){\tiny\ensuremath{u_4}}
\put(12.64,-1.8){\circle*{.2}}

\put(-.2,-3.8){\LARGE\ensuremath{{_{3}X}_{0}}}

\color{LimeGreen}
\put(-3.5,-1.1){\large\ensuremath{{_3A}_{0}^{2}}} \color{Red}
\put(3.8,-1.1){\large\ensuremath{{_3A}_{0}^{3}}} \color{RawSienna}
\put(.1,2.9){\large\ensuremath{{_3A}_{0}^{1}}} \normalcolor
\put(4.5,-1.8){\line(-1,0){8}} \put(.5,-1.8){\line(0,1){4}}
\put(0,-1.4){\small\ensuremath{v_0}} \put(.5,-1.8){\circle*{.4}}
\put(.9,2.1){\small\ensuremath{v_1}} \put(.5,2.2){\circle*{.4}}
\put(-3.7,-2.4){\small\ensuremath{v_4}}
\put(-3.5,-1.8){\circle*{.4}}
\put(4.3,-2.4){\small\ensuremath{v_5}}
\put(4.5,-1.8){\circle*{.4}} \put(-3,-2.4){\small\ensuremath{v_2}}
\put(-2.8,-1.8){\circle*{.4}}
\put(3.6,-2.4){\small\ensuremath{v_3}}
\put(3.8,-1.8){\circle*{.4}}

\thicklines \put(7.25,.2){\vector(-1,0){4}} \thinlines
\put(4.95,.7){\LARGE \ensuremath{{_3\phi}}}

\end{picture}

\hspace{2.8cm} \setlength{\unitlength}{.85cm}
\begin{picture}(0,12)

\thicklines %\put(4.8,-6){\Huge \ensuremath{{_3\phi}}}

\color{LimeGreen} \put(-3.3,-3.8){\large
\ensuremath{{_3A}_{0}^{2}}} \put(-2.3,-5){\line(1,0){7.5}}
\put(-2.3,-5){\line(0,1){2.5}} \put(-2.3,-2.5){\line(1,0){6.25}}

\color{Red} \put(13,-3.8){\large \ensuremath{{_3A}_{0}^{3}}}
\put(5.2,-5){\line(1,0){7.5}} \put(12.7,-5){\line(0,1){2.5}}
\put(12.7,-2.5){\line(-1,0){6.25}}

\color{RawSienna} \put(4.95,4.05){\large
\ensuremath{{_3A}_{0}^{1}}} \put(3.95,-2.5){\line(0,1){6.25}}
\put(6.45,-2.5){\line(0,1){6.25}} \put(3.95,3.75){\line(1,0){2.5}}
\normalcolor

\thinlines \put(5.2,-3.75){\circle*{.168}}

\color{RawSienna} \put(-2.25,-3.85){\scriptsize
\ensuremath{{_3A}_{1}^{1}}}

\put(5.2,-3.75){\line(-1,0){6.9}}

\color{Red} \put(12.1,-3.225){\scriptsize
\ensuremath{{_3A}_{1}^{3}}}

\put(4.702,-3.5){\line(-1,0){3.752}}
\put(4.702,-3.5){\line(3,-1){.498}}
\put(.95,-3.5){\line(0,1){.375}}
\put(.95,-3.125){\line(1,0){3.625}}
\put(4.575,-3.125){\line(0,1){6.25}}
\put(4.575,3.125){\line(1,0){1.25}}
\put(5.825,3.125){\line(0,-1){6.25}}
\put(5.825,-3.125){\line(1,0){6.225}}

\color{LimeGreen} \put(12.1,-4.475){\scriptsize
\ensuremath{{_3A}_{1}^{2}}}

\put(4.702,-4){\line(-1,0){3.752}}
\put(4.702,-4){\line(3,1){.498}} \put(.95,-4){\line(0,-1){.375}}
\put(.95,-4.375){\line(1,0){11.1}} \normalcolor

\end{picture}
\end{figure}
\vspace{4cm}
\begin{figure}[here]
\caption{The bonding map $_3\phi$}
\end{figure}

\newpage

\begin{figure}[here] \vspace{3.4cm}
\setlength{\unitlength}{.85cm} \hspace{2.8cm}
\begin{picture}(0,0)

\put(9.3,-8.3){\LARGE\ensuremath{{_{4}X}_{1}}}

\color{LimeGreen} \put(6,-1.1){\large\ensuremath{{_4A}_{1}^{3}}}
\color{Orchid} \put(13.3,-1.1){\large\ensuremath{{_4A}_{1}^{1}}}
\color{RawSienna} \put(9.6,2.9){\large\ensuremath{{_4A}_{1}^{2}}}
\color{Red} \put(9.6,-6.8){\large\ensuremath{{_4A}_{1}^{4}}}
\normalcolor \put(14,-1.8){\line(-1,0){8}}
\put(10,-1.8){\line(0,1){4}} \put(10,-1.8){\line(0,-1){4}}
\put(9.5,-1.4){\small\ensuremath{v_0}} \put(10,-1.8){\circle*{.4}}
\put(10.4,2.1){\small\ensuremath{v_2}} \put(10,2.2){\circle*{.4}}
\put(5.8,-2.4){\small\ensuremath{v_5}} \put(6,-1.8){\circle*{.4}}
\put(13.8,-2.4){\small\ensuremath{v_1}}
\put(14,-1.8){\circle*{.4}} \put(6.5,-2.4){\small\ensuremath{v_3}}
\put(6.7,-1.8){\circle*{.4}}
\put(10.4,-5.2){\small\ensuremath{v_4}}
\put(10,-5.1){\circle*{.4}}
\put(10.4,-5.9){\small\ensuremath{v_6}}
\put(10,-5.8){\circle*{.4}}

\put(9.4,1.45){\tiny\ensuremath{u_{14}}}
\put(10,1.5){\circle*{.2}}
\put(8.75,-1.55){\tiny\ensuremath{u_{15}}}
\put(8.9,-1.8){\circle*{.2}}
\put(7.65,-1.55){\tiny\ensuremath{u_{16}}}
\put(7.8,-1.8){\circle*{.2}}

\put(10.321,-1.55){\tiny\ensuremath{u_1}}
\put(10.471,-1.8){\circle*{.2}}
\put(10.792,-1.55){\tiny\ensuremath{u_2}}
\put(10.942,-1.8){\circle*{.2}}
\put(11.263,-1.55){\tiny\ensuremath{u_3}}
\put(11.413,-1.8){\circle*{.2}}
\put(11.734,-1.55){\tiny\ensuremath{u_4}}
\put(11.884,-1.8){\circle*{.2}}
\put(12.205,-1.55){\tiny\ensuremath{u_5}}
\put(12.355,-1.8){\circle*{.2}}
\put(12.676,-1.55){\tiny\ensuremath{u_6}}
\put(12.826,-1.8){\circle*{.2}}
\put(13.15,-1.55){\tiny\ensuremath{u_7}}
\put(13.3,-1.8){\circle*{.2}}

\put(9.5,-2.321){\tiny\ensuremath{u_8}}
\put(10,-2.271){\circle*{.2}}
\put(9.5,-2.792){\tiny\ensuremath{u_9}}
\put(10,-2.742){\circle*{.2}}
\put(9.4,-3.263){\tiny\ensuremath{u_{10}}}
\put(10,-3.213){\circle*{.2}}
\put(9.4,-3.734){\tiny\ensuremath{u_{11}}}
\put(10,-3.684){\circle*{.2}}
\put(9.4,-4.205){\tiny\ensuremath{u_{12}}}
\put(10,-4.155){\circle*{.2}}
\put(9.4,-4.676){\tiny\ensuremath{u_{13}}}
\put(10,-4.626){\circle*{.2}}

\put(-.2,-8.3){\LARGE\ensuremath{{_{4}X}_{0}}}

\color{LimeGreen}
\put(-3.5,-1.1){\large\ensuremath{{_4A}_{0}^{3}}} \color{Orchid}
\put(3.8,-1.1){\large\ensuremath{{_4A}_{0}^{1}}} \color{RawSienna}
\put(.1,2.9){\large\ensuremath{{_4A}_{0}^{2}}} \color{Red}
\put(.1,-6.8){\large\ensuremath{{_4A}_{0}^{4}}} \normalcolor
\put(4.5,-1.8){\line(-1,0){8}} \put(.5,-1.8){\line(0,1){4}}
\put(.5,-1.8){\line(0,-1){4}} \put(0,-1.4){\small\ensuremath{v_0}}
\put(.5,-1.8){\circle*{.4}} \put(.9,2.1){\small\ensuremath{v_2}}
\put(.5,2.2){\circle*{.4}} \put(-3.7,-2.4){\small\ensuremath{v_5}}
\put(-3.5,-1.8){\circle*{.4}}
\put(4.3,-2.4){\small\ensuremath{v_1}}
\put(4.5,-1.8){\circle*{.4}} \put(-3,-2.4){\small\ensuremath{v_3}}
\put(-2.8,-1.8){\circle*{.4}}
\put(.9,-5.2){\small\ensuremath{v_4}} \put(.5,-5.1){\circle*{.4}}
\put(.9,-5.9){\small\ensuremath{v_6}} \put(.5,-5.8){\circle*{.4}}

\thicklines \put(7.25,.2){\vector(-1,0){4}} \thinlines
\put(4.95,.7){\LARGE \ensuremath{{_4\phi}}}

\end{picture}

\hspace{2.8cm} \setlength{\unitlength}{.85cm}
\begin{picture}(0,8.7)

\thicklines %\put(4.8,-13.229){\Huge \ensuremath{{_4\phi}}}

\color{Red} \put(4.8,-11.929){\large \ensuremath{{_4A}_{0}^{4}}}

\put(3.95,-5){\line(0,-1){6.25}} \put(6.45,-5){\line(0,-1){6.25}}
\put(3.95,-11.25){\line(1,0){2.5}}

\color{LimeGreen} \put(-3.3,-3.8){\large
\ensuremath{{_4A}_{0}^{3}}}

\put(-2.3,-5){\line(1,0){6.25}} \put(-2.3,-5){\line(0,1){2.5}}
\put(-2.3,-2.5){\line(1,0){6.25}}

\color{Orchid} \put(13,-3.8){\large \ensuremath{{_4A}_{0}^{1}}}

\put(12.7,-5){\line(-1,0){6.25}} \put(12.7,-5){\line(0,1){2.5}}
\put(12.7,-2.5){\line(-1,0){6.25}}

\color{RawSienna} \put(4.95,4.05){\large
\ensuremath{{_4A}_{0}^{2}}}

\put(3.95,-2.5){\line(0,1){6.25}}
\put(6.45,-2.5){\line(0,1){6.25}} \put(3.95,3.75){\line(1,0){2.5}}
\normalcolor

\thinlines \put(5.2,-3.75){\circle*{.168}}

\color{RawSienna} \put(-2.25,-3.766){\scriptsize
\ensuremath{{_4A}_{1}^{2}}}

\put(5.2,-3.666){\line(-1,0){6.9}}

\color{Red} \put(5.775,-11.074){\scriptsize
\ensuremath{{_4A}_{1}^{4}}}

\put(4.702,-3.5){\line(-1,0){3.752}}
\put(4.702,-3.5){\line(3,-1){.498}}
\put(.95,-3.5){\line(0,1){.375}}
\put(.95,-3.125){\line(1,0){3.625}}
\put(4.575,-3.125){\line(0,1){6.416}}
\put(4.575,3.291){\line(1,0){1.416}}
\put(5.991,3.291){\line(0,-1){6.25}}
\put(5.991,-2.959){\line(1,0){6.25}}
\put(5.991,-4.541){\line(1,0){6.25}}
\put(12.241,-4.541){\line(0,1){1.582}}
\put(5.991,-4.541){\line(0,-1){6.084}}

\color{LimeGreen} \put(4.275,-11.074){\scriptsize
\ensuremath{{_4A}_{1}^{3}}}

\put(5.2,-3.834){\line(-1,0){4.416}}
\put(.784,-3.834){\line(0,-1){.707}}
\put(.784,-4.541){\line(1,0){3.791}}
\put(4.575,-4.541){\line(0,-1){6.084}}

\color{Orchid} \put(5.075,-11.074){\scriptsize
\ensuremath{{_4A}_{1}^{1}}}

\put(4.702,-4){\line(-1,0){3.752}}
\put(4.702,-4){\line(3,1){.498}} \put(.95,-4){\line(0,-1){.375}}
\put(.95,-4.375){\line(1,0){11.125}}
\put(12.075,-4.375){\line(0,1){1.25}}
\put(12.075,-3.125){\line(-1,0){6.25}}
\put(5.825,-3.125){\line(0,1){6.25}}
\put(5.825,3.125){\line(-1,0){.3125}}
\put(5.5125,3.125){\line(0,-1){7.417}}
\put(5.5125,-4.458){\line(0,-1){6.167}} \normalcolor

\end{picture}
\end{figure}
\begin{figure}\caption{The bonding map ${_4\phi}$}
\end{figure}

\newpage

\begin{figure}[here] \vspace{3.4cm}
\setlength{\unitlength}{.85cm} \hspace{2.8cm}
\begin{picture}(0,0)

\put(9.3,-7.8){\LARGE\ensuremath{{_{5}X}_{1}}}

\color{LimeGreen} \put(6,-1.1){\large\ensuremath{{_5A}_{1}^{4}}}
\color{Orchid} \put(13.3,-1.1){\large\ensuremath{{_5A}_{1}^{2}}}
\color{RawSienna} \put(9.6,2.9){\large\ensuremath{{_5A}_{1}^{3}}}
\color{Red} \put(6.2716,-5.5284){\large\ensuremath{{_5A}_{1}^{5}}}
\color{ProcessBlue}
\put(12.8284,-5.5284){\large\ensuremath{{_5A}_{1}^{1}}}
\normalcolor \put(14,-1.8){\line(-1,0){8}}
\put(10,-1.8){\line(0,1){4}} \put(10,-1.8){\line(-1,-1){2.8284}}
\put(10,-1.8){\line(1,-1){2.8284}}
\put(9.5,-1.4){\small\ensuremath{v_0}} \put(10,-1.8){\circle*{.4}}
\put(10.4,2.1){\small\ensuremath{v_3}} \put(10,2.2){\circle*{.4}}
\put(5.8,-2.4){\small\ensuremath{v_6}} \put(6,-1.8){\circle*{.4}}
\put(13.8,-2.4){\small\ensuremath{v_2}}
\put(14,-1.8){\circle*{.4}} \put(6.5,-2.4){\small\ensuremath{v_4}}
\put(6.7,-1.8){\circle*{.4}}
\put(7.5716,-4.7284){\small\ensuremath{v_7}}
\put(7.1716,-4.6284){\circle*{.4}}
\put(8.0663,-4.2337){\small\ensuremath{v_5}}
\put(7.6663,-4.1337){\circle*{.4}}
\put(12.1284,-4.7284){\small\ensuremath{v_1}}
\put(12.8284,-4.6284){\circle*{.4}}

\put(9.4,1.45){\tiny\ensuremath{u_{27}}}
\put(10,1.5){\circle*{.2}}
\put(8.75,-1.55){\tiny\ensuremath{u_{28}}}
\put(8.9,-1.8){\circle*{.2}}
\put(7.65,-1.55){\tiny\ensuremath{u_{29}}}
\put(7.8,-1.8){\circle*{.2}}

\put(10.2167,-1.65){\tiny\ensuremath{u_{2}}}
\put(10.3667,-1.8){\circle*{.2}}
\put(10.5834,-1.45){\tiny\ensuremath{u_{4}}}
\put(10.7334,-1.8){\circle*{.2}}
\put(10.9501,-1.65){\tiny\ensuremath{u_{6}}}
\put(11.1001,-1.8){\circle*{.2}}
\put(11.3168,-1.45){\tiny\ensuremath{u_{8}}}
\put(11.4668,-1.8){\circle*{.2}}
\put(11.6835,-1.65){\tiny\ensuremath{u_{10}}}
\put(11.8335,-1.8){\circle*{.2}}
\put(12.0502,-1.45){\tiny\ensuremath{u_{12}}}
\put(12.2002,-1.8){\circle*{.2}}
\put(12.4169,-1.65){\tiny\ensuremath{u_{14}}}
\put(12.5669,-1.8){\circle*{.2}}
\put(12.7836,-1.45){\tiny\ensuremath{u_{16}}}
\put(12.9336,-1.8){\circle*{.2}}
\put(13.15,-1.65){\tiny\ensuremath{u_{18}}}
\put(13.3,-1.8){\circle*{.2}}

\put(9.1407,-2.1093){\tiny\ensuremath{u_{19}}}
\put(9.7407,-2.0593){\circle*{.2}}
\put(8.8814,-2.3686){\tiny\ensuremath{u_{20}}}
\put(9.4814,-2.3186){\circle*{.2}}
\put(8.6221,-2.6279){\tiny\ensuremath{u_{21}}}
\put(9.2221,-2.5779){\circle*{.2}}
\put(8.3628,-2.8872){\tiny\ensuremath{u_{22}}}
\put(8.9628,-2.8372){\circle*{.2}}
\put(8.1035,-3.1465){\tiny\ensuremath{u_{23}}}
\put(8.7035,-3.0965){\circle*{.2}}
\put(7.8442,-3.4058){\tiny\ensuremath{u_{24}}}
\put(8.4442,-3.3558){\circle*{.2}}
\put(7.5849,-3.6651){\tiny\ensuremath{u_{25}}}
\put(8.1849,-3.6151){\circle*{.2}}
\put(7.3256,-3.8744){\tiny\ensuremath{u_{26}}}
\put(7.9256,-3.8744){\circle*{.2}}

\put(10.4093,-2.1093){\tiny\ensuremath{u_{1}}}
\put(10.2593,-2.0593){\circle*{.2}}
\put(10.6686,-2.3686){\tiny\ensuremath{u_{3}}}
\put(10.5186,-2.3186){\circle*{.2}}
\put(10.9279,-2.6279){\tiny\ensuremath{u_{5}}}
\put(10.7779,-2.5779){\circle*{.2}}
\put(11.1872,-2.8872){\tiny\ensuremath{u_{7}}}
\put(11.0372,-2.8372){\circle*{.2}}
\put(11.4465,-3.1465){\tiny\ensuremath{u_{9}}}
\put(11.2965,-3.0965){\circle*{.2}}
\put(11.7058,-3.4058){\tiny\ensuremath{u_{11}}}
\put(11.5558,-3.3558){\circle*{.2}}
\put(11.9651,-3.6651){\tiny\ensuremath{u_{13}}}
\put(11.8151,-3.6151){\circle*{.2}}
\put(12.2244,-3.9244){\tiny\ensuremath{u_{15}}}
\put(12.0744,-3.8744){\circle*{.2}}
\put(12.4837,-4.1837){\tiny\ensuremath{u_{17}}}
\put(12.3337,-4.1337){\circle*{.2}}

\put(-.2,-7.8){\LARGE\ensuremath{{_{5}X}_{0}}}

\color{LimeGreen}
\put(-3.5,-1.1){\large\ensuremath{{_5A}_{0}^{4}}} \color{Orchid}
\put(3.8,-1.1){\large\ensuremath{{_5A}_{0}^{2}}} \color{RawSienna}
\put(.1,2.9){\large\ensuremath{{_5A}_{0}^{3}}} \color{Red}
\put(-3.2284,-5.5284){\large\ensuremath{{_5A}_{0}^{5}}}
\color{ProcessBlue}
\put(3.3284,-5.5284){\large\ensuremath{{_5A}_{0}^{1}}}
\normalcolor \put(4.5,-1.8){\line(-1,0){8}}
\put(.5,-1.8){\line(0,1){4}} \put(.5,-1.8){\line(-1,-1){2.8284}}
\put(.5,-1.8){\line(1,-1){2.8284}}
\put(0,-1.4){\small\ensuremath{v_0}} \put(.5,-1.8){\circle*{.4}}
\put(.9,2.1){\small\ensuremath{v_3}} \put(.5,2.2){\circle*{.4}}
\put(-3.7,-2.4){\small\ensuremath{v_6}}
\put(-3.5,-1.8){\circle*{.4}}
\put(4.3,-2.4){\small\ensuremath{v_2}}
\put(4.5,-1.8){\circle*{.4}} \put(-3,-2.4){\small\ensuremath{v_4}}
\put(-2.8,-1.8){\circle*{.4}}
\put(-1.9284,-4.7284){\small\ensuremath{v_7}}
\put(-2.3284,-4.6284){\circle*{.4}}
\put(-1.4337,-4.2337){\small\ensuremath{v_5}}
\put(-1.8337,-4.1337){\circle*{.4}}
\put(2.6284,-4.7284){\small\ensuremath{v_1}}
\put(3.3284,-4.6284){\circle*{.4}}

\thicklines \put(7.25,.2){\vector(-1,0){4}} \thinlines
\put(4.95,.7){\LARGE \ensuremath{{_5\phi}}}

\end{picture}

\setlength{\unitlength}{.88cm} \hspace{2.8cm}
\begin{picture}(0,10)

\thicklines %\put(4.8,-11.729){\Huge \ensuremath{{_5\phi}}}

\color{Red} \put(-.92,-10.629){\large \ensuremath{{_5A}_{0}^{5}}}

\put(3.691,-5){\line(-1,-1){4.42}}
\put(5.2,-6.509){\line(-1,-1){4.42}}
\put(-.729,-9.42){\line(1,-1){1.509}}

\color{ProcessBlue} \put(10.32,-10.629){\large
\ensuremath{{_5A}_{0}^{1}}}

\put(5.2,-6.509){\line(1,-1){4.42}}
\put(6.709,-5){\line(1,-1){4.42}}
\put(11.129,-9.42){\line(-1,-1){1.509}}

\color{LimeGreen} \put(-3.3,-3.8){\large
\ensuremath{{_5A}_{0}^{4}}}

\put(-2.3,-5){\line(0,1){2.5}} \put(-2.3,-5){\line(1,0){5.991}}
\put(-2.3,-2.5){\line(1,0){6.25}}

\color{Orchid} \put(13,-3.8){\large \ensuremath{{_5A}_{0}^{2}}}

\put(12.7,-5){\line(0,1){2.5}} \put(12.7,-5){\line(-1,0){5.991}}
\put(12.7,-2.5){\line(-1,0){6.25}}

\color{RawSienna} \put(4.95,4.05){\large
\ensuremath{{_5A}_{0}^{3}}}

\put(3.95,-2.5){\line(0,1){6.25}}
\put(6.45,-2.5){\line(0,1){6.25}} \put(3.95,3.75){\line(1,0){2.5}}
\normalcolor

\thinlines \put(5.2,-3.75){\circle*{.168}}

\color{RawSienna} \put(-2.25,-3.766){\scriptsize
\ensuremath{{_5A}_{1}^{3}}}

\put(5.2,-3.666){\line(-1,0){6.9}}

\color{Red} \put(.031,-10.02){\scriptsize
\ensuremath{{_5A}_{1}^{5}}}

\put(4.702,-3.5){\line(-1,0){3.752}}
\put(4.702,-3.5){\line(3,-1){.498}}
\put(.95,-3.5){\line(0,1){.375}}
\put(.95,-3.125){\line(1,0){3.625}}
\put(4.575,-3.125){\line(0,1){6.25}}
\put(4.575,3.125){\line(1,0){1.25}}
\put(5.825,3.125){\line(0,-1){9.009}}
\put(5.825,-5.884){\line(1,-1){3.795}}
\put(9.62,-9.679){\line(1,1){.4}}
\put(10.02,-9.279){\line(-1,1){3.795}}
\put(6.225,-5.484){\line(0,1){.943}}
\put(6.225,-4.541){\line(1,0){5.85}}
\put(12.075,-4.541){\line(0,1){1.416}}
\put(12.075,-3.125){\line(-1,0){5.459}}
\put(6.616,-3.125){\line(-1,-1){.7323}}
\put(5.7663,-3.9747){\line(-1,-1){5.4453}}
\put(.321,-9.42){\line(0,-1){.3}}

\color{ProcessBlue} \put(-.269,-9.72){\scriptsize
\ensuremath{{_5A}_{1}^{1}}}

\put(5.2,-3.834){\line(-1,0){4.416}}
\put(.784,-3.834){\line(0,-1){.707}}
\put(.784,-4.541){\line(1,0){4.333}}
\put(5.283,-4.541){\line(1,0){.376}}
\put(5.659,-4.541){\line(0,-1){1.509}}
\put(5.659,-6.05){\line(1,-1){3.9023}}
\put(9.6787,-9.9523){\line(1,1){.6733}}
\put(10.352,-9.279){\line(-1,1){3.961}}
\put(6.391,-5.318){\line(0,1){.694}}
\put(6.391,-4.624){\line(1,0){5.933}}
\put(12.324,-4.624){\line(0,1){1.748}}
\put(12.324,-2.876){\line(-1,0){6.084}}
\put(6.24,-2.876){\line(0,1){6.25}}
\put(6.24,3.374){\line(-1,0){.6025}}
\put(5.6375,3.374){\line(0,-1){.166}}
\put(5.6375,3.042){\line(0,-1){6.792}}
\put(5.6375,-3.75){\line(-1,-1){.5663}}
\put(5.0125,-4.4165){\line(0,-1){.083}}
\put(4.9538,-4.5997){\line(-1,-1){4.8}}

\color{LimeGreen} \put(-.529,-9.42){\scriptsize
\ensuremath{{_5A}_{1}^{4}}}

\put(5.2,-3.75){\line(-1,0){4.499}}
\put(.701,-3.75){\line(0,-1){.874}}
\put(.701,-4.624){\line(1,0){3.874}}
\put(4.575,-4.624){\line(-1,-1){4.546}}

\color{Orchid} \put(.431,-10.42){\scriptsize
\ensuremath{{_5A}_{1}^{2}}}

\put(4.702,-4){\line(-1,0){3.752}}
\put(4.702,-4){\line(3,1){.498}} \put(.95,-4){\line(0,-1){.375}}
\put(.95,-4.375){\line(1,0){4.333}}
\put(5.449,-4.375){\line(1,0){.3345}}
\put(5.8665,-4.375){\line(1,0){6.1255}}
\put(12.158,-4.375){\line(1,0){.083}}
\put(12.241,-4.375){\line(0,1){1.416}}
\put(12.241,-2.959){\line(-1,0){6.084}}
\put(6.157,-2.959){\line(0,1){6.25}}
\put(6.157,3.291){\line(-1,0){.249}}
\put(5.908,3.291){\line(0,-1){7.041}}
\put(5.908,-3.916){\line(0,-1){.376}}
\put(5.908,-4.458){\line(0,-1){1.343}}
\put(5.908,-5.801){\line(1,-1){3.712}}
\put(9.62,-9.513){\line(0,-1){.083}}
\put(9.62,-9.762){\line(0,-1){.415}}
\put(9.62,-10.177){\line(-1,1){4.42}}
\put(5.2,-5.757){\line(-1,-1){4.32}} \normalcolor

\end{picture}
\end{figure}
\begin{figure}\caption{The bonding map ${_5\phi}$}
\end{figure}

\newpage

Let $\K=\varprojlim\invseq{X}{\Phi}$ be generated by
$\{\nphi,\nphi,\ldots\}$ (Definition \ref{def5.10}).

\begin{proposition} \label{claim1}
Every proper nondegenerate subcontinuum of \K{} is an arc.
\end{proposition}

\begin{proof}
Let $W$ be a nondegenerate subcontinuum of \K, and suppose $W$ is
not an arc.  Since \nphi{} is simplicial, there exists
$v\in\mathrm{V}(\X{1})$ so that $v\in\pi_j(W)$ for infinitely many
$j$.  Then, $v_s\in\pi_j(W)$ for all $j$, where $s=0$, $s=n$, or
$s=n+2$.

Case 1:  Suppose $s=n$ or $s=n+2$.  Since $W$ is not an arc and
\nphi{} embeds $\langle u_{(2n-1)(n-3)+2(n-1)},v_n,v_{n+2}
\rangle$ onto $\langle v_0,v_n,v_{n+2} \rangle$, then
$(\X{j}\setminus ((\langle u_{(2n-1)(n-3)+2(n-1)},v_n \rangle,
\linebreak[0]\X{j})\bigcup(\langle v_n,v_{n+2}
\rangle,\X{j})))\bigcap\pi_j(W)\not=\emptyset$ for infinitely many
$j$ (Definition \ref{def5.1}), implying $(\langle
u_{(2n-1)(n-3)+2(n-1)},v_n \rangle, \X{j})\subseteq\pi_j(W)$ for
infinitely many $j$.  Since $\nPhi{2}{0}((\linebreak[0]\langle
u_{(2n-1)(n-3)+2(n-1)},v_n \rangle, \X{2}))=\X{0}$, then
$\X{j}\subseteq\pi_j(W)$ for all $j$, implying $W=\K$.

Case 2:  Suppose $s=0$.  Since $W$ is not an arc and \nphi{}
embeds $\langle v_0,\linebreak u_{(2n-1)(n-3)+2(n-1)+2} \rangle$
onto $\langle v_0,v_{n-1}\rangle$, then $(\X{j}\setminus(\langle
v_0,u_{(2n-1)(n-3)+2(n-1)+2}
\rangle,\X{j}))\bigcap\linebreak[0]\pi_j(W)\not=\emptyset$ for
infinitely many $j$.  If
$\pi_j(W)\bigcap(\bigcup\limits_{i=1}^{n-3}(\langle v_0,u_i
\rangle,\X{j})\bigcup(\langle v_0,u_{(2n-1)(n-3)+1}
\rangle,\linebreak[0]\X{j})\bigcup (\langle
v_0,u_{(2n-1)(n-3)+2(n-1)+1}
\rangle,\X{j})\setminus\{v_0\})\not=\emptyset$ for some $j$, then
$u_{(2n-1)(n-3)+2(n-1)+2}\linebreak[0]\in\pi_j(W)$. Thus,
$(\langle v_0,u_{(2n-1)(n-3)+2(n-1)+2}
\rangle,\X{j})\subseteq\pi_j(W)$ for infinitely many $j$.  Since
$\nPhi{4}{0}((\langle v_0,u_{(2n-1)(n-3)+2(n-1)+2}
\rangle,\X{4}))=\X{0}$, then $\X{j}\subseteq\pi_j(W)$ for each
$j$, implying $W=\K$.
\end{proof}

\begin{proposition} \label{claim1.5}
\K{} is atriodic.
\end{proposition}

\begin{proof}
The claim follows from Proposition \ref{claim1} and Theorem 3 in
\cite{ingram1}.
\end{proof}

\begin{proposition} \label{claim2}
Every proper nondegenerate subcontinuum of \K{} containing
\branch{} is an arc having \branch{} as an endpoint.
\end{proposition}

\begin{proof}
Let $W\subseteq\K$ be a proper nondegenerate subcontinuum so that
$\branch\in W$.  By Proposition \ref{claim1}, $W$ is an arc. Since
$W$ is proper and $\nPhi{3}{0}((\langle v_0,v_{n-1}
\rangle,\X{3}))=\X{0}$, there exists a nonnegative integer $N$ so
that $v_{n-1}\notin\pi_j(W)$ for all $j\ge N$.  Then,
$\pi_j(W)\bigcap(\bigcup\limits_{i=1}^{n-3}(\langle v_0,u_i
\rangle,\X{j})\bigcup(\langle v_0,u_{(2n-1)(n-3)+1}
\rangle,\X{j})\bigcup(\langle v_0,u_{(2n-1)(n-3)+2(n-1)+1}
\rangle,\linebreak\X{j})\setminus\{v_0\})=\emptyset$ for all $j\ge
N$, implying $\pi_j(W)\subseteq(\langle v_0,v_{n-1}
\rangle,\X{j})$ for all $j\ge N$.  Since \nphi{} embeds $\langle
v_0,u_{(2n-1)(n-3)+2(n-1)+2} \rangle$ onto $\langle v_0,v_{n-1}
\rangle$ and $\nphi(v_0)=v_0$, then \branch{} is an endpoint of
the arc $W$.
\end{proof}

Let $\nS:\mathrm{V}(\X{0})\longrightarrow\{\mbox{nonempty subsets
of }\mathrm{E}(\X{0})\}$ be an edge selection (Definition
\ref{def5.5}) on \X{0} defined as:  \nSdef{i}{0}{i} for all
$i\in\oneton{n}$, \nSdef{n+1}{n-1}{n+1}, \nSdef{n+2}{n}{n+2}, and
$\nS(v_0)=\{\langle v_0,v_i \rangle:i\in\oneton{n-1}\}$.

\begin{proposition} \label{claim3}
\nphi{} preserves $(\nS,\nS)$ (Definition \ref{def5.7}).
\end{proposition}

\begin{proof}
i)  (Recall Definition \ref{def5.1})
\begin{table}[here]
\begin{center}
\begin{tabular}[t]{cl}
$\imphisubedge{i}{v}{0}{v}{i}$&= $\imphiedge{u}{2(n-1)(n-3)+i}{v}{i}$ = $\edge{v}{n}{v}{n+2}\in\imnS{n+2}$\\
&= $\imnSphi{v}{i}$ for all $i\in\oneton{n-3}$,\\
$\imphisubedge{n-2}{v}{0}{v}{n-2}$&=
$\imphiedge{u}{(2n-1)(n-3)+2(n-1)+1}{v}{n-2}$ =
$\edge{v}{n-1}{v}{n+1}\in$\\& $\imnS{n+1}$ = $\imnSphi{v}{n-2}$,\\
$\imphisubedge{n-1}{v}{0}{v}{n-1}$&= $\imphiedge{u}{(2n-1)(n-3)+2(n-1)+3}{v}{n-1}$ = $\edge{v}{0}{v}{n}\in\imnS{n}$\\&= $\imnSphi{v}{n-1}$,\\
$\imphisubedge{n+1}{v}{n-1}{v}{n+1}$&= $\imphiedge{v}{n-1}{v}{n+1}$ = $\edge{v}{n}{v}{n+2}\in\imnS{n+2}$\\&= $\imnSphi{v}{n+1}$,\\
$\imphisubedge{n}{v}{0}{v}{n}$&= $\imphiedge{u}{(2n-1)(n-3)+2(n-1)}{v}{n}$ = $\edge{v}{0}{v}{n}\in\imnS{n}$\\&= $\imnSphi{v}{n}$,\\
$\imphisubedge{n+2}{v}{n}{v}{n+2}$&= $\imphiedge{v}{n}{v}{n+2}$ = $\edge{v}{n}{v}{n+2}\in\imnS{n+2}$\\&= $\imnSphi{v}{n+2}$,\\
$\imphisubedge{0}{v}{0}{v}{i}$&= $\imphiedge{v}{0}{u}{i}$ = $\edge{v}{0}{v}{n-1}\in\imnS{0}$\\&= $\imnSphi{v}{0}$ for all $i\in\oneton{n-3}$,\\
$\imphisubedge{0}{v}{0}{v}{n-2}$&= $\imphiedge{v}{0}{u}{(2n-1)(n-3)+2(n-1)+1}$ = $\edge{v}{0}{v}{n-1}\in\imnS{0}$\\&= $\imnSphi{v}{0}$, and\\
$\imphisubedge{0}{v}{0}{v}{n-1}$&=
$\imphiedge{v}{0}{u}{(2n-1)(n-3)+2(n-1)+2}$ =
$\edge{v}{0}{v}{n-1}\in\imnS{0}$\\&= $\imnSphi{v}{0}.$
\end{tabular}
\end{center}
\end{table}
\newpage
ii)\\
$\imphiedge{v}{0}{u}{i} = \edge{v}{0}{v}{n-1}\in\imnS{n-1} =
\imnSphi{u}{i}$ for all $i\in\oneton{n-3}$,\\
$\imphiedge{u}{(j-1)(n-3)+i}{u}{j(n-3)+i} =
\imphiedge{u}{j(n-3)+i}{u}{(j+1)(n-3)+i} =
\edge{v}{0}{v}{{i-1+\frac{j}{2}}\bmod{n-2}}\in$\\\indent\indent$\imnS{0}\bigcap\imnS{{i-1+\frac{j}{2}}\bmod{n-2}}
= \imnSphi{u}{(j-1)(n-3)+i}\bigcap\imnSphi{u}{(j+1)(n-3)+i}\bigcap$\\
\indent\indent$\imnSphi{u}{j(n-3)+i}$ for all $i\in\oneton{n-3}$
and for all even\\\indent\indent$j\in\oneton[2]{2(n-1)-2}$,\\
$\imphiedge{u}{(2(n-1)-1)(n-3)+i}{u}{2(n-1)(n-3)+i} =
\edge{v}{0}{v}{n}\in\imnS{n} = \imnSphi{u}{2(n-1)(n-3)+i}$
for\\\indent\indent all $i\in\oneton{n-3}$,\\
$\imphiedge{v}{0}{u}{(2n-1)(n-3)+1} = \edge{v}{0}{v}{n-1}\in\imnS{n-1} = \imnSphi{u}{(2n-1)(n-3)+1}$,\\
$\imphiedge{u}{(2n-1)(n-3)+2}{u}{(2n-1)(n-3)+3} =
\edge{v}{0}{v}{n-2}\in\imnS{0}\bigcap\imnS{n-2} =$\\
\indent\indent$\imnSphi{u}{(2n-1)(n-3)+2}\bigcap\imnSphi{u}{(2n-1)(n-3)+3}$,\\
$\imphiedge{u}{(2n-1)(n-3)+i-1}{u}{(2n-1)(n-3)+i} =
\imphiedge{u}{(2n-1)(n-3)+i}{u}{(2n-1)(n-3)+i+1} =
\edge{v}{0}{v}{\frac{i-3}{2}}$\\\indent\indent$\in\imnS{0}\bigcap\imnS{\frac{i-3}{2}}
= \imnSphi{u}{(2n-1)(n-3)+i-1}\bigcap\imnSphi{u}{(2n-1)(n-3)+i+1}\bigcap$\\
\indent\indent$\imnSphi{u}{(2n-1)(n-3)+i}$ for all odd $i\in\oneton[5]{2(n-1)-1}$,\\
$\imphiedge{v}{0}{u}{(2n-1)(n-3)+2(n-1)+1} =
\edge{v}{0}{v}{n-1}\in\imnS{n-1} =
\imnSphi{u}{(2n-1)(n-3)+2(n-1)+1}$,\\\indent\indent and\\
$\imphiedge{v}{0}{u}{(2n-1)(n-3)+2(n-1)+2} =
\imphiedge{u}{(2n-1)(n-3)+2(n-1)+2}{u}{(2n-1)(n-3)+2(n-1)+3} =$\\
\indent\indent$\edge{v}{0}{v}{n-1}\in\imnS{0}\bigcap\imnS{n-1} =
\imnSphi{u}{(2n-1)(n-3)+2(n-1)+3}\bigcap$\\\indent\indent$\imnSphi{u}{(2n-1)(n-3)+2(n-1)+2}$.\\
With (i) above, condition (ii) of Definition \ref{def5.7} is
satisfied whenever $v\in\mathrm{V}(\X{0})$.
\end{proof}

Let $D(\X{0})$ be as in Definition \ref{def2.1}.  Then,
$\mathrm{V}(D(\X{0}))=\{a_i:\in\oneton{n+2}\}$ and
$\mathrm{E}(D(\X{0}))=\{\edge{a}{i}{a}{j}:\linebreak[0]i,j\linebreak[0]\in\linebreak[0]\oneton{n}\linebreak[0]\mbox{
and
}\linebreak[0]i\not=j\}\linebreak[0]\bigcup\linebreak[0]\{\edge{a}{n-1}{a}{n+1},
\linebreak[0]\edge{a}{n}{a}{n+2}\}$, where
$\dualver{a}{i}=\edge{v}{0}{v}{i}$ for all $i\in\oneton{n}$ and
$\dualver{a}{n+1}=\edge{v}{n-1}{v}{n+1}$ and
$\dualver{a}{n+2}=\edge{v}{n}{v}{n+2}$.

Let $D(\nphi,\X{1})$ be as in Definition \ref{def2.4}.  Then,
$\mathrm{V}(D(\nphi,\X{1}))=\{b_i:i\in\oneton[0]{n(n-3)+n+3}\}$
and
$\mathrm{E}(D(\nphi,\X{1}))=\{\edge{b}{0}{b}{i}:\linebreak[0]i\in\linebreak[0]\oneton{n-3}\linebreak[0]\bigcup\linebreak[0]\{n(n-3)+1,\linebreak[0]n(n-3)+n+1,\linebreak[0]n(n-3)+n+2\}\}
\linebreak[0]\bigcup\linebreak[0]\{\edge{b}{j(n-3)+i}{b}{(j+1)(n-3)+i}:\linebreak[0]i\in\oneton{n-3}\linebreak[0]\mbox{
and
}\linebreak[0]j\in\oneton[0]{n-2}\}\linebreak[0]\bigcup\linebreak[0]\{\edge{b}{n(n-3)+i}{b}{n(n-3)+i+1}:\linebreak[0]i\in\oneton{n-1}\}\linebreak[0]\bigcup
\linebreak[0]\{\edge{b}{n(n-3)+n+2}{b}{n(n-3)+n+3}\}$, where $b^*$
for $b\in\mathrm{V}(D(\nphi,\X{1}))$ is given in the following
table. \vspace{-.5cm}
\begin{table}[here]
\begin{center}
\begin{tabular}[t]{cll}
$\dualver{b}{0}$&=&$\bigcup\limits_{i=1}^{n-3}\edge{v}{0}{u}{i}\bigcup\limits_{i=1}^{n-3}\edge{u}{i}{u}{n-3+i}
\bigcup\edge{v}{0}{u}{(2n-1)(n-3)+1}\bigcup$
\brkedge{u}{(2n-1)(n-3)+1}{u}{(2n-1)(n-3)+2}$\bigcup\edge{v}{0}{u}{(2n-1)(n-3)+2(n-1)+1}\bigcup\edge{v}{0}{u}{(2n-1)(n-3)+2(n-1)+2}$\\&&
$\bigcup\edge{u}{(2n-1)(n-3)+2(n-1)+2}{u}{(2n-1)(n-3)+2(n-1)+3}$\\
$\dualver{b}{j(n-3)+i}$&=&$\edge{u}{(2j+1)(n-3)+i}{u}{(2j+2)(n-3)+i}\bigcup
\edge{u}{(2j+2)(n-3)+i}{u}{(2j+3)(n-3)+i}$\\&&
for each $i\in\oneton{n-3}$ and for each $j\in\oneton[0]{n-3}$\\
$\dualver{b}{(n-2)(n-3)+i}$&=&$\edge{u}{(2(n-1)-1)(n-3)+i}{u}{2(n-1)(n-3)+i}$
for each $i\in\oneton{n-3}$\\
$\dualver{b}{(n-1)(n-3)+i}$&=&$\edge{u}{2(n-1)(n-3)+i}{v}{i}$ for
each $i\in\oneton{n-3}$\\
$\dualver{b}{n(n-3)+i}$&=&$\edge{u}{(2n-1)(n-3)+2i}{u}{(2n-1)(n-3)+2i+1}
\bigcup\edge{u}{(2n-1)(n-3)+2i+1}{u}{(2n-1)(n-3)+2i+2}$\\&&
for each $i\in\oneton{n-2}$\\
$\dualver{b}{n(n-3)+n-1}$&=&$\edge{u}{(2n-1)(n-3)+2(n-1)}{v}{n}$\\
$\dualver{b}{n(n-3)+n}$&=&$\edge{v}{n}{v}{n+2}$\\
$\dualver{b}{n(n-3)+n+1}$&=&$\edge{u}{(2n-1)(n-3)+2(n-1)+1}{v}{n-2}$\\
$\dualver{b}{n(n-3)+n+2}$&=&$\edge{u}{(2n-1)(n-3)+2(n-1)+3}{v}{n-1}$\\
$\dualver{b}{n(n-3)+n+3}$&=&$\edge{v}{n-1}{v}{n+1}$
\end{tabular}
\end{center}
\end{table}
\vspace{-1.26cm}
\begin{table}[here]
\caption{Subgraphs of \X{1} corresponding to vertices of
$D(\nphi,\X{1})$}
\end{table}

Let $B^i\subseteq D(\nphi,\X{1})$ for each $i\in\oneton{n}$ be the
arc given by the following table.
\begin{table}[here]
\vspace{-.3cm}
\begin{center}
\begin{tabular}[t]{clc}
arc&\vline&as a subgraph\\\hline $B^i$&\vline&$\langle
b_0,\linebreak[0]b_i,\linebreak[0]b_{n-3+i},\linebreak[0]b_{2(n-3)+i},\linebreak[0]\ldots,\linebreak[0]b_{(n-1)(n-3)+i}
\rangle$ for all $i\in\oneton{n-3}$\\ $B^{n-2}$&\vline&$\langle
b_0,\linebreak[0]b_{n(n-3)+n+1} \rangle$\\
$B^{n-1}$&\vline&$\langle
b_0,\linebreak[0]b_{n(n-3)+n+2},\linebreak[0]b_{n(n-3)+n+3}
\rangle$\\ $B^n$&\vline&$\langle
b_0,\linebreak[0]b_{n(n-3)+1},\linebreak[0]\ldots,\linebreak[0]b_{n(n-3)+n}
\rangle$
\end{tabular}
\caption{Determination of $D(\nphi,\X{1})$}
\end{center}
\end{table}

Then, $\bigcup\limits_{i=1}^{n}B^i= D(\nphi,\X{1})$ and
$B^i\bigcap B^j=\{b_0\}$ for all distinct $i,j\in\oneton{n}$.

\begin{proposition} \label{claim4}
Let \dnphi{} be as in Definition \ref{def2.4}.  Then, \dnphi{} is
an ultra light simplicial map (Definition \ref{def4.1}).
\end{proposition}

\begin{proof}
$\dnphi:D(\nphi,\X{1})\longrightarrow D(\X{0})$, where
$b\in\mathrm{V}(D(\nphi,\X{1}))$, and the subgraphs forming the
inverse images of edges under \dnphi{} are given in the next
tables. \vspace{-.5cm}
\begin{table}[here]
\begin{center}
\begin{tabular}[t]{clc}
$b$&\vline&$\dnphi(b)$\\\hline \imdnphi{0}{n-1}\\
\imdnphi{j(n-3)+i}{{i-1+\frac{2j+2}{2}}\bmod{n-2}}=$a_{{i+j}\bmod{n-2}}$\\&\vline&
for all $i\in\oneton{n-3}$ and for all $j\in\oneton[0]{n-3}$\\
\imdnphi{(n-2)(n-3)+i}{n} for all $i\in\oneton{n-3}$\\
\imdnphi{(n-1)(n-3)+i}{n+2} for all $i\in\oneton{n-3}$\\
\imdnphi{n(n-3)+1}{n-2}\\
\imdnphi{n(n-3)+i}{\frac{2i+1-3}{2}}=$a_{i-1}$ for all
$i\in\oneton[2]{n-2}$\\ \imdnphi{n(n-3)+n-1}{n}\\
\imdnphi{n(n-3)+n}{n+2}\\ \imdnphi{n(n-3)+n+1}{n+1}\\
\imdnphi{n(n-3)+n+2}{n}\\ \imdnphi{n(n-3)+n+3}{n+2}
\end{tabular}
\caption{Determination of \dnphi}
\end{center}
\end{table}

\vspace{.4cm}
\begin{figure}[here]
\setlength{\unitlength}{1cm} \hspace{1.6cm}
\begin{picture}(0,2)

\thicklines \put(2.6,0){\line(1,0){1.4}}
\put(4.2,0){\line(1,0){2.8}} \put(7.2,0){\line(1,0){1.4}}

\put(2.6,0){\line(-1,0){3}}

\put(2.6,0){\line(1,1){3}}

\put(8.6,0){\line(-1,1){3}}

\put(2.6,-3){\line(-1,-1){2}}

\put(2.6,0){\line(0,-1){3}} \put(2.6,-3){\line(1,0){6}}

\put(8.6,0){\line(0,-1){3}}

\put(2.6,-3){\line(1,2){3}}

\put(8.6,-3){\line(-1,2){3}}

\put(2.6,0){\line(2,-1){1.1}} \put(3.9,-.65){\line(2,-1){4.7}}

\put(8.6,0){\line(-2,-1){1.1}} \put(7.3,-.65){\line(-2,-1){1.6}}
\put(5.5,-1.55){\line(-2,-1){2.9}}

\put(2.9,.1){\small\ensuremath{a_4}} \color{LimeGreen}
\put(2.6,0){\circle*{.4}} \normalcolor

\put(-.6,-.4){\small\ensuremath{a_6}} \color{LimeGreen}
\put(-.4,0){\circle*{.4}} \normalcolor

\put(5.45,2.5){\small\ensuremath{a_3}} \color{RawSienna}
\put(5.6,3){\circle*{.4}} \normalcolor

\put(8,.1){\small\ensuremath{a_2}} \color{Orchid}
\put(8.6,0){\circle*{.4}} \normalcolor

\put(.4,-5.4){\small\ensuremath{a_7}} \color{Red}
\put(.6,-5){\circle*{.4}} \normalcolor

\put(2.4,-3.4){\small\ensuremath{a_5}} \color{Red}
\put(2.6,-3){\circle*{.4}} \normalcolor

\put(8.4,-3.35){\small\ensuremath{a_1}} \color{ProcessBlue}
\put(8.6,-3){\circle*{.4}} \normalcolor

\thinlines \color{Orchid} \put(2.3,.3){\line(1,0){4}}
\put(6.4,.3){\line(1,0){2.5}} \put(8.9,.3){\line(-1,1){2.7}}
\put(6.2,3){\line(1,-2){1.325}} \put(7.575,.25){\line(1,-2){.34}}
\put(7.965,-.53){\line(1,-2){.31}}
\put(8.325,-1.25){\line(1,-2){1.025}}
\put(9.05,-3.3){\oval(.6,.6)[br]}
\put(9.05,-3.6){\line(-1,0){6.05}} \put(3,-3.6){\line(-1,-1){1.4}}
\normalcolor \put(9.05,.2){\tiny\ensuremath{b_2}}
\put(8.9,.3){\circle*{.2}} \put(6.35,2.9){\tiny\ensuremath{b_4}}
\put(6.2,3){\circle*{.2}} \put(8.95,-3.9){\tiny\ensuremath{b_6}}
\put(9.05,-3.6){\circle*{.2}} \put(3,-3.45){\tiny\ensuremath{b_8}}
\put(3,-3.6){\circle*{.2}}
\put(1.5,-5.3){\tiny\ensuremath{b_{10}}}
\put(1.6,-5){\circle*{.2}}

\color{RawSienna} \put(2.3,.3){\line(-1,0){2.4}} \normalcolor
\put(-.2,.5){\tiny\ensuremath{b_{16}}} \put(-.1,.3){\circle*{.2}}

\color{Red} \put(2.3,.3){\line(1,1){2.7}}
\put(5,3){\line(1,-2){2.7}} \put(7.7,-1.8){\oval(1.2,1.2)[br]}
\put(8.3,-1.8){\line(0,1){1.5}} \put(8.3,-.3){\line(-2,-1){1.26}}
\put(6.91,-.995){\line(-2,-1){3.41}}
\put(3.5,-2.7){\line(-1,-1){2.3}} \normalcolor
\put(4.9,3.2){\tiny\ensuremath{b_{11}}} \put(5,3){\circle*{.2}}
\put(3.6,-2.9){\tiny\ensuremath{b_{14}}}
\put(3.5,-2.7){\circle*{.2}}
\put(7.85,-.15){\tiny\ensuremath{b_{13}}}
\put(8.3,-.3){\circle*{.2}} \put(1,-5.3){\tiny\ensuremath{b_{15}}}
\put(1.2,-5){\circle*{.2}}
\put(7.65,-2.2){\tiny\ensuremath{b_{12}}}
\put(7.7,-2.4){\circle*{.2}}

\color{ProcessBlue} \put(2.6,0){\oval(.6,.6)[bl]}
\put(2.3,0){\line(0,1){.3}} \put(2.6,-.3){\line(1,0){.3}}
\put(2.9,-.3){\line(2,-1){2.65}}
\put(5.65,-1.675){\line(2,-1){2.5}}
\put(8.15,-2.925){\line(0,-1){.15}}
\put(8.6,-3.075){\oval(.9,.75)[b]}
\put(9.05,-3.075){\line(0,1){.325}}
\put(9.05,-2.65){\line(0,1){2.65}} \put(9.05,.3){\oval(.6,.6)[r]}
\put(9.05,.6){\line(-1,1){2.7}} \put(6.35,3.3){\line(-1,0){.95}}
\put(5.4,3.3){\line(-1,-2){.25}}
\put(5.1,2.7){\line(-1,-2){1.175}}
\put(3.875,.25){\line(-1,-2){.39}}
\put(3.435,-.63){\line(-1,-2){1.035}}
\put(2.4,-2.7){\line(-1,-1){2.3}} \normalcolor
\put(8.825,-3.025){\tiny\ensuremath{b_{1}}}
\put(9,-3.225){\circle*{.2}}
\put(8.95,.8){\tiny\ensuremath{b_{3}}} \put(9.05,.6){\circle*{.2}}
\put(6.25,3.5){\tiny\ensuremath{b_{5}}}
\put(6.35,3.3){\circle*{.2}} \put(2,-2.8){\tiny\ensuremath{b_{7}}}
\put(2.1,-3){\circle*{.2}} \put(0,-5.3){\tiny\ensuremath{b_{9}}}
\put(.1,-5){\circle*{.2}}

\color{LimeGreen} \put(2.2,.3){\line(0,-1){2.7}}
\put(2.2,-2.4){\line(-1,-1){2.6}} \normalcolor
\put(1.7,-2.4){\tiny\ensuremath{b_{17}}}
\put(2.2,-2.4){\circle*{.2}}
\put(-.6,-5.3){\tiny\ensuremath{b_{18}}}
\put(-.4,-5){\circle*{.2}}

\put(2.1,.5){\tiny\ensuremath{b_0}} \put(2.3,.3){\circle*{.2}}

\put(4.25,-.7){\LARGE\ensuremath{D(\X[5]{0})}}
\put(-.6,1){\LARGE\ensuremath{D(\nphi[5],\X[5]{1})}}

\end{picture}
\end{figure}
\vspace{1.6cm}
\begin{figure}
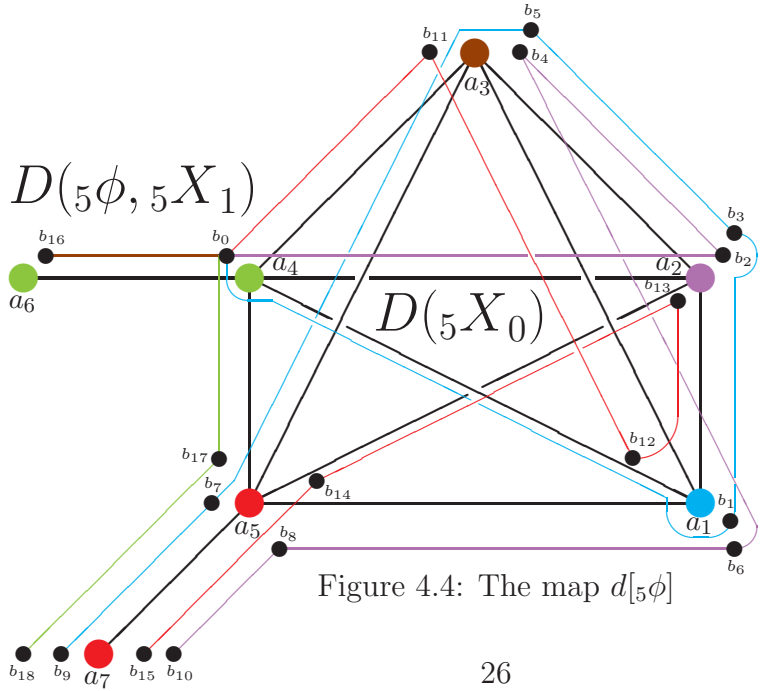

\caption{The map $d[{_5\phi}]$}
\end{figure}

\newpage

\begin{table}[here]
\begin{center}
\begin{tabular}[t]{cll}
\imdnphiinv{n-1}{i}{0}{i} for all $i\in\oneton{n-3}$\\
\imdnphiinv{n-1}{n-2}{0}{n(n-3)+1}\\
\imdnphiinv{n-1}{n+1}{0}{n(n-3)+n+1}\\
\imdnphiinv{n-1}{n}{0}{n(n-3)+n+2}\\
$\dnphi^{-1}(\edge{a}{i}{a}{i+1})$&=&$\bigcup\limits_{j=0}^{i-1}\edge{b}{j(n-3)+i-j}{b}{(j+1)(n-3)+i-j}
\bigcup\limits_{j=1}^{n-4-i}$\brkedge{b}{(n-3-j)(n-3)+i+1+j}{b}{(n-2-j)(n-3)+i+1+j}
$\bigcup\edge{b}{n(n-3)+i+1}{b}{n(n-3)+i+2}$\\&&
for all $i\in\oneton{n-4}$\\
$\dnphi^{-1}(\edge{a}{n-3}{a}{n-2})$&=&$\bigcup\limits_{i=1}^{n-3}\edge{b}{(n-3-i)(n-3)+i}{b}{(n-2-i)(n-3)+i}$\\
\imdnphiinv{i}{n}{(n-3)(n-3)+i+1}{(n-2)(n-3)+i+1} for all
$i\in\oneton{n-4}$\\ \imdnphiinv{n-3}{n}{n(n-3)+n-2}{n(n-3)+n-1}\\
\imdnphiinv{n-2}{n}{(n-3)(n-3)+1}{(n-2)(n-3)+1}\\
\imdnphiinv{n}{n+2}{n(n-3)+n-1}{n(n-3)+n}$\bigcup\edge{b}{n(n-3)+n+2}{b}{n(n-3)+n+3}$\\&&
$\bigcup\limits_{i=1}^{n-3}\edge{b}{(n-2)(n-3)+i}{b}{(n-1)(n-3)+i}$
\end{tabular}
\caption{Edge inverses under \dnphi}
\end{center}
\end{table}
Each subgraph is a union of disjoint edges, and so each component
of an edge inverse mapping onto that edge is an edge.\end{proof}

Let \Y{1} be a subdivision of \X{0}, where \Barm{1} is a
subdivision of \arm{0} for each $i\in\oneton{n}$ defined in the
following table. \vspace{-.1cm}
\begin{table}[here]
\begin{center}
\begin{tabular}[t]{clc}
arc&\vline&as a subgraph\\\hline $\Barm{1}$&\vline&$\Barmonei$ for all $i\in\oneton{n-3}$\\
$\Barm[n-2]{1}$&\vline&$\Barmonenminustwo$\\
$\Barm[n-1]{1}$&\vline&$\Barmonenminusone$\\
$\Barm[n]{1}$&\vline&$\Barmonen$
\end{tabular}
\caption{Definition of \Y{1}}
\end{center}
\end{table}

Define $\nlambda:\Y{1}\longrightarrow D(\nphi,\X{1})$ to be the
simplicial isomorphism determined by the following table where
$v\in\mathrm{V}(\Y{1})$.
\begin{table}[here]
\begin{center}
\begin{tabular}[t]{clc}
$v$&\vline&$\nlambda(v)$\\\hline \llambdadef{v}{0}{0}\\
\llambdadef{v}{i}{(n-1)(n-3)+i} for all
$i\in\oneton{n-3}$\\ \llambdadef{v}{n-2}{n(n-3)+n+1}\\
\llambdadef{v}{n-1}{n(n-3)+n+2}\\ \llambdadef{v}{n}{n(n-3)+n-1}\\
\llambdadef{v}{n+1}{n(n-3)+n+3}\\ \llambdadef{v}{n+2}{n(n-3)+n}\\
\llambdadef{w}{j(n-3)+i}{j(n-3)+i} for all $j\in\oneton[0]{n-2}$
and for all $i\in\oneton{n-3}$\\
\llambdadef{w}{(n-1)(n-3)+i}{n(n-3)+i} for all $i\in\oneton{n-2}$
\end{tabular}
\caption{Definition of \nlambda}
\end{center}
\end{table}

\begin{proposition} \label{claim5}
\nlambda{} is a consistency isomorphism (Definition
\ref{def5.5.1}).
\end{proposition}

\begin{proof} i)
\begin{table}[here]
\begin{center}
\begin{tabular}[t]{cl}
$\subedge{i}{v}{0}{v}{i}$&=$\edge{u}{2(n-1)(n-3)+i}{v}{i}=\dualver{b}{(n-1)(n-3)+i}$\\&=$\dualver{[\nlambda(v_{i})]}{}$
for all $i\in\oneton{n-3}$,\\
$\subedge{n-2}{v}{0}{v}{n-2}$&=$\edge{u}{(2n-1)(n-3)+2(n-1)+1}{v}{n-2}=\dualver{b}{n(n-3)+n+1}$\\&=$\dualver{[\nlambda(v_{n-2})]}{}$,\\
$\subedge{n-1}{v}{0}{v}{n-1}$&=$\edge{u}{(2n-1)(n-3)+2(n-1)+3}{v}{n-1}=\dualver{b}{n(n-3)+n+2}$\\&=$\dualver{[\nlambda(v_{n-1})]}{}$,\\
$\subedge{n+1}{v}{n-1}{v}{n+1}$&=$\edge{v}{n-1}{v}{n+1}=\dualver{b}{n(n-3)+n+3}$\\&=$\dualver{[\nlambda(v_{n+1})]}{}$,\\
$\subedge{n}{v}{0}{v}{n}$&=$\edge{u}{(2n-1)(n-3)+2(n-1)}{v}{n}=\dualver{b}{n(n-3)+n-1}$\\&=$\dualver{[\nlambda(v_{n})]}{}$,
\end{tabular}
\end{center}
\end{table}
\begin{table}[here]
\begin{center}
\begin{tabular}[t]{cl}
$\subedge{n+2}{v}{n}{v}{n+2}$&=$\edge{v}{n}{v}{n+2}=\dualver{b}{n(n-3)+n}$\\&=$\dualver{[\nlambda(v_{n+2})]}{}$,\\
$\subedge{0}{v}{0}{v}{i}$&=$\edge{v}{0}{u}{i}=\dualver{b}{0}$\\&=$\dualver{[\nlambda(v_{0})]}{}$
for all $i\in\oneton{n-3}$,\\
$\subedge{0}{v}{0}{v}{n-2}$&=$\edge{v}{0}{u}{(2n-1)(n-3)+2(n-1)+1}=\dualver{b}{0}$\\&=$\dualver{[\nlambda(v_{0})]}{}$, and\\
$\subedge{0}{v}{0}{v}{n-1}$&=$\edge{v}{0}{u}{(2n-1)(n-3)+2(n-1)+2}=\dualver{b}{0}$\\&=$\dualver{[\nlambda(v_{0})]}{}$.
\end{tabular}
\end{center}
\end{table}

ii)\\
$\dualver{[\nlambda(w_{j(n-3)+i})]}{}=\dualver{b}{j(n-3)+i}=\edge{u}{(2j+1)(n-3)+i}{u}{(2j+2)(n-3)+i}
\bigcup\bbrkedge{u}{(2j+2)(n-3)+i}{u}{(2j+3)(n-3)+i}\subseteq\arm{1}=(\edge{v}{0}{v}{i},\X{1})$
for all $i\in\oneton{n-3}$ and \\\indent\indent for all
$j\in\oneton[0]{n-3}$,\\
$\dualver{[\nlambda(w_{(n-2)(n-3)+i})]}{}=\dualver{b}{(n-2)(n-3)+i}=\edge{u}{(2(n-1)-1)(n-3)+i}{u}{2(n-1)(n-3)+i}
\subseteq\arm{1}\\\indent\indent=(\edge{v}{0}{v}{i},\X{1})$ for all
$i\in\oneton{n-3}$, and\\
$\dualver{[\nlambda(w_{(n-1)(n-3)+i})]}{}=\dualver{b}{n(n-3)+i}=\edge{u}{(2n-1)(n-3)+2i}{u}{(2n-1)(n-3)+2i+1}
\bigcup\\\indent\indent\edge{u}{(2n-1)(n-3)+2i+1}{u}{(2n-1)(n-3)+2i+2}\subseteq\arm{1}=(\edge{v}{0}{v}{n},\X{1})$\\
\indent\indent for each $i\in\oneton{n-2}$.
\end{proof}

\begin{proposition} \label{claim6}
\nphi{} is consistent on \nS{} (Definition \ref{def5.5}).
\end{proposition}

\begin{proof}
The claim follows from Proposition \ref{claim5} and Definition
\ref{def5.5}.
\end{proof}

Let $f:G_1\longrightarrow G_0$ be a light simplicial map between
graphs and $x_0,x_1,\ldots,x_m$ for some $m\in\{1,2,\ldots\}$ be
vertices of $G_1$ so that $x_i$ and $x_{i+1}$ are adjacent for
each $i\in\oneton[0]{m-1}$.  Define $\overline{f}(\langle
x_0,x_1,\ldots,x_m
\rangle)=\edge{f(x_0),f(x_1),\ldots}{}{f(x_m)}{}$.

For each $j\in\{0,1,\ldots\}$, let
$\arm{j}=(\edge{v}{0}{v}{i},\X{j})$ for each $i\in\oneton{n-2}$,
$\arm[n-1]{j}=(\edge{v}{0}{v}{n-1},\X{j})\linebreak[0]\hspace{-.1cm}\bigvee\linebreak[0]\hspace{-.05cm}(\edge{v}{n-1}{v}{n+1},\X{j})$,
and
$\arm[n]{j}=(\edge{v}{0}{v}{n},\X{j})\hspace{-.1cm}\bigvee\hspace{-.05cm}(\edge{v}{n}{v}{n+2},\linebreak[0]\X{j})$
(Definition \ref{def5.1}).

\begin{proposition} \label{claim7}
Let $m\in\{1,2,\ldots\}$ and, for each $i\in\oneton{n}$,
$\xseq{i}{0}{1}{k_i}=\imPhibar{m-1}{0}{i}{m-1}$ for some
$k_i\in\{1,2,\ldots\}$.  Then,

i)
$\imPhibar{m}{0}{i}{m}\hspace{-.1cm}=\hspace{-.1cm}\xxxxxseq{n-1}{0}{1}{k_{n-1}-1}$
\hspace{-.17cm}$\bigvee$
\hspace{-.17cm}$\xxxxxseq{n-1}{k_{n-1}-1}{k_{n-1}-2}{0}$
\hspace{-.25cm}$\bigvee\limits_{j=0}^{n-3}$
\hspace{-.25cm}$(\hspace{-.05cm}\xxxxxseq{{i+j}\bmod{n-2}}{0}{1}{k_{{i+j}\bmod{n-2}}}$
\hspace{-.15cm}$\bigvee$
\hspace{-.15cm}$\xxxxxseq{{i+j}\bmod{n-2}}{k_{{i+j}\bmod{n-2}}}{k_{{i+j}\bmod{n-2}}-1}{0})$
\hspace{-.15cm}$\bigvee$
\hspace{-.15cm}$\imPhibar{m-1}{0}{n}{m-1}$ for each
$i\in\oneton{n-3}$,

ii)  $\imPhibar{m}{0}{n-2}{m}=\imPhibar{m-1}{0}{n-1}{m-1}$,

iii) $\imPhibar{m}{0}{n-1}{m}=\xxxxxseq{n-1}{0}{1}{k_{n-1}-1}$
\hspace{-.15cm}$\bigvee$
\hspace{-.15cm}$\xxxxxseq{n-1}{k_{n-1}-1}{k_{n-1}-2}{0}$
\hspace{-.15cm}$\bigvee$
\hspace{-.15cm}$\imPhibar{m-1}{0}{n}{m-1}$, and

iv) $\imPhibar{m}{0}{n}{m}=\xseq{n-1}{0}{1}{k_{n-1}-1}$
\hspace{-.15cm}$\bigvee$
\hspace{-.15cm}$\xseq{n-1}{k_{n-1}-1}{k_{n-1}-2}{0}$
\hspace{-.15cm}$\bigvee$
\hspace{-.15cm}$\xseq{n-2}{0}{1}{k_{n-2}}$
\hspace{-.15cm}$\bigvee$
\hspace{-.15cm}$\xseq{n-2}{k_{n-2}}{k_{n-2}-1}{0}$
\hspace{-.15cm}$\bigvee\limits_{j=1}^{n-3}$
\hspace{-.15cm}$(\xseq{j}{0}{1}{k_j}$ \hspace{-.15cm}$\bigvee$
\hspace{-.15cm}$\xseq{j}{k_j}{k_j-1}{0})$
\linebreak[0]\hspace{-.15cm}$\bigvee$
\hspace{-.15cm}$\imPhibar{m-1}{0}{n}{m-1}$.
\end{proposition}

\begin{proof}
\underline{\scriptsize \bf CLAIM}:
$\subdivedge{n-1}{n}{i}=\edge{v}{n-1}{v}{n+1}$ for each
$i\in\{1,2,\ldots\}$.  Since \nPhi{i}{i-1} is a subdivision of
\nphi{} matching \X{i-1} for each $i\in\{1,2,\ldots\}$ and \nphi{}
embeds \edge{v}{n}{v}{n+2} onto \edge{v}{n}{v}{n+2}, then
$\subdivedge{n}{n+2}{i}=\edge{v}{n}{v}{n+2}$ for each
$i\in\{1,2,\ldots\}$. Since \nphi{} embeds \edge{v}{n-1}{v}{n+1}
onto \edge{v}{n}{v}{n+2}, then
$\subdivedge{n-1}{n+1}{i}=\edge{v}{n-1}{v}{n+1}$ for each
$i\in\{1,2,\ldots\}$.

i)  Since \nPhi{m}{m-1} is a subdivision of \nphi{} matching
\X{m-1} and \nphi{} embeds \edge{v}{0}{u}{i} onto
\edge{v}{0}{v}{n-1}, \hspace{-.05cm}\edge{u}{i}{u}{n-3+i}
\hspace{-.05cm}onto \hspace{-.05cm}\edge{v}{n-1}{v}{0},
\hspace{-.05cm}\edge{u}{j(n-3)+i}{u}{(j+1)(n-3)+i}
\hspace{-.05cm}onto
\hspace{-.05cm}\edge{v}{0}{v}{{i-1+\frac{j+1}{2}}\bmod{n-2}} and
\edge{u}{(j+1)(n-3)+i}{u}{(j+2)(n-3)+i} onto
\edge{v}{{i-1+\frac{j+1}{2}}\bmod{n-2}}{v}{0} for each odd
$j\in\oneton{2(n-1)-3}$, and $\langle
u_{(2(n-1)-1)(n-3)+i},u_{2(n-1)(n-3)+i},v_{i} \rangle$ onto
\armzeron, for each $i\in\oneton{n-3}$, then
$\imPhibar{m}{m-1}{i}{m}=\subdivedge{0}{n-1}{m-1}\bigvee\subdivedge{n-1}{0}{m-1}
\bigvee\limits_{j=0}^{n-3}(\subdivedge{0}{{i+j}\bmod{n-2}}{m-1}
\bigvee\subdivedge{{i+j}\bmod{n-2}}{0}{m-1})\bigvee\arm[n]{m-1}$
for each $i\in\oneton{n-3}$.  By the claim,
$\imedgePhibar{m-1}{0}{0}{n-1}{m-1}=\xseq{n-1}{0}{1}{k_{n-1}-1}$.
Thus,
$\imPhibar{m}{0}{i}{m}=\overline{\nPhi{m-1}{0}}(\imPhibar{m}{m-1}{i}{m})
=\overline{\nPhi{m-1}{0}}(\subdivedge{0}{n-1}{m-1}\bigvee\subdivedge{n-1}{0}{m-1}
\bigvee\limits_{j=0}^{n-3}\linebreak[0](\subdivedge{0}{{i+j}\bmod{n-2}}{m-1}\hspace{-.1cm}\bigvee\hspace{-.05cm}\subdivedge{{i+j}\bmod{n-2}}{0}{m-1})\hspace{-.1cm}\bigvee\hspace{-.1cm}\arm[n]{m-1})
=\imedgePhibar{m-1}{0}{0}{n-1}{m-1}
\hspace{-.09cm}\bigvee\hspace{-.09cm}\imedgePhibar{m-1}{0}{n-1}{0}{m-1}\hspace{-.09cm}\bigvee\limits_{j=0}^{n-3}\hspace{-.09cm}(\imedgePhibar{m-1}{0}{0}{{i+j}\bmod{n-2}}{m-1}\hspace{-.09cm}\bigvee
\hspace{-.09cm}\imedgePhibar{m-1}{0}{{i+j}\bmod{n-2}}{0}{m-1})$
$\bigvee$ $\imPhibar{m-1}{0}{n}{m-1}=\xseq{n-1}{0}{1}{k_{n-1}-1}$
$\bigvee$ $\xxxseq{n-1}{k_{n-1}-1}{k_{n-1}-2}{0}$
$\bigvee\limits_{j=0}^{n-3}$
$(\xxxxxxxxseq{{i+j}\bmod{n-2}}{0}{1}{k_{{i+j}\bmod{n-2}}}$
$\bigvee$
$\xxxseq{{i+j}\bmod{n-2}}{k_{{i+j}\bmod{n-2}}}{k_{{i+j}\bmod{n-2}}-1}{0})
\bigvee\imPhibar{m-1}{0}{n}{m-1}$ for each $i\in\oneton{n-3}$.

ii)  Since \nPhi{m}{m-1} is a subdivision of \nphi{} matching
\X{m-1} and \nphi{} embeds $\langle
v_0,\linebreak[0]u_{(2n-1)(n-3)+2(n-1)+1},\linebreak[0]v_{n-2}
\rangle$ onto \armzeronminusone, then
$\imPhibar{m}{m-1}{n-2}{m}=\arm[n-1]{m-1}$, giving
$\imPhibar{m}{0}{n-2}{m}=\overline{\nPhi{m-1}{0}}(\imPhibar{m}{m-1}{n-2}{m})=\imPhibar{m-1}{0}{n-1}{m-1}$.

iii)  Since \nPhi{m}{m-1} is a subdivision of \nphi{} matching
\X{m-1} and \nphi{} embeds \edge{v}{0}{u}{(2n-1)(n-3)+2(n-1)+2}
onto \edge{v}{0}{v}{n-1},
\edge{u}{(2n-1)(n-3)+2(n-1)+2}{u}{(2n-1)(n-3)+2(n-1)+3} onto
\edge{v}{n-1}{v}{0}, and $\langle
u_{(2n-1)(n-3)+2(n-1)+3},v_{n-1},v_{n+1} \rangle$ onto \armzeron,
then
$\imPhibar{m}{m-1}{n-1}{m}\linebreak[0]=\subdivedge{0}{n-1}{m-1}\bigvee\subdivedge{n-1}{0}{m-1}\bigvee\arm[n]{m-1}$,
and so by the claim,
$\imPhibar{m}{0}{n-1}{m}\linebreak[0]=\overline{\nPhi{m-1}{0}}(\imPhibar{m}{m-1}{n-1}{m})
=\imedgePhibar{m-1}{0}{0}{n-1}{m-1}
\bigvee\imedgePhibar{m-1}{0}{n-1}{0}{m-1}
\linebreak[0]\bigvee\imPhibar{m-1}{0}{n}{m-1}=\xseq{n-1}{0}{1}{k_{n-1}-1}\bigvee
\xseq{n-1}{k_{n-1}-1}{k_{n-1}-2}{0}\bigvee\imPhibar{m-1}{0}{n}{m-1}$.

iv)  Since \nPhi{m}{m-1} is a subdivision of \nphi{} matching
\X{m-1} and \nphi{} embeds \edge{v}{0}{u}{(2n-1)(n-3)+1} onto
\edge{v}{0}{v}{n-1}, \edge{u}{(2n-1)(n-3)+1}{u}{(2n-1)(n-3)+2}
onto \edge{v}{n-1}{v}{0},
\linebreak\edge{u}{(2n-1)(n-3)+2}{u}{(2n-1)(n-3)+3} onto
\edge{v}{0}{v}{n-2}, \edge{u}{(2n-1)(n-3)+3}{u}{(2n-1)(n-3)+4}
onto \linebreak\edge{v}{n-2}{v}{0},
\edge{u}{(2n-1)(n-3)+i}{u}{(2n-1)(n-3)+i+1} onto
\edge{v}{0}{v}{\frac{i-2}{2}} and
\bbbbrkedge{u}{(2n-1)(n-3)+i+1}{u}{(2n-1)(n-3)+i+2} onto
\edge{v}{\frac{i-2}{2}}{v}{0} for each even
$i\in\oneton[4]{2(n-1)-2}$, and \linebreak$\langle
u_{((2n-1)(n-3)+2(n-1))(n-3)+i},v_{n},{v}_{n+2} \rangle$ onto
\armzeron, then
$\imPhibar{m}{m-1}{n}{m}=\ssssubdivedge{0}{n-1}{m-1}$ $\bigvee$
$\subdivedge{n-1}{0}{m-1}$ $\bigvee$ $\subdivedge{0}{n-2}{m-1}$
$\bigvee$ $\subdivedge{n-2}{0}{m-1}$ $\bigvee\limits_{i=1}^{n-3}$
$\linebreak(\subdivedge{0}{i}{m-1}
\bigvee\subdivedge{i}{0}{m-1})\bigvee\arm[n]{m-1}$, and so by the
claim,
$\imPhibar{m}{0}{n}{m}=\overline{\nPhi{m-1}{0}}(\imPhibar{m}{m-1}{n}{m})
=\imedgePhibar{m-1}{0}{0}{n-1}{m-1}$ $\bigvee$
$\imedgePhibar{m-1}{0}{n-1}{0}{m-1}$ $\bigvee$
$\imedgePhibar{m-1}{0}{0}{n-2}{m-1}$ \hspace{-.1cm}$\bigvee$
\hspace{-.1cm}$\imedgePhibar{m-1}{0}{n-2}{0}{m-1}$
\hspace{-.1cm}$\bigvee\limits_{i=1}^{n-3}$
\hspace{-.1cm}$(\imedgePhibar{m-1}{0}{0}{i}{m-1}$ $\bigvee$
$\imedgePhibar{m-1}{0}{i}{0}{m-1})$ $\bigvee$
$\imPhibar{m-1}{0}{n}{m-1}=\xseq{n-1}{0}{1}{k_{n-1}-1}$ $\bigvee$
$\xseq{n-1}{k_{n-1}-1}{k_{n-1}-2}{0}$ $\bigvee$
$\xseq{n-2}{0}{1}{k_{n-2}}$ $\bigvee$
$\xseq{n-2}{k_{n-2}}{k_{n-2}-1}{0}$ $\bigvee\limits_{i=1}^{n-3}$
$(\xseq{i}{0}{1}{k_i}$ $\bigvee$ $\xseq{i}{k_i}{k_i-1}{0})$
$\bigvee$ $\imPhibar{m-1}{0}{n}{m-1}$.

\end{proof}

Let $p:\oneton{n-2}\longrightarrow\oneton{n-2}$ be the bijection
defined as $p(i)=i+1$ for each $i\in\oneton{n-3}$ and $p(n-2)=1$.

Let $C=\{\vvseqonea{}{}{}$\hspace{-.05cm},\linebreak[0]
\hspace{-.075cm}$\vseqpa$,\linebreak[0] $\vseqtwo$,\linebreak[0]
$\vvvvseqthree$,\linebreak[0] $\vseqfour:$\linebreak[0]
$j\linebreak[0]\in\linebreak[0]\oneton{n-2}\}$.

\begin{proposition} \label{claim8}
Let $e\in C$, $k\in\{1,2,\ldots\}$, and
$\{x_0,x_1,\ldots,x_k\}\subseteq\mathrm{V}(\X{0})$ so that
$e=\xseq{}{0}{1}{k}$.  Then, $\xseq{}{k}{k-1}{0}=e$.
\end{proposition}

\begin{proof}
If $v_{n+1}\in e$, then $v_{n+1}$ is a vertex of symmetry for $e$.
If $v_{n+2}\in e$, then $v_{n+2}$ is a vertex of symmetry for $e$.
Otherwise, $v_n$ is a vertex of symmetry for $e$.
\end{proof}

\begin{proposition} \label{claim9}
Let $m\ge 3$.  Then, $\imPhibar{m-1}{0}{i}{m-1}=\vseqthree\bigvee
e_1^{i,m-1}\bigvee e_2^{i,m-1}\bigvee\cdots\bigvee
e_{k_{i,m-1}}^{i,m-1}\bigvee\langle
v_0,\linebreak[0]v_{n-1},\linebreak[0]v_0,\linebreak[0]v_{n-2},\linebreak[0]v_0,\linebreak[0]v_1,\linebreak[0]v_0,\linebreak[0]\underline{v_2,\linebreak[0]v_0},\linebreak[0]\ldots,\linebreak[0]\underline{v_{n-3},\linebreak[0]v_0},\linebreak[0]v_n,\linebreak[0]v_{n+2}
\rangle$ for some $k_{i,m-1}\in\{0,1,\ldots\}$ for each
$i\in\oneton{n}\setminus\{n-2\}$ and
$\imPhibar{m}{0}{n-2}{m}=\vseqthree\bigvee e_1^{n-2,m}\bigvee
e_2^{n-2,m}\bigvee\cdots\bigvee
e_{k_{n-2,m}}^{n-2,m}\bigvee\langle
v_0,\linebreak[0]v_{n-1},\linebreak[0]v_0,\linebreak[0]v_{n-2},\linebreak[0]v_0,\linebreak[0]v_1,\linebreak[0]v_0,\linebreak[0]\underline{v_2,\linebreak[0]v_0},\linebreak[0]\ldots,\linebreak[0]\underline{v_{n-3},\linebreak[0]v_0},\linebreak[0]v_n,\linebreak[0]v_{n+2}
\rangle$ for some $k_{n-2,m}\in\{0,1,\ldots\}$, where
$e_k^{i,m-1}\linebreak[0]\in C$ for each $k\in\oneton{k_{i,m-1}}$
and each $i\in\oneton{n}\setminus\{n-2\}$ and $e_k^{n-2,m}\in C$
for each $k\in\oneton{k_{n-2,m}}$.
\end{proposition}

\begin{proof}
By Proposition \ref{claim7}, $\imPhibar{0}{0}{i}{0}=\armzeroi$ for
all $i\in\oneton{n-2}$, \
$\imPhibar{0}{0}{n-1}{0}=\armzeronminusone$, \
$\imPhibar{0}{0}{n}{0}=\armzeron$; \
$\imPhibar{1}{0}{i}{1}=\vseqphalfa[i-1]$ for all
$i\in\oneton{n-3}$, \
$\imPhibar{1}{0}{n-2}{1}\linebreak[0]=\vseqtwohalf$, \
$\imPhibar{1}{0}{n-1}{1}=\vseqfourhalf$, \
$\imPhibar{1}{0}{n}{1}=\vseqonehalfa$; \
$\imPhibar{2}{0}{i}{2}=\vseqthree
\linebreak[0]\bigvee\limits_{j=i-1}^{n-4}\vvvvseqpa{}{}{}$
$\bigvee$ $\vseqtwo$ $\bigvee\limits_{j=0}^{i-2}$
$\vvvseqpb{}{}{}$ \hspace{-.15cm}$\bigvee$
\hspace{-.12cm}$\vseqonehalf$ for all $i\in\oneton{n-3}$, \
$\imPhibar{2}{0}{n-2}{2}=\vseqfourhalf$, \
$\imPhibar{2}{0}{n-1}{2}=\vseqthree$ $\bigvee$ $\vseqonehalf$, \
$\imPhibar{2}{0}{n}{2}=\vseqthree$ $\bigvee$ $\vseqtwo$
\hspace{-.07cm}$\bigvee\limits_{i=1}^{n-3}$
\hspace{-.07cm}$\vseqpc[i-1]$ $\bigvee$ $\vvseqonehalfa{}$; \
$\imPhibar{3}{0}{n-2}{3}=\vseqthree$ $\bigvee$ $\vseqonehalfa$.
Thus, the claim is true for $m=3$.

By Proposition \ref{claim7} and Proposition \ref{claim8},
$\imPhibar{3}{0}{i}{3}=\vseqthree\bigvee\eseq{n-1,2}{1}{2}{k_{n-1,2}}
\bigvee\vseqonehalf[]\linebreak[0]$ $\bigvee$
\hspace{-.15cm}$\langle
v_n,v_0,v_{n-3},v_0,\underline{v_{n-4},v_0},\ldots,\underline{v_1,v_0},v_{n-2},v_0,v_{n-1},v_0
\rangle$ \hspace{-.15cm}$\bigvee$
\hspace{-.1cm}$\eseq{n-1,2}{k_{n-1,2}}{k_{n-1,2}-1}{1}\linebreak[0]$
$\bigvee$ $\vseqthree$ $\bigvee\limits_{j=i}^{n-3}(\vseqthree
\bigvee\eseq{j,2}{1}{2}{k_{j,2}}\linebreak[0]\bigvee\linebreak[0]$
$\vseqonehalfa\bigvee\langle
v_{n+2},\linebreak[0]v_n,\linebreak[0]v_0,\linebreak[0]v_{n-3},\linebreak[0]v_0,\linebreak[0]\underline{v_{n-4},}$
$\underline{\linebreak[0]v_0},\linebreak[0]\ldots,\linebreak[0]\underline{v_1,\linebreak[0]v_0},\linebreak[0]v_{n-2},\linebreak[0]v_0,\linebreak[0]v_{n-1},\linebreak[0]v_0
\rangle$ $\bigvee$ $\eseq{j,2}{k_{j,2}}{k_{j,2}-1}{1}$ $\bigvee$
$\vseqthree) \linebreak[0]$ $\bigvee\linebreak[0]$ $\vseqfour$
$\bigvee\limits_{j=1}^{i-1}(\vseqthree$ $\bigvee$
$\eseq{j,2}{1}{2}{k_{j,2}}$ $\bigvee$ $\vseqonehalfa$ $\bigvee$
$\langle
v_{n+2},\linebreak[0]v_n,\linebreak[0]v_0,\linebreak[0]v_{n-3},\linebreak[0]v_0,\linebreak[0]\underline{v_{n-4},\linebreak[0]v_0},\linebreak[0]\ldots,\linebreak[0]\underline{v_1,\linebreak[0]v_0},\linebreak[0]v_{n-2},\linebreak[0]v_0,\linebreak[0]v_{n-1},\linebreak[0]v_0
\rangle$ \hspace{-.1cm}$\bigvee$
$\eseq{j,2}{k_{j,2}}{k_{j,2}-1}{1}$ $\bigvee$
\hspace{-.1cm}$\vseqthree)$ $\bigvee$ $\vseqthree$ $\bigvee$
$\eseq{n,2}{1}{2}{k_{n,2}}\bigvee\vseqonehalf=\vseqthree$
$\bigvee$ $\eseq{n-1,2}{1}{2}{k_{n-1,2}}$ $\bigvee$ $\vseqone$
$\bigvee$ $\eseq{n-1,2}{k_{n-1,2}}{k_{n-1,2}-1}{1}$ $\bigvee$
$\vseqthree\linebreak[0]$ $\bigvee\limits_{j=i}^{n-3}(\vseqthree$
\hspace{-.1cm}$\bigvee$ $\eseq{j,2}{1}{2}{k_{j,2}}$ $\bigvee$
\hspace{-.1cm}$\vvvseqoneb[v_{n+2},v_n,]{}{}{}$ $\bigvee$
$\eseq{j,2}{k_{j,2}}{k_{j,2}-1}{1}$ $\bigvee$ $\vseqthree)$
$\bigvee$ $\vseqfour \linebreak[0]$
$\bigvee\limits_{j=1}^{i-1}(\vseqthree$ $\bigvee$
$\eseq{j,2}{1}{2}{k_{j,2}}$ $\bigvee$
$\vvvseqoneb[v_{n+2},v_n,]{}{}{}$ $\bigvee$
$\eseq{j,2}{k_{j,2}}{k_{j,2}-1}{1}$ $\bigvee$ $\vseqthree$
$\bigvee\vseqthree$ $\bigvee$ $\eseq{n,2}{1}{2}{k_{n,2}}$
$\bigvee$ $\vseqonehalfa$  for all $i\in\oneton{n-3}$, \
$\imPhibar{3}{0}{n-1}{3}=\vseqthree$ $\bigvee$
$\eseq{n-1,2}{1}{2}{k_{n-1,2}}$ $\bigvee$ $\vseqonehalf[]$
$\bigvee$ $\langle
v_n,\linebreak[0]v_0,\linebreak[0]v_{n-3},\linebreak[0]v_0,\linebreak[0]\underline{v_{n-4},\linebreak[0]v_0},\linebreak[0]\ldots,\linebreak[0]\underline{v_1,\linebreak[0]v_0},\linebreak[0]v_{n-2},\linebreak[0]v_0,\linebreak[0]v_{n-1},\linebreak[0]v_0
\rangle$ $\bigvee$ $\eseq{n-1,2}{k_{n-1,2}}{k_{n-1,2}-1}{1}$
$\bigvee$ $\vseqthree$ $\bigvee$ $\vseqthree$ $\bigvee$
$\eseq{n,2}{1}{2}{k_{n,2}}$ $\bigvee$ $\vseqonehalf= \vseqthree$
$\bigvee$ $\eseq{n-1,2}{1}{2}{k_{n-1,2}}$ $\bigvee\vseqoneb$
$\bigvee$ $\eseq{n-1,2}{k_{n-1,2}}{k_{n-1,2}-1}{1}$ $\bigvee$
$\vseqthree)$ $\bigvee$ $\linebreak[0]\vseqthree$ $\bigvee$
$\eseq{n,2}{1}{2}{k_{n,2}}$ $\bigvee$ $\vseqonehalf$, \
$\imPhibar{3}{0}{n}{3}=\vseqthree$ $\bigvee$
$\eseq{n-1,2}{1}{2}{k_{n-1,2}}$ $\bigvee$ $\vseqonehalf[]$
$\bigvee$ $\langle
v_n,\linebreak[0]v_0,\linebreak[0]v_{n-3},\linebreak[0]v_0,\linebreak[0]\underline{v_{n-4},\linebreak[0]v_0},\linebreak[0]\ldots,\linebreak[0]\underline{v_1,\linebreak[0]v_0},\linebreak[0]v_{n-2},\linebreak[0]v_0,\linebreak[0]v_{n-1},\linebreak[0]v_0
\rangle$ $\bigvee$ $\eseq{n-1,2}{k_{n-1,2}}{k_{n-1,2}-1}{1}$
$\bigvee$ $\vseqthree$ $\linebreak[0]\bigvee$
$\linebreak[0]\vseqfour$ $\bigvee\limits_{i=i}^{n-3}(\vseqthree$
$\bigvee$ $\eseq{i,2}{1}{2}{k_{i,2}}$ $\bigvee$ $\vseqonehalf$
$\bigvee$ $\langle
v_{n+2},\linebreak[0]v_n,\linebreak[0]v_0,\linebreak[0]v_{n-3},\linebreak[0]v_0,\linebreak[0]\underline{v_{n-4},\linebreak[0]v_0},\linebreak[0]\ldots,\linebreak[0]\underline{v_1,\linebreak[0]v_0},\linebreak[0]v_{n-2},\linebreak[0]v_0,\linebreak[0]v_{n-1},\linebreak[0]v_0
\rangle$ $\bigvee$ $\eseq{i,2}{k_{i,2}}{k_{i,2}-1}{1}$ $\bigvee$
$\vseqthree)$ $\bigvee$ $\vseqthree$ $\bigvee$
$\eseq{n,2}{1}{2}{k_{n,2}}$ $\bigvee$ $\vseqonehalf=\vseqthree$
$\bigvee$ $\eseq{n-1,2}{1}{2}{k_{n-1,2}}$ $\bigvee$
$\vvvseqoneb{}{}{}$ $\bigvee$
$\eseq{n-1,2}{k_{n-1,2}}{k_{n-1,2}-1}{1}$ $\bigvee$ $\vseqthree$
$\bigvee$ $\vseqfour$ $\bigvee\limits_{i=1}^{n-3}(\vseqthree$
$\bigvee$ $\eseq{i,2}{1}{2}{k_{i,2}}$ $\linebreak[0]\bigvee$
$\vvvseqone[v_{n+2},v_n,]{}{}{}$ $\bigvee$
$\eseq{i,2}{k_{i,2}}{k_{i,2}-1}{1}$ $\bigvee$ $\vseqthree)$
$\bigvee$ $\vseqthree$ \hspace{.02cm}$\bigvee$
$\eseq{n,2}{1}{2}{k_{n,2}}$ $\bigvee$
\hspace{.02cm}$\vseqonehalfa$; \
$\imPhibar{4}{0}{n-2}{4}=\imPhibar{3}{0}{n-1}{3}$.  Thus, the
claim is true for $m=4$.

Let $m\ge4$ and suppose the claim is true for all
$\tilde{m}\in\oneton[3]{m}$.  By Proposition \ref{claim7} and
Proposition \ref{claim8}, $\imPhibar{m}{0}{i}{m}=\vseqthree$
$\bigvee$
$\eseq{n-1,m-1}{1}{2}{k_{n-1,m-1}}\bigvee\vseqonehalf[]\bigvee\langle
v_n,\linebreak[0]v_0,\linebreak[0]v_{n-3},\linebreak[0]v_0,\linebreak[0]\underline{v_{n-4},\linebreak[0]v_0},\linebreak[0]\ldots,\linebreak[0]\underline{v_1,\linebreak[0]v_0},\linebreak[0]v_{n-2},\linebreak[0]v_0,\linebreak[0]v_{n-1},\linebreak[0]v_0
\rangle$ $\bigvee$ $\eseq{n-1,m-1}{k_{n-1,m-1}}{k_{n-1,m-1}-1}{1}$
$\linebreak[0]\bigvee$ $\vseqthree$
$\bigvee\limits_{j=0}^{n-3}(\vseqthree$ $\bigvee$
$\eeeeseq{{i+j}\bmod{n-2},m-1}{1}{2}{k_{{i+j}\bmod{n-2},m-1}}\hspace{-.05cm}\bigvee\hspace{-.05cm}\vseqonehalf\bigvee\langle
v_{n+2},v_n,v_0,v_{n-3},v_0,\underline{v_{n-4},v_0},\ldots,\underline{v_1,v_0},v_{n-2},v_0,v_{n-1},v_0
\rangle$ $\bigvee$
$\eseq{{i+j}\bmod{n-2},m-1}{k_{{i+j}\bmod{n-2},m-1}}{k_{{i+j}\bmod{n-2},m-1}-1}{1}$
$\bigvee$ $\vseqthree)$ $\bigvee$
$\vseqthree\bigvee\eseq{n,m-1}{1}{2}{k_{n,m-1}}\bigvee\vseqonehalf=\vseqthree$
$\bigvee$ $\eeeseq{n-1,m-1}{1}{2}{k_{n-1,m-1}}$
$\linebreak[0]\bigvee$ $\vseqoneb$ $\bigvee$
$\eseq{n-1,m-1}{k_{n-1,m-1}}{k_{n-1,m-1}-1}{1}$ $\bigvee$
$\vseqthree$ $\bigvee\limits_{j=0}^{n-3}\linebreak[0](\vseqthree$
$\bigvee$
$\eseqa{{i+j}\bmod{n-2},m-1}{1}{2}{k_{{i+j}\bmod{n-2},m-1}}$
$\linebreak[0]\bigvee$ $\vvseqonec[v_{n+2},v_n,]{}{}{}$ $\bigvee$
$\eeseqa{{i+j}\bmod{n-2},m-1}{k_{{i+j}\bmod{n-2},m-1}}{k_{{i+j}\bmod{n-2},m-1}-1}{1}$
$\bigvee$ $\vseqthree)$ $\bigvee$ $\vseqthree$ $\bigvee$
$\eseq{n,m-1}{1}{2}{k_{n,m-1}}$ $\linebreak[0]\bigvee$
$\vseqonehalf$ for all $i\in\oneton{n-3}$, \
$\imPhibar{m}{0}{n-1}{m}=\vseqthree$ $\bigvee$
$\eseqa{n-1,m-1}{1}{2}{k_{n-1,m-1}}\bigvee\linebreak[0]\vseqonehalf[]\bigvee\langle
v_n,\linebreak[0]v_0,\linebreak[0]v_{n-3},\linebreak[0]v_0,\linebreak[0]\underline{v_{n-4},\linebreak[0]v_0},\linebreak[0]\ldots,\linebreak[0]\underline{v_1,\linebreak[0]v_0},\linebreak[0]v_{n-2},\linebreak[0]v_0,\linebreak[0]v_{n-1},\linebreak[0]v_0
\rangle\bigvee$ $\eseq{n-1,m-1}{k_{n-1,m-1}}{k_{n-1,m-1}-1}{1}$
$\bigvee$ $\vseqthree$ $\linebreak[0]\bigvee$ $\vseqthree$
$\bigvee$ $\eseq{n,m-1}{1}{2}{k_{n,m-1}}$ $\bigvee$ $\vseqonehalf=
\vseqthree\bigvee\eseqa{n-1,m-1}{1}{2}{k_{n-1,m-1}}\bigvee\vseqonea$
\hspace{-.1cm}$\bigvee$
\hspace{-.1cm}$\eseqq{n-1,m-1}{k_{n-1,m-1}}{k_{n-1,m-1}-1}{1}\bigvee\vseqthree)\bigvee\vseqthree$
$\bigvee$ $\eseq{n,m-1}{1}{2}{k_{n,m-1}}$ $\bigvee$
$\vvvvvseqonehalf{}$, \
\linebreak[0]$\imPhibar{m}{0}{n}{m}=\vseqthree$
$\bigvee\eseq{n-1,m-1}{1}{2}{k_{n-1,m-1}}$ $\bigvee$
$\vseqonehalf[]\bigvee\langle
v_n,v_0,v_{n-3},v_0,\underline{v_{n-4},v_0},\ldots,\underline{v_1,v_0}\linebreak,v_{n-2},v_0,v_{n-1},v_0
\rangle\bigvee\eseq{n-1,m-1}{k_{n-1,m-1}}{k_{n-1,m-1}-1}{1}\bigvee\vseqthree$
$\bigvee$ $\vseqthree$ $\bigvee$
$\eseq{n-2,m-1}{1}{2}{k_{n-2,m-1}}$ $\bigvee$ $\vvvvvseqonehalf{}$
$\bigvee$ $\langle
v_{n+2},\linebreak[0]v_n,\linebreak[0]v_0,\linebreak[0]v_{n-3},\linebreak[0]v_0,\linebreak[0]\underline{v_{n-4},\linebreak[0]v_0},\linebreak[0]\hspace{.08cm}\ldots\hspace{.08cm},\linebreak\underline{v_1,\linebreak[0]v_0},\linebreak[0]v_{n-2},\linebreak[0]v_0,\linebreak[0]v_{n-1},\linebreak[0]v_0
\rangle$ $\bigvee$ $\eseq{n-2,m-1}{k_{n-2,m-1}}{k_{n-2,m-1}-1}{1}$
$\bigvee$ $\vseqthree$ $\bigvee\limits_{i=i}^{n-3}(\vseqthree$
$\bigvee$ $\eseq{i,m-1}{1}{2}{k_{i,m-1}}$ $\bigvee$
$\vseqonehalf\hspace{-.12cm}\bigvee\hspace{-.03cm}\langle
v_{n+2},\linebreak[0]v_n,\linebreak[0]v_0,\linebreak[0]v_{n-3},\linebreak[0]v_0,\linebreak[0]\underline{v_{n-4},\linebreak[0]v_0},\linebreak[0]\ldots,\linebreak[0]\underline{v_1,\linebreak[0]v_0},\linebreak[0]v_{n-2},\linebreak[0]v_0,\linebreak[0]v_{n-1},\linebreak[0]v_0
\rangle$ $\bigvee$ $\eseq{i,m-1}{k_{i,m-1}}{k_{i,m-1}-1}{1}$
$\linebreak[0]\bigvee$ $\linebreak[0]\vseqthree)$ $\bigvee$
$\vseqthree$ $\bigvee$ $\eseq{n,m-1}{1}{2}{k_{n,m-1}}$ $\bigvee$
$\vseqonehalf\hspace{-.05cm}=\hspace{-.05cm}\vseqthree\bigvee\eseq{n-1,m-1}{1}{2}{k_{n-1,m-1}}$
$\bigvee$
$\vseqonec\hspace{-.04cm}\bigvee\hspace{-.04cm}\eseq{n-1,m-1}{k_{n-1,m-1}}{k_{n-1,m-1}-1}{1}\bigvee\vseqthree$
$\bigvee$ $\vseqthree$ $\bigvee$
$\eseq{n-2,m-1}{1}{2}{k_{n-2,m-1}}\bigvee\vvvseqonea[v_{n+2},v_n,]{}{}{}$
$\bigvee$ $\eseq{n-2,m-1}{k_{n-2,m-1}}{k_{n-2,m-1}-1}{1}$
$\linebreak[0]\bigvee$
$\linebreak[0]\vseqthree\hspace{-.05cm}\bigvee\limits_{i=1}^{n-3}\hspace{-.05cm}(\vseqthree\hspace{-.05cm}\bigvee\hspace{-.05cm}\eseq{i,m-1}{1}{2}{k_{i,m-1}}\hspace{-.05cm}\bigvee\hspace{-.05cm}\vvvvseqonea[v_{n+2},v_n,]{}{}{}$
$\bigvee$ $\eseq{i,m-1}{k_{i,m-1}}{k_{i,m-1}-1}{1}$ $\bigvee$
$\vseqthree)$ $\bigvee$ $\vseqthree$ $\bigvee$
$\eseq{n,m-1}{1}{2}{k_{n,m-1}}$ $\bigvee$ $\vseqonehalfa$; \
$\imPhibar{m+1}{0}{n-2}{m+1}=\imPhibar{m}{0}{n-1}{m}$, and so the
claim holds by induction.
\end{proof}

\section{Factoring}

In this section, the complexity of \K{} is addressed by resolving
whether or not it is simpler (Proposition \ref{claim18}). Were it
to be, in keeping with the strategy by Minc, a contradiction to a
factoring such as given by Theorem \ref{claim17} is obtained. In
line with what was developed in \cite{minc2}, a certain factoring
allowing for control of the complexity of the dual system is
deduced, through properties of the bonding map, when given a
factoring in general. In particular, a factoring such that the
dual of the simpler object through which the bonding map is being
factored remains simpler, deduced, here, from the more general
factoring in which branch point of the simpler object goes to
branch point. For \K, the symmetry of the bonding maps described
in the previous section (Proposition \ref{claim9}) facilitates
this with Propositons \ref{claim13} and \ref{claim14}. The
argument of \cite{minc2} is then further adapted, with the
conditions of Minc's program being satisfied as established in the
previous section, to allow for a ``shortening'' and ``lifting'' of
such a factoring after passing to the dual (Proposition
\ref{claim15}), showing that a factoring cannot occur (Proposition
\ref{claim16}), provided that certain induction conditions are
preserved and that the defining bonding map (\nphi) cannot be
factored.

For Propositions \ref{claim10} through \ref{claim14}, let $m\ge2$,
$\nt\in\oneton{n-1}$, and $T$ be a simple-\nt-od.  Let
$S^1,S^2,\ldots,S^{\nt}$ be \nt{} arcs so that
$\bigcup\limits_{i=1}^{\nt}S^i=T$ and $S^i\bigcap S^j=\{s_0\}$ for
each distinct $i,j\in\oneton{\nt}$, and so $s_0$ is an endpoint of
$S^i$ for each $i\in\oneton{\nt}$.  Define $T$ as a graph with
subgraph $S^i=\sseqzero{i}{k_i}$ for some $k_i\in\{1,2,\ldots\}$
for each $i\in\oneton{\nt}$.  Suppose $\alpha:\X{m}\longrightarrow
T$ and $\beta:T\longrightarrow\X{0}$ are simplicial maps so that
$\alpha$ is surjective, $\beta\circ\alpha=\nPhi{m}{0}$, and
$\beta(s_0)=v_0$.  Let $\arm{m}=\xseq{i}{0}{1}{\ell_i}$ for some
$\ell_i\in\{1,2,\ldots\}$ for each $i\in\oneton{n}$ where
$x_0^i=v_0$ for all $i\in\oneton{n}$.

In the case of the existence of distinct
$x,y,z\in\mathrm{V}(\X{m})$ so that $x\mbox{ and }y$ are adjacent,
$y\mbox{ and }z$ are adjacent, $\alpha(y)=s_0$,
$\nPhi{m}{0}(z)\not=v_{n-1}$, and $\nPhi{m}{0}(x)=v_{n}$,
Propositions \ref{claim10} through \ref{claim12} show that for all
vertices of \X{m} mapped by $\alpha$ to $s_0$ a certain uniqueness
applies in describing how they are ``locally'' mapped by
\nPhi{m}{0}. This uniqueness ensures certain instances, in which
the above case holds, do not prevent the  ``controllable''
factoring arranged in Proposition \ref{claim13} from being
established as well-defined.

\begin{proposition} \label{claim10}
Let $i\in\oneton{n}$, $r\in\oneton[0]{\ell_i}$, and
$j\in\oneton{n-2}$ so that
$\overline{\nPhi{m}{0}}(\xseq{i}{r+1}{r+2}{r+2(n-2)})=\vpjseq[]$,
$\overline{\alpha}(\xseq{i}{r+1}{r+2}{r+t})=\sseq[,s_0]{\hat{\imath}}{t-1}{t-2}{1}$
for some even $t\in\oneton[2]{2(n-2)}$ and for some
$\hat{\imath}\in\oneton{\nt}$, and
$\{v_n,v_{\impj{\frac{\tilde{t}}{2}}}:\tilde{t}\mbox{ is even and
}\tilde{t}\in\oneton[t]{2(n-2)}\}\subseteq\{\beta(s_1^{\delta}):\delta\in\oneton{\nt}\}$.
Then,

i)  there exists $r^{\prime}\in\oneton[0]{\ell_i}$ so that
$r^{\prime}<r$,
$\overline{\nPhi{m}{0}}(\xseq{i}{r^{\prime}+1}{r^{\prime}+2}{r^{\prime}+2(n-2)})=\vpjseq[]$,
$\overline{\alpha}(\hspace{-.05cm}\xseq{i}{r^{\prime}+1}{r^{\prime}+2}{r^{\prime}+t^{\prime}}\hspace{-.05cm})\hspace{-.125cm}=\hspace{-.125cm}\sseq[,s_0]{\hat{\imath}^{\prime}}{t^{\prime}-1}{t^{\prime}-2}{1}$
for some even $t^{\prime}\in\oneton[2]{2(n-2)}$ and for some
$\hat{\imath}^{\prime}\in\oneton{\nt}$, and
$\{v_n,v_{\impj{\frac{\tilde{t}}{2}}}:\tilde{t}\mbox{ is even and
}\tilde{t}\in\oneton[t^{\prime}]{2(n-2)}\}\subseteq\{\beta(s_1^{\delta}):\delta\in\oneton{\nt}\}$,
or

ii) there exists \hspace{.05cm}$r^{\prime}\in\oneton[0]{\ell_i}$
so that $r^{\prime}<r$,
$\overline{\nPhi{m}{0}}(\xseq{i}{r^{\prime}+1}{r^{\prime}+2}{r^{\prime}+4(n-2)+5})\hspace{-.1cm}=
\vpjseq$ (or
$\overline{\nPhi{m}{0}}(\xseq{i}{r^{\prime}+1}{r^{\prime}+2}{r^{\prime}+4(n-2)+3})\linebreak[0]=
\langle
v_{n-2},v_0,v_1,v_0,\underline{v_2,v_0},\ldots,\underline{v_{n-3},v_0},
v_n,\linebreak[0]v_0,v_{n-3},v_0,\underline{v_{n-4},v_0},\ldots,\underline{v_1,v_0},v_{n-2},v_0,v_{n-1}
\rangle$),
$\overline{\alpha}(\xseq{i}{r^{\prime}+4(n-2)+5-(t+1)}{r^{\prime}+4(n-2)+5-t}{r^{\prime}+4(n-2)+5})
\hspace{-.125cm}=\hspace{-.125cm}\sseqzero{\hat{\imath}}{t+1}$ (or
$\overline{\alpha}(\hspace{-.075cm}\xxxxxseq{i}{r^{\prime}+4(n-2)+3-(t+1)}{r^{\prime}+4(n-2)+3-t}{r^{\prime}+4(n-2)+3}\hspace{-.05cm})
=\sseqzero{\hat{\imath}}{t+1}$), and
$\overline{\alpha}(\xseq{i}{0}{1}{r^{\prime}+2(n-2)+1})\subseteq
S^{\hat{\imath}^{\prime}}$ where
$\hat{\imath}^{\prime}\in\oneton{\nt}$ so that
$\imalpha{i}{r^{\prime}+1}\in S^{\hat{\imath}^{\prime}}$  or

iii) $\overline{\alpha}(\xseq{i}{0}{1}{r})\subseteq
S^{\hat{\imath}}$.
\end{proposition}

\begin{proof}
Suppose $\overline{\alpha}(\xseq{i}{0}{1}{r})\not\subseteq
S^{\hat{\imath}}$.  By Proposition \ref{claim9},
$\overline{\nPhi{m}{0}}(\xedge{i}{r-1}{r})=\edge{v}{n-1}{v}{0}$.
If $t=2$, then
$|\{v_n,v_{\impj{\frac{\tilde{t}}{2}}}:\tilde{t}\mbox{ is even and
}\tilde{t}\in\oneton[t]{2(n-2)}\}|=\nt$.  Hence,
$\overline{\alpha}(\xedge{i}{r-1}{r})=\xedge[s]{\hat{i}}{t+1}{t}$.
Then, by Proposition \ref{claim9}, there exists
$\nt[r]^{\prime}<r$ so that
$\overline{\nPhi{m}{0}}(\xseq{i}{\nt[r]^{\prime}+1}{\nt[r]^{\prime}+2}{\nt[r]^{\prime}+4(n-2)+5})=
\vpjseq$ (or
$\overline{\nPhi{m}{0}}(\xseq{i}{\nt[r]^{\prime}+1}{\nt[r]^{\prime}+2}{\nt[r]^{\prime}+4(n-2)+3})\linebreak[0]=
\langle
v_{n-2},v_0,v_1,v_0,\underline{v_2,v_0},\ldots,\underline{v_{n-3},v_0},
v_n,v_0,v_{n-3},v_0,\underline{v_{n-4},v_0},\ldots,\underline{v_1,v_0},v_{n-2},v_0,v_{n-1}
\rangle$) \linebreak[0] and
$\overline{\alpha}(\xseq{i}{\nt[r]^{\prime}+4(n-2)+5-(t+1)}{\nt[r]^{\prime}+4(n-2)+5-t}{\nt[r]^{\prime}+4(n-2)+5})
=\sseqzero{\hat{\imath}}{t+1}$ (or
$\overline{\alpha}(\linebreak\xseq{i}{\nt[r]^{\prime}+4(n-2)+3-(t+1)}{\nt[r]^{\prime}+4(n-2)+3-t}{\nt[r]^{\prime}+4(n-2)+3})
=\sseqzero{\hat{\imath}}{t+1}$).

Let $\nt[\hat{\imath}]^{\prime}\in\oneton{\nt}$ so that
$\imalpha{i}{\nt[r]^{\prime}+1}\in
S^{\nt[\hat{\imath}]^{\prime}}$.

Case 1:  $\imalpha{i}{\nt[r]^{\prime}+2(n-2)+1}\in
S^{\nt[\hat{\imath}]^{\prime}}$.

Then, there exists an even $\nt[t]^{\prime}\in\oneton[t]{2(n-2)}$
so that
$\imalpha{i}{\nt[r]^{\prime}+4(n-2)+5-(\nt[t]^{\prime}+2)}=s_1^{\nt[\hat{\imath}]^{\prime}}$
(or
$\imalpha{i}{\nt[r]^{\prime}+4(n-2)+3-(\nt[t]^{\prime}+2)}=s_1^{\nt[\hat{\imath}]^{\prime}}$).
If $\overline{\alpha}(\xseq{i}{0}{1}{\nt[r]^{\prime}})\subseteq
S^{\nt[\hat{\imath}]^{\prime}}$, then conclusion (ii) holds with
$r^{\prime}=\nt[r]^{\prime}$ and
$\hat{\imath}^{\prime}=\nt[\hat{\imath}]^{\prime}$.  Suppose
otherwise.  By Proposition \ref{claim9} there exists
$r^{\prime}<\nt[r]^{\prime}$ so that
$\overline{\nPhi{m}{0}}(\xseq{i}{r^{\prime}+1}{r^{\prime}+2}{r^{\prime}+2(n-2)+1})=\vpjseq[,v_n]$
and
$\overline{\alpha}(\xedge{i}{r^{\prime}+\nt[t]^{\prime}}{r^{\prime}+\nt[t]^{\prime}+1})
=\edge{s}{0}{s_1^{\nt[\hat{\imath}]^{\prime}}}{}$.  Let
$\hat{\imath}^{\prime}\in\oneton{\nt}$ so that
$\imalpha{i}{r^{\prime}+1}\in S^{\hat{\imath}^{\prime}}$.  Then,
there exists an even $t^{\prime}\in\oneton[2]{\nt[t]^{\prime}}$ so
that
$\imalphabar{i}{r^{\prime}+1}{r^{\prime}+2}{r^{\prime}+t^{\prime}}=
\sseq[,s_0]{\hat{\imath}^{\prime}}{t^{\prime}-1}{t^{\prime}-2}{1}$
and $\{\imalpha{i}{r^{\prime}+\nt[t]+1}:\nt[t]\mbox{ is even and
}\nt[t]\in\oneton[t^{\prime}]{\nt[t]^{\prime}}\}\subseteq\{s_1^{\delta}:\delta\in\oneton{\nt}\}$,
giving $\{v_{\impj{\frac{\tilde{t}}{2}}}:\tilde{t}\mbox{ is even
and
}\tilde{t}\in\oneton[t^{\prime}]{\nt[t]^{\prime}}\}\subseteq\{\beta(s_1^{\delta}):\delta\in\oneton{\nt}\}$.
Since $\nt[t]^{\prime}\ge t$ and by hypothesis
$\{v_n,v_{\impj{\frac{\tilde{t}}{2}}}:\tilde{t}\mbox{ is even and
}\tilde{t}\in\oneton[t]{2(n-2)}\}\subseteq\{\beta(s_1^{\delta}):\delta\in\oneton{\nt}\}$,
then $\{v_n,v_{\impj{\frac{\tilde{t}}{2}}}:\tilde{t}\mbox{ is even
and
}\tilde{t}\in\oneton[t^{\prime}]{2(n-2)}\}\subseteq\{\beta(s_1^{\delta}):\delta\in\oneton{\nt}\}$,
and conclusion (i) is reached.

Case 2: $\imalpha{i}{\nt[r]^{\prime}+2(n-2)+1}\notin
S^{\nt[\hat{\imath}]^{\prime}}$.

Then, there exists an even $t^{\prime}\in\oneton[2]{2(n-2)}$ so
that
$\imalphabar{i}{\nt[r]^{\prime}+1}{\nt[r]^{\prime}+2}{\nt[r]^{\prime}+t^{\prime}}=
\sseq[s_0]{\nt[\hat{\imath}]^{\prime}}{t^{\prime}-1}{t^{\prime}-2}{1}$.
Let $\ntt[\hat{\imath}]^{\prime}\in\oneton{\nt}$ so that
$\imalpha{i}{\nt[r]^{\prime}+2(n-2)+1}\in
S^{\ntt[\hat{\imath}]^{\prime}}$. Then, there exists an even
$\nt[t]^{\prime}\in\{\mbox{max}\{t,t^{\prime}\},\ldots,2(n-2)\}$
so that
$\overline{\alpha}(\xedge{i}{\nt[r]^{\prime}+\nt[t]^{\prime}}{\nt[r]^{\prime}+\nt[t]^{\prime}+1})=
\edge{s}{0}{s_1^{\ntt[\hat{\imath}]^{\prime}}}{}$ and
$\{\imalpha{i}{\nt[r]^{\prime}+\nt[t]+1}:\nt[t]\mbox{ is even and
}\nt[t]\in\oneton[t^{\prime}]{\nt[t]^{\prime}}\}\subseteq\{s_1^{\delta}:\delta\in\oneton{\nt}\}$,
giving $\{v_{\impj{\frac{\tilde{t}}{2}}}:\tilde{t}\mbox{ is even
and
}\tilde{t}\in\oneton[t^{\prime}]{\nt[t]^{\prime}}\}\subseteq\{\beta(s_1^{\delta}):\delta\in\oneton{\nt}\}$.
Since $\nt[t]^{\prime}\ge t$ and by hypothesis
$\{v_n,v_{\impj{\frac{\tilde{t}}{2}}}:\tilde{t}\mbox{ is even and
}\tilde{t}\in\oneton[t]{2(n-2)}\}\subseteq\{\beta(s_1^{\delta}):\delta\in\oneton{\nt}\}$,
then $\{v_n,v_{\impj{\frac{\tilde{t}}{2}}}:\tilde{t}\mbox{ is even
and
}\tilde{t}\in\oneton[t^{\prime}]{2(n-2)}\}\subseteq\{\beta(s_1^{\delta}):\delta\in\oneton{\nt}\}$,
and conclusion (i) holds with $r^{\prime}=\nt[r]^{\prime}$ and
$\hat{\imath}^{\prime}=\nt[\hat{\imath}]^{\prime}$.
\end{proof}

\begin{proposition} \label{claim11}
Let $x\in\mathrm{V}(\X{m})$ so that $\alpha(x)=s_0$. Suppose there
exist $i\in\oneton{n}$, $r\in\oneton[0]{\ell_i}$, and
$j\in\oneton{n-2}$ so that
$\overline{\nPhi{m}{0}}(\xseq{i}{r+1}{r+2}{r+2(n-2)})=\vpjseq[]$,
$\overline{\alpha}(\xseq{i}{r+1}{r+2}{r+t})=\sseq[,s_0]{\hat{\imath}}{t-1}{t-2}{1}$
for some even $t\in\oneton[2]{2(n-2)}$ and for some
$\hat{\imath}\in\oneton{\nt}$, and
$\{v_n,v_{\impj{\frac{\tilde{t}}{2}}}:\tilde{t}\mbox{ is even and
}\tilde{t}\in\oneton[t]{2(n-2)}\}\subseteq\{\beta(s_1^{\delta}):\delta\in\oneton{\nt}\}$.
Then, there exist $\acute{\imath}\in\oneton{n}$,
$\acute{r}\in\oneton[0]{\ell_{\acute{\imath}}}$, and an even
$\acute{t}\in\oneton[2]{2(n-2)}$ so that

i)
$\overline{\nPhi{m}{0}}(\xseq{\acute{\imath}}{\acute{r}+1}{\acute{r}+2}{\acute{r}+2(n-2)})=\vpjseq[]$,
$x=x_{\acute{r}+\acute{t}}^{\acute{\imath}}$, and
$\{v_n,v_{\impj{\frac{\tilde{t}}{2}}}:\tilde{t}\mbox{ is even and
}\tilde{t}\in\oneton[\acute{t}]{2(n-2)}\}\subseteq\{\beta(s_1^{\delta}):\delta\in\oneton{\nt}\}$,
or

ii)
$\overline{\nPhi{m}{0}}(\xseq{\acute{\imath}}{\acute{r}-1}{\acute{r}-2}{\acute{r}-2(n-2)})\hspace{-.06cm}=\hspace{-.06cm}\vpjseq[]$,
$x=x_{\acute{r}-\acute{t}}^{\acute{\imath}}$, and
$\{v_n,v_{\impj{\frac{\tilde{t}}{2}}}:\tilde{t}\mbox{ is even and
}\tilde{t}\in\oneton[\acute{t}]{2(n-2)}\}\subseteq\{\beta(s_1^{\delta}):\delta\in\oneton{\nt}\}$.
\end{proposition}

\begin{proof}
By Proposition \ref{claim10}, there exists
$r^{\prime}\in\oneton[0]{\ell_i}$ so that $r^{\prime}\le r$,
$\overline{\nPhi{m}{0}}(\xseq{i}{r^{\prime}+1}{r^{\prime}+2}{r^{\prime}+2(n-2)})=\vpjseq[]$,
$\overline{\alpha}(\xxxxxxseq{i}{r^{\prime}+1}{r^{\prime}+2}{r^{\prime}+t^{\prime}})=\sseq[,s_0]{\hat{\imath}^{\prime}}{t^{\prime}-1}{t^{\prime}-2}{1}$
for some even $t^{\prime}\in\oneton[2]{2(n-2)}$ and for some
$\hat{\imath}^{\prime}\in\oneton{\nt}$,
$\{v_n,v_{\impj{\frac{\tilde{t}}{2}}}:\tilde{t}\mbox{ is even and
}\tilde{t}\in\oneton[t^{\prime}]{2(n-2)}\}\subseteq\{\beta(s_1^{\delta}):\delta\in\oneton{\nt}\}$,
and a)  $\overline{\alpha}(\xseq{i}{0}{1}{r^{\prime}})\subseteq
S^{\hat{\imath}^{\prime}}$  or b) there exists
$\nt[r]^{\prime}\in\oneton[0]{\ell_i}$ so that
$\nt[r]^{\prime}<r^{\prime}$,
$\overline{\nPhi{m}{0}}(\xseq{i}{\nt[r]^{\prime}+1}{\nt[r]^{\prime}+2}{\nt[r]^{\prime}+4(n-2)+5})=
\vpjseq$ (or
$\overline{\nPhi{m}{0}}(\xseq{i}{\nt[r]^{\prime}+1}{\nt[r]^{\prime}+2}{\nt[r]^{\prime}+4(n-2)+3})\linebreak[0]=
\langle
v_{n-2},v_0,v_1,v_0,\underline{v_2,v_0},\ldots,\underline{v_{n-3},v_0},
v_n,v_0,v_{n-3},v_0,\underline{v_{n-4},v_0},\ldots,\underline{v_1,v_0},\linebreak[0]v_{n-2},v_0,v_{n-1}
\rangle$),
$\overline{\alpha}(\xseq{i}{\nt[r]^{\prime}+4(n-2)+5-(t^{\prime}+1)}{\nt[r]^{\prime}+4(n-2)+5-t^{\prime}}{\nt[r]^{\prime}+4(n-2)+5})
=\sseqzero{\hat{\imath}^{\prime}}{t^{\prime}+1}$ (or
$\overline{\alpha}(\xxxxxxxseq{i}{\nt[r]^{\prime}+4(n-2)+3-(t^{\prime}+1)}{\nt[r]^{\prime}+4(n-2)+3-t^{\prime}}{\nt[r]^{\prime}+4(n-2)+3})
=\sseqzero{\hat{\imath}^{\prime}}{t^{\prime}+1}$), and
$\overline{\alpha}(\xseq{i}{0}{1}{\nt[r]^{\prime}+2(n-2)+1})\subseteq
S^{\nt[\hat{\imath}]^{\prime}}$ where
$\nt[\hat{\imath}]^{\prime}\in\oneton{\nt}$ so that
$\imalpha{i}{\nt[r]^{\prime}+1}\in
S^{\nt[\hat{\imath}]^{\prime}}$.

Let $\breve{\imath}\in\oneton{n}$ and
$\breve{r}\in\oneton[0]{\ell_{\breve{\imath}}}$ so that
$\alpha^{-1}(\{s_0\})\bigcap\xseq{\breve{\imath}}{0}{1}{\breve{r}}=\{x_{\breve{r}}^{\breve{\imath}}\}$.
If $\breve{r}$ exists and (a) holds, then by the hypothesis and
Proposition \ref{claim9}, conclusion (i) of the claim holds for
$x_{\breve{r}}^{\breve{\imath}}$ with
$\acute{\imath}=\breve{\imath}$, $\acute{r}=\breve{r}-t^{\prime}$,
and $\acute{t}=t^{\prime}$. Suppose $\breve{r}$ exists and (b)
holds, and let
$\nt[t]^{\prime}\in\oneton[2(n-2)+4]{4(n-2)+4-t^{\prime}}$ (or
$\nt[t]^{\prime}\in\oneton[2(n-2)+2]{4(n-2)+2-t^{\prime}}$) so
that
$\alpha^{-1}(\{s_0\})\bigcap\xseq{i}{0}{1}{\nt[r]^{\prime}+\nt[t]^{\prime}}=
\{x_{\nt[r]^{\prime}+\nt[t]^{\prime}}^{i}\}$.  Then by the
hypothesis and Proposition \ref{claim9}, conclusion (ii) of the
claim holds for $x_{\breve{r}}^{\breve{\imath}}$ with
$\acute{\imath}=\breve{\imath}$,
$\acute{r}=\breve{r}+4(n-2)+4-\nt[t]^{\prime}$ (or
$\acute{r}=\breve{r}+4(n-2)+2-\nt[t]^{\prime}$), and
$\acute{t}=4(n-2)+4-\nt[t]^{\prime}$ (or
$\acute{t}=4(n-2)+2-\nt[t]^{\prime}$).

Suppose $\breve{r}$ exists and let
$\breve{r}_{\theta}\in\oneton[\breve{r}+1]{\ell_{\breve{\imath}}}$
so that the claim holds for all
$x\in\{x_{\nt[r]}^{\breve{\imath}}:\nt[r]\in\oneton[0]{\breve{r}_{\theta}-1}\mbox{
and }\alpha(x_{\nt[r]}^{\breve{\imath}})=s_0\}$ and
$\alpha(x_{\breve{r}_{\theta}}^{\breve{\imath}})=s_0$.  By
Proposition \ref{claim9},
$\alpha(x_{\breve{r}_{\theta}-1}^{\breve{\imath}})=s_1^{\hat{\imath}_{-}}$
for some $\hat{\imath}_{-}\in\oneton{\nt}$ and
$\alpha(x_{\breve{r}_{\theta}+1}^{\breve{\imath}})=s_1^{\hat{\imath}_{+}}$
for some $\hat{\imath}_{+}\in\oneton{\nt}$.

Case i: $\alpha(x_{\breve{r}_{\theta}-2}^{\breve{\imath}})=s_0$.

By the induction hypothesis, there exist
$\acute{r}_0\in\oneton[0]{\ell_{\breve{\imath}}}$ and an even
$\acute{t}_0\in\oneton[2]{2(n-2)}$ so that

1)
$\overline{\nPhi{m}{0}}(\xseq{\breve{\imath}}{\acute{r}_0+1}{\acute{r}_0+2}{\acute{r}_0+2(n-2)})=\vpjseq[]$,
$\breve{r}_{\theta}-2=\acute{r}_0+\acute{t}_0$, and
$\{v_n,v_{\impj{\frac{\tilde{t}}{2}}}:\tilde{t}\mbox{ is even and
}\tilde{t}\in\oneton[\acute{t}_0]{2(n-2)}\}\subseteq\{\beta(s_1^{\delta}):\delta\in\oneton{\nt}\}$,
or

2)
$\overline{\nPhi{m}{0}}(\xseq{\breve{\imath}}{\acute{r}_0-1}{\acute{r}_0-2}{\acute{r}_0-2(n-2)})=\vpjseq[]$,
$\breve{r}_{\theta}-2=\acute{r}_0-\acute{t}_0$, and
$\{v_n,v_{\impj{\frac{\tilde{t}}{2}}}:\tilde{t}\mbox{ is even and
}\tilde{t}\linebreak[0]\in\oneton[\acute{t}_0]{2(n-2)}\}\subseteq\{\beta(s_1^{\delta}):\delta\in\oneton{\nt}\}$.

\ \ Case i.a: Suppose (1) holds.

If $\acute{t}_0\in\oneton[2]{2(n-2)-2}$, then conclusion (i) of
the claim holds for $x_{\breve{r}_{\theta}}^{\breve{\imath}}$ with
$\acute{\imath}=\breve{\imath}$, $\acute{r}=\acute{r}_0$, and
$\acute{t}=\acute{t}_0+2$.  If $\acute{t}_0=2(n-2)$, then by
Proposition \ref{claim9}, conclusion (ii) of the claim holds for
$x_{\breve{r}_{\theta}}^{\breve{\imath}}$ with
$\acute{\imath}=\breve{\imath}$, $\acute{r}=\acute{r}_0+4(n-2)+2$,
and $\acute{t}=2(n-2)$.

\ \  Case i.b: Suppose (2) holds.

If $\acute{t}_0\in\oneton[4]{2(n-2)}$, then conclusion (ii) of the
claim holds for $x_{\breve{r}_{\theta}}^{\breve{\imath}}$ with
$\acute{\imath}=\breve{\imath}$, $\acute{r}=\acute{r}_0$, and
$\acute{t}=\acute{t}_0-2$.  If $\acute{t}_0=2$, then
$|\{v_n,v_{\impj{\frac{\tilde{t}}{2}}}:\tilde{t}\mbox{ is even and
}\tilde{t}\in\oneton[\acute{t}_0]{2(n-2)}\}|=\nt$, but
$\beta(s_1^{\hat{\imath}_{+}})=v_{n-1}$, since by Proposition
\ref{claim9}
$\nPhi{m}{0}(x_{\breve{r}_{\theta}+1}^{\breve{\imath}})=v_{n-1}$,
giving a contradiction.

Case ii:
$\alpha(x_{\breve{r}_{\theta}-2}^{\breve{\imath}})\not=s_0$.

Then,
$\alpha(x_{\breve{r}_{\theta}-2}^{\breve{\imath}})=s_2^{\hat{\imath}_{-}}$,
and by the induction hypothesis, there exist
$\acute{r}_0\in\oneton[0]{\ell_{\breve{\imath}}}$ and an even
$\acute{t}_0\in\oneton[2]{2(n-2)}$ so that

1) $\acute{r}_0+\acute{t}_0<\breve{r}_{\theta}$,
$\overline{\nPhi{m}{0}}(\xseq{\breve{\imath}}{\acute{r}_0+1}{\acute{r}_0+2}{\acute{r}_0+2(n-2)})=\vvpjseq[]$,
$\overline{\alpha}(\edge{x^{\breve{\imath}}_{\acute{r}_0+\acute{t}_0}}{}{x^{\breve{\imath}}_{\acute{r}_0+\acute{t}_0+1},x^{\breve{\imath}}_{\acute{r}_0+\acute{t}_0+2}}{})=
\edge{s}{0}{s_1^{\hat{\imath}_{-}},s_2^{\hat{\imath}_{-}}}{}$, and
$\{v_n,v_{\impj{\frac{\tilde{t}}{2}}}:\tilde{t}\mbox{ is even}
\linebreak[0]\mbox{and
}\tilde{t}\in\oneton[\acute{t}_0]{2(n-2)}\}\subseteq\{\beta(s_1^{\delta}):\delta\in\oneton{\nt}\}$,
or

2)  $\acute{r}_0-\acute{t}_0<\breve{r}_{\theta}$,
$\overline{\nPhi{m}{0}}(\xseq{\breve{\imath}}{\acute{r}_0-1}{\acute{r}_0-2}{\acute{r}_0-2(n-2)})=\vvpjseq[]$,
$\overline{\alpha}(\edge{x^{\breve{\imath}}_{\acute{r}_0-\acute{t}_0}}{}{x^{\breve{\imath}}_{\acute{r}_0-\acute{t}_0+1},x^{\breve{\imath}}_{\acute{r}_0-\acute{t}_0+2}}{})=
\edge{s}{0}{s_1^{\hat{\imath}_{-}},s_2^{\hat{\imath}_{-}}}{}$, and
$\{v_n,v_{\impj{\frac{\tilde{t}}{2}}}:\tilde{t}\mbox{ is even}
\linebreak[0]\mbox{and
}\tilde{t}\in\oneton[\acute{t}_0]{2(n-2)}\}\subseteq\{\beta(s_1^{\delta}):\delta\in\oneton{\nt}\}$.

\ \ Case ii.a: Suppose (1) holds.

By Proposition \ref{claim9}, conclusion (ii) of the claim holds
for $x_{\breve{r}_{\theta}}^{\breve{\imath}}$ with
$\acute{\imath}=\breve{\imath}$,
$\acute{r}=\breve{r}_{\theta}+\acute{t}_0$, and
$\acute{t}=\acute{t}_0$.

\ \ Case ii.b: Suppose (2) holds.

By Proposition \ref{claim9}, conclusion (i) of the claim holds for
$x_{\breve{r}_{\theta}}^{\breve{\imath}}$ with
$\acute{\imath}=\breve{\imath}$,
$\acute{r}=\breve{r}_{\theta}-\acute{t}_0$, and
$\acute{t}=\acute{t}_0$.

Thus, the claim follows by induction.
\end{proof}

\begin{proposition} \label{claim12}
Let distinct $x,y,z\in\mathrm{V}(\X{m})$ be so that $x\mbox{ and
}y$ are adjacent, $y\mbox{ and }z$ are adjacent, $\alpha(y)=s_0$,
$\nPhi{m}{0}(z)\not=v_{n-1}$, and $\nPhi{m}{0}(x)=v_{n}$. Then,

i) there exist $i\in\oneton{n}$, $r\in\oneton[0]{\ell_i}$, and
$j\in\oneton{n-2}$ so that
$\overline{\nPhi{m}{0}}(\xseq{i}{r+1}{r+2}{r+2(n-2)})=\vpjseq[]$,
$\overline{\alpha}(\xseq{i}{r+1}{r+2}{r+t})=\sseq[,s_0]{\hat{\imath}}{t-1}{t-2}{1}$
for some even $t\in\oneton[2]{2(n-2)}$ and for some
$\hat{\imath}\in\oneton{\nt}$, and
$\{v_n,v_{\impj{\frac{\tilde{t}}{2}}}:\tilde{t}\mbox{ is even and
}\tilde{t}\in\oneton[t]{2(n-2)}\}\subseteq\{\beta(s_1^{\delta}):\delta\in\oneton{\nt}\}$,
or

ii)  there exists $j\in\oneton{n-2}$ so that if
$\nt[x]\in\mathrm{V}(\X{m})$ and $\alpha(\nt[x])=s_0$, then there
exist $i\in\oneton{n}$ and $r\in\oneton[0]{\ell_i}$ so that
$\overline{\nPhi{m}{0}}(\xseq{i}{r-1}{r-2}{r-2(n-2)})=\vpjseq[]$,
$\nt[x]=x_{r-t}^i$ for some even $t\in\oneton[2]{2(n-2)}$, and
$\{v_n,v_{\impj{\frac{\tilde{t}}{2}}}:\tilde{t}\mbox{ is even and
}\tilde{t}\in\oneton[t]{2(n-2)}\}\subseteq\{\beta(s_1^{\delta}):\delta\in\oneton{\nt}\}$.
\end{proposition}

\begin{proof}
By Proposition \ref{claim9}, there exist $i^{\prime}\in\oneton{n}$
and $r^{\prime}\in\oneton[0]{\ell_{i^{\prime}}}$ so that for some
$j\in\oneton{n-2}$,

a)
$\overline{\nPhi{m}{0}}(\xseq{i^{\prime}}{r^{\prime}+1}{r^{\prime}+2}{r^{\prime}+2(n-2)})=\vpjseq[]$
and
$\edge{x}{}{y,z}{}=\edge{x_{r^{\prime}+2(n-2)+1}^{i^{\prime}}}{}{x_{r^{\prime}+2(n-2)}^{i^{\prime}},x_{r^{\prime}+2(n-2)-1}^{i^{\prime}}}{}$,
or

b)
$\overline{\nPhi{m}{0}}(\xseq{i^{\prime}}{r^{\prime}-1}{r^{\prime}-2}{r^{\prime}-2(n-2)})=\vpjseq[]$
and
$\edge{x}{}{y,z}{}=\edge{x_{r^{\prime}-2(n-2)-1}^{i^{\prime}}}{}{x_{r^{\prime}-2(n-2)}^{i^{\prime}},x_{r^{\prime}-2(n-2)+1}^{i^{\prime}}}{}$.

If (a) holds, then there exist an even $t\in\oneton[2]{2(n-2)}$
and $\hat{\imath}\in\oneton{\nt}$ so that the conditions of
conclusion (i) are satisfied with $r=r^{\prime}$ and
$i=i^{\prime}$. Suppose (b) is the case, and let
$\hat{\imath}^{\prime}\in\oneton{\nt}$ be so that
$\alpha(x)=s_{1}^{\hat{\imath}^{\prime}}$. If
$\alpha(\xseq{i^{\prime}}{0}{1}{r^{\prime}-2(n-2)-1})\not\subseteq
S^{\hat{\imath}^{\prime}}$, then by Proposition \ref{claim9},
there exist $r\in\oneton[0]{r^{\prime}-4(n-2)-4}$ (or
$r\in\oneton[0]{r^{\prime}-4(n-2)-2}$), an even
$t\in\oneton[2]{2(n-2)}$, and $\hat{\imath}\in\oneton{\nt}$ so
that the conditions of conclusion (i) are satisfied with
$i=i^{\prime}$.  Suppose
$\alpha(\xseq{i^{\prime}}{0}{1}{r^{\prime}-2(n-2)-1})\subseteq
S^{\hat{\imath}^{\prime}}$.  By Proposition \ref{claim9}, if there
exist $i\in\oneton{n}$ and $r_i\in\oneton[0]{\ell_i}$ so that
$\alpha(x_{r_i}^i)=s_0$ and $\alpha(\xseq{i}{0}{1}{r_i})\subseteq
S^{\hat{\imath}^{\prime}}$, then
$\alpha^{-1}(\{s_0\})\bigcap\xseq{i}{0}{1}{r_i}=\{x_{r_i}^i\}$ and
$\overline{\nPhi{m}{0}}(\xseq{i}{\nt[r]+1}{\nt[r]+2}{\nt[r]+2(n-2)})=\vpjseq[]$
where $\nt[r]=r_i-2(n-2)-4$ (or $\nt[r]=r_i-2(n-2)-2$).  Suppose
$r_i$ exists for some $i\in\oneton{n}$.  Let
$\nt[r]^{\prime}=\nt[r]+4(n-2)+4$ (or
$\nt[r]^{\prime}=\nt[r]+4(n-2)+2$).  By Proposition \ref{claim9},
$\overline{\nPhi{m}{0}}(\xseq{i}{\nt[r]^{\prime}-1}{\nt[r]^{\prime}-2}{\nt[r]^{\prime}-2(n-2)})=\vpjseq[]$.
Let
$\hat{\imath}\in\oneton{\nt}\setminus\{\hat{\imath}^{\prime}\}$ so
that
$\overline{\alpha}(\xseq{i}{\nt[r]^{\prime}-1}{\nt[r]^{\prime}-2}{\nt[r]^{\prime}-t_i})=
\sseq[,s_0]{\hat{\imath}}{t_i-1}{t_i-2}{1}$ for some even
$t_i\in\oneton[2]{2(n-2)}$.  By Proposition \ref{claim9},
$\overline{\nPhi{m}{0}}(\xedge{i}{\nt[r]^{\prime}}{\nt[r]^{\prime}+1})=\edge{v}{0}{v}{n-1}$.
If $t_i=2$, then
$|\{v_n,v_{\impj{\frac{\tilde{t}}{2}}}:\tilde{t}\mbox{ is even and
}\tilde{t}\in\oneton[t_i]{2(n-2)}\}|=\nt$.  Hence,
$\overline{\alpha}(\xedge{i}{\nt[r]^{\prime}}{\nt[r]^{\prime}+1})=\xedge[s]{\hat{\imath}}{t_i}{t_i+1}$.
If there exists
$\nt[r]_i^{\prime}\in\oneton[\tilde{r}^{\prime}+1]{\ell_i}$ so
that
$\xseq{i}{\nt[r]^{\prime}}{\nt[r]^{\prime}+1}{\nt[r]_i^{\prime}}\subseteq
S^{\hat{\imath}}$ and $\alpha(x_{\nt[r]_i^{\prime}}^{i})=s_0$,
then by Proposition \ref{claim9}, the conditions of conclusion (i)
are satisfied with $r=\nt[r]_i^{\prime}-t_i$ and $t=t_i$.  Suppose
$\nt[r]_i^{\prime}$ does not exist for each $i\in\oneton{n}$, and
let $\nt[x]\in\mathrm{V}(\X{m})$ so that $\alpha(\nt[x])=s_0$.
Then, there exist $i\in\oneton{n}$ and some even
$t\in\oneton[t_i]{2(n-2)}$ so that the conditions of conclusion
(ii) are satisfied with $r=r_i+2(n-2)$.
\end{proof}

\begin{proposition} \label{claim13}
If $m,\nt,T,\alpha,$ and $\beta$ exist, then for some
$\nt_0\in\oneton{\nt}$, there exist a simple-$\nt_0$-od $\nt[T]$,
$\nt_0$ arcs $\nt[S]^1,\nt[S]^2,\ldots,\nt[S]^{\nt_0}$ so that
$\bigcup\limits_{i=1}^{\nt_0}\nt[S]^i=\nt[T]$, $\nt[S]^i\bigcap
\nt[S]^j=\{\nt[s]_0\}$ for each distinct $i,j\in\oneton{\nt_0}$,
so $\nt[s]_0$ is an endpoint of $\nt[S]^i$ for each
$i\in\oneton{\nt_0}$, and $\nt[S]^i=\stseqzero{i}{\nt[k]_i}$ for
some $\nt[k]_i\in\{1,2,\ldots\}$ for each $i\in\oneton{\nt_0}$,
and simplicial maps $\nt[\alpha]:\X{m}\longrightarrow \nt[T]$ and
$\nt[\beta]:\nt[T]\longrightarrow\X{0}$ so that
$\nt[\beta]\circ\nt[\alpha]=\nPhi{m}{0}$,
$\nt[\beta](\nt[s]_0)=v_0$, and if distinct
$x,y,z\in\mathrm{V}(\X{m})$ are such that $x\mbox{ and }y$ are
adjacent, $y\mbox{ and }z$ are adjacent, and
$\nt[\alpha](y)=\nt[s]_0$, then $\nPhi{m}{0}(x)=v_{n-1}$ or
$\nPhi{m}{0}(z)=v_{n-1}$.
\end{proposition}

\begin{proof}
Let $D=\{\delta\in\oneton{\nt}:\mbox{ there exist distinct
}x,y,z\in\mathrm{V}(\X{m})\mbox{ so that
}\edge{x}{}{y,z}{}\subseteq\X{m},\alpha(y)=s_0,\nPhi{m}{0}(x)\not=v_{n-1},\nPhi{m}{0}(z)\not=v_{n-1},\mbox{
and }\alpha(x)=s_1^{\delta}\}$, $D_0=\{\delta\in
D:s_2^{\delta}\mbox{ does not exist}\mbox{ and
}\beta(s_{1}^{\delta})\not=v_n\}$, $D_1=\{\delta\in
D:s_2^{\delta}\mbox{ exists and }
\overline{\beta}(\sseq{\delta}{1}{2}{2[(n-2)-t_{\delta}]+3})=\langle
v_{\impj[j_{\delta}]{t_{\delta}}},v_0,\underline{v_{\impj[j_{\delta}]{t_{\delta}+1}},v_0},\ldots,\underline{v_{\impj[j_{\delta}]{n-2}},v_0},\linebreak[0]v_n
\rangle\mbox{ for some }j_{\delta}\in\oneton{n-2}\mbox{ and }
t_{\delta}\in\oneton{n-2}\}$, $D_2=\{\delta\in
D:\beta(s_1^{\delta})=v_n\}$, and $D_3=\{\delta\in
D:s_2^{\delta}\mbox{ exists and
}\overline{\beta}(\sseq{\delta}{1}{2}{2t_{\delta}+1})=\langle
v_{\impj[j_{\delta}]{t_{\delta}}},v_0,\underline{v_{\impj[j_{\delta}]{t_{\delta}-1}},v_0},\linebreak[0]\ldots,\underline{v_{\impj[j_{\delta}]{1}},v_0},v_{n-1}
\rangle\mbox{ for}\linebreak[0]\mbox{some
}j_{\delta}\in\oneton{n-2}\mbox{ and
}t_{\delta}\in\oneton{n-2}\}$. Then, $D_i\bigcap D_j=\emptyset$
for each distinct $i,j\in\oneton[0]{3}$.  Without loss of
generality, suppose $D_0$ is empty or there exists
$\nt[\delta]\in\oneton{\nt}$ so that
$D_0=\{\nt[\delta],\nt[\delta]+1,\ldots,\nt\}$.  By Proposition
\ref{claim9}, $\bigcup\limits_{i=0}^{3}D_i=D$, and $j_{\delta}$
and $t_{\delta}$ are unique for each $\delta\in D_1\bigcup D_3$.
If $\delta\in D_2$, let $j_{\delta}$ be the unique integer $j$
guaranteed by the conclusions of Propositions \ref{claim12} and
\ref{claim11}.

Let $\nt[k]_{\delta}=k_{\delta}+2(t_{\delta}-1)$ if $\delta\in
D_1$, $\nt[k]_{\delta}=k_{\delta}+2(n-2)$ if $\delta\in D_2$,
$\nt[k]_{\delta}=k_{\delta}-2t_{\delta}$ if $\delta\in D_3$, and
$\nt[k]_{\delta}=k_{\delta}$ if $\delta\in\oneton{\nt}\setminus
D$. Let $\nt_0=\nt-|D_0|$ and
$\nt[S]^1,\nt[S]^2,\ldots,\nt[S]^{\nt_0}$ be $\nt_0$ arcs so that
$\nt[S]^i=\stseqzero{i}{\nt[k]_i}$ for each $i\in\oneton{\nt_0}$
and $\nt[S]^i\bigcap\nt[S]^j=\{\nt[s]_0\}$ for each distinct
$i,j\in\oneton{\nt_0}$. Let $\nt[T]$ be a simple-$\nt_0$-od so
that $\bigcup\limits_{i=1}^{\nt_0}\nt[S]^i=\nt[T]$ and
$\nt[\beta]:\nt[T]\longrightarrow\X{0}$ be the simplicial map
defined as:  $\nt[\beta](\nt[s]_0)=\beta(s_0)$,
$\nt[\beta](\stseq{\delta}{1}{2}{2(t_{\delta}-1)})=\langle
v_{\impj[j_{\delta}]{1}},v_0,\underline{v_{\impj[j_{\delta}]{2}},v_0},\ldots,\underline{v_{\impj[j_{\delta}]{t_{\delta}-1}},v_0}
\rangle$ if $\delta\in D_1$ and $t_{\delta}>1$,
$\nt[\beta](\nt[s]_{k}^{\delta})=\beta(s_{k-2(t_{\delta}-1)}^{\delta})$
for $k\in\oneton[2(t_{\delta}-1)+1]{\nt[k]_{\delta}}$ if
$\delta\in D_1$, $\nt[\beta](\stseq{\delta}{1}{2}{2(n-2)})=\langle
v_{\impj[j_{\delta}]{1}},v_0,\underline{v_{\impj[j_{\delta}]{2}},v_0},\ldots,\underline{v_{\impj[j_{\delta}]{n-2}},v_0}
\rangle$ if $\delta\in D_2$,
$\nt[\beta](\nt[s]_{k}^{\delta})=\beta(s_{k-2(n-2)}^{\delta})$ for
$k\in\oneton[2(n-2)+1]{\nt[k]_{\delta}}$ if $\delta\in D_2$,
$\nt[\beta](\nt[s]_{k}^{\delta})=\beta(s_{k+2t_{\delta}}^{\delta})$
for $k\in\oneton{\nt[k]_{\delta}}$ if $\delta\in D_3$, and
$\nt[\beta](\nt[s]_{k}^{\delta})=\beta(s_{k}^{\delta})$ for
$k\in\oneton{\nt[k]_{\delta}}$ if $\delta\in\oneton{\nt}\setminus
D$.

Let
$E=\{\xseq{}{0}{1}{2(n-2)}\subseteq\X{m}:x_0,x_1,\ldots,x_{2(n-2)}\in\mathrm{V}(\X{m})$
are distinct and
$\overline{\nPhi{m}{0}}(\xseq{}{0}{1}{2(n-2)+1})=\langle
v_0,v_{\impj[j_e]{1}},v_0,\underline{v_{\impj[j_e]{2}},v_0},\ldots,\underline{v_{\impj[j_e]{n-2}},v_0},v_n
\rangle$ for some $j_e\in\oneton{n-2}\mbox{ and some
}x_{2(n-2)+1}\in\mathrm{V}(\X{m})\mbox{ adjacent to
}x_{2(n-2)}\}$.  By Proposition \ref{claim9}, $j_e$ is unique for
each $e\in E$, and if $e_1,e_2\in E$ are distinct, $\X{m}\supseteq
e_1\bigcap e_2=\emptyset$.

Let $e\in E$.  By Proposition \ref{claim9}, there exist
$i\in\oneton{n}$ and $r\in\oneton[0]{\ell_i}$ so that

a) $\xseq{}{0}{1}{2(n-2)}=\xseq{i}{r}{r+1}{r+2(n-2)}$ or

b) $\xseq{}{0}{1}{2(n-2)}=\xseq{i}{r}{r-1}{r-2(n-2)}$.

If (a), let $\delta_e\in\oneton{\nt}$ so that
$\alpha(x_{r+2(n-2)+1}^i)\in S^{\delta_e}\setminus\{s_0\}$ and
$t_e\in\oneton[0]{2(n-2)+1}$ so that
$t_e=\mathrm{min}\{t\in\oneton[0]{2(n-2)+1}:\alpha(x_{r+t}^i)\in
S^{\delta_e}\setminus\{s_0\}\}$. If (b), let
$\delta_e\in\oneton{\nt}$ so that $\alpha(x_{r-2(n-2)-1}^i)\in
S^{\delta_e}\setminus\{s_0\}$ and $t_e\in\oneton[0]{2(n-2)+1}$ so
that
$t_e=\mathrm{min}\{t\in\oneton[0]{2(n-2)+1}:\alpha(x_{r-t}^i)\in
S^{\delta_e}\setminus\{s_0\}\}$.  Let $E_0=\{e\in E:t_e>1\}$.

Let $\nt[\alpha]:\X{m}\longrightarrow\nt[T]$ be defined as:
$\nt[\alpha](x)=\nt[s]_0$ if $x=x_0$ where $x_0\in e$ for some
$e\in E_0$, $\nt[\alpha](x)=\nt[s]_t^{\delta_e}$ if $x=x_t$ where
$t>0$ and $x_t\in e$ for some $e\in E_0$,
$\nt[\alpha](x)=\nt[s]_0$ if $x\in\mathrm{V}(\X{m}\setminus\bigcup
E_0)$ with $\alpha(x)=s_0$,
$\nt[\alpha](x)=\nt[s]_{k+2(t_{\delta}-1)}^{\delta}$ if
$x\in\mathrm{V}(\X{m}\setminus\bigcup E_0)$ with
$\alpha(x)=s_k^{\delta}$ and $\delta\in D_1$,
$\nt[\alpha](x)=\nt[s]_{k+2(n-2)}^{\delta}$ if
$x\in\mathrm{V}(\X{m}\setminus\bigcup E_0)$ with
$\alpha(x)=s_k^{\delta}$ and $\delta\in D_2$,
$\nt[\alpha](x)=\nt[s]_{0}$ if
$x\in\mathrm{V}(\X{m}\setminus\bigcup E_0)$ with
$\alpha(x)=s_{2t_{\delta}}^{\delta}$ and $\delta\in D_3$,
$\nt[\alpha](x)=\nt[s]_{k-2t_{\delta}}^{\delta}$ if
$x\in\mathrm{V}(\X{m}\setminus\bigcup E_0)$ with
$\alpha(x)=s_k^{\delta}$ where $k>2t_{\delta}$ and $\delta\in
D_3$, and $\nt[\alpha](x)=\nt[s]_{k}^{\delta}$ if
$x\in\mathrm{V}(\X{m}\setminus\bigcup E_0)$ with
$\alpha(x)=s_k^{\delta}$ and $\delta\in\oneton{\nt}\setminus D$.

\underline{\scriptsize \bf CLAIM}:  $\nt[\alpha]$ is well-defined.

Let $e\in E_0$. Then, $\delta_e\in D_1\bigcup D_2$,
$\nt[k]_{\delta_e}=k_{\delta_e}+2(t_{\delta_e}-1)\ge
2[(n-2)-t_{\delta_e}]+3+2(t_{\delta_e}-1)=2[(n-2)-(t_{\delta_e}-1)]+1+2(t_{\delta_e}-1)=2(n-2)+1$
if $\delta_e\in D_1$, $\nt[k]_{\delta_e}=k_{\delta_e}+2(n-2)\ge
1+2(n-2)$ if $\delta_e\in D_2$, and so, $\nt[s]_{t}^{\delta_e}$
exists for all $t\in\oneton{2(n-2)}$.

Let $x\in\mathrm{V}(\X{m}\setminus\bigcup E_0)$, $\delta\in D_1$,
and $k\in\oneton{k_{\delta}}$ so that $\alpha(x)=s_{k}^{\delta}$.
Then, $k+2(t_{\delta}-1)\le
k_{\delta}+2(t_{\delta}-1)=\nt[k]_{\delta}$, and so,
$s_{k+2(t_{\delta}-1)}^{\delta}$ exists.

Let $x\in\mathrm{V}(\X{m}\setminus\bigcup E_0)$, $\delta\in D_2$,
and $k\in\oneton{k_{\delta}}$ so that $\alpha(x)=s_{k}^{\delta}$.
Then, $k+2(n-2)\le k_{\delta}+2(n-2)=\nt[k]_{\delta}$, and so,
$s_{k+2(n-2)}^{\delta}$ exists.

Let $x\in\mathrm{V}(\X{m})$, $\delta\in D_3$, and
$k\in\oneton{2t_{\delta}-1}$ so that $\alpha(x)=s_{k}^{\delta}$.
By Proposition \ref{claim9}, $x\in e$ for some $e\in E_0$, and so,
$\nt[\alpha](x)$ is defined.  Thus, if
$x\in\mathrm{V}(\X{m}\setminus\bigcup E_0)$, $\delta\in D_3$, and
$k\in\oneton{k_{\delta}}$ so that $\alpha(x)=s_{k}^{\delta}$, then
$k\in\oneton[2t_{\delta}]{k_{\delta}}$.

Let $x\in\mathrm{V}(\X{m}\setminus\bigcup E_0)$,
$\delta\in\oneton{\nt}\setminus D$, and $k\in\oneton{k_{\delta}}$
so that $\alpha(x)=s_{k}^{\delta}$.  Then $k\le
k_{\delta}=\nt[k]_{\delta}$, and so, $\nt[s]_{k}^{\delta}$ exists.

Let $x\in\mathrm{V}(\X{m})$ and $\delta\in D_0$ so that
$\alpha(x)=s_{1}^{\delta}$.  Then,
$\nPhi{m}{0}(x)\in\{v_i:i\in\oneton{n-2}\}$. By Proposition
\ref{claim9} and definition of $D_0$, $x\in e$ for some $e\in
E_0$, and so, $\nt[\alpha](x)$ is defined.

\underline{\scriptsize \bf CLAIM}:  $\nt[\alpha]$ is simplicial.

Let $x,y\in\mathrm{V}(\X{m})$ be distinct and adjacent.

\ \ Case 1:  $x,y\in e$ for some $e\in E_0$.

Then, $\nt[\alpha](x)$ and $\nt[\alpha](y)$ are adjacent by
definition of $\nt[\alpha]$ on $e$.

\ \ Case 2:  $x\in e$ for some $e\in E_0$ and
$y\in\mathrm{V}(\X{m}\setminus\bigcup E_0)$.

Then, $x=x_0$ or $x=x_{2(n-2)}$ where $e=\xseq{}{0}{1}{2(n-2)}$.

\ \ \ \ Case 2.a:  $x=x_0$.

By Proposition \ref{claim9}, $\nPhi{m}{0}(y)=v_{n-1}$.  Let
$\delta_{e}^{\prime}\in\oneton{\nt}$ so that $\alpha(x_1)\in
S^{\delta_{e}^{\prime}}\setminus\{s_0\}$.

\ \ \ \ \ \ Case 2.a.i:  $\alpha(x_0)=s_0$.

Then, $\alpha(y)=s_{1}^{\delta}$ for some $\delta\in\oneton{\nt}$,
and since
$\beta(s_{1}^{\delta})=\beta(\alpha(y))=\nPhi{m}{0}(y)=v_{n-1}$,
$\delta\notin D$.  So, $\nt[\alpha](y)=\nt[s]_{1}^{\delta}$, and
since $\nt[\alpha](x_0)=\nt[s]_0$, $\nt[\alpha](x)$ and
$\nt[\alpha](y)$ are adjacent.

\ \ \ \ \ \ Case 2.a.ii:
$\alpha(x_0)=s_{k+1}^{\delta_{e}^{\prime}}$, where
$\alpha(x_1)=s_{k}^{\delta_{e}^{\prime}}$ for some
$k\in\oneton{k_{\delta_{e}^{\prime}}}$.

Then, $\alpha(y)=s_{k+2}^{\delta_{e}^{\prime}}$ and
$\delta_{e}^{\prime}\in D_3$, giving
$k+2=2t_{\delta_{e}^{\prime}}+1$, and so,
$\nt[\alpha](y)=s_{1}^{\delta_{e}^{\prime}}$.  Since
$\nt[\alpha](x_0)=\nt[s]_0$, $\nt[\alpha](x)$ and $\nt[\alpha](y)$
are adjacent.

Since $\nPhi{m}{0}(y)=v_{n-1}$, by definition of $e$ and $E_0$,
these are all of the subcases of case 2.a.

\ \ \ \ Case 2.b:  $x=x_{2(n-2)}$.

By definition of $E$, $y=x_{2(n-2)+1}$ and $\nPhi{m}{0}(y)=v_n$.
Let $\delta_{e}\in\oneton{\nt}$ and $k\in\oneton{k_{\delta_{e}}}$
so that $\alpha(x_{2(n-2)+1})=s_k^{\delta_{e}}$. Then, by
definition of $e$ and $E_0$, $\delta_{e}\in D_1\bigcup D_2$.

\ \ \ \ \ \ Case 2.b.i:  $\delta_{e}\in D_1$.

Then, $k=2[(n-2)-t_{\delta_{e}]+3}$, giving
$\nt[\alpha](y)=\nt[s]_{k+2(t_{\delta_e}-1)}^{\delta_{e}}=
\nt[s]_{2(n-2)+1}^{\delta_{e}}$.  Since
$\nt[\alpha](x_{2(n-2)})=\nt[s]_{2(n-2)}^{\delta_{e}}$,
$\nt[\alpha](x)$ and $\nt[\alpha](y)$ are adjacent.

\ \ \ \ \ \ Case 2.b.ii:  $\delta_{e}\in D_2$.

Then, $k=1$, giving
$\nt[\alpha](y)=\nt[s]_{1+2(n-2)}^{\delta_{e}}$, and since
$\nt[\alpha](x_{2(n-2)})=\nt[s]_{2(n-2)}^{\delta_{e}}$,
$\nt[\alpha](x)$ and $\nt[\alpha](y)$ are adjacent.

\ \ Case 3:  $x,y\in\mathrm{V}(\X{m}\setminus\bigcup E_0)$.

\ \ \ \ Case 3.a:  $\alpha(x),\alpha(y)\in
S^{\delta}\setminus\{s_0\}$ for some
$\delta\in\oneton{\nt}\setminus D_0$.

Since $\alpha$ is simplicial, then $\alpha(x)$ and $\alpha(y)$ are
adjacent by Proposition \ref{claim9}, and so, $\nt[\alpha](x)$ and
$\nt[\alpha](y)$ are adjacent by definition of $\nt[\alpha]$.

\ \ \ \ Case 3.b:  $\alpha(x)=s_0$.

\ \ \ \ \ \ Case 3.b.i:  $\alpha(y)=s_{1}^{\delta}$ for some
$\delta\in\oneton{\nt}\setminus D$.

Then, $\nt[\alpha](y)=\nt[s]_{1}^{\delta}$ and
$\nt[\alpha](x)=\nt[s]_0$, and so, $\nt[\alpha](x)$ and
$\nt[\alpha](y)$ are adjacent.

\ \ \ \ \ \ Case 3.b.ii:  $\alpha(y)=s_{1}^{\delta}$ for some
$\delta\in D_1$.

Then,
$\nPhi{m}{0}(y)=\beta(\alpha(y))=\beta(s_{1}^{\delta})=v_{\impj[j_{\delta}]{t_{\delta}}}$.
Since $y\in\mathrm{V}(\X{m}\setminus\bigcup E_0)$ and by
Proposition \ref{claim9}, $y\in e$ for some $e\in E$, then
$t_{\delta}=1$.  Since
$\nt[\alpha](y)=\nt[s]_{1+2(t_{\delta}-1)}^{\delta}$, then
$\nt[\alpha](y)=\nt[s]_{1+2(1-1)}^{\delta}=s_{1}^{\delta}$, and
so, $\nt[\alpha](x)$ and $\nt[\alpha](y)$ are adjacent since
$\nt[\alpha](x)=\nt[s]_0$.

By Propositions \ref{claim9}, \ref{claim12}, and \ref{claim11},
these are all of the subcases of Case 3.b.

\underline{\scriptsize \bf CLAIM}:  $\nt[\beta](\nt[s]_0)=v_0$.

By definition of $\nt[\beta]$, $\nt[\beta](\nt[s]_0)=\beta(s_0)$,
and by definition of $\beta$, $\beta(s_0)=v_0$.

\underline{\scriptsize \bf CLAIM}:
$\nt[\beta]\circ\nt[\alpha]=\nPhi{m}{0}$.

Let $x\in\mathrm{V}(\X{m})$.

\ \ Case 1:  $x=x_0$ where $x_0\in e$ for some $e\in E_0$.

Then,
$\nPhi{m}{0}(x)=v_0=\beta(s_0)=\nt[\beta](\nt[s]_0)=\nt[\beta](\nt[\alpha](x))$.

\ \ Case 2:  $x=x_t$ where $t>0$ and $x_t\in e$ for some $e\in
E_0$.

Then, $\delta_e\in D_1\bigcup D_2$, $j=j_{\delta_e}$ where
$\nPhi{m}{0}(e)=\langle
v_0,v_{\impj{1}},v_0,\underline{v_{\impj{2}},v_0},\ldots,\underline{v_{\impj{n-2}},}$
$\underline{v_0} \rangle$, $t_e=2t_{\delta_e}-1$ if $\delta_e\in
D_1$, and $t_e=2(n-2)+1$ if $\delta_e\in D_2$.

\ \ \ \ Case 2.a:  $t\ge t_e$.

Then, $\delta_e\notin D_2$.

\ \ \ \ \ \ Case 2.a.i:  $t$ is even.

Since $t_e$ is odd, $t-t_e+1$ is even and $1=t_e-t_e+1\le
t-t_e+1=t-(2t_{\delta_e}-1)+1=t-2t_{\delta_e}+2\le
2(n-2)-2t_{\delta_e}+2=2[(n-2)-(t_{\delta_e}-1)]$.  Then,
$\beta(s_{t-t_e+1}^{\delta_e})=v_0$, giving
$\nPhi{m}{0}(x)=v_0=\beta(s_{t-t_e+1}^{\delta_e})=\beta(s_{t-2t_{\delta_e}+2}^{\delta_e})=
\beta(s_{t-2(t_{\delta_e}-1)}^{\delta_e})=\nt[\beta](\nt[s]_{t}^{\delta_e})=\nt[\beta](\nt[\alpha](x))$.

\ \ \ \ \ \ Case 2.a.ii:  $t$ is odd.

Since $t_e$ is odd, $t-t_e+1$ is odd and $1=t_e-t_e+1\le
t-t_e+1=t-(2t_{\delta_e}-1)+1=t-2t_{\delta_e}+2\le
2(n-2)-1-2t_{\delta_e}+2=2[(n-2)-(t_{\delta_e}-1)]-1$.  Then,
$\beta(s_{t-t_e+1}^{\delta_e})=v_{\impj[j_{\delta_e}]{t_{\delta_e}+\frac{t-t_e+1-1}{2}}}$,
giving
$\nPhi{m}{0}(x)=v_{\impj{\frac{t+1}{2}}}=v_{\impj[j_{\delta_e}]{t_{\delta_e}-t_{\delta_e}+\frac{t+1}{2}}}=
v_{\impj[j_{\delta_e}]{t_{\delta_e}-\frac{(2t_{\delta_e}-1)-t}{2}}}=
v_{\impj[j_{\delta_e}]{t_{\delta_e}+\frac{t-t_e}{2}}}=
v_{\impj[j_{\delta_e}]{t_{\delta_e}+\frac{t-t_e+1-1}{2}}}=
\beta(s_{t-t_e+1}^{\delta_e})=\beta(s_{t-2t_{\delta_e}+2}^{\delta_e})=\beta(s_{t-2(t_{\delta_e}-1)}^{\delta_e})=
\nt[\beta](\nt[s]_{t}^{\delta_e})=\nt[\beta](\nt[\alpha](x))$.

\ \ \ \ Case 2.b:  $t<t_e$.

Then, $1\le t\le t_e-1=2t_{\delta_e}-2=2(t_{\delta_e}-1)$ if
$\delta_e\in\ D_1$ and $1\le t\le t_e-1=2(n-2)+1-1=2(n-2)$ if
$\delta_e\in D_2$.

\ \ \ \ \ \ Case 2.b.i:  $t$ is even.

Then,
$\nPhi{m}{0}(x)=v_0=\nt[\beta](\nt[s]_{t}^{\delta_e})=\nt[\beta](\nt[\alpha](x))$.

\ \ \ \ \ \ Case 2.b.ii:  $t$ is odd.

Then,
$\nPhi{m}{0}(x)=v_{\impj{\frac{t+1}{2}}}=v_{\impj[j_{\delta_e}]{\frac{t+1}{2}}}=
\nt[\beta](\nt[s]_{t}^{\delta_e})=\nt[\beta](\nt[\alpha](x))$.

\ \ Case 3:  $x\in\mathrm{V}(\X{m}\setminus\bigcup E_0)$ and
$\alpha(x)=s_0$.

Then,
$\nPhi{m}{0}(x)=\beta(\alpha(x))=\beta(s_0)=\nt[\beta](\nt[s]_0)=\nt[\beta](\nt[\alpha](x))$.

\ \ Case 4:  $x\in\mathrm{V}(\X{m}\setminus\bigcup E_0)$ and
$\alpha(x)=s_{k}^{\delta}$ for some $\delta\in D_1\bigcup D_2$.

Then, $\nPhi{m}{0}(x)=\beta(\alpha(x))=\beta(s_{k}^{\delta})=
\nt[\beta](\nt[s]_{k+2(t_{\delta}-1)}^{\delta})=\nt[\beta](\nt[\alpha](x))$
if $\delta\in D_1$ and
$\nPhi{m}{0}(x)=\beta(\alpha(x))=\beta(s_{k}^{\delta})=
\nt[\beta](\nt[s]_{k+2(n-2)}^{\delta})=\nt[\beta](\nt[\alpha](x))$
if $\delta\in D_2$.

\ \ Case 5:  $x\in\mathrm{V}(\X{m}\setminus\bigcup E_0)$ and
$\alpha(x)=s_{k}^{\delta}$ for some $\delta\in D_3$.

Then, $\nPhi{m}{0}(x)=\beta(\alpha(x))=\beta(s_{k}^{\delta})=
\nt[\beta](\nt[s]_{k-2t_{\delta}}^{\delta})=\nt[\beta](\nt[\alpha](x))$.

\ \ Case 6:  $x\in\mathrm{V}(\X{m}\setminus\bigcup E_0)$ and
$\alpha(x)=s_{k}^{\delta}$ for some
$\delta\in\oneton{\nt}\setminus D$.

Then, $\nPhi{m}{0}(x)=\beta(\alpha(x))=\beta(s_{k}^{\delta})=
\nt[\beta](\nt[s]_{k}^{\delta})=\nt[\beta](\nt[\alpha](x))$.

\underline{\scriptsize \bf CLAIM}:  If $x,y,z\in\mathrm{V}(\X{m})$
so that $x,y,z$ are distinct, $\edge{x}{}{y,z}{}\subseteq\X{m}$,
and $\nt[\alpha](y)=\nt[s]_0$, then $\nPhi{m}{0}(x)=v_{n-1}$ or
$\nPhi{m}{0}(z)=v_{n-1}$.

Suppose $\nPhi{m}{0}(x)\not=v_{n-1}$. By Proposition \ref{claim9},
since
$\nPhi{m}{0}(y)=\nt[\beta](\nt[\alpha](y))=\nt[\beta](\nt[s]_0)=\beta(s_0)=v_0$,
$\nPhi{m}{0}(x)\not=v_0$ and $\nPhi{m}{0}(z)\not=v_0$, and so,
there exists $\delta\in\oneton{\nt}$ so that
$\nt[\alpha](x)=s_{1}^{\delta}$.  By Proposition \ref{claim9},
$\edge{x}{}{y,z}{}\subseteq\arm{m}$ for some $i\in\oneton{\nt}$.

\ \ Case 1:  $\delta\in D_1\bigcup D_2$.

\ \ \ \ Case 1.a:  $x\in e$ for some $e\in E_0$.

Then, $\delta_e=\delta$, $x=x_1$, and $y=x_0$ where $x_0,x_1\in
e$, and so by Proposition \ref{claim9}, $\nPhi{m}{0}(z)=v_{n-1}$.

\ \ \ \ Case 1.b:  $x\in\mathrm{V}(\X{m}\setminus\bigcup E_0)$.

Then $\delta\notin D_2$.  Since
$\nt[\alpha](x)=\nt[s]_{k+2(t_{\delta}-1)}^{\delta}$ where
$\alpha(x)=s_{k}^{\delta}$, $\nt[\alpha](x)=\nt[s]_{1}^{\delta}$,
and $1\le k\le k+2(t_{\delta}-1)=1$, then $k=1$ and
$t_{\delta}=1$, giving $\alpha(x)=s_{1}^{\delta}$.  If
$\alpha(y)=s_{2}^{\delta}$ and $y\in e$ for some $e\in E_0$, then
$x\in e$, giving a contradiction.  Thus, $\alpha(y)=s_0$.  By
Proposition \ref{claim9}, since $\nPhi{m}{0}(x)\not=v_{n-1}$,
there exists $\delta^{\prime}\in\oneton{\nt}\setminus\{\delta\}$
so that $\alpha(z)=s_{1}^{\delta^{\prime}}$.  If
$\nPhi{m}{0}(z)\not=v_{n-1}$, then $x\in e$ for some $e\in E_0$ by
Proposition \ref{claim9}, resulting in a contradiction.  Thus,
$\nPhi{m}{0}(z)=v_{n-1}$.

\ \ Case 2:  $\delta\in D_3$.

Then,
$\nPhi{m}{0}(x)=\nt[\beta](\nt[\alpha](x))=\nt[\beta](\nt[s]_{1}^{\delta}=
\beta(s_{1+2t_{\delta}}^{\delta})=v_{n-1}$, giving a
contradiction.  Thus, case 2 does not occur.

\ \ Case 3:  $\delta\in\oneton{\nt}\setminus D$.

\ \ \ \ Case 3.a:  $x\in e$ for some $e\in E_0$.

Then, $\delta_e=\delta$ and $\delta_e\in D_1\bigcup D_2$, giving a
contradiction.  Thus, Case 3.a does not occur.

\ \ \ \ Case 3.b:  $x\in\mathrm{V}(\X{m})\setminus\bigcup E_0$.

Then, $\alpha(x)=s_{k}^{\delta}$, where
$\nt[\alpha](x)=\nt[s]_{k}^{\delta}$, and so,
$\alpha(x)=s_{1}^{\delta}$.  If $\alpha(y)=s_{2}^{\delta}$ and
$y\in e$ for some $e\in E_0$, then $\delta\in D$, resulting in a
contradiction.  Thus, $\alpha(y)=s_0$.  Since
$\nPhi{m}{0}(x)\not=v_{n-1}$ and $\delta\notin D$, then
$\nPhi{m}{0}(z)=v_{n-1}$.
\end{proof}

\begin{proposition} \label{claim14}
If $m,\nt,T,\alpha,\mbox{ and }\beta$ exist, then for some
$\nt_1\in\oneton{\nt_0}$ where $\nt_0$ is as given in Proposition
\ref{claim13}, there exist a simple-$\nt_1$-od $\hat{T}$, $\nt_1$
arcs $\hat{S}^1,\hat{S}^2,\ldots,\linebreak[0]\hat{S}^{\nt_1}$ so
that $\bigcup\limits_{i=1}^{\nt_1}\hat{S}^i=\hat{T}$,
$\hat{S}^i\bigcap \hat{S}^j=\{\hat{s}_0\}$ for each distinct
$i,j\in\oneton{\nt_1}$, so $\hat{s}_0$ is an endpoint of
$\hat{S}^i$ for each $i\in\oneton{\nt_1}$, and
$\hat{S}^i=\shseqzero{i}{\hat{k}_i}$ for some
$\hat{k}_i\in\{1,2,\ldots\}$ for each $i\in\oneton{\nt_1}$, and
simplicial maps $\hat{\alpha}:\X{m}\longrightarrow \hat{T}$ and
$\hat{\beta}:\hat{T}\longrightarrow\X{0}$ so that
$\hat{\beta}\circ\hat{\alpha}=\nPhi{m}{0}$,
$\hat{\beta}(\hat{s}_0)=v_0$, and
$\hat{\beta}(\hat{s}_{1}^{\delta})=v_{n-1}$ for each
$\delta\in\oneton{\nt_1}$.
\end{proposition}

\begin{proof}
Let $\nt[T], \nt[\alpha],\mbox{ and } \nt[\beta]$ be as given in
Proposition \ref{claim13}.

Let
$F_0=\{\delta\in\oneton{\nt_0}:\nt[\beta](\nt[s]_{1}^{\delta})\not=
v_{n-1}\}$, $F_1=\{\delta\in\oneton{\nt_0}:\nt[s]_{3}^{\delta}$
exists and $\overline{\nt[\beta]}(\langle
\nt[s]_{1}^{\delta},\nt[s]_{2}^{\delta},\nt[s]_{3}^{\delta}
\rangle)=\edge{v_{n-1}}{}{v_{n+1},v_{n-1}}{}$, or
$\nt[s]_{3}^{\delta}$ exists and $\overline{\nt[\beta]}(\langle
\nt[s]_{1}^{\delta},\nt[s]_{2}^{\delta},\nt[s]_{3}^{\delta}
\rangle)=\edge{v_{n-1}}{}{v_0,v}{}$ for some
$v\in\{v_1,v_2,\ldots,v_n\}\setminus\{v_{n-1}\}\}$,
$F_2=\{\delta\in\oneton{\nt_0}:s_{3}^{\delta}$ exists and
$\overline{\nt[\beta]}(\langle
\nt[s]_{1}^{\delta},\nt[s]_{2}^{\delta},\nt[s]_{3}^{\delta}
\rangle)=\edge{v_{n-1}}{}{v_0,v_{n-1}}{}\}$,
$F_3=\{\delta\in\oneton{\nt_0}:\nt[s]_{3}^{\delta}$ does not exist
and $\nt[\beta](\nt[s]_{1}^{\delta})= v_{n-1}\}$, and
$F=\bigcup\limits_{i=0}^{3}F_i$.  Then, $F_i\bigcap F_j=\emptyset$
for each distinct $i,j\in\oneton[0]{3}$, and by Proposition
\ref{claim9}, $|F|=\nt_0$.  Without loss of generality, suppose
$F_3$ is empty or there exists $\nt[\delta]\in\oneton{\nt_0}$ so
that $F_3=\{\nt[\delta],\nt[\delta]+1,\ldots,\nt_0\}$.

Let $\hat{k}_{\delta}=\nt[k]_{\delta}+2$ if $\delta\in F_0$,
$\hat{k}_{\delta}=\nt[k]_{\delta}$ if $\delta\in F_1$, and
$\hat{k}_{\delta}=\nt[k]_{\delta}-2$ if $\delta\in F_2$. Let
$\nt_1=\nt_0-|F_3|$ and
$\hat{S}^1,\hat{S}^2,\ldots,\hat{S}^{\nt_1}$ be $\nt_1$ arcs so
that $\hat{S}^i=\shseqzero{i}{\hat{k}_i}$ for each
$i\in\oneton{\nt_1}$ and $\hat{S}^i\bigcap
\hat{S}^j=\{\hat{s}_0\}$ for each distinct $i,j\in\oneton{\nt_1}$.
Let $\hat{T}$ be a simple-$\nt_1$-od so that
$\bigcup\limits_{i=1}^{\nt_1}\hat{S}^i=\hat{T}$ and
$\hat{\beta}:\hat{T}\longrightarrow\X{0}$ be the simplicial map
defined as:  $\hat{\beta}(\hat{s}_0)=\nt[\beta](\nt[s]_0)$,
$\overline{\hat{\beta}}(\langle
\hat{s}_{1}^{\delta},\hat{s}_{2}^{\delta}
\rangle)=\edge{v}{n-1}{v}{0}$ if $\delta\in F_0$,
$\hat{\beta}(\hat{s}_{k}^{\delta})=\nt[\beta](\nt[s]_{k-2}^{\delta})$
for $k\in\oneton[3]{\hat{k}_{\delta}}$ if $\delta\in F_0$,
$\hat{\beta}(\hat{s}_{k}^{\delta})=\nt[\beta](\nt[s]_{k}^{\delta})$
for $k\in\oneton{\hat{k}_{\delta}}$ if $\delta\in F_1$, and
$\hat{\beta}(\hat{s}_{k}^{\delta})=\nt[\beta](\nt[s]_{k+2}^{\delta})$
for $k\in\oneton{\hat{k}_{\delta}}$ if $\delta\in F_2$.

Let $G_0=\{\langle x_0^i,x_1^i,x_2^i \rangle:i\in\oneton{n}\}$,
$G_1=\{\langle x_r^i,x_{r+1}^i,x_{r+2}^i,x_{r+3}^i,x_{r+4}^i
\rangle:\overline{\nPhi{m}{0}}(\langle
x_r^i,x_{r+1}^i,x_{r+2}^i,x_{r+3}^i,x_{r+4}^i \rangle)=\langle
v_0,v_{n-1},v_0,v_{n-1},v_0 \rangle$ for some $i\in\oneton{n}$ and
for some $r\in\oneton[0]{\ell_i}\}$, $G_2=\{\langle
x_r^i,x_{r+1}^i,x_{r+2}^i,x_{r+3}^i,x_{r+4}^i
\rangle:\overline{\nPhi{m}{0}}(\langle
x_r^i,x_{r+1}^i,x_{r+2}^i,\linebreak[0]x_{r+3}^i,x_{r+4}^i
\rangle)=\langle v_0,v_{n-1},v_0,v_{n-1},v_{n+1} \rangle$ for some
$i\in\oneton{n}$ and for some $r\in\oneton[0]{\ell_i}\}$,
$G_3=\{\langle x_r^i,x_{r+1}^i,x_{r+2}^i,x_{r+3}^i,x_{r+4}^i
\rangle:\overline{\nPhi{m}{0}}(\langle
x_r^i,x_{r+1}^i,x_{r+2}^i,x_{r+3}^i,x_{r+4}^i \rangle)=\langle
v_{n+1},v_{n-1},v_0,v_{n-1},v_0 \rangle$ for some $i\in\oneton{n}$
and for some $r\in\oneton[0]{\ell_i}\}$, $G_4=\{\langle
x_r^i,x_{r+1}^i,x_{r+2}^i,x_{r+3}^i,x_{r+4}^i
\rangle:\overline{\nPhi{m}{0}}(\langle
x_r^i,x_{r+1}^i,x_{r+2}^i,x_{r+3}^i,x_{r+4}^i \rangle)=\langle
v_{n+1},v_{n-1},v_0,\linebreak[0]v_{n-1},\linebreak[0]v_{n+1}
\rangle$ for some $i\in\oneton{n}$ and for some
$r\in\oneton[0]{\ell_i}\}$, and $G=\bigcup\limits_{i=0}^{4} G_i$.
Then, $G_i\bigcap G_j=\emptyset$ for each distinct
$i,j\in\oneton[0]{4}$.

Let $g_1,g_2\in G$ be distinct.  By Proposition \ref{claim9}, if
$\X{m}\supseteq g_1\bigcap g_2\not=\emptyset$, then
$g_1,g_2\in\bigcup\limits_{i=1}^{4}G_i$ with
$x_{r_1+4}^i=x_{r_2}^i$ and
$\nPhi{m}{0}(x_{r_1+4}^{i})=\nPhi{m}{0}(x_{r_2}^{i})=v_{n+1}$ (or
$x_{r_2+4}^{i}=x_{r_1}^i$ and
$\nPhi{m}{0}(x_{r_1}^{i})=\nPhi{m}{0}(x_{r_2+4}^{i})=v_{n+1}$),
where $g_1=\xseq{i}{r_1}{r_1+1}{r_1+4}$ and
$g_2=\xseq{i}{r_2}{r_2+1}{r_2+4}$, or $g_1,g_2\in G_0$ with
$x_{0}^{i_1}=x_{0}^{i_2}=v_0$, where $g_1=\langle
x_{0}^{i_1},x_{1}^{i_1},x_{2}^{i_1} \rangle$ and $g_2=\langle
x_{0}^{i_2},x_{1}^{i_2},x_{2}^{i_2} \rangle$.

Let $G_0^H=\{\langle x_0^i,x_1^i,x_2^i \rangle\in
G_0:\nt[\alpha](x_k^i)=\nt[s]_0$ for each $k\in H\mbox{ and
}\nt[\alpha](x_k^i)\not=\nt[s]_0$ for each $k\in\{0,1,2\}\setminus
H\}$ for $H\subseteq\{0,1,2\}$ and $G_j^H=\{\langle
x_r^i,x_{r+1}^i,x_{r+2}^i,x_{r+3}^i,x_{r+4}^i \rangle\in
G_j:\nt[\alpha](x_{r+k}^i)=\nt[s]_0\mbox{ for each }k\in H\mbox{
and }\nt[\alpha](x_{r+k}^i)\not=\nt[s]_0\mbox{ for each
}k\in\oneton[0]{4}\setminus H\}$ for $H\subseteq\oneton[0]{4}$ and
for $j\in\oneton{4}$.  Then, if $G_{j_1}^{H_1}\bigcap
G_{j_2}^{H_2}\not=\emptyset$, $H_1=H_2$ and $j_1=j_2$.

Let $G_0^{\prime}=G_0$ if $G_0^{\{0\}}\not=\emptyset$ or
$G_0^{\{0,2\}}\not=\emptyset$, and $G_0^{\prime}=\emptyset$
otherwise.  Let $G_0^{\prime\prime}=G_0$ if
$G_0^{\{2\}}\not=\emptyset$, and $G_0^{\prime\prime}=\emptyset$
otherwise.  Then $\mathcal{G}_1\bigcap\mathcal{G}_2=\emptyset$ for
each distinct $\mathcal{G}_1,\mathcal{G}_2\in\{G_0\setminus
(G^{\prime}\bigcup
G^{\prime\prime}),G^{\prime},G^{\prime\prime}\}$. Let
$G^{\prime}=G_0^{\prime}\bigcup G_0^{\prime\prime}\bigcup
G_1^{\{0\}}\bigcup G_1^{\{0,2\}}\bigcup G_1^{\{4\}}\bigcup
G_1^{\{2,4\}}\bigcup G_1^{\{0,4\}}\bigcup\linebreak[0]
G_1^{\{0,2,4\}}\bigcup G_2^{\{0\}}\bigcup G_2^{\{0,2\}}\bigcup
G_3^{\{4\}}\bigcup G_3^{\{2,4\}}$.

Let $g=\xseq{i}{r}{r+1}{r+t}\in G$ and $\nt[t]\in\oneton[0]{t}$.
Let ${_g\delta}_{\nt[t]}\in\oneton{\nt_0}$ so that
$\nt[\alpha](x_{r+\nt[t]}^i)\in\nt[S]^{{_g\delta}_{\nt[t]}}\setminus\{\nt[s]_0\}$
whenever $\nt[\alpha](x_{r+\nt[t]}^i)\not=\nt[s]_0$,
${_g\delta}_{t+1}\in\oneton{\nt_0}$ so that
$\nt[\alpha](x_{r+t+1}^i)\in\nt[S]^{{_g\delta}_{t+1}}\setminus\{\nt[s]_0\}$
whenever $x_{r+t+1}^i$ exists and
$\nt[\alpha](x_{r+t+1}^i)\not=\nt[s]_0$, and
${_g\delta}_{-1}\in\oneton{\nt_0}$ so that
$\nt[\alpha](x_{r-1}^i)\in\nt[S]^{{_g\delta}_{-1}}\setminus\{\nt[s]_0\}$
whenever $x_{r-1}^i$ exists and
$\nt[\alpha](x_{r-1}^i)\not=\nt[s]_0$.  By Proposition
\ref{claim9}, $x_{r+t+1}^i$ exists, and if
$g\in\bigcup\limits_{i=1}^{4}G_i$, $x_{r-1}^i$ exists.

Let $\hat{\alpha}(\langle x_0^i,x_1^i,x_2^i \rangle)=\langle
\hat{s}_0,\hat{s}_1^{{_g\delta}_3},\hat{s}_2^{{_g\delta}_3}
\rangle$ if $g=\langle x_0^i,x_1^i,x_2^i \rangle\in
G_0^{\prime}\bigcup G_0^{\prime\prime}$, \
$\overline{\hat{\alpha}}(\langle
x_r^i,x_{r+1}^i,\linebreak[0]x_{r+2}^i,x_{r+3}^i,x_{r+4}^i)=\langle
\hat{s}_2^{{_g\delta}_{-1}},\hat{s}_1^{{_g\delta}_{-1}},\hat{s}_0,\hat{s}_1^{{_g\delta}_4},\hat{s}_2^{{_g\delta}_4}
\rangle$ if $g=\xseq{i}{r}{r+1}{r+4}\hspace{-.1cm}\in
\hspace{-.08cm}G_1^{\{0\}}\bigcup G_1^{\{0,2\}}$, \
$\overline{\hat{\alpha}}(\langle
x_r^i,x_{r+1}^i,x_{r+2}^i,x_{r+3}^i,x_{r+4}^i)=\langle
\hat{s}_2^{{_g\delta}_0},\hat{s}_1^{{_g\delta}_0},\hat{s}_0,\hat{s}_1^{{_g\delta}_5},\hat{s}_2^{{_g\delta}_5}
\rangle$ if $g=\xseq{i}{r}{r+1}{r+4}\in G_1^{\{4\}}\bigcup
G_1^{\{2,4\}}$, \ $\overline{\hat{\alpha}}(\langle
x_r^i,x_{r+1}^i,x_{r+2}^i,x_{r+3}^i,x_{r+4}^i)=\langle
\hat{s}_2^{{_g\delta}_{-1}},\hat{s}_1^{{_g\delta}_{-1}},\hat{s}_0,\hat{s}_1^{{_g\delta}_5},\hat{s}_2^{{_g\delta}_5}
\rangle$ if $g=\xseq{i}{r}{r+1}{r+4}\in G_1^{\{0,4\}}\bigcup
G_1^{\{0,2,4\}}$, \ $\overline{\hat{\alpha}}(\langle
x_r^i,x_{r+1}^i,x_{r+2}^i,x_{r+3}^i,x_{r+4}^i)=\langle
\hat{s}_2^{{_g\delta}_{-1}},\hat{s}_1^{{_g\delta}_{-1}},\hat{s}_0,\linebreak[0]\hat{s}_1^{{_g\delta}_4},\hat{s}_2^{{_g\delta}_4}
\rangle$ if $g=\xseq{i}{r}{r+1}{r+4}\in G_2^{\{0\}}\bigcup
G_2^{\{0,2\}}$, \ $\overline{\hat{\alpha}}(\langle
x_r^i,x_{r+1}^i,x_{r+2}^i,x_{r+3}^i,x_{r+4}^i)=\langle
\hat{s}_2^{{_g\delta}_0},\hat{s}_1^{{_g\delta}_0},\hat{s}_0,\hat{s}_1^{{_g\delta}_5},\hat{s}_2^{{_g\delta}_5}
\rangle$ if $g=\xseq{i}{r}{r+1}{r+4}\in G_3^{\{4\}}\bigcup
G_3^{\{2,4\}}$, \ $\hat{\alpha}(x)=\hat{s}_0$ if
$x\in\mathrm{V}(\X{m}\setminus\bigcup G^{\prime})$ and
$\nt[\alpha](x)=\nt[s]_0$, where $G^{\prime}=G_0^{\{0\}}\bigcup
G_0^{\{2\}}\bigcup G_0^{\{0,2\}}\bigcup G_1^{\{0\}}\bigcup
G_1^{\{0,2\}}\bigcup G_1^{\{4\}}\linebreak[0]\bigcup
G_1^{\{2,4\}}\bigcup G_1^{\{0,4\}} \bigcup G_1^{\{0,2,4\}}\bigcup
G_2^{\{0\}}\bigcup G_2^{\{0,2\}}\bigcup G_3^{\{4\}}\bigcup
G_3^{\{2,4\}}$, \ $\hat{\alpha}(x)=\hat{s}_{k+2}^{\delta}$ if
$x\in\mathrm{V}(\X{m}\linebreak[0]\setminus\bigcup G^{\prime})$,
$\nt[\alpha](x)=\nt[s]_k^{\delta}$, and $\delta\in F_0$, \
$\hat{\alpha}(x)=\hat{s}_{k}^{\delta}$ if
$x\in\mathrm{V}(\X{m}\setminus\bigcup G^{\prime})$,
$\nt[\alpha](x)=\nt[s]_k^{\delta}$, and $\delta\in F_1$, \
$\hat{\alpha}(x)=\hat{s}_0$ if
$x\in\mathrm{V}(\X{m}\setminus\bigcup G^{\prime})$,
$\nt[\alpha](x)=\nt[s]_2^{\delta}$, and $\delta\in F_2$, and
$\hat{\alpha}(x)=\hat{s}_{k-2}^{\delta}$ if
$x\in\mathrm{V}(\X{m}\setminus\bigcup G^{\prime})$,
$\nt[\alpha](x)=\nt[s]_k^{\delta}$, $\delta\in F_2$, and
$k\in\oneton[3]{\nt[k]_{\delta}}$.

\underline{\scriptsize \bf CLAIM}:  $\hat{\alpha}$ is
well-defined.

\ \ Case 1:  Let $g=\langle x_0^i,x_1^i,x_2^i \rangle\in
G_0^{\prime}\bigcup G_0^{\prime\prime}$.

By Proposition \ref{claim9}, $\nPhi{m}{0}(x_3^i)=v_n$.

\ \ \ \ Case1.a:  Suppose $g\in G_0^{\prime}$.

Then,

i) $\overline{\nt[\alpha]}(\langle x_0^i,x_1^i,x_2^i,x_3^i
\rangle)=\langle
\nt[s]_0,\nt[s]_1^{\delta_3},\nt[s]_2^{\delta_3},\nt[s]_3^{\delta_3}
\rangle$ for some $\delta_3\in\oneton{\nt_0}$ or

ii) $\overline{\nt[\alpha]}(\langle x_0^i,x_1^i,x_2^i,x_3^i
\rangle)=\langle
\nt[s]_0,\nt[s]_1^{\delta_1},\nt[s]_0,\nt[s]_1^{\delta_3} \rangle$
for some $\delta_1,\delta_3\in\oneton{\nt_0}$ so that
$\delta_1\not=\delta_3$.

If (i), then $\delta_3\in F_1$ and ${_g\delta}_3=\delta_3$, giving
$3\le\nt[k]_{\delta_3}=\hat{k}_{{_g\delta}_3}$, and so,
$\hat{s}_1^{{_g\delta}_3}$ and $\hat{s}_2^{{_g\delta}_3}$ exist.
If (ii), then $\delta_3\in F_0$ and ${_g\delta}_3=\delta_3$,
giving $3=1+2\le\nt[k]_{\delta_3}+2=\hat{k}_{{_g\delta}_3}$, and
so, $\hat{s}_1^{{_g\delta}_3}$ and $\hat{s}_2^{{_g\delta}_3}$
exist.

\ \ \ \ Case1.b:  Suppose $g\in G_0^{\prime\prime}$.

Then,

i) $\overline{\nt[\alpha]}(\langle x_0^i,x_1^i,x_2^i,x_3^i
\rangle)=\langle
\nt[s]_2^{\delta_0},\nt[s]_1^{\delta_0},\nt[s]_0,\nt[s]_1^{\delta_3}
\rangle$ for some $\delta_0,\delta_3\in\oneton{\nt_0}$ so that
$\delta_0\not=\delta_3$,

ii)  $\overline{\nt[\alpha]}(\langle x_0^i,x_1^i,x_2^i,x_3^i
\rangle)=\langle
\nt[s]_2^{\delta_3},\nt[s]_1^{\delta_3},\nt[s]_2^{\delta_3},\nt[s]_3^{\delta_3}
\rangle$ for some $\delta_3\in\oneton{\nt_0}$, or

iii) $\overline{\nt[\alpha]}(\langle x_0^i,x_1^i,x_2^i,x_3^i
\rangle)=\langle
\nt[s]_2^{\delta_3},\nt[s]_3^{\delta_3},\nt[s]_4^{\delta_3},\nt[s]_5^{\delta_3}
\rangle$ for some $\delta_3\in\oneton{\nt_0}$.

If (i), then $\delta_3\in F_0$ and ${_g\delta}_3=\delta_3$, giving
$3=1+2\le\nt[k]_{\delta_3}+2=\hat{k}_{{_g\delta}_3}$, and so,
$\hat{s}_1^{{_g\delta}_3}$ and $\hat{s}_2^{{_g\delta}_3}$ exist.
If (ii), then $\delta_3\in F_1$ and ${_g\delta}_3=\delta_3$,
giving $3\le\nt[k]_{\delta_3}=\hat{k}_{{_g\delta}_3}$, and so,
$\hat{s}_1^{{_g\delta}_3}$ and $\hat{s}_2^{{_g\delta}_3}$ exist.
If (iii), then $\delta_3\in F_2$ and ${_g\delta}_3=\delta_3$,
giving $3=5-2\le\nt[k]_{\delta_3}-2=\hat{k}_{{_g\delta}_3}$, and
so, $\hat{s}_1^{{_g\delta}_3}$ and $\hat{s}_2^{{_g\delta}_3}$
exist.

\ \ Case 2:  Let $g=\xseq{i}{r}{r+1}{r+4}\in G_1^{\{0\}}\bigcup
G_1^{\{0,2\}}$ (or $G_1^{\{4\}}\bigcup G_1^{\{2,4\}}$).

By Proposition \ref{claim9},
$\nPhi{m}{0}(x_{r-1}^i),\nPhi{m}{0}(x_{r+5}^i)\notin\{v_0,v_{n-1},v_{n+1},v_{n+2}\}$.
Then,

i)  $\overline{\nt[\alpha]}(\xseq{i}{r-1}{r}{r+5})=\langle
\nt[s]_1^{\delta_{-1}},\nt[s]_0,\nt[s]_1^{\delta_4},\nt[s]_2^{\delta_4},\nt[s]_3^{\delta_4},\nt[s]_4^{\delta_4},
\nt[s]_5^{\delta_4} \rangle$ for some
$\delta_{-1},\delta_4\in\oneton{\nt_0}$ so that
$\delta_{-1}\not=\delta_4$ (or
$\overline{\nt[\alpha]}(\xseq{i}{r-1}{r}{r+5})=\langle
\nt[s]_5^{\delta_0},\nt[s]_4^{\delta_0},\nt[s]_3^{\delta_0},\nt[s]_2^{\delta_0},\nt[s]_1^{\delta_0},\nt[s]_0,
\nt[s]_1^{\delta_5} \rangle$ for some
$\delta_0,\delta_5\in\oneton{\nt_0}$ so that
$\delta_0\not=\delta_5$),

ii) $\overline{\nt[\alpha]}(\xseq{i}{r-1}{r}{r+5})=\langle
\nt[s]_1^{\delta_{-1}},\nt[s]_0,\nt[s]_1^{\delta_4},\nt[s]_2^{\delta_4},\nt[s]_1^{\delta_4},\linebreak[0]\nt[s]_2^{\delta_4},
\nt[s]_3^{\delta_4} \rangle$ for some
$\delta_{-1},\delta_4\in\oneton{\nt_0}$ so that
$\delta_{-1}\not=\delta_4$ (or
$\overline{\nt[\alpha]}(\xseq{i}{r-1}{r}{r+5})=\langle
\nt[s]_3^{\delta_0},\nt[s]_2^{\delta_0},\nt[s]_1^{\delta_0},\nt[s]_2^{\delta_0},\nt[s]_1^{\delta_0},\nt[s]_0,
\nt[s]_1^{\delta_5} \rangle$ for some
$\delta_0,\delta_5\in\oneton{\nt_0}$ so that
$\delta_0\not=\delta_5$), or

iii) $\overline{\nt[\alpha]}(\xseq{i}{r-1}{r}{r+5})=\langle
\nt[s]_1^{\delta_{-1}},\nt[s]_0,\nt[s]_1^{\delta_1},\nt[s]_0,\nt[s]_1^{\delta_4},\nt[s]_2^{\delta_4},
\nt[s]_3^{\delta_4} \rangle$ for some
$\delta_{-1},\delta_1,\delta_4\in\oneton{\nt_0}$ so that
$\delta_{-1}\not=\delta_1$ and $\delta_{-1}\not=\delta_4$ (or
$\overline{\nt[\alpha]}(\xseq{i}{r-1}{r}{r+5})=\langle
\nt[s]_3^{\delta_0},\nt[s]_2^{\delta_0},\nt[s]_1^{\delta_0},\nt[s]_0,\nt[s]_1^{\delta_3},\linebreak[0]\nt[s]_0,
\nt[s]_1^{\delta_5} \rangle$ for some
$\delta_0,\delta_3,\delta_5\in\oneton{\nt_0}$ so that
$\delta_5\not=\delta_3$ and $\delta_5\not=\delta_0$).

If (i), then $\delta_{-1}\in F_0$, $\delta_4\in F_2$,
${_g\delta}_{-1}=\delta_{-1}$, and ${_g\delta}_4=\delta_4$, giving
$3=1+2\le\nt[k]_{\delta_{-1}}+2=\hat{k}_{{_g\delta}_{-1}}$ and
$3=5-2\le\nt[k]_{\delta_4}-2=\hat{k}_{{_g\delta}_4}$, and so,
$\hat{s}_1^{{_g\delta}_{-1}}$, $\hat{s}_2^{{_g\delta}_{-1}}$,
$\hat{s}_1^{{_g\delta}_4}$, and $\hat{s}_2^{{_g\delta}_4}$ exist
(or $\delta_0\in F_2$, $\delta_5\in F_0$, ${_g\delta}_0=\delta_0$,
and ${_g\delta}_5=\delta_5$, giving
$3=5-2\le\nt[k]_{\delta_0}-2=\hat{k}_{{_g\delta}_0}$ and
$3=1+2\le\nt[k]_{\delta_5}+2=\hat{k}_{{_g\delta}_5}$, and so,
$\hat{s}_1^{{_g\delta}_0}$, $\hat{s}_2^{{_g\delta}_0}$,
$\hat{s}_1^{{_g\delta}_5}$, and $\hat{s}_2^{{_g\delta}_5}$ exist).

If (ii), then $\delta_{-1}\in F_0$, $\delta_4\in F_1$,
${_g\delta}_{-1}=\delta_{-1}$, and ${_g\delta}_4=\delta_4$, giving
$3=1+2\le\nt[k]_{\delta_{-1}}+2=\hat{k}_{{_g\delta}_{-1}}$ and
$3\le\nt[k]_{\delta_4}=\hat{k}_{{_g\delta}_4}$, and so,
$\hat{s}_1^{{_g\delta}_{-1}}$, $\hat{s}_2^{{_g\delta}_{-1}}$,
$\hat{s}_1^{{_g\delta}_4}$, and $\hat{s}_2^{{_g\delta}_4}$ exist
(or $\delta_0\in F_1$, $\delta_5\in F_0$, ${_g\delta}_0=\delta_0$,
and ${_g\delta}_5=\delta_5$, giving
$3\le\nt[k]_{\delta_0}=\hat{k}_{{_g\delta}_0}$ and
$3=1+2\le\nt[k]_{\delta_5}+2=\hat{k}_{{_g\delta}_5}$, and so,
$\hat{s}_1^{{_g\delta}_0}$, $\hat{s}_2^{{_g\delta}_0}$,
$\hat{s}_1^{{_g\delta}_5}$, and $\hat{s}_2^{{_g\delta}_5}$ exist).

If (iii), then $\delta_{-1}\in F_0$, $\delta_1\in F_1\bigcup F_2$,
$\delta_4\in F_1$, ${_g\delta}_{-1}=\delta_{-1}$,
${_g\delta}_1=\delta_1$, and ${_g\delta}_4=\delta_4$, giving
$3=1+2\le\nt[k]_{\delta_{-1}}+2=\hat{k}_{{_g\delta}_{-1}}$ and
$3\le\nt[k]_{\delta_4}=\hat{k}_{{_g\delta}_4}$, and so,
$\hat{s}_1^{{_g\delta}_{-1}}$, $\hat{s}_2^{{_g\delta}_{-1}}$,
$\hat{s}_1^{{_g\delta}_4}$, and $\hat{s}_2^{{_g\delta}_4}$ exist
(or $\delta_0\in F_1$, $\delta_3\in F_1\bigcup F_2$, $\delta_5\in
F_0$, ${_g\delta}_0=\delta_0$, ${_g\delta}_3=\delta_3$, and
${_g\delta}_5=\delta_5$, giving
$3\le\nt[k]_{\delta_0}=\hat{k}_{{_g\delta}_0}$ and
$3=1+2\le\nt[k]_{\delta_5}+2=\hat{k}_{{_g\delta}_5}$, and so,
$\hat{s}_1^{{_g\delta}_0}$, $\hat{s}_2^{{_g\delta}_0}$,
$\hat{s}_1^{{_g\delta}_5}$, and $\hat{s}_2^{{_g\delta}_5}$ exist).

\ \ Case 3:  Let $g=\xseq{i}{r}{r+1}{r+4}\in G_1^{\{0,4\}}\bigcup
G_1^{\{0,2,4\}}$.

By Proposition \ref{claim9},
$\nPhi{m}{0}(x_{r-1}^i),\nPhi{m}{0}(x_{r+5}^i)\notin\{v_0,v_{n-1},v_{n+1},v_{n+2}\}$,
giving

i)
$\overline{\nt[\alpha]}(\linebreak[0]\xseq{i}{r-1}{r}{r+5})=\langle
\nt[s]_1^{\delta_{-1}},\nt[s]_0,\nt[s]_1^{\delta_2},\nt[s]_2^{\delta_2},\nt[s]_1^{\delta_2},\nt[s]_0,\nt[s]_1^{\delta_5}
\rangle$ for some $\delta_{-1},\delta_2,\delta_5\in\oneton{\nt_0}$
so that $\delta_2\not=\delta_{-1}$ and $\delta_2\not=\delta_5$, or

ii)  $\overline{\nt[\alpha]}(\xseq{i}{r-1}{r}{r+5})=\langle
\nt[s]_1^{\delta_{-1}},\nt[s]_0,\nt[s]_1^{\delta_1},\nt[s]_0,\nt[s]_1^{\delta_3},\nt[s]_0,\nt[s]_1^{\delta_5}
\rangle$ for some
$\delta_{-1},\delta_1,\delta_3,\delta_5\in\oneton{\nt_0}$ so that
$\delta_{-1},\delta_5\notin\{\delta_1,\delta_3\}$.

Then, $\delta_{-1},\delta_5\in F_0$,
${_g\delta}_{-1}=\delta_{-1}$, and ${_g\delta}_5=\delta_5$, giving
$3=1+2\le\nt[k]_{\delta_{-1}}+2=\hat{k}_{{_g\delta}_{-1}}$ and
$3=1+2\le\nt[k]_{\delta_5}+2=\hat{k}_{{_g\delta}_5}$, and so,
$\hat{s}_1^{{_g\delta}_{-1}}$, $\hat{s}_2^{{_g\delta}_{-1}}$,
$\hat{s}_1^{{_g\delta}_5}$, and $\hat{s}_2^{{_g\delta}_5}$ exist.

\ \ Case 4:  Let $g=\xseq{i}{r}{r+1}{r+4}\in G_2^{\{0\}}\bigcup
G_2^{\{0,2\}}$ (or $G_3^{\{4\}}\bigcup G_3^{\{2,4\}}$).

By Proposition \ref{claim9},
$\nPhi{m}{0}(x_{r-1}^i)\notin\{v_0,v_{n-1},v_{n+1},v_{n+2}\}$ (or
$\nPhi{m}{0}(x_{r+5}^i)\notin\{v_0,v_{n-1},\linebreak[0]v_{n+1},v_{n+2}\}$),
giving

i) $\overline{\nt[\alpha]}(\xseq{i}{r-1}{r}{r+4})=\langle
\nt[s]_1^{\delta_{-1}},\nt[s]_0,\nt[s]_1^{\delta_4},\nt[s]_2^{\delta_4},\nt[s]_3^{\delta_4},\nt[s]_4^{\delta_4}
\rangle$ for some
$\delta_{-1},\linebreak[0]\delta_4\in\oneton{\nt_0}$ so that
$\delta_{-1}\not=\delta_4$ (or
$\overline{\nt[\alpha]}(\xseq{i}{r}{r+1}{r+5})=\langle
\nt[s]_4^{\delta_0},\nt[s]_3^{\delta_0},\nt[s]_2^{\delta_0},\nt[s]_1^{\delta_0},\nt[s]_0,\nt[s]_1^{\delta_5}
\rangle$ for some $\delta_0,\delta_5\in\oneton{\nt_0}$ so that
$\delta_0\not=\delta_5$), or

ii) $\overline{\nt[\alpha]}(\xseq{i}{r-1}{r}{r+4})=\langle
\nt[s]_1^{\delta_{-1}},\nt[s]_0,\linebreak[0]\nt[s]_1^{\delta_1},\nt[s]_0,\nt[s]_1^{\delta_4},\nt[s]_2^{\delta_4}
\rangle$ for some $\delta_{-1},\delta_1,\delta_4\in\oneton{\nt_0}$
so that $\delta_{-1}\not=\delta_1$ and $\delta_{-1}\not=\delta_4$
(or $\overline{\nt[\alpha]}(\xseq{i}{r}{r+1}{r+5})=\langle
\nt[s]_2^{\delta_0},\nt[s]_1^{\delta_0},\nt[s]_0,\nt[s]_1^{\delta_3},\nt[s]_0,\nt[s]_1^{\delta_5}
\rangle$ for some $\delta_0,\delta_3,\delta_5\in\oneton{\nt_0}$ so
that $\delta_5\not=\delta_0$ and $\delta_5\not=\delta_3$).

If (i), then $\delta_{-1}\in F_0$, $\delta_4\in F_2$,
${_g\delta}_{-1}=\delta_{-1}$, and ${_g\delta}_4=\delta_4$, giving
$3=1+2\le\nt[k]_{\delta_{-1}}+2=\hat{k}_{{_g\delta}_{-1}}$ and
$2=4-2\le\nt[k]_{\delta_4}-2=\hat{k}_{{_g\delta}_4}$, and so,
$\hat{s}_1^{{_g\delta}_{-1}}$, $\hat{s}_2^{{_g\delta}_{-1}}$,
$\hat{s}_1^{{_g\delta}_4}$, and $\hat{s}_2^{{_g\delta}_4}$ exist
(or $\delta_0\in F_2$, $\delta_5\in F_0$, ${_g\delta}_0=\delta_0$,
and ${_g\delta}_5=\delta_5$, giving
$2=4-2\le\nt[k]_{\delta_0}-2=\hat{k}_{{_g\delta}_0}$ and
$3=1+2\le\nt[k]_{\delta_5}+2=\hat{k}_{{_g\delta}_5}$, and so,
$\hat{s}_1^{{_g\delta}_0}$, $\hat{s}_2^{{_g\delta}_0}$,
$\hat{s}_1^{{_g\delta}_5}$, and $\hat{s}_2^{{_g\delta}_5}$ exist).

If (ii), then $\delta_{-1}\in F_0$, $\delta_4\in F_1$,
${_g\delta}_{-1}=\delta_{-1}$, and ${_g\delta}_4=\delta_4$, giving
$3=1+2\le\nt[k]_{\delta_{-1}}+2=\hat{k}_{{_g\delta}_{-1}}$ and
$2\le\nt[k]_{\delta_4}=\hat{k}_{{_g\delta}_4}$, and so,
$\hat{s}_1^{{_g\delta}_{-1}}$, $\hat{s}_2^{{_g\delta}_{-1}}$,
$\hat{s}_1^{{_g\delta}_4}$, and $\hat{s}_2^{{_g\delta}_4}$ exist
(or $\delta_0\in F_1$, $\delta_5\in F_0$, ${_g\delta}_0=\delta_0$,
and ${_g\delta}_5=\delta_5$, giving
$2\le\nt[k]_{\delta_0}=\hat{k}_{{_g\delta}_0}$ and
$3=1+2\le\nt[k]_{\delta_5}+2=\hat{k}_{{_g\delta}_5}$, and so,
$\hat{s}_1^{{_g\delta}_0}$, $\hat{s}_2^{{_g\delta}_0}$,
$\hat{s}_1^{{_g\delta}_5}$, and $\hat{s}_2^{{_g\delta}_5}$ exist).

\ \ Case 5:  Let $x\in\mathrm{V}(\X{m}\setminus\bigcup
G^{\prime})$, $\delta\in F_0$, and $k\in\oneton{\nt[k]_{\delta}}$
so that $\nt[\alpha](x)=\nt[s]_k^{\delta}$.

Then, $k+2\le\nt[k]_{\delta}+2=\hat{k}_{\delta}$, and so,
$\hat{s}_{k+2}^{\delta}$ exists.

\ \ Case 6:  Let $x\in\mathrm{V}(\X{m})$ and $\delta\in F_2$ so
that $\nt[\alpha](x)=\nt[s]_1^{\delta}$.

Then,
$\nPhi{m}{0}(x)=\nt[\beta](\nt[\alpha](x))=\nt[\beta](\nt[s]_1^{\delta})=v_{n-1}$,
and by Proposition \ref{claim9}, $x\in g$ for some $g\in G$.

Suppose $g\notin G_0$, and let $\nt[t]\in\oneton[0]{4}$ so that
$x=x_{r+\nt[t]}^i$, where $g=\xseq{i}{r}{r+1}{r+4}$.  Then,
$\nt[t]\in\{1,3\}$.  Suppose $\nt[\alpha](x_r^i)\not=\nt[s]_0$ and
$\nt[\alpha](x_{r+4}^i)\not=\nt[s]_0$.

If $\nt[t]=1$, then $\nt[\alpha](x_r^i)=\nt[s]_2^{\delta}$, giving
$\nPhi{m}{0}(x_r^i)=v_0$, and
$\nt[\alpha](x_{r-1}^i)=\nt[s]_1^{\delta}$ or
$\nt[\alpha](x_{r-1}^i)=\nt[s]_3^{\delta}$.  Thus,
$\nPhi{m}{0}(x_{r-1}^i)=v_{n-1}$ and
$\overline{\nPhi{m}{0}}(\xseq{i}{r-1}{r}{r+3})=\edge{v_{n-1}}{}{v_0,v_{n-1},v_0,v_{n-1}}{}$,
contradicting Proposition \ref{claim9}.

If $\nt[t]=3$, then $\nt[\alpha](x_{r+4}^i)=\nt[s]_2^{\delta}$,
giving $\nPhi{m}{0}(x_{r+4}^i)=v_0$, and
$\nt[\alpha](x_{r+5}^i)=\nt[s]_1^{\delta}$ or
$\nt[\alpha](x_{r+5}^i)=\nt[s]_3^{\delta}$.  Thus,
$\nPhi{m}{0}(x_{r+5}^i)=v_{n-1}$ and
$\overline{\nPhi{m}{0}}(\xseq{i}{r+1}{r+2}{r+5})=\edge{v_{n-1}}{}{v_0,v_{n-1},v_0,v_{n-1}}{}$,
contradicting Proposition \ref{claim9}.

Thus, $\nt[\alpha](x_r^i)=\nt[s]_0$ or
$\nt[\alpha](x_{r+4}^i)=\nt[s]_0$, and so, $g\in G^{\prime}$,
giving that $\hat{\alpha}(x)$ is defined.  If $g\in G_0$, then by
Proposition \ref{claim9}, $x=x_1^i$ and $\nPhi{m}{0}(x_3^i)=v_n$,
where $g=\langle x_0^i,x_1^i,x_2^i \rangle$, and so,
$\nt[\alpha](x_2^i)\not=\nt[s]_2^{\delta}$.  Then,
$\nt[\alpha](x_2^i)=\nt[s]_0$, giving $g\in G_0^{\prime}\bigcup
G_0^{\prime\prime}\subseteq G^{\prime}$, and so, $\hat{\alpha}(x)$
is defined.

\ \ Case 6.5:  Let $x\in\mathrm{V}(\X{m})$ and $\delta\in F_3$ so
that $\nt[\alpha](x)=\nt[s]_{1}^{\delta}$.

Then, by Proposition \ref{claim9} and definition of $F_3$, $x\in
g$ for some $g\in G^{\prime}$, and so, $\hat{\alpha}(x)$ is
defined.

\ \ Case 7:  Let $x\in\mathrm{V}(\X{m}\setminus\bigcup
G^{\prime})$, $\delta\in F_2$, and $k\in\oneton{\nt[k]_{\delta}}$
so that $\nt[\alpha](x)=\nt[s]_k^{\delta}$.

From Case 6, $k\in\oneton[2]{\nt[k]_{\delta}}$.  If $k=2$, then
$\hat{\alpha}(x)$ is defined.  If
$k\in\oneton[3]{\nt[k]_{\delta}}$, then
$\hat{k}_{\delta}=\nt[k]_{\delta}-2\ge k-2\ge 1$, and so,
$\hat{s}_{k-2}^{\delta}$ exists.

\ \ Case 8:  Let $x\in\mathrm{V}(\X{m}\setminus\bigcup
G^{\prime})$, $\delta\in F_1$, and $k\in\oneton{\nt[k]_{\delta}}$
so that $\nt[\alpha](x)=\nt[s]_k^{\delta}$.

Then, $1\le k\le\nt[k]_{\delta}=\hat{k}_{\delta}$, and so,
$\hat{s}_k^{\delta}$ exists.

\ \ Case 9:  Let $g_1=\langle x_0^{i_1},x_1^{i_1},x_2^{i_1}
\rangle,g_2=\langle x_0^{i_2},x_1^{i_2},x_2^{i_2} \rangle\in G_0$
so that $i_1\not=i_2$ and $\X{m}\supseteq g_1\bigcap
g_2\not=\emptyset$.

Then, $x_0^{i_1}=x_0^{i_2}=v_0$.  If $D_0^{\prime}\bigcup
D_0^{\prime\prime}\not=\emptyset$, then
$\hat{\alpha}(x_0^{i_1})=\hat{\alpha}(x_0^{i_2})=\hat{s}_0$.  If
$D_0^{\prime}\bigcup D_0^{\prime\prime}=\emptyset$, then
$x_0^{i_1}=x_0^{i_2}=x$ for some $x$ as given in case 5, case 7,
or case 8.

\ \ Case 10:  Let
$g_1=\xseq{i}{r_1}{r_1+1}{r_1+4},g_2=\xseq{i}{r_2}{r_2+1}{r_2+4}\in\bigcup\limits_{i=1}^{4}G_i$
so that $r_1\not=r_2$ and $\X{m}\supseteq g_1\bigcap
g_2\not=\emptyset$.

Then, $x_{r_1+4}^i=x_{r_2}^i$  and
$\nPhi{m}{0}(x_{r_1+4}^i)=\nPhi{m}{0}(x_{r_2}^i)=v_{n+1}$, or
$x_{r_1}^i=x_{r_2+4}^i$ and
$\nPhi{m}{0}(x_{r_1}^i)=\nPhi{m}{0}(x_{r_2+4}^i)=v_{n+1}$. Without
loss of generality, suppose $x_{r_1+4}^i=x_{r_2}^i$.  Then,
$g_1\in G_2\bigcup G_4$ and $g_2\in G_3\bigcup G_4$.

\ \ \ \ Case 10.a:  $g_1,g_2\in G^{\prime}$.

Then, $g_1\in G_2^{\{0\}}\bigcup G_2^{\{0,2\}}$ and $g_2\in
G_3^{\{4\}}\bigcup G_3^{\{2,4\}}$, giving

i) $\overline{\nt[\alpha]}(\xseq{i}{r_1-1}{r_1}{r_1+4})=\langle
\nt[s]_1^{{_{g_1}\delta}_{-1}},\nt[s]_0,\nt[s]_1^{{_{g_1}\delta}_4},\nt[s]_2^{{_{g_1}\delta}_4},
\nt[s]_3^{{_{g_1}\delta}_4},\nt[s]_4^{{_{g_1}\delta}_4} \rangle$
for some
${_{g_1}\delta}_{-1},{_{g_1}\delta}_4\linebreak[0]\in\oneton{\nt_0}$
so that ${_{g_1}\delta}_{-1}\not={_{g_1}\delta}_4$ or

ii) $\overline{\nt[\alpha]}(\xseq{i}{r_1-1}{r_1}{r_1+4})=\langle
\nt[s]_1^{{_{g_1}\delta}_{-1}},\nt[s]_0,\nt[s]_1^{{_{g_1}\delta}_1},\nt[s]_0,
\nt[s]_1^{{_{g_1}\delta}_4},\nt[s]_2^{{_{g_1}\delta}_4} \rangle$
for some
${_{g_1}\delta}_{-1},{_{g_1}\delta}_1,\linebreak[0]{_{g_1}\delta}_4\in\oneton{\nt_0}$
so that ${_{g_1}\delta}_{-1}\not={_{g_1}\delta}_1$ and
${_{g_1}\delta}_{-1}\not={_{g_1}\delta}_4$, from case 4.

If (i), then
$\overline{\nt[\alpha]}(\xseq{i}{r_2}{r_2+1}{r_2+5})=\langle
\nt[s]_4^{{_{g_2}\delta}_0},\nt[s]_3^{{_{g_2}\delta}_0},\nt[s]_2^{{_{g_2}\delta}_0},\nt[s]_1^{{_{g_2}\delta}_0},
\nt[s]_0,\nt[s]_1^{{_{g_2}\delta}_5} \rangle$ for some
${_{g_2}\delta}_0,{_{g_2}\delta}_5\in\oneton{\nt_0}$ so that
${_{g_2}\delta}_0\not={_{g_2}\delta}_5$, from case 4.  So,
${_{g_1}\delta}_4={_{g_2}\delta}_0$, giving
$\hat{\alpha}(x_{r_1+4}^i)=\hat{s}_2^{{_{g_1}\delta}_4}=\hat{s}_2^{{_{g_2}\delta}_0}=\hat{\alpha}(x_{r_2}^i)$.

If (ii), then
$\overline{\nt[\alpha]}(\xseq{i}{r_2}{r_2+1}{r_2+5})=\langle
\nt[s]_2^{{_{g_2}\delta}_0},\nt[s]_1^{{_{g_2}\delta}_0},\nt[s]_0,\nt[s]_1^{{_{g_2}\delta}_3},
\nt[s]_0,\nt[s]_1^{{_{g_2}\delta}_5} \rangle$ for some
${_{g_2}\delta}_0,{_{g_2}\delta}_3,{_{g_2}\delta}_5\in\oneton{\nt_0}$
so that ${_{g_2}\delta}_5\not={_{g_2}\delta}_0$ and
${_{g_2}\delta}_5\not={_{g_2}\delta}_3$, from case 4.  So,
${_{g_1}\delta}_4={_{g_2}\delta}_0$, giving
$\hat{\alpha}(x_{r_1+4}^i)=\hat{s}_2^{{_{g_1}\delta}_4}=\hat{s}_2^{{_{g_2}\delta}_0}=\hat{\alpha}(x_{r_2}^i)$.

Thus, $\hat{\alpha}(x_{r_1+4}^i)$ is unique, since if $g_3\in G$
so that $g_3\not= g_1$ and $g_3\not= g_2$, $x_{r_1+4}^i\notin
g_3$.

\ \ \ \ Case 10.b:  $g_1\in G^{\prime}$ and $g_2\notin
G^{\prime}$, or $g_2\in G^{\prime}$ and $g_1\notin G^{\prime}$.

Since if $g_3\in G$ so that $g_3\not= g_1$ and $g_3\not= g_2$,
$x_{r_1+4}^i\notin g_3$, then $g_1$ is the unique element of
$G^{\prime}$ containing $x_{r_1+4}^i$ or $g_2$ is the unique
element of $G^{\prime}$ containing $x_{r_2}^i$, and so,
$\hat{\alpha}(x_{r_1+4}^i)$ is unique.

\ \ \ \ Case 10.c:  $g_1,g_2\notin G^{\prime}$.

Since if $g_3\in G$ so that $g_3\not=g_1$ and $g_3\not= g_2$,
$x_{r_1+4}^i\notin g_3$, then,
$x_{r_1+4}^i\in\mathrm{V}(\X{m}\setminus\bigcup G^{\prime})$, and
so, $\hat{\alpha}(x_{r_1+4}^i)$ is unique.

\underline{\scriptsize \bf CLAIM}:  $\hat{\alpha}$ is simplicial.

Let $x,y\in\mathrm{V}(\X{m})$ be distinct and adjacent.

\ \ Case 1:  $x,y\in g$ for some $g\in G^{\prime}$.

Then, $\hat{\alpha}(x)$ and $\hat{\alpha}(y)$ are adjacent by
definition of $\hat{\alpha}$ on $g$.

\ \ Case 2:  $x\in g$ for some $g\in G^{\prime}$ and
$y\in\mathrm{V}(\X{m}\setminus\bigcup G^{\prime})$.

\ \ \ \ Case 2.a:  $g\in G_0^{\prime}\bigcup G_0^{\prime\prime}$.

Then, $x=x_2^i$ and $y=x_3^i$ where $g=\langle x_0^i,x_1^i,x_2^i
\rangle$ for some $i\in\oneton{n}$.

\ \ \ \ \ \ Case 2.a.i:  $g\in G_0^{\prime}$.

From Case 1.a of the previous claim,

i) $\overline{\nt[\alpha]}(\langle x_0^i,x_1^i,x_2^i,x_3^i
\rangle)=\langle
\nt[s]_0,\nt[s]_1^{{_g\delta}_3},\nt[s]_2^{{_g\delta}_3},\nt[s]_3^{{_g\delta}_3}
\rangle$ for some ${_g\delta}_3\in F_1$ or

ii) $\overline{\nt[\alpha]}(\langle x_0^i,x_1^i,x_2^i,x_3^i
\rangle)=\langle
\nt[s]_0,\nt[s]_1^{{_g\delta}_1},\nt[s]_0,\nt[s]_1^{{_g\delta}_3}
\rangle$ for some ${_g\delta}_1\in F_1\bigcup F_2$ and for some
${_g\delta}_3\in F_0$.

If (i), then
$\hat{\alpha}(x)=\hat{\alpha}(x_2^i)=\hat{s}_2^{{_g\delta}_3}$ and
$\hat{\alpha}(y)=\hat{\alpha}(x_3^i)=\hat{s}_3^{{_g\delta}_3}$,
and so, $\hat{\alpha}(x)$ and $\hat{\alpha}(y)$ are adjacent.

If (ii), then
$\hat{\alpha}(x)=\hat{\alpha}(x_2^i)=\hat{s}_2^{{_g\delta}_3}$ and
$\hat{\alpha}(y)=\hat{\alpha}(x_3^i)=\hat{s}_{1+2}^{{_g\delta}_3}=\hat{s}_3^{{_g\delta}_3}$,
and so, $\hat{\alpha}(x)$ and $\hat{\alpha}(y)$ are adjacent.

\ \ \ \ \ \ Case 2.a.ii:  $g\in G_0^{\prime\prime}$.

From Case 1.b of the previous claim,

i) $\overline{\nt[\alpha]}(\langle x_0^i,x_1^i,x_2^i,x_3^i
\rangle)=\langle
\nt[s]_2^{{_g\delta}_0},\nt[s]_1^{{_g\delta}_0},\nt[s]_0,\nt[s]_1^{{_g\delta}_3}
\rangle$ for some ${_g\delta}_3\in F_0$,

ii) $\overline{\nt[\alpha]}(\langle x_0^i,x_1^i,x_2^i,x_3^i
\rangle)=\langle
\nt[s]_2^{{_g\delta}_3},\nt[s]_1^{{_g\delta}_3},\nt[s]_2^{{_g\delta}_3},\nt[s]_3^{{_g\delta}_3}
\rangle$ for some ${_g\delta}_3\in F_1$, or

iii) $\overline{\nt[\alpha]}(\langle x_0^i,x_1^i,x_2^i,x_3^i
\rangle)=\langle
\nt[s]_2^{{_g\delta}_3},\nt[s]_3^{{_g\delta}_3},\nt[s]_4^{{_g\delta}_3},\nt[s]_5^{{_g\delta}_3}
\rangle$ for some ${_g\delta}_3\in F_2$.

If (i), then
$\hat{\alpha}(x)=\hat{\alpha}(x_2^i)=\hat{s}_2^{{_g\delta}_3}$ and
$\hat{\alpha}(y)=\hat{\alpha}(x_3^i)=\hat{s}_{1+2}^{{_g\delta}_3}=\hat{s}_3^{{_g\delta}_3}$,
and so, $\hat{\alpha}(x)$ and $\hat{\alpha}(y)$ are adjacent.

If (ii), then
$\hat{\alpha}(x)=\hat{\alpha}(x_2^i)=\hat{s}_2^{{_g\delta}_3}$ and
$\hat{\alpha}(y)=\hat{\alpha}(x_3^i)=\hat{s}_3^{{_g\delta}_3}$,
and so, $\hat{\alpha}(x)$ and $\hat{\alpha}(y)$ are adjacent.

If (iii), then
$\hat{\alpha}(x)=\hat{\alpha}(x_2^i)=\hat{s}_2^{{_g\delta}_3}$ and
$\hat{\alpha}(y)=\hat{\alpha}(x_3^i)=\hat{s}_{5-2}^{{_g\delta}_3}=\hat{s}_3^{{_g\delta}_3}$,
and so, $\hat{\alpha}(x)$ and $\hat{\alpha}(y)$ are adjacent.

\ \ \ \ Case 2.b:  $g\in\bigcup\limits_{i=1}^{4}G_i$.

Then, $x=x_r^i$ or $x=x_{r+4}^i$ where $g=\xseq{i}{r}{r+1}{r+4}$.

\ \ \ \ \ \ Case 2.b.i:  $g\in G_1^{\{0\}}\bigcup G_1^{\{0,2\}}$.

From Case 2 of the previous claim,

i) $\overline{\nt[\alpha]}(\xseq{i}{r-1}{r}{r+5})=\langle
\nt[s]_1^{{_g\delta}_{-1}},\nt[s]_0,\nt[s]_1^{{_g\delta}_4},\nt[s]_2^{{_g\delta}_4},\linebreak[0]\nt[s]_3^{{_g\delta}_4},
\nt[s]_4^{{_g\delta}_4},\nt[s]_5^{{_g\delta}_4} \rangle$ for some
${_g\delta}_{-1}\in F_0$ and ${_g\delta}_4\in F_2$,

ii) $\overline{\nt[\alpha]}(\xseq{i}{r-1}{r}{r+5})=\langle
\nt[s]_1^{{_g\delta}_{-1}},\nt[s]_0,\nt[s]_1^{{_g\delta}_4},\nt[s]_2^{{_g\delta}_4},\nt[s]_1^{{_g\delta}_4},
\nt[s]_2^{{_g\delta}_4},\nt[s]_3^{{_g\delta}_4} \rangle$ for some
${_g\delta}_{-1}\in F_0$ and ${_g\delta}_4\in F_1$, or

iii) $\overline{\nt[\alpha]}(\xseq{i}{r-1}{r}{r+5})=\langle
\nt[s]_1^{{_g\delta}_{-1}},\nt[s]_0,\nt[s]_1^{{_g\delta}_1},\nt[s]_0,\nt[s]_1^{{_g\delta}_4},
\nt[s]_2^{{_g\delta}_4},\nt[s]_3^{{_g\delta}_4} \rangle$ for some
${_g\delta}_{-1}\in F_0$, ${_g\delta}_1\in F_1\bigcup F_2$, and
${_g\delta}_4\in F_1$.

\ \ \ \ \ \ \ \ Case 2.b.i.1:  $x=x_r^i$.

Then, $y=x_{r-1}^i$.

If (i),(ii), or (iii), then
$\hat{\alpha}(x)=\hat{\alpha}(x_r^i)=\hat{s}_2^{{_g\delta}_{-1}}$
and
$\hat{\alpha}(y)=\hat{\alpha}(x_{r-1}^i)=\hat{s}_{1+2}^{{_g\delta}_{-1}}=\hat{s}_3^{{_g\delta}_{-1}}$,
and so, $\hat{\alpha}(x)$ and $\hat{\alpha}(y)$ are adjacent.

\ \ \ \ \ \ \ \ Case 2.b.i.2:  $x=x_{r+4}^i$.

Then, $y=x_{r+5}^i$.

If (i), then
$\hat{\alpha}(x)=\hat{\alpha}(x_{r+4}^i)=\hat{s}_2^{{_g\delta}_4}$
and
$\hat{\alpha}(y)=\hat{\alpha}(x_{r+5}^i)=\hat{s}_{5-2}^{{_g\delta}_4}=\hat{s}_3^{{_g\delta}_4}$,
and so, $\hat{\alpha}(x)$ and $\hat{\alpha}(y)$ are adjacent.

If (ii) or (iii), then
$\hat{\alpha}(x)=\hat{\alpha}(x_{r+4}^i)=\hat{s}_2^{{_g\delta}_4}$
and
$\hat{\alpha}(y)=\hat{\alpha}(x_{r+5}^i)=\hat{s}_3^{{_g\delta}_4}$,
and so, $\hat{\alpha}(x)$ and $\hat{\alpha}(y)$ are adjacent.

\ \ \ \ \ \ Case 2.b.ii:  $g\in G_1^{\{4\}}\bigcup G_1^{\{2,4\}}$.

From Case 2 of the previous claim,

i) $\overline{\nt[\alpha]}(\xseq{i}{r-1}{r}{r+5})=\langle
\nt[s]_5^{{_g\delta}_0},\nt[s]_4^{{_g\delta}_0}\nt[s]_3^{{_g\delta}_0},\nt[s]_2^{{_g\delta}_0},\linebreak[0]\nt[s]_1^{{_g\delta}_0},
\nt[s]_0,\nt[s]_1^{{_g\delta}_5} \rangle$ for some
${_g\delta}_0\in F_2$ and ${_g\delta}_5\in F_0$,

ii) $\overline{\nt[\alpha]}(\xseq{i}{r-1}{r}{r+5})=\langle
\nt[s]_3^{{_g\delta}_0},\nt[s]_2^{{_g\delta}_0},\linebreak[0]\nt[s]_1^{{_g\delta}_0},\nt[s]_2^{{_g\delta}_0},\nt[s]_1^{{_g\delta}_0},
\nt[s]_0,\nt[s]_1^{{_g\delta}_5} \rangle$ for some
${_g\delta}_0\in F_1$ and ${_g\delta}_5\in F_0$, or

iii) $\overline{\nt[\alpha]}(\xseq{i}{r-1}{r}{r+5})=\langle
\nt[s]_3^{{_g\delta}_0},\nt[s]_2^{{_g\delta}_0},\nt[s]_1^{{_g\delta}_0},\nt[s]_0,\nt[s]_1^{{_g\delta}_3},
\nt[s]_0,\nt[s]_1^{{_g\delta}_5} \rangle$ for some
${_g\delta}_0\in F_1$, ${_g\delta}_3\in F_1\bigcup F_2$, and
${_g\delta}_5\in F_0$.

\ \ \ \ \ \ \ \ Case 2.b.ii.1:  $x=x_{r+4}^i$.

Then, $y=x_{r+5}^i$.

If (i),(ii), or (iii), then
$\hat{\alpha}(x)=\hat{\alpha}(x_{r+4}^i)=\hat{s}_2^{{_g\delta}_5}$
and
$\hat{\alpha}(y)=\hat{\alpha}(x_{r+5}^i)=\hat{s}_{1+2}^{{_g\delta}_5}=\hat{s}_3^{{_g\delta}_5}$,
and so, $\hat{\alpha}(x)$ and $\hat{\alpha}(y)$ are adjacent.

\ \ \ \ \ \ \ \ Case 2.b.ii.2:  $x=x_r^i$.

Then, $y=x_{r-1}^i$.

If (i), then
$\hat{\alpha}(x)=\hat{\alpha}(x_r^i)=\hat{s}_2^{{_g\delta}_0}$ and
$\hat{\alpha}(y)=\hat{\alpha}(x_{r-1}^i)=\hat{s}_{5-2}^{{_g\delta}_0}=\hat{s}_3^{{_g\delta}_0}$,
and so, $\hat{\alpha}(x)$ and $\hat{\alpha}(y)$ are adjacent.

If (ii) or (iii), then
$\hat{\alpha}(x)=\hat{\alpha}(x_r^i)=\hat{s}_2^{{_g\delta}_0}$ and
$\hat{\alpha}(y)=\hat{\alpha}(x_{r-1}^i)=\hat{s}_3^{{_g\delta}_0}$,
and so, $\hat{\alpha}(x)$ and $\hat{\alpha}(y)$ are adjacent.

\ \ \ \ \ \ Case 2.b.iii:  $g\in G_1^{\{0,4\}}\bigcup
G_1^{\{0,2,4\}}$.

From Case 3 of the previous claim,

i) $\overline{\nt[\alpha]}(\xseq{i}{r-1}{r}{r+5})=\langle
\nt[s]_1^{{_g\delta}_{-1}},\nt[s]_0,\nt[s]_1^{{_g\delta}_2},\nt[s]_2^{{_g\delta}_2},\linebreak[0]\nt[s]_1^{{_g\delta}_2},
\nt[s]_0,\nt[s]_1^{{_g\delta}_5} \rangle$ for some
${_g\delta}_{-1},{_g\delta}_5\in F_0$ and ${_g\delta}_2\in
F_1\bigcup F_2$ or

ii) $\overline{\nt[\alpha]}(\xseq{i}{r-1}{r}{r+5})=\langle
\nt[s]_1^{{_g\delta}_{-1}},\nt[s]_0,\nt[s]_1^{{_g\delta}_1},\nt[s]_0,\nt[s]_1^{{_g\delta}_3},
\nt[s]_0,\nt[s]_1^{{_g\delta}_5} \rangle$ for some
${_g\delta}_{-1},{_g\delta}_5\in F_0$ and
${_g\delta}_1,{_g\delta}_3\in F_1\bigcup F_2$.

\ \ \ \ \ \ \ \ Case 2.b.iii.1:  $x=x_r^i$.

Then, $y=x_{r-1}^i$.

If (i) or (ii), then
$\hat{\alpha}(x)=\hat{\alpha}(x_r^i)=\hat{s}_2^{{_g\delta}_{-1}}$
and
$\hat{\alpha}(y)=\hat{\alpha}(x_{r-1}^i)=\hat{s}_{1+2}^{{_g\delta}_{-1}}=\hat{s}_3^{{_g\delta}_{-1}}$,
and so, $\hat{\alpha}(x)$ and $\hat{\alpha}(y)$ are adjacent.

\ \ \ \ \ \ \ \ Case 2.b.iii.2:  $x=x_{r+4}^i$.

Then, $y=x_{r+5}^i$.

If (i) or (ii), then
$\hat{\alpha}(x)=\hat{\alpha}(x_{r+4}^i)=\hat{s}_2^{{_g\delta}_5}$
and
$\hat{\alpha}(y)=\hat{\alpha}(x_{r+5}^i)=\hat{s}_{1+2}^{{_g\delta}_5}=\hat{s}_3^{{_g\delta}_5}$,
and so, $\hat{\alpha}(x)$ and $\hat{\alpha}(y)$ are adjacent.

\ \ \ \ \ \ Case 2.b.iv:  $g\in G_2^{\{0\}}\bigcup G_2^{\{0,2\}}$.

From Case 4 of the previous claim,

i) $\overline{\nt[\alpha]}(\xseq{i}{r-1}{r}{r+4})=\langle
\nt[s]_1^{{_g\delta}_{-1}},\nt[s]_0,\nt[s]_1^{{_g\delta}_4},\nt[s]_2^{{_g\delta}_4},\linebreak[0]\nt[s]_3^{{_g\delta}_4},
\nt[s]_4^{{_g\delta}_4} \rangle$ for some ${_g\delta}_{-1}\in F_0$
and ${_g\delta}_4\in F_2$ or

ii) $\overline{\nt[\alpha]}(\xseq{i}{r-1}{r}{r+4})=\langle
\nt[s]_1^{{_g\delta}_{-1}},\nt[s]_0,\nt[s]_1^{{_g\delta}_1},\linebreak[0]\nt[s]_0,\nt[s]_1^{{_g\delta}_4},
\nt[s]_2^{{_g\delta}_4} \rangle$ for some ${_g\delta}_{-1}\in
F_0$, ${_g\delta}_1\in F_1\bigcup F_2$, and ${_g\delta}_4\in F_1$.

\ \ \ \ \ \ \ \ Case 2.b.iv.1:  $x=x_r^i$.

Then, $y=x_{r-1}^i$.

If (i) or (ii), then
$\hat{\alpha}(x)=\hat{\alpha}(x_r^i)=\hat{s}_2^{{_g\delta}_{-1}}$
and
$\hat{\alpha}(y)=\hat{\alpha}(x_{r-1}^i)=\hat{s}_{1+2}^{{_g\delta}_{-1}}=\hat{s}_3^{{_g\delta}_{-1}}$,
and so, $\hat{\alpha}(x)$ and $\hat{\alpha}(y)$ are adjacent.

\ \ \ \ \ \ \ \ Case 2.b.iv.2:  $x=x_{r+4}^i$.

Then, $y=x_{r+5}^i$.

If (i), then by Proposition \ref{claim9}, $x_{r+5}^i$ exists and
$\nPhi{m}{0}(x_{r+5}^i)=v_{n-1}$, and so,
$\nt[\alpha](x_{r+5}^i)=\nt[s]_3^{{_g\delta}_4}$ or
$\nt[\alpha](x_{r+5}^i)=\nt[s]_5^{{_g\delta}_4}$.  Then,
$\hat{\alpha}(x)=\hat{\alpha}(x_{r+4}^i)=\hat{s}_2^{{_g\delta}_4}$,
and
$\hat{\alpha}(y)=\hat{\alpha}(x_{r+5}^i)=\hat{s}_{3-2}^{{_g\delta}_4}=\hat{s}_1^{{_g\delta}_4}$
or
$\hat{\alpha}(y)=\hat{\alpha}(x_{r+5}^i)=\hat{s}_{5-2}^{{_g\delta}_4}=\hat{s}_3^{{_g\delta}_4}$,
and so, $\hat{\alpha}(x)$ and $\hat{\alpha}(y)$ are adjacent.

If (ii), then by Proposition \ref{claim9}, $x_{r+5}^i$ exists and
$\nPhi{m}{0}(x_{r+5}^i)=v_{n-1}$, and so,
$\nt[\alpha](x_{r+5}^i)=\nt[s]_1^{{_g\delta}_4}$ or
$\nt[\alpha](x_{r+5}^i)=\nt[s]_3^{{_g\delta}_4}$.  Then,
$\hat{\alpha}(x)=\hat{\alpha}(x_{r+4}^i)=\hat{s}_2^{{_g\delta}_4}$,
and
$\hat{\alpha}(y)=\hat{\alpha}(x_{r+5}^i)=\hat{s}_1^{{_g\delta}_4}$
or
$\hat{\alpha}(y)=\hat{\alpha}(x_{r+5}^i)=\hat{s}_3^{{_g\delta}_4}$,
and so, $\hat{\alpha}(x)$ and $\hat{\alpha}(y)$ are adjacent.

\ \ \ \ \ \ Case 2.b.v:  $g\in G_3^{\{4\}}\bigcup G_3^{\{2,4\}}$.

From Case 4 of the previous claim,

i) $\overline{\nt[\alpha]}(\xseq{i}{r}{r+1}{r+5})=\langle
\nt[s]_4^{{_g\delta}_0},\nt[s]_3^{{_g\delta}_0},\nt[s]_2^{{_g\delta}_0},\nt[s]_1^{{_g\delta}_0},\linebreak[0]
\nt[s]_0,\nt[s]_1^{{_g\delta}_5} \rangle$ for some
${_g\delta}_0\in F_2$ and ${_g\delta}_5\in F_0$ or

ii) $\overline{\nt[\alpha]}(\xseq{i}{r}{r+1}{r+5})=\langle
\nt[s]_2^{{_g\delta}_0},\nt[s]_1^{{_g\delta}_0},\nt[s]_0,\linebreak[0]\nt[s]_1^{{_g\delta}_3},
\nt[s]_0,\nt[s]_1^{{_g\delta}_5} \rangle$ for some
${_g\delta}_0\in F_1$, ${_g\delta}_3\in F_1\bigcup F_2$, and
${_g\delta}_5\in F_0$.

\ \ \ \ \ \ \ \ Case 2.b.v.1:  $x=x_{r+4}^i$.

Then, $y=x_{r+5}^i$.

If (i) or (ii), then
$\hat{\alpha}(x)=\hat{\alpha}(x_{r+4}^i)=\hat{s}_2^{{_g\delta}_5}$
and
$\hat{\alpha}(y)=\hat{\alpha}(x_{r+5}^i)=\hat{s}_{1+2}^{{_g\delta}_5}=\hat{s}_3^{{_g\delta}_5}$,
and so, $\hat{\alpha}(x)$ and $\hat{\alpha}(y)$ are adjacent.

\ \ \ \ \ \ \ \ Case 2.b.v.2:  $x=x_r^i$.

Then, $y=x_{r-1}^i$.

If (i), then by Proposition \ref{claim9}, $x_{r-1}^i$ exists and
$\nPhi{m}{0}(x_{r-1}^i)=v_{n-1}$, and so,
$\nt[\alpha](x_{r-1}^i)=\nt[s]_3^{{_g\delta}_0}$ or
$\nt[\alpha](x_{r-1}^i)=\nt[s]_5^{{_g\delta}_0}$.  Then,
$\hat{\alpha}(x)=\hat{\alpha}(x_r^i)=\hat{s}_2^{{_g\delta}_0}$,
and
$\hat{\alpha}(y)=\hat{\alpha}(x_{r-1}^i)=\hat{s}_{3-2}^{{_g\delta}_0}=\hat{s}_1^{{_g\delta}_0}$
or
$\hat{\alpha}(y)=\hat{\alpha}(x_{r-1}^i)=\hat{s}_{5-2}^{{_g\delta}_0}=\hat{s}_3^{{_g\delta}_0}$,
and so, $\hat{\alpha}(x)$ and $\hat{\alpha}(y)$ are adjacent.

If (ii), then by Proposition \ref{claim9}, $x_{r-1}^i$ exists and
$\nPhi{m}{0}(x_{r-1}^i)=v_{n-1}$, and so,
$\nt[\alpha](x_{r-1}^i)=\nt[s]_1^{{_g\delta}_0}$ or
$\nt[\alpha](x_{r-1}^i)=\nt[s]_3^{{_g\delta}_0}$.  Then,
$\hat{\alpha}(x)=\hat{\alpha}(x_r^i)=\hat{s}_2^{{_g\delta}_0}$,
and
$\hat{\alpha}(y)=\hat{\alpha}(x_{r-1}^i)=\hat{s}_1^{{_g\delta}_0}$
or
$\hat{\alpha}(y)=\hat{\alpha}(x_{r-1}^i)=\hat{s}_3^{{_g\delta}_0}$,
and so, $\hat{\alpha}(x)$ and $\hat{\alpha}(y)$ are adjacent.

\ \ Case 3:  $x,y\in\mathrm{V}(\X{m}\setminus\bigcup G^{\prime}$.

\ \ \ \ Case 3.a:
$\nt[\alpha](x),\nt[\alpha](y)\in\nt[S]^{\delta}\setminus\{\nt[s]_0\}$
for some $\delta\in\oneton{\nt_0}$.

Since $\nt[\alpha]$ is simplicial, then $\nt[\alpha](x)$ and
$\nt[\alpha](y)$ are adjacent by Proposition \ref{claim9}, and so,
$\hat{\alpha}(x)$ and $\hat{\alpha}(y)$ are adjacent by definition
of $\hat{\alpha}$.

\ \ \ \ Case 3.b:  $\nt[\alpha](x)=\nt[s]_0$.

\ \ \ \ \ \ Case 3.b.i:  $\nt[\alpha](y)=\nt[s]_1^{\delta}$ for
some $\delta\in F_0$.

Then,
$\nPhi{m}{0}(y)=\nt[\beta](\nt[\alpha](y))=\nt[\beta](\nt[s]_1^{\delta}\not=v_{n-1}$,
and so by Proposition \ref{claim9}, there exist $i\in\oneton{n}$
and $z\in\mathrm{V}(\X{m})$ so that $z$ is adjacent to $x$ and
$z\not=y$ and $\edge{z}{}{x,y}{}\subseteq\arm{m}$.  By Proposition
\ref{claim13}, $\nPhi{m}{0}(z)=v_{n-1}$, and so by Proposition
\ref{claim9}, $z\in g$ for some $g\in G$.

Suppose $g\notin G_0$.  Then, $z=x_{r+1}^i$ or $z=x_{r+3}^i$ where
$g=\xseq{i}{r}{r+1}{r+4}$.  Since $\nPhi{m}{0}(y)\not=v_{n-1}$,
then $\edge{z}{}{x,y}{}=\langle x_{r+1}^i,x_r^i,x_{r-1}^i \rangle$
or $\edge{z}{}{x,y}{}=\langle x_{r+3}^i,x_{r+4}^i,x_{r+5}^i
\rangle$, and so, $\nt[\alpha](x_r^i)=\nt[s]_0$ or
$\nt[\alpha](x_{r+4}^i)=\nt[s]_0$.  Thus, $g\in G^{\prime}$,
contradicting $x\in g$.

Suppose $g\in G_0$.  Then, $z=x_1^i$ and $x=x_2^i$ where
$g=\langle x_0^i,x_1^i,x_2^i \rangle$, and so $g\in
G_0^{\prime}\bigcup G_0^{\prime\prime}\subseteq G^{\prime}$,
contradicting $x\in g$.

Thus, Case 3.b.i does not occur.

\ \ \ \ \ \ Case 3.b.ii:  $\nt[\alpha](y)=\nt[s]_1^{\delta}$ for
some $\delta\in F_2$.

From Case 6 of the previous claim, $y\in g$ for some $g\in
G^{\prime}$, giving a contradiction.

Thus, Case 3.b.ii does not occur.

\ \ \ \ \ \ Case 3.b.iii:  $\nt[\alpha](y)=\nt[s]_1^{\delta}$ for
some $\delta\in F_1$.

Then, $\hat{\alpha}(x)=\hat{s}_0$ and
$\hat{\alpha}(y)=\hat{s}_1^{\delta}$, and so, $\hat{\alpha}(x)$
and $\hat{\alpha}(y)$ are adjacent.

\underline{\scriptsize \bf CLAIM}:  $\hat{\beta}(\hat{s}_0)=v_0$.

By definition of $\hat{\beta}$,
$\hat{\beta}(\hat{s}_0)=\nt[\beta](\nt[s]_0)$, and by hypothesis,
$\nt[\beta](\nt[s]_0)=v_0$.

\underline{\scriptsize \bf CLAIM}:
$\hat{\beta}\circ\hat{\alpha}=\nPhi{m}{0}$.

Let $x\in\mathrm{V}(\X{m})$.

\ \ Case 1:  $x\in g$ for some $g\in G^{\prime}$.

\ \ \ \ Case 1.a:  $g\in G_0^{\prime}\bigcup G_0^{\prime\prime}$.

Let $g=\langle x_0^i,x_1^i,x_2^i \rangle$ for some
$i\in\oneton{n}$.  By Proposition \ref{claim9},
$\nPhi{m}{0}(\langle x_0^i,x_1^i,x_2^i
\rangle)\linebreak[0]=\edge{v}{0}{v_{n-1},v_0}{}$.

\ \ \ \ \ \ Case1.a.i:  $g\in G_0^{\prime}$.

From Case 2.a.i of the second claim, (i) or (ii) holds.

If (i), then $\overline{\nPhi{m}{0}}(\langle x_0^i,x_1^i,x_2^i
\rangle)=\overline{\nt[\beta]}(\overline{\nt[\alpha]}(\langle
x_0^i,x_1^i,x_2^i \rangle))=\overline{\nt[\beta]}(\langle
\nt[s]_0,\nt[s]_1^{{_g\delta}_3},\nt[s]_2^{{_g\delta}_3}
\rangle)=\overline{\hat{\beta}}(\langle
\hat{s}_0,\hat{s}_1^{{_g\delta}_3},\linebreak[0]\hat{s}_2^{{_g\delta}_3}
\rangle)=\overline{\hat{\beta}}(\overline{\hat{\alpha}}(\langle
x_0^i,x_1^i,x_2^i \rangle))$.

If (ii), then $\overline{\nPhi{m}{0}}(\langle x_0^i,x_1^i,x_2^i
\rangle)=\edge{v}{0}{v_{n-1},v_0}{}=\overline{\hat{\beta}}(\langle
\hat{s}_0,\hat{s}_1^{{_g\delta}_3},\hat{s}_2^{{_g\delta}_3}
\rangle)=\overline{\hat{\beta}}(\overline{\hat{\alpha}}(\langle
x_0^i,x_1^i,\linebreak[0]x_2^i \rangle))$.

\ \ \ \ \ \ Case1.a.ii:  $g\in G_0^{\prime\prime}$.

From Case 2.a.ii of the second claim, (i),(ii), or (iii) holds.

If (i), then $\overline{\nPhi{m}{0}}(\langle x_0^i,x_1^i,x_2^i
\rangle)=\edge{v}{0}{v_{n-1},v_0}{}=\overline{\hat{\beta}}(\langle
\hat{s}_0,\hat{s}_1^{{_g\delta}_3},\hat{s}_2^{{_g\delta}_3}
\rangle)=\overline{\hat{\beta}}(\overline{\hat{\alpha}}(\langle
x_0^i,x_1^i,\linebreak[0]x_2^i \rangle))$.

If (ii), then $\overline{\nPhi{m}{0}}(\langle x_0^i,x_1^i,x_2^i
\rangle)=\overline{\nt[\beta]}(\overline{\nt[\alpha]}(\langle
x_0^i,x_1^i,x_2^i \rangle))=\overline{\nt[\beta]}(\langle
\nt[s]_0,\nt[s]_1^{{_g\delta}_3},\nt[s]_2^{{_g\delta}_3}
\rangle)=\overline{\hat{\beta}}(\langle
\hat{s}_0,\hat{s}_1^{{_g\delta}_3},\linebreak[0]\hat{s}_2^{{_g\delta}_3}
\rangle)=\overline{\hat{\beta}}(\overline{\hat{\alpha}}(\langle
x_0^i,x_1^i,x_2^i \rangle))$.

If (iii), then $\overline{\nPhi{m}{0}}(\langle x_0^i,x_1^i,x_2^i
\rangle)=\overline{\nt[\beta]}(\overline{\nt[\alpha]}(\langle
x_0^i,x_1^i,x_2^i \rangle))=\overline{\nt[\beta]}(\langle
\nt[s]_2^{{_g\delta}_3},\nt[s]_3^{{_g\delta}_3},\nt[s]_4^{{_g\delta}_3}
\rangle)=\overline{\hat{\beta}}(\langle
\hat{s}_0,\linebreak[0]\hat{s}_1^{{_g\delta}_3},\hat{s}_2^{{_g\delta}_3}
\rangle)=\overline{\hat{\beta}}(\overline{\hat{\alpha}}(\langle
x_0^i,x_1^i,x_2^i \rangle))$.

\ \ \ \ Case 1.b:  $g\in\bigcup\limits_{i=1}^{4}G_i$.

Let $g=\xseq{i}{r}{r+1}{r+4}$ for some $i\in\oneton{n}$ and
$r\in\oneton[0]{\ell_i}$.

\ \ \ \ \ \ Case 1.b.i:  $g\in G_1^{\{0\}}\bigcup G_1^{\{0,2\}}$.

From Case 2.b.i of the second claim, (i),(ii), or (iii) holds.

If (i), then
$\overline{\nPhi{m}{0}}(\xseq{i}{r}{r+1}{r+4})=\overline{\nPhi{m}{0}}(\xseq{i}{r}{r+1}{r+3})\bigvee
\overline{\nPhi{m}{0}}(\edge{x_{r+3}^i}{}{x_{r+4}^i}{})=\langle
v_0,v_{n-1},v_0,v_{n-1} \rangle\bigvee
\overline{\nPhi{m}{0}}(\edge{x_{r+3}^i}{}{x_{r+4}^i}{})=\overline{\hat{\beta}}(\langle
\hat{s}_2^{{_g\delta}_{-1}},\hat{s}_1^{{_g\delta}_{-1}},\hat{s}_0,\hat{s}_1^{{_g\delta}_4}
\rangle)\bigvee\overline{\nt[\beta]}(\overline{\nt[\alpha]}(\linebreak\edge{x_{r+3}^i}{}{x_{r+4}^i}{}))=
\overline{\hat{\beta}}(\overline{\hat{\alpha}}(\xseq{i}{r}{r+1}{r+3}))\bigvee\overline{\nt[\beta]}(\langle
\nt[s]_3^{{_g\delta}_4},\nt[s]_4^{{_g\delta}_4}
\rangle)=\overline{\hat{\beta}}(\overline{\hat{\alpha}}(\xseq{i}{r}{r+1}{r+3}\linebreak[0]))\bigvee\overline{\hat{\beta}}(\langle
\hat{s}_1^{{_g\delta}_4},\hat{s}_2^{{_g\delta}_4}
\rangle)=\overline{\hat{\beta}}(\overline{\hat{\alpha}}(\xseq{i}{r}{r+1}{r+3}))\bigvee\overline{\hat{\beta}}(\overline{\hat{\alpha}}((\langle
x_{r+3}^i,x_{r+4}^i
\rangle))=\overline{\hat{\beta}}(\overline{\hat{\alpha}}(\xseq{i}{r}{r+1}{r+4}))$,
and if (ii) or (iii), then
$\overline{\nPhi{m}{0}}(\xseq{i}{r}{r+1}{r+4})=\overline{\nPhi{m}{0}}(\xseq{i}{r}{r+1}{r+3})\bigvee
\overline{\nPhi{m}{0}}(\edge{x_{r+3}^i}{}{x_{r+4}^i}{})=\langle
v_0,v_{n-1},v_0,v_{n-1} \rangle\bigvee
\overline{\nPhi{m}{0}}(\edge{x_{r+3}^i}{}{x_{r+4}^i}{})=\overline{\hat{\beta}}(\langle
\hat{s}_2^{{_g\delta}_{-1}},\hat{s}_1^{{_g\delta}_{-1}},\linebreak[0]\hat{s}_0,\hat{s}_1^{{_g\delta}_4}
\rangle)\bigvee\overline{\nt[\beta]}(\overline{\nt[\alpha]}(\edge{x_{r+3}^i}{}{x_{r+4}^i}{}))=
\overline{\hat{\beta}}(\overline{\hat{\alpha}}(\xseq{i}{r}{r+1}{r+3}))\bigvee\overline{\nt[\beta]}(\langle
\nt[s]_1^{{_g\delta}_4},\nt[s]_2^{{_g\delta}_4}
\rangle)=\overline{\hat{\beta}}(\overline{\hat{\alpha}}(\xseq{i}{r}{r+1}{r+3}))\bigvee\overline{\hat{\beta}}(\langle
\hat{s}_1^{{_g\delta}_4},\hat{s}_2^{{_g\delta}_4}
\rangle)=\overline{\hat{\beta}}(\overline{\hat{\alpha}}(\xseq{i}{r}{r+1}{r+3}))\bigvee\overline{\hat{\beta}}(\overline{\hat{\alpha}}((\langle
x_{r+3}^i,x_{r+4}^i
\rangle))=\overline{\hat{\beta}}(\overline{\hat{\alpha}}(\xseq{i}{r}{r+1}{r+4}))$,
since
$\hat{\beta}(\hat{s}_1^{{_g\delta}_4})=\nt[\beta](\nt[s]_{1+2}^{{_g\delta}_4})=\nt[\beta](\nt[s]_{3}^{{_g\delta}_4})=v_{n-1}$
if ${_g\delta}_4\in F_2$ and
$\hat{\beta}(\hat{s}_1^{{_g\delta}_4})=\nt[\beta](\nt[s]_{1}^{{_g\delta}_4})=v_{n-1}$
if ${_g\delta}_4\in F_1$.

\ \ \ \ \ \ Case 1.b.ii:  $g\in G_1^{\{4\}}\bigcup G_1^{\{2,4\}}$.

From Case 2.b.ii of the second claim, (i),(ii), or (iii) holds.

If (i), then
$\overline{\nPhi{m}{0}}(\xseq{i}{r}{r+1}{r+4})=\overline{\nPhi{m}{0}}(\edge{x_{r}^i}{}{x_{r+1}^i}{})
\bigvee\overline{\nPhi{m}{0}}(\xseq{i}{r+1}{r+2}{r+4})=\overline{\nPhi{m}{0}}(\edge{x_{r}^i}{}{x_{r+1}^i}{})\bigvee\langle
v_{n-1},v_0,v_{n-1},v_0
\rangle=\overline{\nt[\beta]}(\overline{\nt[\alpha]}(\edge{x_{r}^i}{}{x_{r+1}^i}{}))\bigvee\overline{\hat{\beta}}(\langle
\hat{s}_1^{{_g\delta}_0},\hat{s}_0,\hat{s}_1^{{_g\delta}_5},\linebreak[0]\hat{s}_2^{{_g\delta}_5}
\rangle)=\overline{\nt[\beta]}(\langle
\nt[s]_4^{{_g\delta}_0},\nt[s]_3^{{_g\delta}_0}
\rangle)\bigvee\overline{\hat{\beta}}(\overline{\hat{\alpha}}(\xseq{i}{r+1}{r+2}{r+4}))=\overline{\hat{\beta}}(\langle
\hat{s}_2^{{_g\delta}_0},\hat{s}_1^{{_g\delta}_0}
\rangle)\bigvee\overline{\hat{\beta}}(\overline{\hat{\alpha}}(\xseq{i}{r+1}{r+2}{r+4}))
=\overline{\hat{\beta}}(\overline{\hat{\alpha}}(\edge{x_r^i}{}{x_{r+1}^i}{}))
\bigvee\overline{\hat{\beta}}(\overline{\hat{\alpha}}(\xseq{i}{r+1}{r+2}{r+4}))=\overline{\hat{\beta}}(\overline{\hat{\alpha}}(\xseq{i}{r}{r+1}{r+4}))$,
and if (ii) or (iii), then
$\overline{\nPhi{m}{0}}(\xseq{i}{r}{r+1}{r+4})=\overline{\nPhi{m}{0}}(\edge{x_{r}^i}{}{x_{r+1}^i}{})
\bigvee\linebreak[0]\overline{\nPhi{m}{0}}(\xseq{i}{r+1}{r+2}{r+4})=\overline{\nPhi{m}{0}}(\edge{x_{r}^i}{}{x_{r+1}^i}{})
\bigvee\langle v_{n-1},v_0,v_{n-1},v_0
\rangle=\overline{\nt[\beta]}(\overline{\nt[\alpha]}(\edge{x_{r}^i}{}{x_{r+1}^i}{}))\linebreak[0]
\bigvee\overline{\hat{\beta}}(\langle
\hat{s}_1^{{_g\delta}_0},\hat{s}_0,\hat{s}_1^{{_g\delta}_5},\hat{s}_2^{{_g\delta}_5}
\rangle)=\overline{\nt[\beta]}(\langle
\nt[s]_2^{{_g\delta}_0},\nt[s]_1^{{_g\delta}_0}
\rangle)\bigvee\overline{\hat{\beta}}(\overline{\hat{\alpha}}(\xseq{i}{r+1}{r+2}{r+4}))=
\overline{\hat{\beta}}(\langle
\hat{s}_2^{{_g\delta}_0},\hat{s}_1^{{_g\delta}_0}\linebreak[0]
\rangle)\bigvee\overline{\hat{\beta}}(\overline{\hat{\alpha}}(\xseq{i}{r+1}{r+2}{r+4}))
=\overline{\hat{\beta}}(\overline{\hat{\alpha}}((\langle
x_{r}^i,x_{r+1}^i
\rangle))\bigvee\overline{\hat{\beta}}(\overline{\hat{\alpha}}(\xseq{i}{r+1}{r+2}{r+4}))=\overline{\hat{\beta}}(\overline{\hat{\alpha}}(\xseq{i}{r}{r+1}{r+4}))$,
since
$\hat{\beta}(\hat{s}_1^{{_g\delta}_0})=\nt[\beta](\nt[s]_{1+2}^{{_g\delta}_0})=\nt[\beta](\nt[s]_{3}^{{_g\delta}_0})=v_{n-1}$
if ${_g\delta}_0\in F_2$ and
$\hat{\beta}(\hat{s}_1^{{_g\delta}_0})=\nt[\beta](\nt[s]_{1}^{{_g\delta}_0})=v_{n-1}$
if ${_g\delta}_0\in F_1$.

\ \ \ \ \ \ Case 1.b.iii:  $g\in G_1^{\{0,4\}}\bigcup
G_1^{\{0,2,4\}}$.

From Case 2.b.iii of the second claim, (i) or (ii) holds.

If (i) or (ii), then $\nPhi{m}{0}(\xseq{i}{r}{r+1}{r+4})=\langle
v_0,v_{n-1},v_0,v_{n-1},v_0 \rangle=\overline{\hat{\beta}}(\langle
\hat{s}_2^{{_g\delta}_{-1}},\linebreak[0]\hat{s}_1^{{_g\delta}_{-1}},\hat{s}_0,\hat{s}_1^{{_g\delta}_5},\hat{s}_2^{{_g\delta}_5}
\rangle)=\overline{\hat{\beta}}(\overline{\hat{\alpha}}(\xseq{i}{r}{r+1}{r+4}))$.

\ \ \ \ \ \ Case 1.b.iv:  $g\in G_2^{\{0\}}\bigcup G_2^{\{0,2\}}$.

From Case 2.b.iv of the second claim, (i) or (ii) holds.

If (i), then
$\overline{\nPhi{m}{0}}(\xseq{i}{r}{r+1}{r+4})=\overline{\nPhi{m}{0}}(\xseq{i}{r}{r+1}{r+3})\bigvee
\overline{\nPhi{m}{0}}(\edge{x_{r+3}^i}{}{x_{r+4}^i}{})=\langle
v_0,v_{n-1},v_0,v_{n-1} \rangle\bigvee
\overline{\nPhi{m}{0}}(\edge{x_{r+3}^i}{}{x_{r+4}^i}{})=\overline{\hat{\beta}}(\langle
\hat{s}_2^{{_g\delta}_{-1}},\hat{s}_1^{{_g\delta}_{-1}},\hat{s}_0,\hat{s}_1^{{_g\delta}_4}
\rangle)\bigvee\overline{\nt[\beta]}(\overline{\nt[\alpha]}(\linebreak\edge{x_{r+3}^i}{}{x_{r+4}^i}{}))=
\overline{\hat{\beta}}(\overline{\hat{\alpha}}(\xseq{i}{r}{r+1}{r+3}))\bigvee\overline{\nt[\beta]}(\langle
\nt[s]_3^{{_g\delta}_4},\nt[s]_4^{{_g\delta}_4}
\rangle)=\overline{\hat{\beta}}(\overline{\hat{\alpha}}(\xseq{i}{r}{r+1}{r+3}\linebreak[0]))\bigvee\overline{\hat{\beta}}(\langle
\hat{s}_1^{{_g\delta}_4},\hat{s}_2^{{_g\delta}_4}
\rangle)=\overline{\hat{\beta}}(\overline{\hat{\alpha}}(\xseq{i}{r}{r+1}{r+3}))\bigvee\overline{\hat{\beta}}(\overline{\hat{\alpha}}((\langle
x_{r+3}^i,x_{r+4}^i
\rangle))=\overline{\hat{\beta}}(\overline{\hat{\alpha}}(\xseq{i}{r}{r+1}{r+4}))$,
and if (ii), then
$\overline{\nPhi{m}{0}}(\xseq{i}{r}{r+1}{r+4})=\overline{\nPhi{m}{0}}(\xseq{i}{r}{r+1}{r+3})\bigvee\linebreak[0]
\overline{\nPhi{m}{0}}(\edge{x_{r+3}^i}{}{x_{r+4}^i}{})=\langle
v_0,v_{n-1},v_0,v_{n-1} \rangle\bigvee
\overline{\nPhi{m}{0}}(\edge{x_{r+3}^i}{}{x_{r+4}^i}{})=\overline{\hat{\beta}}(\langle
\hat{s}_2^{{_g\delta}_{-1}},\hat{s}_1^{{_g\delta}_{-1}},\hat{s}_0,\hat{s}_1^{{_g\delta}_4}
\rangle)\linebreak[0]\bigvee\overline{\nt[\beta]}(\overline{\nt[\alpha]}(\edge{x_{r+3}^i}{}{x_{r+4}^i}{}))=
\overline{\hat{\beta}}(\overline{\hat{\alpha}}(\xseq{i}{r}{r+1}{r+3}))\bigvee\overline{\nt[\beta]}(\langle
\nt[s]_1^{{_g\delta}_4},\nt[s]_2^{{_g\delta}_4}
\rangle)=\overline{\hat{\beta}}(\overline{\hat{\alpha}}(\xseq{i}{r}{r+1}{r+3}))\bigvee\overline{\hat{\beta}}(\langle
\hat{s}_1^{{_g\delta}_4},\hat{s}_2^{{_g\delta}_4}
\rangle)=\overline{\hat{\beta}}(\overline{\hat{\alpha}}(\xseq{i}{r}{r+1}{r+3}))\bigvee\overline{\hat{\beta}}(\overline{\hat{\alpha}}((\langle
x_{r+3}^i,x_{r+4}^i
\rangle))=\overline{\hat{\beta}}(\overline{\hat{\alpha}}\linebreak[0](\xseq{i}{r}{r+1}{r+4}))$,
since
$\hat{\beta}(\hat{s}_1^{{_g\delta}_4})=\nt[\beta](\nt[s]_{1+2}^{{_g\delta}_4})=\nt[\beta](\nt[s]_{3}^{{_g\delta}_4})=v_{n-1}$
if ${_g\delta}_4\in F_2$ and
$\hat{\beta}(\hat{s}_1^{{_g\delta}_4})=\nt[\beta](\nt[s]_{1}^{{_g\delta}_4})=v_{n-1}$
if ${_g\delta}_4\in F_1$.

\ \ \ \ \ \ Case 1.b.v:  $g\in G_3^{\{4\}}\bigcup G_3^{\{2,4\}}$.

From Case 2.b.v of the second claim, (i) or (ii) holds.

If (i), then
$\overline{\nPhi{m}{0}}(\xseq{i}{r}{r+1}{r+4})=\overline{\nPhi{m}{0}}(\edge{x_{r}^i}{}{x_{r+1}^i}{})
\bigvee\overline{\nPhi{m}{0}}(\xseq{i}{r+1}{r+2}{r+4})=\overline{\nPhi{m}{0}}(\edge{x_{r}^i}{}{x_{r+1}^i}{})\bigvee\langle
v_{n-1},v_0,v_{n-1},v_0
\rangle=\overline{\nt[\beta]}(\overline{\nt[\alpha]}(\edge{x_{r}^i}{}{x_{r+1}^i}{}))\bigvee\overline{\hat{\beta}}(\langle
\hat{s}_1^{{_g\delta}_0},\hat{s}_0,\hat{s}_1^{{_g\delta}_5},\linebreak[0]\hat{s}_2^{{_g\delta}_5}
\rangle)=\overline{\nt[\beta]}(\langle
\nt[s]_4^{{_g\delta}_0},\nt[s]_3^{{_g\delta}_0}
\rangle)\bigvee\overline{\hat{\beta}}(\overline{\hat{\alpha}}(\xseq{i}{r+1}{r+2}{r+4}))=\overline{\hat{\beta}}(\langle
\hat{s}_2^{{_g\delta}_0},\hat{s}_1^{{_g\delta}_0}
\rangle)\bigvee\overline{\hat{\beta}}(\overline{\hat{\alpha}}(\xseq{i}{r+1}{r+2}{r+4}))
=\overline{\hat{\beta}}(\overline{\hat{\alpha}}(\edge{x_r^i}{}{x_{r+1}^i}{}))
\bigvee\overline{\hat{\beta}}(\overline{\hat{\alpha}}(\xseq{i}{r+1}{r+2}{r+4}))=\overline{\hat{\beta}}(\overline{\hat{\alpha}}(\xseq{i}{r}{r+1}{r+4}))$,
and if (ii), then
$\overline{\nPhi{m}{0}}(\xseq{i}{r}{r+1}{r+4})=\overline{\nPhi{m}{0}}(\edge{x_{r}^i}{}{x_{r+1}^i}{})
\bigvee\overline{\nPhi{m}{0}}(\xseq{i}{r+1}{r+2}{r+4})=\overline{\nPhi{m}{0}}(\edge{x_{r}^i}{}{x_{r+1}^i}{})
\bigvee\langle v_{n-1},v_0,v_{n-1},v_0
\rangle=\overline{\nt[\beta]}(\overline{\nt[\alpha]}(\edge{x_{r}^i}{}{x_{r+1}^i}{}))
\bigvee\overline{\hat{\beta}}(\langle
\hat{s}_1^{{_g\delta}_0},\linebreak[0]\hat{s}_0,\hat{s}_1^{{_g\delta}_5},\hat{s}_2^{{_g\delta}_5}
\rangle)=\overline{\nt[\beta]}(\langle
\nt[s]_2^{{_g\delta}_0},\nt[s]_1^{{_g\delta}_0}
\rangle)\bigvee\overline{\hat{\beta}}(\overline{\hat{\alpha}}(\xseq{i}{r+1}{r+2}{r+4}))=
\overline{\hat{\beta}}(\langle
\hat{s}_2^{{_g\delta}_0},\hat{s}_1^{{_g\delta}_0}
\rangle)\bigvee\overline{\hat{\beta}}(\overline{\hat{\alpha}}(\linebreak[0]\xseq{i}{r+1}{r+2}{r+4}))
=\overline{\hat{\beta}}(\overline{\hat{\alpha}}((\langle
x_{r}^i,x_{r+1}^i
\rangle))\bigvee\overline{\hat{\beta}}(\overline{\hat{\alpha}}(\xseq{i}{r+1}{r+2}{r+4}))=\overline{\hat{\beta}}(\overline{\hat{\alpha}}(\xseq{i}{r}{r+1}{r+4}))$,
since
$\hat{\beta}(\hat{s}_1^{{_g\delta}_0})=\nt[\beta](\nt[s]_{1+2}^{{_g\delta}_0})=\nt[\beta](\nt[s]_{3}^{{_g\delta}_0})=v_{n-1}$
if ${_g\delta}_0\in F_2$ and
$\hat{\beta}(\hat{s}_1^{{_g\delta}_0})=\nt[\beta](\nt[s]_{1}^{{_g\delta}_0})=v_{n-1}$
if ${_g\delta}_0\in F_1$.

\ \ Case 2:  $x\in\mathrm{V}(\X{m}\setminus\bigcup G^{\prime})$.

\ \ \ \ Case 2.a:  $\nt[\alpha](x)=\nt[s]_0$.

Then,
$\nPhi{m}{0}(x)=\nt[\beta](\nt[\alpha](x))=\nt[\beta](\nt[s]_0)=\hat{\beta}(\hat{s}_0)=\hat{\beta}(\hat{\alpha}(x))$.

\ \ \ \ Case 2.b:  $\nt[\alpha](x)=\nt[s]_k^{\delta}$ for some
$\delta\in F_0$.

Then,
$\nPhi{m}{0}(x)=\nt[\beta](\nt[\alpha](x))=\nt[\beta](\nt[s]_k^{\delta})=\hat{\beta}(\hat{s}_{k+2}^{\delta})=\hat{\beta}(\hat{\alpha}(x))$.

\ \ \ \ Case 2.c:  $\nt[\alpha](x)=\nt[s]_k^{\delta}$ for some
$\delta\in F_2$.

\ \ \ \ \ \ Case 2.c.i:  $k=2$.

Then,
$\nPhi{m}{0}(x)=\nt[\beta](\nt[\alpha](x))=\nt[\beta](\nt[s]_2^{\delta})=v_0=\nt[\beta](\nt[s]_0)=\hat{\beta}(\hat{s}_0)=
\hat{\beta}(\hat{\alpha}(x))$.

\ \ \ \ \ \ Case 2.c.ii:  $k\in\oneton[3]{\nt[k]_{\delta}}$.

Then,
$\nPhi{m}{0}(x)=\nt[\beta](\nt[\alpha](x))=\nt[\beta](\nt[s]_k^{\delta})=\hat{\beta}(\hat{s}_{k-2}^{\delta})=
\hat{\beta}(\hat{\alpha}(x))$.

\ \ \ \ Case 2.d:  $\nt[\alpha](x)=\nt[s]_k^{\delta}$ for some
$\delta\in\ F_1$.

Then,
$\nPhi{m}{0}(x)=\nt[\beta](\nt[\alpha](x))=\nt[\beta](\nt[s]_k^{\delta})=\hat{\beta}(\hat{s}_k^{\delta})=\hat{\beta}(\hat{\alpha}(x))$.
\end{proof}

\newpage
\begin{figure}[here] \hspace{2cm}
\setlength{\unitlength}{.35cm}
\begin{picture}(0,4.29)

\put(8.8,-6.1){\Large\ensuremath{{_{5}X}_{1}}}

%\color{LimeGreen} \put(6,-1.1){\large\ensuremath{{_5A}_{1}^{4}}}
%\color{Orchid} \put(13.3,-1.1){\large\ensuremath{{_5A}_{1}^{2}}}
%\color{RawSienna} \put(9.6,2.9){\large\ensuremath{{_5A}_{1}^{3}}}
%\color{Red} \put(6.2716,-5.5284){\large\ensuremath{{_5A}_{1}^{5}}}
%\color{ProcessBlue}
%\put(12.8284,-5.5284){\large\ensuremath{{_5A}_{1}^{1}}}
\normalcolor \put(14,-1.8){\line(-1,0){8}}
\put(10,-1.8){\line(0,1){4}} \color{Peach}
\put(10,-1.8){\line(-1,-1){2.8284}}
\put(10,-1.8){\line(1,-1){2.8284}} \normalcolor
\put(9,-1.3){\footnotesize\ensuremath{v_0}}
\put(10,-1.8){\circle*{.4}}
%\put(10.4,2.1){\small\ensuremath{v_3}}
\put(10,2.2){\circle*{.4}}
%\put(5.8,-2.4){\small\ensuremath{v_6}}
\put(6,-1.8){\circle*{.4}}
%\put(13.8,-2.4){\small\ensuremath{v_2}}
\put(14,-1.8){\circle*{.4}}
\put(6.5,-2.6){\footnotesize\ensuremath{v_4}}
\put(6.7,-1.8){\circle*{.4}}
%\put(7.5716,-4.7284){\small\ensuremath{v_7}}
\put(7.1716,-4.6284){\circle*{.4}}
%\put(8.0663,-4.2337){\small\ensuremath{v_5}}
\put(7.6663,-4.1337){\circle*{.4}}
%\put(12.1284,-4.7284){\small\ensuremath{v_1}}
\put(12.8284,-4.6284){\circle*{.4}}

%\put(9.4,1.45){\tiny\ensuremath{u_{27}}}
\put(10,1.5){\circle*{.2}}
%\put(8.75,-1.55){\tiny\ensuremath{u_{28}}}
\put(8.9,-1.8){\circle*{.2}}
%\put(7.65,-1.55){\tiny\ensuremath{u_{29}}}
\put(7.8,-1.8){\circle*{.2}}

%\put(10.2167,-1.65){\tiny\ensuremath{u_{2}}}
\put(10.3667,-1.8){\circle*{.2}}
%\put(10.5834,-1.45){\tiny\ensuremath{u_{4}}}
\put(10.7334,-1.8){\circle*{.2}}
%\put(10.9501,-1.65){\tiny\ensuremath{u_{6}}}
\put(11.1001,-1.8){\circle*{.2}}
%\put(11.3168,-1.45){\tiny\ensuremath{u_{8}}}
\put(11.4668,-1.8){\circle*{.2}}
%\put(11.6835,-1.65){\tiny\ensuremath{u_{10}}}
\put(11.8335,-1.8){\circle*{.2}}
%\put(12.0502,-1.45){\tiny\ensuremath{u_{12}}}
\put(12.2002,-1.8){\circle*{.2}}
%\put(12.4169,-1.65){\tiny\ensuremath{u_{14}}}
\put(12.5669,-1.8){\circle*{.2}}
%\put(12.7836,-1.45){\tiny\ensuremath{u_{16}}}
\put(12.9336,-1.8){\circle*{.2}}
%\put(13.15,-1.65){\tiny\ensuremath{u_{18}}}
\put(13.3,-1.8){\circle*{.2}}

%\put(9.1407,-2.1093){\tiny\ensuremath{u_{19}}}
\put(9.7407,-2.0593){\circle*{.2}}
%\put(8.8814,-2.3686){\tiny\ensuremath{u_{20}}}
\put(9.4814,-2.3186){\circle*{.2}}
%\put(8.6221,-2.6279){\tiny\ensuremath{u_{21}}}
\put(9.2221,-2.5779){\circle*{.2}}
%\put(8.3628,-2.8872){\tiny\ensuremath{u_{22}}}
\put(8.9628,-2.8372){\circle*{.2}}
%\put(8.1035,-3.1465){\tiny\ensuremath{u_{23}}}
\put(8.7035,-3.0965){\circle*{.2}}
%\put(7.8442,-3.4058){\tiny\ensuremath{u_{24}}}
\put(8.4442,-3.3558){\circle*{.2}}
%\put(7.5849,-3.6651){\tiny\ensuremath{u_{25}}}
\put(8.1849,-3.6151){\circle*{.2}}
%\put(7.3256,-3.8744){\tiny\ensuremath{u_{26}}}
\put(7.9256,-3.8744){\circle*{.2}}

%\put(10.4093,-2.1093){\tiny\ensuremath{u_{1}}}
\put(10.2593,-2.0593){\circle*{.2}}
%\put(10.6686,-2.3686){\tiny\ensuremath{u_{3}}}
\put(10.5186,-2.3186){\circle*{.2}}
%\put(10.9279,-2.6279){\tiny\ensuremath{u_{5}}}
\put(10.7779,-2.5779){\circle*{.2}}
%\put(11.1872,-2.8872){\tiny\ensuremath{u_{7}}}
\put(11.0372,-2.8372){\circle*{.2}}
%\put(11.4465,-3.1465){\tiny\ensuremath{u_{9}}}
\put(11.2965,-3.0965){\circle*{.2}}
%\put(11.7058,-3.4058){\tiny\ensuremath{u_{11}}}
\put(11.5558,-3.3558){\circle*{.2}}
%\put(11.9651,-3.6651){\tiny\ensuremath{u_{13}}}
\put(11.8151,-3.6151){\circle*{.2}}
%\put(12.2244,-3.9244){\tiny\ensuremath{u_{15}}}
\put(12.0744,-3.8744){\circle*{.2}}
%\put(12.4837,-4.1837){\tiny\ensuremath{u_{17}}}
\put(12.3337,-4.1337){\circle*{.2}}

\color{Black} \put(10.5186,-2.3186){\line(1,-1){2.3098}}
\put(9.4814,-2.3186){\line(-1,-1){2.3098}} \normalcolor

\color{Peach} \put(10,-1.8){\circle*{.4}}
\put(10.2593,-2.0593){\circle*{.2}}
\put(10.3667,-1.8){\circle*{.2}}
\put(9.7407,-2.0593){\circle*{.2}} \put(10,1.5){\circle*{.2}}
\put(8.9,-1.8){\circle*{.2}} \put(10.5186,-2.3186){\circle*{.2}}
\put(10.7334,-1.8){\circle*{.2}}
\put(9.4814,-2.3186){\circle*{.2}} \put(7.8,-1.8){\circle*{.2}}
\put(10,-1.8){\line(1,0){.7334}} \put(10,-1.8){\line(0,1){3.3}}
\put(10,-1.8){\line(-1,0){2.2}} \normalcolor

\put(-.7,-6.1){\Large\ensuremath{{_{5}X}_{0}}}

%\color{LimeGreen}
%\put(-3.5,-1.1){\large\ensuremath{{_5A}_{0}^{4}}} \color{Orchid}
%\put(3.8,-1.1){\large\ensuremath{{_5A}_{0}^{2}}} \color{RawSienna}
%\put(.1,2.9){\large\ensuremath{{_5A}_{0}^{3}}} \color{Red}
%\put(-3.2284,-5.5284){\large\ensuremath{{_5A}_{0}^{5}}}
%\color{ProcessBlue}
%\put(3.3284,-5.5284){\large\ensuremath{{_5A}_{0}^{1}}}
\normalcolor \put(4.5,-1.8){\line(-1,0){8}}
\put(.5,-1.8){\line(0,1){4}} \put(.5,-1.8){\line(-1,-1){2.8284}}
\put(.5,-1.8){\line(1,-1){2.8284}}
\put(-.5,-1.3){\footnotesize\ensuremath{v_0}}
\put(.5,-1.8){\circle*{.4}}
%\put(.9,2.1){\small\ensuremath{v_3}}
\put(.5,2.2){\circle*{.4}}
%\put(-3.7,-2.4){\small\ensuremath{v_6}}
\put(-3.5,-1.8){\circle*{.4}}
%\put(4.3,-2.4){\small\ensuremath{v_2}}
\put(4.5,-1.8){\circle*{.4}}
\put(-3,-2.6){\footnotesize\ensuremath{v_4}}
\put(-2.8,-1.8){\circle*{.4}}
%\put(-1.9284,-4.7284){\small\ensuremath{v_7}}
\put(-2.3284,-4.6284){\circle*{.4}}
%\put(-1.4337,-4.2337){\small\ensuremath{v_5}}
\put(-1.8337,-4.1337){\circle*{.4}}
%\put(2.6284,-4.7284){\small\ensuremath{v_1}}
\put(3.3284,-4.6284){\circle*{.4}}

\color{Peach} \put(.5,-1.8){\circle*{.4}}
\put(.5,-1.8){\line(-1,0){3.3}} \put(-2.8,-1.8){\circle*{.4}}
\normalcolor

\color{Peach} \thicklines \put(7.25,.2){\vector(-1,0){4}}
\normalcolor \thinlines \put(4.95,.8){\Large
\ensuremath{{_5\phi}}}

\end{picture}

\vspace{-.75cm} \setlength{\unitlength}{.35cm} \hspace{11.45cm}
\begin{picture}(0,0)

\put(-.7,-6.1){\Large\ensuremath{{_{5}X}_{m}}}

%\color{LimeGreen}
%\put(-3.5,-1.1){\large\ensuremath{{_5A}_{0}^{4}}} \color{Orchid}
%\put(3.8,-1.1){\large\ensuremath{{_5A}_{0}^{2}}} \color{RawSienna}
%\put(.1,2.9){\large\ensuremath{{_5A}_{0}^{3}}} \color{Red}
%\put(-3.2284,-5.5284){\large\ensuremath{{_5A}_{0}^{5}}}
%\color{ProcessBlue}
%\put(3.3284,-5.5284){\large\ensuremath{{_5A}_{0}^{1}}}
\normalcolor \put(4.5,-1.8){\line(-1,0){8}}
\put(.5,-1.8){\line(0,1){4}} \color{Peach}
\put(.5,-1.8){\line(-1,-1){2.8284}}
\put(.5,-1.8){\line(1,-1){2.8284}} \normalcolor
\put(-.5,-1.3){\footnotesize\ensuremath{v_0}}
\put(.5,-1.8){\circle*{.4}}
%\put(.9,2.1){\small\ensuremath{v_3}} \put(.5,2.2){\circle*{.4}}
%\put(-3.7,-2.4){\small\ensuremath{v_6}}
%\put(-3.5,-1.8){\circle*{.4}}
%\put(4.3,-2.4){\small\ensuremath{v_2}}
%\put(4.5,-1.8){\circle*{.4}} \put(-3,-2.4){\small\ensuremath{v_4}}
%\put(-2.8,-1.8){\circle*{.4}}
%\put(-1.9284,-4.7284){\small\ensuremath{v_7}}
%\put(-2.3284,-4.6284){\circle*{.4}}
%\put(-1.4337,-4.2337){\small\ensuremath{v_5}}
%\put(-1.8337,-4.1337){\circle*{.4}}
%\put(2.6284,-4.7284){\small\ensuremath{v_1}}
%\put(3.3284,-4.6284){\circle*{.4}}

\color{Black} \put(1.5186,-2.8186){\line(1,-1){1.8098}}
\put(-0.5186,-2.8186){\line(-1,-1){1.8098}} \normalcolor

\color{Peach} %\put(8.75,-1.55){\tiny\ensuremath{u_{28}}}
\put(-.6,-1.8){\circle*{.2}} \put(.5,-.7){\circle*{.2}}
\put(1.6,-1.8){\circle*{.2}} \put(-.3,-2.6){\circle*{.2}}
\put(1.3,-2.6){\circle*{.2}} \put(.5,-1.8){\line(-1,0){1.4}}
\put(.5,-1.8){\line(1,0){1.4}} \put(.5,-1.8){\line(0,1){1.4}}
\put(.5,-1.8){\line(-1,-1){1}} \put(.5,-1.8){\line(1,-1){1}}
\put(.5,-1.8){\circle*{.4}} \normalcolor

\color{Peach} \thicklines \put(-10.05,.2){\vector(-1,0){4}}
\put(-2.05,.2){\vector(-1,0){4}} \put(7.05,.2){\vector(-1,0){4}}
\normalcolor \put(-9,0){\ensuremath{\cdots}}
\put(-8,0){\ensuremath{\cdots}} \put(7.8,0){\ensuremath{\cdots}}

\color{Peach} \put(-4.05,2){\vector(-1,0){8}} \normalcolor
\thinlines \put(-8.9,2.8){\Large \ensuremath{\nPhi[5]{m}{1}}}

\end{picture}

\vspace{3cm} \setlength{\unitlength}{.35cm} \hspace{7.15cm}
\begin{picture}(0,0)

\put(0,-7.6){\Large\ensuremath{T}}

%\color{LimeGreen}
%\put(-3.5,-1.1){\large\ensuremath{{_4A}_{0}^{3}}} \color{Orchid}
%\put(3.8,-1.1){\large\ensuremath{{_4A}_{0}^{1}}} \color{RawSienna}
%\put(.1,2.9){\large\ensuremath{{_4A}_{0}^{2}}} \color{Red}
%\put(.1,-6.8){\large\ensuremath{{_4A}_{0}^{4}}} \normalcolor
\put(4.5,-1.8){\line(-1,0){8}} \put(.5,-1.8){\line(0,1){4}}
\put(.5,-1.8){\line(0,-1){4}} %\put(0,-1.4){\small\ensuremath{v_0}}
\put(.5,-1.8){\circle*{.4}} %\put(.9,2.1){\small\ensuremath{v_2}}
%\put(.5,2.2){\circle*{.4}} %\put(-3.7,-2.4){\small\ensuremath{v_5}}
%\put(-3.5,-1.8){\circle*{.4}}
%\put(4.3,-2.4){\small\ensuremath{v_1}}
%\put(4.5,-1.8){\circle*{.4}} %\put(-3,-2.4){\small\ensuremath{v_3}}
%\put(-2.8,-1.8){\circle*{.4}}
%\put(.9,-5.2){\small\ensuremath{v_4}}
%\put(.5,-5.1){\circle*{.4}}
%\put(.9,-5.9){\small\ensuremath{v_6}}
%\put(.5,-5.8){\circle*{.4}}

\color{Peach} \put(.5,-1.8){\circle*{.4}} \normalcolor

\thicklines \put(9,3.5){\vector(-1,-1){3}} \color{Peach}
\put(-5,.5){\vector(-1,1){3}} \normalcolor \thinlines
\put(7.9,1.25){\Large \ensuremath{\alpha}} \put(-8.35,1){\Large
\ensuremath{\beta}}

\end{picture}
\end{figure}

\vspace{6cm} \begin{figure}[here] \hspace{2cm}
\setlength{\unitlength}{.35cm}
\begin{picture}(0,0)

\put(8.8,-6.1){\Large\ensuremath{{_{5}X}_{1}}}

%\color{LimeGreen} \put(6,-1.1){\large\ensuremath{{_5A}_{1}^{4}}}
%\color{Orchid} \put(13.3,-1.1){\large\ensuremath{{_5A}_{1}^{2}}}
%\color{RawSienna} \put(9.6,2.9){\large\ensuremath{{_5A}_{1}^{3}}}
%\color{Red} \put(6.2716,-5.5284){\large\ensuremath{{_5A}_{1}^{5}}}
%\color{ProcessBlue}
%\put(12.8284,-5.5284){\large\ensuremath{{_5A}_{1}^{1}}}
\normalcolor \put(14,-1.8){\line(-1,0){8}}
\put(10,-1.8){\line(0,1){4}} \color{Peach}
\put(10,-1.8){\line(-1,-1){2.8284}}
\put(10,-1.8){\line(1,-1){2.8284}} \normalcolor
\put(9,-1.3){\footnotesize\ensuremath{v_0}}
\put(10,-1.8){\circle*{.4}}
%\put(10.4,2.1){\small\ensuremath{v_3}}
\put(10,2.2){\circle*{.4}}
%\put(5.8,-2.4){\small\ensuremath{v_6}}
\put(6,-1.8){\circle*{.4}}
%\put(13.8,-2.4){\small\ensuremath{v_2}}
\put(14,-1.8){\circle*{.4}}
\put(6.5,-2.6){\footnotesize\ensuremath{v_4}}
\put(6.7,-1.8){\circle*{.4}}
%\put(7.5716,-4.7284){\small\ensuremath{v_7}}
\put(7.1716,-4.6284){\circle*{.4}}
%\put(8.0663,-4.2337){\small\ensuremath{v_5}}
\put(7.6663,-4.1337){\circle*{.4}}
%\put(12.1284,-4.7284){\small\ensuremath{v_1}}
\put(12.8284,-4.6284){\circle*{.4}}

%\put(9.4,1.45){\tiny\ensuremath{u_{27}}}
\put(10,1.5){\circle*{.2}}
%\put(8.75,-1.55){\tiny\ensuremath{u_{28}}}
\put(8.9,-1.8){\circle*{.2}}
%\put(7.65,-1.55){\tiny\ensuremath{u_{29}}}
\put(7.8,-1.8){\circle*{.2}}

%\put(10.2167,-1.65){\tiny\ensuremath{u_{2}}}
\put(10.3667,-1.8){\circle*{.2}}
%\put(10.5834,-1.45){\tiny\ensuremath{u_{4}}}
\put(10.7334,-1.8){\circle*{.2}}
%\put(10.9501,-1.65){\tiny\ensuremath{u_{6}}}
\put(11.1001,-1.8){\circle*{.2}}
%\put(11.3168,-1.45){\tiny\ensuremath{u_{8}}}
\put(11.4668,-1.8){\circle*{.2}}
%\put(11.6835,-1.65){\tiny\ensuremath{u_{10}}}
\put(11.8335,-1.8){\circle*{.2}}
%\put(12.0502,-1.45){\tiny\ensuremath{u_{12}}}
\put(12.2002,-1.8){\circle*{.2}}
%\put(12.4169,-1.65){\tiny\ensuremath{u_{14}}}
\put(12.5669,-1.8){\circle*{.2}}
%\put(12.7836,-1.45){\tiny\ensuremath{u_{16}}}
\put(12.9336,-1.8){\circle*{.2}}
%\put(13.15,-1.65){\tiny\ensuremath{u_{18}}}
\put(13.3,-1.8){\circle*{.2}}

%\put(9.1407,-2.1093){\tiny\ensuremath{u_{19}}}
\put(9.7407,-2.0593){\circle*{.2}}
%\put(8.8814,-2.3686){\tiny\ensuremath{u_{20}}}
\put(9.4814,-2.3186){\circle*{.2}}
%\put(8.6221,-2.6279){\tiny\ensuremath{u_{21}}}
\put(9.2221,-2.5779){\circle*{.2}}
%\put(8.3628,-2.8872){\tiny\ensuremath{u_{22}}}
\put(8.9628,-2.8372){\circle*{.2}}
%\put(8.1035,-3.1465){\tiny\ensuremath{u_{23}}}
\put(8.7035,-3.0965){\circle*{.2}}
%\put(7.8442,-3.4058){\tiny\ensuremath{u_{24}}}
\put(8.4442,-3.3558){\circle*{.2}}
%\put(7.5849,-3.6651){\tiny\ensuremath{u_{25}}}
\put(8.1849,-3.6151){\circle*{.2}}
%\put(7.3256,-3.8744){\tiny\ensuremath{u_{26}}}
\put(7.9256,-3.8744){\circle*{.2}}

%\put(10.4093,-2.1093){\tiny\ensuremath{u_{1}}}
\put(10.2593,-2.0593){\circle*{.2}}
%\put(10.6686,-2.3686){\tiny\ensuremath{u_{3}}}
\put(10.5186,-2.3186){\circle*{.2}}
%\put(10.9279,-2.6279){\tiny\ensuremath{u_{5}}}
\put(10.7779,-2.5779){\circle*{.2}}
%\put(11.1872,-2.8872){\tiny\ensuremath{u_{7}}}
\put(11.0372,-2.8372){\circle*{.2}}
%\put(11.4465,-3.1465){\tiny\ensuremath{u_{9}}}
\put(11.2965,-3.0965){\circle*{.2}}
%\put(11.7058,-3.4058){\tiny\ensuremath{u_{11}}}
\put(11.5558,-3.3558){\circle*{.2}}
%\put(11.9651,-3.6651){\tiny\ensuremath{u_{13}}}
\put(11.8151,-3.6151){\circle*{.2}}
%\put(12.2244,-3.9244){\tiny\ensuremath{u_{15}}}
\put(12.0744,-3.8744){\circle*{.2}}
%\put(12.4837,-4.1837){\tiny\ensuremath{u_{17}}}
\put(12.3337,-4.1337){\circle*{.2}}

\color{Black} \put(10.5186,-2.3186){\line(1,-1){2.3098}}
\put(9.4814,-2.3186){\line(-1,-1){2.3098}} \normalcolor

\color{Peach} \put(10,-1.8){\circle*{.4}}
\put(10.2593,-2.0593){\circle*{.2}}
\put(10.3667,-1.8){\circle*{.2}}
\put(9.7407,-2.0593){\circle*{.2}} \put(10,1.5){\circle*{.2}}
\put(8.9,-1.8){\circle*{.2}} \put(10.5186,-2.3186){\circle*{.2}}
\put(10.7334,-1.8){\circle*{.2}}
\put(9.4814,-2.3186){\circle*{.2}} \put(7.8,-1.8){\circle*{.2}}
\put(10,-1.8){\line(1,0){.7334}} \put(10,-1.8){\line(0,1){3.3}}
\put(10,-1.8){\line(-1,0){2.2}} \normalcolor

\put(-.7,-6.1){\Large\ensuremath{{_{5}X}_{0}}}

%\color{LimeGreen}
%\put(-3.5,-1.1){\large\ensuremath{{_5A}_{0}^{4}}} \color{Orchid}
%\put(3.8,-1.1){\large\ensuremath{{_5A}_{0}^{2}}} \color{RawSienna}
%\put(.1,2.9){\large\ensuremath{{_5A}_{0}^{3}}} \color{Red}
%\put(-3.2284,-5.5284){\large\ensuremath{{_5A}_{0}^{5}}}
%\color{ProcessBlue}
%\put(3.3284,-5.5284){\large\ensuremath{{_5A}_{0}^{1}}}
\normalcolor \put(4.5,-1.8){\line(-1,0){8}}
\put(.5,-1.8){\line(0,1){4}} \put(.5,-1.8){\line(-1,-1){2.8284}}
\put(.5,-1.8){\line(1,-1){2.8284}}
\put(-.5,-1.3){\footnotesize\ensuremath{v_0}}
\put(.5,-1.8){\circle*{.4}}
%\put(.9,2.1){\small\ensuremath{v_3}}
\put(.5,2.2){\circle*{.4}}
%\put(-3.7,-2.4){\small\ensuremath{v_6}}
\put(-3.5,-1.8){\circle*{.4}}
%\put(4.3,-2.4){\small\ensuremath{v_2}}
\put(4.5,-1.8){\circle*{.4}}
\put(-3,-2.6){\footnotesize\ensuremath{v_4}}
\put(-2.8,-1.8){\circle*{.4}}
%\put(-1.9284,-4.7284){\small\ensuremath{v_7}}
\put(-2.3284,-4.6284){\circle*{.4}}
%\put(-1.4337,-4.2337){\small\ensuremath{v_5}}
\put(-1.8337,-4.1337){\circle*{.4}}
%\put(2.6284,-4.7284){\small\ensuremath{v_1}}
\put(3.3284,-4.6284){\circle*{.4}}

\color{Peach} \put(.5,-1.8){\circle*{.4}}
\put(.5,-1.8){\line(-1,0){3.3}} \put(-2.8,-1.8){\circle*{.4}}
\normalcolor

\color{Peach} \thicklines \put(7.25,.2){\vector(-1,0){4}}
\normalcolor \thinlines \put(4.95,.8){\Large
\ensuremath{{_5\phi}}}

\end{picture}

\vspace{-.75cm} \setlength{\unitlength}{.35cm} \hspace{11.45cm}
\begin{picture}(0,0)

\put(-.7,-6.1){\Large\ensuremath{{_{5}X}_{m}}}

%\color{LimeGreen}
%\put(-3.5,-1.1){\large\ensuremath{{_5A}_{0}^{4}}} \color{Orchid}
%\put(3.8,-1.1){\large\ensuremath{{_5A}_{0}^{2}}} \color{RawSienna}
%\put(.1,2.9){\large\ensuremath{{_5A}_{0}^{3}}} \color{Red}
%\put(-3.2284,-5.5284){\large\ensuremath{{_5A}_{0}^{5}}}
%\color{ProcessBlue}
%\put(3.3284,-5.5284){\large\ensuremath{{_5A}_{0}^{1}}}
\normalcolor \put(4.5,-1.8){\line(-1,0){8}}
\put(.5,-1.8){\line(0,1){4}} \color{Peach}
\put(.5,-1.8){\line(-1,-1){2.8284}}
\put(.5,-1.8){\line(1,-1){2.8284}} \normalcolor
\put(-.5,-1.3){\footnotesize\ensuremath{v_0}}
\put(.5,-1.8){\circle*{.4}}
%\put(.9,2.1){\small\ensuremath{v_3}} \put(.5,2.2){\circle*{.4}}
%\put(-3.7,-2.4){\small\ensuremath{v_6}}
%\put(-3.5,-1.8){\circle*{.4}}
%\put(4.3,-2.4){\small\ensuremath{v_2}}
%\put(4.5,-1.8){\circle*{.4}} \put(-3,-2.4){\small\ensuremath{v_4}}
%\put(-2.8,-1.8){\circle*{.4}}
%\put(-1.9284,-4.7284){\small\ensuremath{v_7}}
%\put(-2.3284,-4.6284){\circle*{.4}}
%\put(-1.4337,-4.2337){\small\ensuremath{v_5}}
%\put(-1.8337,-4.1337){\circle*{.4}}
%\put(2.6284,-4.7284){\small\ensuremath{v_1}}
%\put(3.3284,-4.6284){\circle*{.4}}

\color{Black} \put(1.5186,-2.8186){\line(1,-1){1.8098}}
\put(-0.5186,-2.8186){\line(-1,-1){1.8098}} \normalcolor

\color{Peach} %\put(8.75,-1.55){\tiny\ensuremath{u_{28}}}
\put(-.6,-1.8){\circle*{.2}} \put(.5,-.7){\circle*{.2}}
\put(1.6,-1.8){\circle*{.2}} \put(-.3,-2.6){\circle*{.2}}
\put(1.3,-2.6){\circle*{.2}} \put(.5,-1.8){\line(-1,0){1.4}}
\put(.5,-1.8){\line(1,0){1.4}} \put(.5,-1.8){\line(0,1){1.4}}
\put(.5,-1.8){\line(-1,-1){1}} \put(.5,-1.8){\line(1,-1){1}}
\put(.5,-1.8){\circle*{.4}} \normalcolor

\color{Peach} \thicklines \put(-10.05,.2){\vector(-1,0){4}}
\put(-2.05,.2){\vector(-1,0){4}} \put(7.05,.2){\vector(-1,0){4}}
\normalcolor \put(-9,0){\ensuremath{\cdots}}
\put(-8,0){\ensuremath{\cdots}} \put(7.8,0){\ensuremath{\cdots}}

\color{Peach} \put(-4.05,2){\vector(-1,0){8}} \normalcolor
\thinlines \put(-8.9,2.8){\Large \ensuremath{\nPhi[5]{m}{1}}}

\end{picture}

\vspace{3cm} \setlength{\unitlength}{.35cm} \hspace{7.15cm}
\begin{picture}(0,0)

\put(0,-7.6){\Large\ensuremath{T}}

%\color{LimeGreen}
%\put(-3.5,-1.1){\large\ensuremath{{_4A}_{0}^{3}}} \color{Orchid}
%\put(3.8,-1.1){\large\ensuremath{{_4A}_{0}^{1}}} \color{RawSienna}
%\put(.1,2.9){\large\ensuremath{{_4A}_{0}^{2}}} \color{Red}
%\put(.1,-6.8){\large\ensuremath{{_4A}_{0}^{4}}} \normalcolor
\put(4.5,-1.8){\line(-1,0){8}} \put(.5,-1.8){\line(0,1){4}}
\put(.5,-1.8){\line(0,-1){4}} %\put(0,-1.4){\small\ensuremath{v_0}}
\put(.5,-1.8){\circle*{.4}} %\put(.9,2.1){\small\ensuremath{v_2}}
%\put(.5,2.2){\circle*{.4}} %\put(-3.7,-2.4){\small\ensuremath{v_5}}
%\put(-3.5,-1.8){\circle*{.4}}
%\put(4.3,-2.4){\small\ensuremath{v_1}}
%\put(4.5,-1.8){\circle*{.4}} %\put(-3,-2.4){\small\ensuremath{v_3}}
%\put(-2.8,-1.8){\circle*{.4}}
%\put(.9,-5.2){\small\ensuremath{v_4}}
%\put(.5,-5.1){\circle*{.4}}
%\put(.9,-5.9){\small\ensuremath{v_6}}
%\put(.5,-5.8){\circle*{.4}}

\color{Peach} \put(.5,-1.8){\circle*{.4}}
\put(.5,-1.8){\line(-1,0){1.4}} \put(-.6,-1.8){\circle*{.2}}
\put(.5,-1.8){\line(0,1){1.4}} \put(.5,-.7){\circle*{.2}}
\put(.5,-1.8){\line(1,0){1.4}} \put(1.6,-1.8){\circle*{.2}}
\put(.5,-1.8){\line(0,-1){1.4}} \put(.5,-2.9){\circle*{.2}}
\normalcolor

\thicklines \put(9,3.5){\vector(-1,-1){3}} \color{Peach}
\put(-5,.5){\vector(-1,1){3}} \normalcolor \thinlines
\put(7.9,1.25){\Large \ensuremath{\alpha}} \put(-8.35,1){\Large
\ensuremath{\beta}}

\end{picture}
\end{figure}
\vspace{2.5cm}
\begin{figure}[here]
\caption[Factoring of ${_{5}\Phi}_{0}^{m}$]{(top) Factoring of
\nPhi[5]{m}{0} given by Theorem \ref{claim17}. (bottom) Factoring
given by Proposition \ref{claim14}.}
\end{figure}

\newpage

For Propositions \ref{claim15} and \ref{claim16} let
$\varprojlim\invseq{Y}{\Gamma}$ be generated by
$d[\nphi]\circ{_n\lambda},\nphi,\nphi,\ldots$ (Definition
\ref{def5.10}).

\begin{proposition} \label{claim15} Let $m\in\{1,2,\ldots\}$,
$\nt\in\oneton{n-1}$, $\nt[T]$ be a simple-$\nt$-od, and
$\nt[\alpha]:{_nY}_m\longrightarrow\nt[T]$ and
$\nt[\beta]:\nt[T]\longrightarrow {_nY}_1$ be simplicial maps so
that $\nt[\beta]\circ\nt[\alpha]={_n\Gamma}_1^m$ and
$\nt[\beta](\nt[s]_0)=v_0$ where $\nt[s]_0$ is the branch point of
$\nt[T]$.  Then, there exist a simple-$\nt$-od $T$ and simplicial
maps $\alpha:\X{m-1}\longrightarrow T$ and $\beta:T\longrightarrow
\X{0}$ so that $\beta\circ\alpha=\nPhi{m-1}{0}$ and
$\beta(s_0)=v_0$ where $s_0$ is the branch point of $T$.
\end{proposition}

\begin{proof}
Let $\nt[S]^1,\nt[S]^2,\ldots,\nt[S]^{\nt}$ be $\nt$ arcs so that
$\bigcup\limits_{i=1}^{\nt}\nt[S]^i=\nt[T]$ and
$\nt[S]^i\bigcap\nt[S]^j=\{\nt[s]_0\}$ for each distinct
$i,j\in\oneton{\nt}$, and so $\nt[s]_0$ is an endpoint of
$\nt[S]^i$ for each $i\in\oneton{\nt}$.

\underline{\scriptsize \bf CLAIM}:  $\nt[\alpha]^{-1}(\nt[S]^i)$
is a subdivision of a subgraph $\X{m-1}^i$ of $\X{m-1}$ for each
$i\in\oneton{\nt}$.

Let $i\in\oneton{\nt}$ and $x,y\in\mathrm{V}(\X{m-1})$ so that $x$
and $y$ are distinct and adjacent and
$(\edge{x}{}{y}{},{_nY}_m)\bigcap\nt[\alpha]^{-1}(\nt[S]^i)\not=\emptyset$.
Since $\nt[\beta](\nt[s]_0)=v_0$ and ${_n\Gamma}_1^m$ is a
subdivision of $\nPhi{m-1}{0}$ matching ${_nY}_1$, then
$\nt[\alpha]^{-1}(\nt[s]_0)\subseteq\mathrm{V}(\X{m-1})$.  Suppose
$\nt[\alpha](x)\in\nt[S]^i$.  If $\nt[\alpha](y)\notin\nt[S]^i$,
then
$(\edge{x}{}{y}{},{_nY}_m)\bigcap\nt[\alpha]^{-1}(\nt[S]^i)=\{x\}$,
and if $\nt[\alpha](y)\in\nt[S]^i$, then
$(\edge{x}{}{y}{},{_nY}_m)\subseteq\nt[\alpha]^{-1}(\nt[S]^i)$.
Suppose $\nt[\alpha](x)\notin\nt[S]^i$.  If
$\nt[\alpha](y)\notin\nt[S]^i$, then
$(\edge{x}{}{y}{},{_nY}_m)\bigcap\nt[\alpha]^{-1}(\nt[S]^i)=\emptyset$,
and if $\nt[\alpha](y)\in\nt[S]^i$, then
$(\edge{x}{}{y}{},{_nY}_m)\bigcap\nt[\alpha]^{-1}(\nt[S]^i)=\{y\}$.

Thus, the claim holds.

By Proposition 5.13 in \cite{minc1}, for each $i\in\oneton{\nt}$,
there exist an arc $S^i$, simplicial maps
$\beta^i:S^i\longrightarrow\nPhi{m-1}{0}(\X{m-1}^i)$ and
$\alpha^i:\X{m-1}^i\longrightarrow S^i$ so that
$\beta^i\circ\alpha^i=\nPhi{m-1}{0}|_{\X{m-1}^i}$, and an endpoint
$s_0^i$ of $S^i$ so that $\beta^i(s_0^i)=\nt[\beta](\nt[s]_0)$ and
${\alpha^i}^{-1}(s_0^i)={\nt[\alpha]}^{-1}(\nt[s]_0)$.

Let $T=\bigcup\limits_{i=1}^{\nt}S^i/\sim$, where $\sim$
identifies $s_0^i$ to $s_0$ for all $i\in\oneton{\nt}$, and
$q:\bigcup\limits_{i=1}^{\nt}S^i\longrightarrow\bigcup\limits_{i=1}^{\nt}S^i/\sim$
be the quotient map, and define $\beta:T\longrightarrow\X{0}$ as
$\beta(y)=\beta^i(x)$ if $y=q(x)$ for some
$x\in\mathrm{V}(S^i\setminus\{s_0^i\})$ and $i\in\oneton{\nt}$ and
$\beta(s_0)=\nt[\beta](\nt[s]_0)$, and
$\alpha:\X{m-1}\longrightarrow T$ as $\alpha(x)=q(\alpha^i(x))$ if
$x\in\mathrm{V}(\X{m-1}^i)$.

Since $\beta^i$ is simplicial and
$\beta^i(s_0^i)=\nt[\beta](\nt[s]_0)$ for all $i\in\oneton{\nt}$,
then $\beta$ is simplicial.

\underline{\scriptsize \bf CLAIM}:  $\alpha$ is well-defined.

Let $x\in\mathrm{V}(\X{m-1}^i\bigcap\X{m-1}^j)$ for distinct
$i,j\in\oneton{\nt}$.  Then, $x\in\nt[\alpha]^{-1}(\nt[s]_0)$, and
so, $q(\alpha^i(x))=q(s_0^i)=s_0=q(s_0^j)=q(\alpha^j(x))$.

\underline{\scriptsize \bf CLAIM}:  $\alpha$ is simplicial.

Let $x,y\in\mathrm{V}(\X{m-1})$ be distinct and adjacent.

By the first claim, $x,y\in\X{m-1}^i$ for some $i\in\oneton{\nt}$.
Since $q$ is simplicial, $\alpha^i$ is simplicial, and
$\nPhi{m-1}{0}$ is light, then $\alpha(x)=q(\alpha^i(x))$ and
$\alpha(y)=q(\alpha^i(y))$ are adjacent.

\underline{\scriptsize \bf CLAIM}:
$\beta\circ\alpha=\nPhi{m-1}{0}$.

Let $x\in\mathrm{V}(\X{m-1})$.  Then, $x\in\mathrm{V}(\X{m-1}^i)$
for some $i\in\oneton{\nt}$.  If $x\in\nt[\alpha]^{-1}(\nt[s]_0)$,
then
$\nPhi{m-1}{0}(x)=\beta^i(\alpha^i(x))=\beta^i(s_0^i)=\nt[\beta](\nt[s]_0)=\beta(s_0)=\beta(q(s_0^i))=\beta(q(\alpha^i(x)))=\beta(\alpha(x))$,
and if $x\notin\nt[\alpha]^{-1}(\nt[s]_0)$, then
$\alpha^i(x)\in\mathrm{V}(S^i\setminus\{s_0^i\})$ and
$\nPhi{m-1}{0}(x)=\beta^i(\alpha^i(x))=\beta(q(\alpha^i(x)))=\beta(\alpha(x))$.

\underline{\scriptsize \bf CLAIM}:  $\beta(s_0)=v_0$.

By definition, $\beta(s_0)=\nt[\beta](\nt[s]_0)$, and by
hypothesis, $\nt[\beta](\nt[s]_0)=v_0$.
\end{proof}

\begin{proposition} \label{claim16}
There do not exist a simple-$\nt$-od $T$ and simplicial maps
$\alpha:\X{m}\longrightarrow T$ and $\beta:T\longrightarrow\X{0}$
so that $\beta\circ\alpha=\nPhi{m}{0}$ and $\beta(s_0)=v_0$, where
$s_0$ is the branch point of $T$, for each $m\in\{1,2,\ldots\}$
and for each $\nt\in\oneton{n-1}$.
\end{proposition}

\begin{proof}
Let $\nt\in\oneton{n-1}$ and $\arm{1}=\xseq{i}{0}{1}{\ell_i}$ for
some $\ell_i\in\{1,2,\ldots\}$ for each $i\in\oneton{n}$ where
$x_0^i=v_0$ for all $i\in\oneton{n}$.

Let $m=1$ and suppose there exist a simple-$\nt$-od $T$ and
simplicial maps $\alpha:\X{1}\longrightarrow T$ and
$\beta:T\longrightarrow\X{0}$ so that
$\beta\circ\alpha=\nPhi{1}{0}$ and $\beta(s_0)=v_0$ where $s_0$ is
the branch point of $T$.

Since $\overline{\nPhi{1}{0}}(\arm{1})=\langle
v_0,v_{n-1},v_0,v_{\impj[i-1]{1}},v_0,\underline{v_{\impj[i-1]{2}},v_0},\ldots,\underline{v_{\impj[i-1]{n-2}},v_0},v_n,v_{n+2}
\rangle$ for each $i\in\oneton{n-3}$,
$\overline{\nPhi{1}{0}}(\arm[n-2]{1})=\langle v_0,v_{n-1},v_{n+1}
\rangle$, $\overline{\nPhi{1}{0}}(\arm[n-1]{1})=\langle
v_0,v_{n-1},\linebreak[0]v_0,v_n,v_{n+2} \rangle$, and
$\overline{\nPhi{1}{0}}(\arm[n]{1})=\langle
v_0,v_{n-1},v_0,v_{n-2},v_0,v_1,v_0,\underline{v_2,v_0},\ldots,\underline{v_{n-3},v_0},v_n,v_{n+2}
\rangle$, then $\nPhi{1}{0}(x_{t_i}^i)=\nPhi{1}{0}(x_t^j)$ for
some $t<t_j$ where
$x_{t_i}^i,x_{t_j}^j\in\{x_2^{n-2},x_3^k:k\in\oneton{n}\setminus\{n-2\}\}$
are distinct, giving a contradiction.

Let $m>1$, and suppose the claim holds for all
$\nt[m]\in\oneton{m-1}$.  Suppose there exist a simple-$\nt$-od
$T$ and simplicial maps $\alpha:\X{m}\longrightarrow T$ and
$\beta:T\longrightarrow\X{0}$ so that
$\beta\circ\alpha=\nPhi{m}{0}$ and $\beta(s_0)=v_0$ where $s_0$ is
the branch point of $T$.

Since $\X{m}$ is connected, $\alpha$ may be assumed to be
surjective.  By Proposition \ref{claim14}, there exist
$\nt_1,\hat{T},\hat{\alpha}$, and $\hat{\beta}$ as defined in
Proposition \ref{claim14}.  Since
$\hat{\beta}(\bigcup\limits_{i=1}^{\nt_1}\edge{\hat{s}}{0}{\hat{s}^i}{1})=\edge{v}{0}{v}{n-1}$,
then
$\bigcup\limits_{i=1}^{\nt_1}\edge{\hat{s}}{0}{\hat{s}^i}{1}\subseteq
c_0^*$ and $d[\hat{\beta}](c_0)=a_{n-1}$ where $c_0$ is a vertex
of $D(\hat{\beta},\hat{T})$.  So, if
$c_i,c_j\in\mathrm{V}(D(\hat{\beta},\hat{T}))$ are adjacent, then
$c_i^*\bigcup c_j^*\subseteq\hat{S}^{\delta}$ for some
$\delta\in\oneton{\nt}$, or $c_k=c_0$ for some $k\in\{i,j\}$ and
$c_{\ell}^*\subseteq\hat{S}^{\delta}$ for some
$\delta\in\oneton{\nt}$ where $\ell\in\{i,j\}$ and $\ell\not=k$,
giving $D(\hat{\beta},\hat{T})$ is a simple-$\nt$-od with branch
point $c_0$.

Since $\hat{\alpha}$ is surjective by definition,
$d[\hat{\beta},\hat{\alpha}]$ is surjective.  By Proposition
\ref{claim4}, Proposition 2.6 in \cite{minc1}, and Theorem 4.3 in
\cite{minc1}, there exists a simplicial map
$\sigma:D(\hat{\beta},\hat{T})\longrightarrow
D(\nPhi{1}{0},\X{1})$ so that $\sigma\circ
d[\hat{\beta},\hat{\alpha}]=d[\nPhi{1}{0},\nPhi{m}{1}]$.  By the
proof of Proposition \ref{claim4}, $\dnphi^{-1}(a_{n-1})=\{b_0\}$,
and so, since $d[\hat{\beta}](c_0)=a_{n-1}$, $\sigma(c_0)=b_0$.

By Proposition \ref{claim3}, Proposition \ref{claim6}, and Theorem
5.11 in \cite{minc1}, $\varprojlim\dinvseq{X}{\Phi}$ is isomorphic
to $\varprojlim\invseq{Y}{\Gamma}$ (Definition \ref{def5.10.1}).
Thus, there exist a simple-$\nt$-od $\nt[T]$ and simplicial maps
$\nt[\alpha]:{_nY}_m\longrightarrow\nt[T]$ and
$\nt[\beta]:\nt[T]\longrightarrow{_nY}_1$ so that
$\nt[\beta]\circ\nt[\alpha]={_n\Gamma}_1^m$ and
$\nt[\beta](\nt[s]_0)=v_0$ where $\nt[s]_0$ is the branch point of
$\nt[T]$.

By Proposition \ref{claim15}, there exist a simple-$\nt$-od $T$
and simplicial maps $\alpha:\X{m-1}\linebreak[0]\longrightarrow T$
and $\beta:T\longrightarrow\X{0}$ so that
$\beta\circ\alpha=\nPhi{m-1}{0}$ and $\beta(s_0)=v_0$ where $s_0$
is the branch point of $T$, contradicting the induction
hypothesis.  Therefore, the claim follows by induction.
\end{proof}

\newpage
\begin{figure}[here]
\setlength{\unitlength}{.35cm}
\begin{picture}(0,5.29)

\thicklines \put(2.6,0){\line(1,0){1.4}}
\put(4.2,0){\line(1,0){2.8}} \put(7.2,0){\line(1,0){1.4}}

\put(2.6,0){\line(-1,0){3}}

\put(2.6,0){\line(1,1){3}}

\put(8.6,0){\line(-1,1){3}}

\put(2.6,-3){\line(-1,-1){2}}

\put(2.6,0){\line(0,-1){3}} \put(2.6,-3){\line(1,0){6}}

\put(8.6,0){\line(0,-1){3}}

\put(2.6,-3){\line(1,2){3}}

\put(8.6,-3){\line(-1,2){3}}

\put(2.6,0){\line(2,-1){1.1}} \put(3.9,-.65){\line(2,-1){4.7}}

\put(8.6,0){\line(-2,-1){1.1}} \put(7.3,-.65){\line(-2,-1){1.6}}
\put(5.5,-1.55){\line(-2,-1){2.9}}

\put(1.9,.6){\footnotesize\ensuremath{a_4}} \color{Peach}
\put(2.6,0){\circle*{.4}} \normalcolor

%\put(-.6,-.4){\small\ensuremath{a_6}} \color{LimeGreen}
\put(-.4,0){\circle*{.4}} \normalcolor

%\put(5.45,2.5){\small\ensuremath{a_3}} \color{RawSienna}
\put(5.6,3){\circle*{.4}} \normalcolor

%\put(8,.1){\small\ensuremath{a_2}} \color{Orchid}
\put(8.6,0){\circle*{.4}} \normalcolor

%\put(.4,-5.4){\small\ensuremath{a_7}} \color{Red}
\put(.6,-5){\circle*{.4}} \normalcolor

%\put(2.4,-3.4){\small\ensuremath{a_5}} \color{Red}
\put(2.6,-3){\circle*{.4}} \normalcolor

%\put(8.4,-3.35){\small\ensuremath{a_1}} \color{ProcessBlue}
\put(8.6,-3){\circle*{.4}} \normalcolor \thinlines

%\put(2.1,.5){\tiny\ensuremath{b_0}} \put(2.3,.3){\circle*{.2}}

\put(3.25,-6){\Large\ensuremath{D(\X[5]{0})}}
%\put(-.6,1){\LARGE\ensuremath{D(\nphi[5],\X[5]{1})}}

\end{picture}

\vspace{-1cm} \setlength{\unitlength}{.35cm} \hspace{2.1cm}
\begin{picture}(0,0)

\put(6.3,-6.6){\Large\ensuremath{D(\nphi[5],\X[5]{1})}}

%\color{LimeGreen} \put(6,-1.1){\large\ensuremath{{_5A}_{1}^{4}}}
%\color{Orchid} \put(13.3,-1.1){\large\ensuremath{{_5A}_{1}^{2}}}
%\color{RawSienna} \put(9.6,2.9){\large\ensuremath{{_5A}_{1}^{3}}}
%\color{Red} \put(6.2716,-5.5284){\large\ensuremath{{_5A}_{1}^{5}}}
%\color{ProcessBlue}
%\put(12.8284,-5.5284){\large\ensuremath{{_5A}_{1}^{1}}}
\normalcolor \put(14,-1.8){\line(-1,0){8}}
\put(10,-1.8){\line(0,1){4}} %\color{Peach}
\put(10,-1.8){\line(-1,-1){2.8284}}
\put(10,-1.8){\line(1,-1){2.8284}} %\normalcolor
\put(9,-1.3){\footnotesize\ensuremath{v_0}}
\put(10,-1.8){\circle*{.4}}
%\put(10.4,2.1){\small\ensuremath{v_3}}
\put(10,2.2){\circle*{.4}}
%\put(5.8,-2.4){\small\ensuremath{v_6}}
\put(6,-1.8){\circle*{.4}}
%\put(13.8,-2.4){\small\ensuremath{v_2}}
\put(14,-1.8){\circle*{.4}}
\put(6.5,-2.6){\footnotesize\ensuremath{v_4}}
\put(6.7,-1.8){\circle*{.4}}
%\put(7.5716,-4.7284){\small\ensuremath{v_7}}
\put(7.1716,-4.6284){\circle*{.4}}
%\put(8.0663,-4.2337){\small\ensuremath{v_5}}
\put(7.6663,-4.1337){\circle*{.4}}
%\put(12.1284,-4.7284){\small\ensuremath{v_1}}
\put(12.8284,-4.6284){\circle*{.4}}

%\put(9.4,1.45){\tiny\ensuremath{u_{27}}}
%\put(10,1.5){\circle*{.2}}
%\put(8.75,-1.55){\tiny\ensuremath{u_{28}}}
%\put(8.9,-1.8){\circle*{.2}}
%\put(7.65,-1.55){\tiny\ensuremath{u_{29}}}
%\put(7.8,-1.8){\circle*{.2}}

%\put(10.2167,-1.65){\tiny\ensuremath{u_{2}}}
%\put(10.3667,-1.8){\circle*{.2}}
%\put(10.5834,-1.45){\tiny\ensuremath{u_{4}}}
%\put(10.7334,-1.8){\circle*{.2}}
%\put(10.9501,-1.65){\tiny\ensuremath{u_{6}}}
\put(11.1001,-1.8){\circle*{.2}}
%\put(11.3168,-1.45){\tiny\ensuremath{u_{8}}}
%\put(11.4668,-1.8){\circle*{.2}}
%\put(11.6835,-1.65){\tiny\ensuremath{u_{10}}}
\put(11.8335,-1.8){\circle*{.2}}
%\put(12.0502,-1.45){\tiny\ensuremath{u_{12}}}
%\put(12.2002,-1.8){\circle*{.2}}
%\put(12.4169,-1.65){\tiny\ensuremath{u_{14}}}
\put(12.5669,-1.8){\circle*{.2}}
%\put(12.7836,-1.45){\tiny\ensuremath{u_{16}}}
%\put(12.9336,-1.8){\circle*{.2}}
%\put(13.15,-1.65){\tiny\ensuremath{u_{18}}}
\put(13.3,-1.8){\circle*{.2}}

%\put(9.1407,-2.1093){\tiny\ensuremath{u_{19}}}
%\put(9.7407,-2.0593){\circle*{.2}}
%\put(8.8814,-2.3686){\tiny\ensuremath{u_{20}}}
%\put(9.4814,-2.3186){\circle*{.2}}
%\put(8.6221,-2.6279){\tiny\ensuremath{u_{21}}}
\put(9.2221,-2.5779){\circle*{.2}}
%\put(8.3628,-2.8872){\tiny\ensuremath{u_{22}}}
%\put(8.9628,-2.8372){\circle*{.2}}
%\put(8.1035,-3.1465){\tiny\ensuremath{u_{23}}}
\put(8.7035,-3.0965){\circle*{.2}}
%\put(7.8442,-3.4058){\tiny\ensuremath{u_{24}}}
%\put(8.4442,-3.3558){\circle*{.2}}
%\put(7.5849,-3.6651){\tiny\ensuremath{u_{25}}}
\put(8.1849,-3.6151){\circle*{.2}}
%\put(7.3256,-3.8744){\tiny\ensuremath{u_{26}}}
%\put(7.9256,-3.8744){\circle*{.2}}

%\put(10.4093,-2.1093){\tiny\ensuremath{u_{1}}}
%\put(10.2593,-2.0593){\circle*{.2}}
%\put(10.6686,-2.3686){\tiny\ensuremath{u_{3}}}
%\put(10.5186,-2.3186){\circle*{.2}}
%\put(10.9279,-2.6279){\tiny\ensuremath{u_{5}}}
\put(10.7779,-2.5779){\circle*{.2}}
%\put(11.1872,-2.8872){\tiny\ensuremath{u_{7}}}
%\put(11.0372,-2.8372){\circle*{.2}}
%\put(11.4465,-3.1465){\tiny\ensuremath{u_{9}}}
\put(11.2965,-3.0965){\circle*{.2}}
%\put(11.7058,-3.4058){\tiny\ensuremath{u_{11}}}
%\put(11.5558,-3.3558){\circle*{.2}}
%\put(11.9651,-3.6651){\tiny\ensuremath{u_{13}}}
\put(11.8151,-3.6151){\circle*{.2}}
%\put(12.2244,-3.9244){\tiny\ensuremath{u_{15}}}
%\put(12.0744,-3.8744){\circle*{.2}}
%\put(12.4837,-4.1837){\tiny\ensuremath{u_{17}}}
\put(12.3337,-4.1337){\circle*{.2}}

\color{Black} \put(10.5186,-2.3186){\line(1,-1){2.3098}}
\put(9.4814,-2.3186){\line(-1,-1){2.3098}} \normalcolor

\color{Peach} \put(10,-1.8){\circle*{.4}}
%\put(10.2593,-2.0593){\circle*{.2}}
%\put(10.3667,-1.8){\circle*{.2}}
%\put(9.7407,-2.0593){\circle*{.2}} \put(10,1.5){\circle*{.2}}
%\put(8.9,-1.8){\circle*{.2}} \put(10.5186,-2.3186){\circle*{.2}}
%\put(10.7334,-1.8){\circle*{.2}}
%\put(9.4814,-2.3186){\circle*{.2}} \put(7.8,-1.8){\circle*{.2}}
%\put(10,-1.8){\line(1,0){.7334}} %\put(10,-1.8){\line(0,1){3.3}}
%\put(10,-1.8){\line(-1,0){2.2}}
\normalcolor

\color{Peach} \thicklines \put(7.25,.2){\vector(-1,0){4}}
\normalcolor \thinlines \put(3.85,1){\Large
\ensuremath{d[{_5\phi}}]}

\end{picture}

\vspace{-.8cm} \setlength{\unitlength}{.35cm} \hspace{11.45cm}
\begin{picture}(0,0)

\put(-3.8,-6.6){\Large\ensuremath{D(\nPhi[5]{m}{0},\X[5]{m})}}

%\color{LimeGreen}
%\put(-3.5,-1.1){\large\ensuremath{{_5A}_{0}^{4}}} \color{Orchid}
%\put(3.8,-1.1){\large\ensuremath{{_5A}_{0}^{2}}} \color{RawSienna}
%\put(.1,2.9){\large\ensuremath{{_5A}_{0}^{3}}} \color{Red}
%\put(-3.2284,-5.5284){\large\ensuremath{{_5A}_{0}^{5}}}
%\color{ProcessBlue}
%\put(3.3284,-5.5284){\large\ensuremath{{_5A}_{0}^{1}}}
\normalcolor \put(4.5,-1.8){\line(-1,0){8}}
\put(.5,-1.8){\line(0,1){4}} \put(.5,-1.8){\line(-1,-1){2.8284}}
\put(.5,-1.8){\line(1,-1){2.8284}}
\put(-.5,-1.3){\footnotesize\ensuremath{v_0}}
\put(.5,-1.8){\circle*{.4}}
%\put(.9,2.1){\small\ensuremath{v_3}} \put(.5,2.2){\circle*{.4}}
%\put(-3.7,-2.4){\small\ensuremath{v_6}}
%\put(-3.5,-1.8){\circle*{.4}}
%\put(4.3,-2.4){\small\ensuremath{v_2}}
%\put(4.5,-1.8){\circle*{.4}} \put(-3,-2.4){\small\ensuremath{v_4}}
%\put(-2.8,-1.8){\circle*{.4}}
%\put(-1.9284,-4.7284){\small\ensuremath{v_7}}
%\put(-2.3284,-4.6284){\circle*{.4}}
%\put(-1.4337,-4.2337){\small\ensuremath{v_5}}
%\put(-1.8337,-4.1337){\circle*{.4}}
%\put(2.6284,-4.7284){\small\ensuremath{v_1}}
%\put(3.3284,-4.6284){\circle*{.4}}

\color{Peach}
%\put(8.75,-1.55){\tiny\ensuremath{u_{28}}}
%\put(-.6,-1.8){\circle*{.2}} \put(.5,-.7){\circle*{.2}}
%\put(1.6,-1.8){\circle*{.2}} \put(-.3,-2.6){\circle*{.2}}
%\put(1.3,-2.6){\circle*{.2}} \put(.5,-1.8){\line(-1,0){1.1}}
%\put(.5,-1.8){\line(1,0){1.1}} \put(.5,-1.8){\line(0,1){1.1}}
%\put(.5,-1.8){\line(-1,-1){.8}} \put(.5,-1.8){\line(1,-1){.8}}
\put(.5,-1.8){\circle*{.4}} \normalcolor

\color{Peach} \thicklines \put(-10.05,.2){\vector(-1,0){4}}
\put(-2.05,.2){\vector(-1,0){4}} \put(7.05,.2){\vector(-1,0){4}}
\normalcolor \put(-9,0){\ensuremath{\cdots}}
\put(-8,0){\ensuremath{\cdots}} \put(7.8,0){\ensuremath{\cdots}}

\color{Peach} \put(-4.05,2){\vector(-1,0){8}} \normalcolor
\thinlines \put(-11.1,2.8){\Large
\ensuremath{d[\nphi[5],\nPhi[5]{m}{1}]}}

\end{picture}

\vspace{4cm} \setlength{\unitlength}{.35cm} \hspace{7.15cm}
\begin{picture}(0,0)

\put(-1.8,-7.6){\Large\ensuremath{D(\beta,T)}}

%\color{LimeGreen}
%\put(-3.5,-1.1){\large\ensuremath{{_4A}_{0}^{3}}} \color{Orchid}
%\put(3.8,-1.1){\large\ensuremath{{_4A}_{0}^{1}}} \color{RawSienna}
%\put(.1,2.9){\large\ensuremath{{_4A}_{0}^{2}}} \color{Red}
%\put(.1,-6.8){\large\ensuremath{{_4A}_{0}^{4}}} \normalcolor
\put(4.5,-1.8){\line(-1,0){8}} \put(.5,-1.8){\line(0,1){4}}
\put(.5,-1.8){\line(0,-1){4}} %\put(0,-1.4){\small\ensuremath{v_0}}
\put(.5,-1.8){\circle*{.4}} %\put(.9,2.1){\small\ensuremath{v_2}}
%\put(.5,2.2){\circle*{.4}} %\put(-3.7,-2.4){\small\ensuremath{v_5}}
%\put(-3.5,-1.8){\circle*{.4}}
%\put(4.3,-2.4){\small\ensuremath{v_1}}
%\put(4.5,-1.8){\circle*{.4}} %\put(-3,-2.4){\small\ensuremath{v_3}}
%\put(-2.8,-1.8){\circle*{.4}}
%\put(.9,-5.2){\small\ensuremath{v_4}}
%\put(.5,-5.1){\circle*{.4}}
%\put(.9,-5.9){\small\ensuremath{v_6}}
%\put(.5,-5.8){\circle*{.4}}

\color{Peach} \put(.5,-1.8){\circle*{.4}} \normalcolor

\thicklines \color{Peach} \put(-5,2.5){\vector(-1,1){3}}
\normalcolor \put(9,5.5){\vector(-1,-1){3}} \thinlines
\put(8.1,3.25){\Large \ensuremath{d[\beta,\alpha]}}
\put(-9.85,3){\Large \ensuremath{d[\beta]}} \put(-1.6,4.1){\Large
\ensuremath{\sigma}} \color{Peach} \put(-2,2.5){\vector(0,1){4}}
\normalcolor

\end{picture}
\end{figure}
\vspace{2.5cm}
\begin{figure}[here]
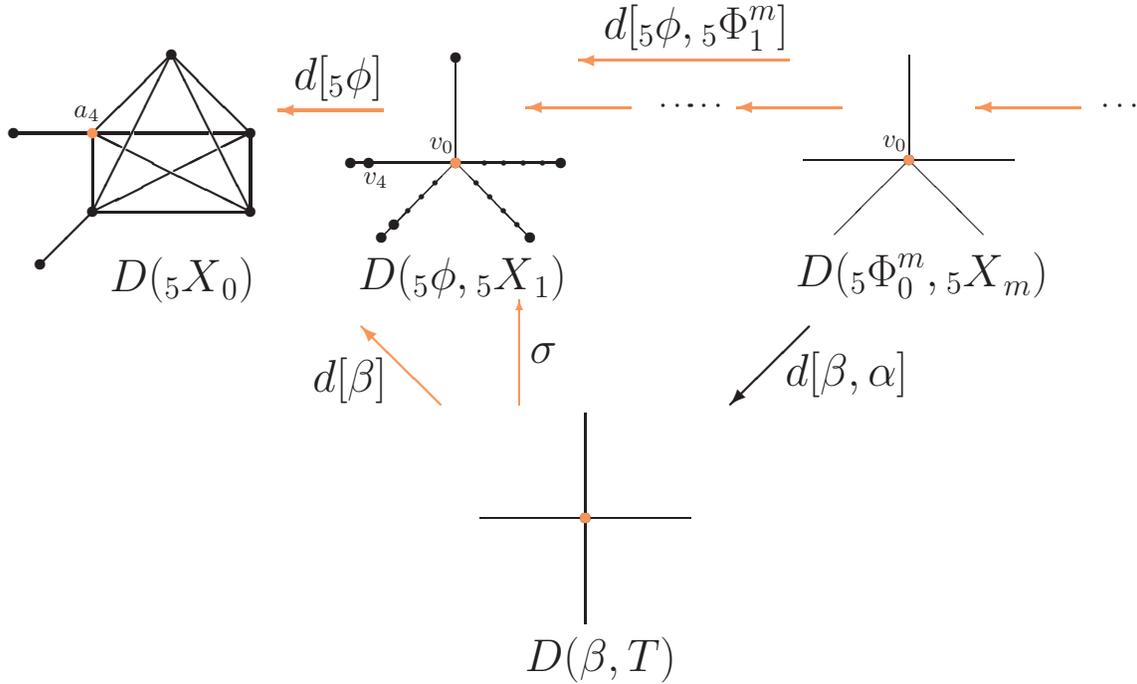

\caption{Factoring of ${_5\Gamma}_1^m$}
\end{figure}

\begin{proposition} \label{claim18}
$\K$ is not simple-$(n-1)$-od-like.
\end{proposition}

\begin{proof}
If $\K$ is simple-$(n-1)$-od-like, then the conclusion of Theorem
\ref{claim17} for $j=0$ contradicts Proposition \ref{claim16}.
\end{proof}

\begin{proposition} \label{claim19}
\K{} is indecomposable.
\end{proposition}

\begin{proof}
The claim follows from Proposition \ref{claim1}, Proposition
\ref{claim1.5}, Proposition \ref{claim18}, and Corollary 1 in
\cite{bing2}.
\end{proof}

\begin{theorem}
For each integer $n$ greater than or equal to 3, there exists a
simple-$n$-od-like continuum (\K) having the properties of not
being simple-$(n-1)$-od-like and of every proper nondegenerate
subcontinuum being an arc.
\end{theorem}

\begin{proof}
The claim follows from Propositions \ref{claim1} and
\ref{claim18}.
\end{proof}

%%%%%%%%%%%%%%%%%%%%%%%%%%%%%%%%%%%%%%%%%%%%%%%%%%%%%%%%%%%%%%%%%%%%%%%%%%%%%%%%%%%%%%%%%%%%%
\chapter{FURTHER QUESTIONS}

This dissertation provides, for each integer $n\ge 3$, a
simple-$n$-od-like continuum {\K} which is not
simple-$(n-1)$-od-like and whose every proper nondegenerate
subcontinuum is an arc.  Examples for the case $n=5$ and higher
were not proven previously.  The existence of such continua is
related to the problem of being able to distinguish among
tree-like continua those that are chainable.  A fundamental open
question in this area is the following (Question 5.1) due to L.
Mohler. One partial positive answer has been given by L.G.
Oversteegen in \cite{oversteegen} for continua satisfying the
additional conditions of being the continuous image of a chainable
continuum under an induced map and of having surjective semispan
equal to zero.  Another partial positive answer has been given by
P. Minc in \cite{minc3}, as another application of the
combinatorial machinery developed in \cite{minc1}, for continua
which are the inverse limits of trees with simplicial bonding
maps.

\begin{question} [Problem 16 \cite{lewis2}] Is every
atriodic tree-like continuum that is the continuous image of a
chainable continuum chainable?
\end{question}

Examining possible patterns for nested intersections of covers as
determined by the pattern of the bonding map \nphi{} (for $n\ge
4$), the examples presented here do not appear to be planar. The
question of just how ``simple'' examples with these properties
could be, corresponding to questions regarding the complexity of
subcontinua of the plane, is raised in the following.

\begin{question}  For each integer $n$ greater than 3, does there
exist a simple-$n$-od-like planar continuum having the properties
of not being simple-$(n-1)$-od-like and of every proper
nondegenerate subcontinuum being an arc (being atriodic)?
\end{question}

Do there exist examples with the above-mentioned properties,
replacing arc with pseudo-arc?  The combinatorial techniques
employed here, being dependent on a simplicial setting, could not
be directly utilized in showing that such a given example is not
simpler (Definition \ref{simpler2}).  The case for $n=3$ is known,
as mentioned in the Introduction, with an example in
\cite{ingram3}, constructed so as to be continuously mapped onto
the example in \cite{ingram2}. This implies nonchainabilty of the
example in \cite{ingram3} since the example in \cite{ingram2} is
not the continuous image of a chainable continuum.  This does not
naturally generalize to a method for showing that a hereditarily
indecomposable simple-$n$-od-like continuum is not
simple-$(n-1)$-od-like. Specifically, can the bonding map for \K{}
be modified, in a manner similar to that done to the bonding map
in \cite{ingram2} to give the example in \cite{ingram3}, to
produce a continuum which is not simple-$(n-1)$-od-like?  What
techniques could be used to recognize that such a continuum is not
simple-$(n-1)$-od-like?

\begin{question}  For each integer $n$ greater
than 3, does there exist a simple-$n$-od-like (planar) continuum
having the properties of not being simple-$(n-1)$-od-like and of
being hereditarily indecomposable (every proper nondegenerate
subcontinuum being a pseudo-arc)?
\end{question}

The following question concerns the existence of more
``complicated'' atriodic examples known to be like a certain graph
and not like any simpler graph (Definition \ref{simpler1}).  For
trees, in constructing an atriodic continuum like a given tree and
not like anything simpler, the examples \K{} could be a way to
control the order of the branching.

\begin{question}  For a graph (tree) $G$, does
there exist a $G$-like continuum having the properties of not
being $H$-like, for all graphs $H$ simpler (Definition
\ref{simpler1}) than $G$, and of being atriodic (being
hereditarily indecomposable)?
\end{question}

\begin{question} Does there exist a continuum $K$, with the
property of being atriodic (of every proper nondegenerate
subcontinuum being an arc), which is not simple-$n$-od-like for
each positive integer $n$ and which for each $\epsilon>0$ there
exist a positive integer $m$ and an open cover $\mathcal{U}$ of
$K$ so that $\mathrm{mesh}(\mathcal{U})<\epsilon$ and the nerve of
$\mathcal{U}$ is a simple-$m$-od?
\end{question}

If $v$ in a continuum $K$ is a branch point of $K$ of order $n$,
then $v$ is a branch point of $K$ of order $n+1$.  If $K$ is
simple-$(n-1)$-od-like, is $v$ a branch point of $K$ of order
$n-1$?  For \K, by construction, \branch{} is a branch point of
\K{} of order $n$. In supposing \K{} is simple-$(n-1)$-od-like, by
Proposition \ref{claim0}, \branch{} is ``close'' to being a branch
point of \K{} of order $n-1$.  Although sufficient in showing \K{}
not being simple-$(n-1)$-od-like, the argument could be made more
concise in the case of \branch{} necessarily being a branch point
of \K{} of order $n-1$.

\begin{question}  Does there exist a continuum $K$ with $v\in K$ so
that, for some integer $n$, $v$ is a branch point of $K$ of order
$n$, $v$ is not a branch point of $K$ of order $n-1$, and $K$ is
simple-$(n-1)$-od-like?
\end{question}

%%%%%%%%%%%%%%%%%%%%%%%%%%%%%%%%%%%%%%%%%%%%%%%%%%%%%%%%%%%%%%%%%%%%%%%%%%%%%%
\renewcommand{\baselinestretch}{1}\selectfont
\addcontentsline{toc}{chapter}{BIBLIOGRAPHY}
\bibliographystyle{plain}

\renewcommand{\baselinestretch}{1.5}\selectfont

%%%%%%%%%%%%%%%%%%%%%%%%%%%%%%%%%%%%%%%%%%%%%%%%%%%%%%%%%%%%%%%%%%%%%%%%%%%%%%%%%%%%%%%%%%%%%%%
\newpage

\addcontentsline{toc}{chapter}{APPENDIX:  SOME BONDING MAPS}
\begin{appendix}
\begin{center}APPENDIX \\ SOME BONDING MAPS\end{center}

\addcontentsline{lof}{figure}{A.1 \ \ Ingram bonding map}

\begin{figure}[here] \vspace{4.3cm}
\hspace{2.8cm} \setlength{\unitlength}{.85cm}
\begin{picture}(0,0)

\thicklines %\put(2.5,-6){\LARGE Ingram bonding map}

\color{LimeGreen} \put(-3.3,-3.8){\large \ensuremath{{A}_{0}^{2}}}
\put(-2.3,-5){\line(1,0){7.5}} \put(-2.3,-5){\line(0,1){2.5}}
\put(-2.3,-2.5){\line(1,0){6.25}}

\color{Red} \put(13,-3.8){\large \ensuremath{{A}_{0}^{3}}}
\put(5.2,-5){\line(1,0){7.5}} \put(12.7,-5){\line(0,1){2.5}}
\put(12.7,-2.5){\line(-1,0){6.25}}

\color{RawSienna} \put(4.95,4.05){\large \ensuremath{{A}_{0}^{1}}}
\put(3.95,-2.5){\line(0,1){6.25}}
\put(6.45,-2.5){\line(0,1){6.25}} \put(3.95,3.75){\line(1,0){2.5}}
\normalcolor

\thinlines

\color{LimeGreen} \put(12.1,-3.85){\scriptsize
\ensuremath{{A}_{1}^{2}}}

%\put(-1.675,-3.75){\line(1,0){13.775}}
\put(-1.675,-3.75){\line(1,0){6.5625}}
\put(4.8875,-3.75){\line(0,1){3.75}}
\put(4.8875,0){\line(1,0){.625}} \put(5.5125,0){\line(0,-1){3.75}}
\put(5.5125,-3.75){\line(1,0){6.5875}}

\color{RawSienna} \put(12.1,-3.225){\scriptsize
\ensuremath{{A}_{1}^{1}}}

%\put(4.702,-3.5){\line(-1,0){3.752}}
%\put(.95,-3.5){\line(0,1){.375}}
\put(4.575,-3.125){\line(-1,0){6.25}}
\put(-1.675,-3.125){\line(0,-1){.625}}
\put(4.575,3.125){\line(0,-1){6.25}}
\put(4.575,3.125){\line(1,0){1.25}}
\put(5.825,3.125){\line(0,-1){6.25}}
\put(5.825,-3.125){\line(1,0){6.225}}

\color{Red} \put(12.1,-4.475){\scriptsize
\ensuremath{{A}_{1}^{3}}}

%\put(4.702,-4){\line(-1,0){3.752}}
\put(-1.675,-4.375){\line(0,1){.625}}
%\put(.95,-4){\line(0,-1){.375}}
\put(-1.675,-4.375){\line(1,0){13.725}} \normalcolor

\put(-1.675,-3.75){\circle*{.168}}

\end{picture}
\end{figure}
\vspace{4cm}
\begin{figure}[here]
\begin{center} Figure A.1:  The bonding map from \cite{ingram2}
\end{center}
\end{figure}

The bonding map above is used by W.T. Ingram as the single bonding
map in the construction of an inverse limit, published in 1972,
being the first proven counterexample to a question of Bing from
1951 as to whether every atriodic nonseparating plane continuum is
chainable. Nonchainability of Ingram's continuum follows from his
proof of positive span of the continuum, and the nature of the
bonding map ensures every proper nondegenerate subcontinuum being
an arc, implying atriodicity.  As evident by examining the pattern
above, the continuum is embeddable in the plane.  Ingram's
continuum is the first continuum which, by construction, is
simple-3-od-like and shown not to be simple-2-od-like.

\newpage
\addcontentsline{lof}{figure}{A.2 \ \ Davis-Ingram bonding map}

\begin{figure}[here] \vspace{4.3cm}
\hspace{2.8cm} \setlength{\unitlength}{.85cm}
\begin{picture}(0,0)

\thicklines %\put(1.5,-6){\LARGE Davis-Ingram bonding map}

\color{LimeGreen} \put(-3.3,-3.8){\large \ensuremath{{A}_{0}^{2}}}
\put(-2.3,-5){\line(1,0){7.5}} \put(-2.3,-5){\line(0,1){2.5}}
\put(-2.3,-2.5){\line(1,0){6.25}}

\color{Red} \put(13,-3.8){\large \ensuremath{{A}_{0}^{3}}}
\put(5.2,-5){\line(1,0){7.5}} \put(12.7,-5){\line(0,1){2.5}}
\put(12.7,-2.5){\line(-1,0){6.25}}

\color{RawSienna} \put(4.95,4.05){\large \ensuremath{{A}_{0}^{1}}}
\put(3.95,-2.5){\line(0,1){6.25}}
\put(6.45,-2.5){\line(0,1){6.25}} \put(3.95,3.75){\line(1,0){2.5}}
\normalcolor

\thinlines

\color{RawSienna} \put(-2.25,-3.85){\scriptsize
\ensuremath{{A}_{1}^{1}}}

\put(.95,-3.75){\line(-1,0){2.8}}

\color{Red} \put(12.1,-3.225){\scriptsize
\ensuremath{{A}_{1}^{3}}}

%\put(4.702,-3.5){\line(-1,0){3.752}}
%\put(4.702,-3.5){\line(3,-1){.498}}
\put(.95,-3.75){\line(0,1){.625}}
\put(.95,-3.125){\line(1,0){3.625}}
\put(4.575,-3.125){\line(0,1){6.25}}
\put(4.575,3.125){\line(1,0){1.25}}
\put(5.825,3.125){\line(0,-1){6.25}}
\put(5.825,-3.125){\line(1,0){6.225}}

\color{LimeGreen} \put(12.1,-4.475){\scriptsize
\ensuremath{{A}_{1}^{2}}}

%\put(4.702,-4){\line(-1,0){3.752}}
%\put(4.702,-4){\line(3,1){.498}}
\put(.95,-3.75){\line(0,-1){.625}}
\put(.95,-4.375){\line(1,0){11.1}} \normalcolor

\put(.95,-3.75){\circle*{.168}}

\end{picture}
\end{figure}
\vspace{4cm}
\begin{figure}[here]
\begin{center} Figure A.2:  The bonding map from \cite{davis}
\end{center}
\end{figure}

The bonding map above is used by J.F. Davis and W.T. Ingram as the
single bonding map in the construction of an inverse limit,
published in 1988, with the properties of having positive span and
of having every proper nondegenerate subcontinuum as an arc in
common with the previous example, with similar techniques utilized
in demonstrating positive span.  The example is constructed to
satisfy an additional property.  The nature of the bonding map
allows for a continuous map to be induced from the Davis-Ingram
continuum to a chainable continuum, having only one nondegenerate
point inverse which is an arc.  The Davis-Ingram continuum is the
first known example, having the previous example's properties,
admitting a continuous monotone map to a chainable continuum.

\newpage
\addcontentsline{lof}{figure}{A.3 \ \ Minc bonding map}

\begin{figure}[here] \vspace{4.3cm}
\hspace{2.8cm} \setlength{\unitlength}{.85cm}
\begin{picture}(0,0)

\thicklines %\put(2.5,-12.729){\LARGE Minc bonding map}

\color{Red} \put(4.8,-11.929){\large \ensuremath{{A}_{0}^{4}}}

\put(3.95,-5){\line(0,-1){6.25}} \put(6.45,-5){\line(0,-1){6.25}}
\put(3.95,-11.25){\line(1,0){2.5}}

\color{LimeGreen} \put(-3.3,-3.8){\large \ensuremath{{A}_{0}^{3}}}

\put(-2.3,-5){\line(1,0){6.25}} \put(-2.3,-5){\line(0,1){2.5}}
\put(-2.3,-2.5){\line(1,0){6.25}}

\color{Orchid} \put(13,-3.8){\large \ensuremath{{A}_{0}^{1}}}

\put(12.7,-5){\line(-1,0){6.25}} \put(12.7,-5){\line(0,1){2.5}}
\put(12.7,-2.5){\line(-1,0){6.25}}

\color{RawSienna} \put(4.95,4.05){\large \ensuremath{{A}_{0}^{2}}}

\put(3.95,-2.5){\line(0,1){6.25}}
\put(6.45,-2.5){\line(0,1){6.25}} \put(3.95,3.75){\line(1,0){2.5}}
\normalcolor

\thinlines

\color{RawSienna} \put(-2.25,-3.934){\scriptsize
\ensuremath{{A}_{1}^{2}}}

\put(.95,-3.834){\line(-1,0){2.8}}

\color{Red} \put(4.975,-11.074){\scriptsize
\ensuremath{{A}_{1}^{4}}}

\put(.784,-3.666){\line(1,0){.166}}
\put(.784,-3.666){\line(0,1){.707}}
\put(.784,-2.959){\line(1,0){3.791}}

%\put(4.702,-3.5){\line(-1,0){3.752}}
%\put(4.702,-3.5){\line(3,-1){.498}}
%\put(.95,-3.5){\line(0,1){.375}}
%\put(.95,-3.125){\line(1,0){3.625}}
\put(4.575,-2.959){\line(0,1){6.25}}
\put(4.575,3.291){\line(1,0){.625}}
%\put(5.991,3.291){\line(0,-1){6.25}}
%\put(5.991,-2.959){\line(1,0){6.25}}
%\put(5.991,-4.541){\line(1,0){6.25}}
%\put(12.241,-4.541){\line(0,1){1.582}}
%\put(5.991,-4.541){\line(0,-1){6.084}}
\put(5.2,3.291){\line(0,-1){7.583}}
\put(5.2,-4.458){\line(0,-1){6.167}}

\color{LimeGreen} \put(4.375,-11.074){\scriptsize
\ensuremath{{A}_{1}^{3}}}

%\put(4.702,-3.5){\line(-1,0){3.752}}
\put(.95,-3.75){\line(0,1){.625}}
\put(.95,-3.125){\line(1,0){4.167}}
\put(12.075,-3.75){\line(0,1){.625}}
\put(12.075,-3.125){\line(-1,0){6.792}}
\put(12.075,-3.75){\line(-1,0){6.792}}
\put(5.117,-3.75){\line(-1,0){.542}}
\put(4.575,-3.75){\line(0,-1){.542}}

%\put(5.2,-3.834){\line(-1,0){4.416}}
%\put(.784,-3.666){\line(0,1){.707}}
%\put(.784,-2.959){\line(1,0){3.791}}
%\put(4.575,-4.541){\line(0,-1){6.084}}
\put(4.575,-4.458){\line(0,-1){6.167}}

\color{Orchid} \put(5.775,3.406){\scriptsize
\ensuremath{{A}_{1}^{1}}}

%\put(4.702,-4){\line(-1,0){3.752}}
%\put(4.702,-4){\line(3,1){.498}}
\put(.95,-3.75){\line(0,-1){.625}}
\put(.95,-4.375){\line(1,0){4.709}}
%\put(12.075,-4.375){\line(0,1){1.25}}
%\put(12.075,-3.125){\line(-1,0){6.25}}
%\put(5.825,-3.125){\line(0,1){6.25}}
%\put(5.825,3.125){\line(-1,0){.3125}}
%\put(5.5125,3.125){\line(0,-1){7.417}}
\put(5.659,-4.375){\line(0,-1){3.542}}
\put(5.659,-7.917){\line(1,0){.332}}

\put(5.991,3.291){\line(0,-1){6.25}}
\put(5.991,-2.959){\line(1,0){6.25}}
\put(5.991,-4.375){\line(1,0){6.25}}
\put(12.241,-4.375){\line(0,1){1.416}}
\put(5.991,-4.375){\line(0,-1){3.542}} \normalcolor

\put(.95,-3.75){\circle*{.168}}

\end{picture}
\end{figure}
\vspace{10cm}
\begin{figure}[here]
\begin{center} Figure A.3:  The bonding map from \cite{minc2}
\end{center}
\end{figure}

\vspace{-.3cm} The bonding map above is used by P. Minc as the
single bonding map in the construction of an inverse limit,
published in 1993, with the properties of not being
simple-3-od-like and of being atriodic.  As with the others, the
nature of the bonding map ensures every proper nondegenerate
subcontinuum being an arc.  In showing the continuum to be not
simple-3-od-like, alternate techniques to span are needed.  In so
doing, Minc adapts combinatorial techniques from \cite{minc1}.
Minc's continuum is the first continuum which, by construction, is
simple-4-od-like and shown not to be simple-3-od-like.

\end{appendix}

\end{document}